\documentclass[]{article}
\usepackage{amsmath,amsthm,amssymb,amsfonts,amscd} 
\usepackage{color}
\usepackage{comment}
\usepackage{paralist}
\usepackage{mathrsfs}
\usepackage{hyperref} 
\usepackage{setspace}
\usepackage{xcolor}
\usepackage{titlesec}
\titleformat{\subsection}{\normalfont\scshape\filcenter}{\thesubsection}{1em}{}
\titleformat{\subsubsection}[runin]{\normalfont\bfseries}{\thesubsubsection. }{.2em}{}[.]

\newcommand\thefontsize[1]{{#1 The current font size is: \f@size pt\par}}

\newtheorem{lem}{Lemma}[section]
\newtheorem{define}{Definition}[section]
\newtheorem{remark}{Remark}[section]
\newtheorem{cor}[lem]{Corollary}

\newtheorem{THM}{Theorem}
\newtheorem{Question}{Question \ignorespaces}

\newtheorem{prop}[lem]{Proposition}

\newcommand{\R}{{\mathbb R}}
\newcommand{\N}{{\mathbb{N}}}
\newcommand{\Z}{{\mathbb Z}}

\newcommand{\e}{\epsilon}

\newcommand{\pt}{\widetilde{P}}
\newcommand{\al}{\alpha}
\newcommand{\ma}{\mathcal{A}}
\newcommand{\mm}{\mathcal{M}}
\newcommand{\mab}{\overline{\mathcal{A}}}
\newcommand{\mpp}{\mathcal{P}}
\newcommand{\mb}{\mathcal{B}}
\newcommand{\lra}{\leftrightarrow}
\newcommand{\ka}{\kappa}
\newcommand{\D}{\Delta}
\newcommand{\ab}{\bar{\alpha}}
\newcommand{\ph}{\hat{P}}
\newcommand{\J}{\mathscr{J}}
\newcommand{\p}{\partial}
\newcommand{\fib}{\bar{\varphi}}
\newcommand{\bb}{\bar{\beta}}
\newcommand{\Se}{\mathscr{S}}
\newcommand{\tf}{T(f,\vec{P},\vec{S})}
\newcommand{\tfr}{T(f,\vec{R},\vec{R}^*)}
\newcommand{\goodchi}{\protect\raisebox{2pt}{$\chi$}}
\newcommand{\oh}{\hat{\omega}}
\newcommand{\ot}{\widetilde{\omega}}
\newcommand{\bpt}{\mathfrak{B}_{CZ}}
\newcommand{\bkpt}{\mathfrak{B}_{key}}

\DeclareMathOperator{\diam}{diam}
\DeclareMathOperator{\dist}{dist}
\DeclareMathOperator{\supp}{supp}
\DeclareMathOperator{\inte}{int}
\DeclareMathOperator{\ctr}{ctr}

\DeclareMathOperator{\Cl}{Cl}

\numberwithin{equation}{section}
 
\title{Approximate Extension in Sobolev Space}
\author{Marjorie K. Drake}
\date{} 

\begin{document}
\maketitle

\begin{abstract}
    Let $L^{m,p}(\mathbb{R}^n)$ be the homogeneous Sobolev space for $p \in (n,\infty)$, $\mu$ be a Borel regular measure on $\mathbb{R}^n$, and $L^{m,p}(\mathbb{R}^n) + L^p(d\mu)$ be the space of Borel measurable functions with finite seminorm $\|f\|_{L^{m,p}(\mathbb{R}^n) + L^p(d\mu)} := \\\inf_{f_1 +f_2 = f} \left\{  \|f_1\|_{L^{m,p}(\mathbb{R}^n)}^p + \int_{\mathbb{R}^n} |f_2|^p d\mu  \right\}^{1/p}$. We construct a linear operator $T:L^{m,p}(\mathbb{R}^n) + L^p(d\mu) \to L^{m,p}(\mathbb{R}^n)$, that nearly optimally decomposes every function in the sum space: $\|Tf\|_{L^{m,p}(\mathbb{R}^n)}^p + \int_{\mathbb{R}^n} |Tf-f|^p d\mu \leq C \|f\|_{L^{m,p}(\mathbb{R}^n) + L^p(d\mu)}^p$ with $C$ dependent on $m$, $n$, and $p$ only. For $E \subset \mathbb{R}^n$, let $L^{m,p}(E)$ denote the space of all restrictions to $E$ of functions $F \in L^{m,p}(\mathbb{R}^n)$, equipped with the standard trace seminorm. For $p \in (n, \infty)$, we construct a linear extension operator $T:L^{m,p}(E) \to L^{m,p}(\mathbb{R}^n)$ satisfying $Tf|_E = f|_E$ and $\|Tf\|_{L^{m,p}(\mathbb{R}^n)} \leq C \|f\|_{L^{m,p}(E)}$, where $C$ depends only on $n$, $m$, and $p$. We show these operators can be expressed through a collection of linear functionals whose supports have bounded overlap. 
\end{abstract}


\tableofcontents

\section{Introduction}
\label{sec:intro}

The homogeneous Sobolev space $L^{m,p}(\R^n)$, consists of all functions $F : \R^n \rightarrow \R$ whose distributional partial derivatives of order $m$ belong to $L^p(\R^n)$. We define the $L^{m,p}(\R^n)$ seminorm by
\begin{align*}
\|F\|_{L^{m,p}(\R^n)} = \max_{ |\al| = m}   \| \p^\al F \|_{L^p(\R^n)}  \qquad (F \in L^{m,p}(\R^n)).
\end{align*}
Given a Borel regular measure $\mu$ on $\R^n$, let $L^p(d \mu)$ be the space of Borel measurable functions $g : \R^n \rightarrow \R$ with finite norm $\| g \|_{L^p(d \mu)} := \left( \int_{\R^n} g(x) d \mu \right)^{1/p} < \infty$.  Let $L^{m,p}(\R^n) + L^p(d\mu)$ be the space of Borel measurable functions $f : \R^n \rightarrow \R$ with finite seminorm 
\[
\|f\|_{L^{m,p}(\R^n) + L^p(d\mu)} := \inf_{f_1 +f_2 = f} \big\{ \|f_1\|_{L^{m,p}(\R^n)}^p + \| f_2 \|_{L^p(d \mu)}^p \big\}^{1/p} < \infty.
\]
We consider the topic of nearly optimal decomposition of functions in this sum space. Precisely:
\begin{Question} 
Can we construct a linear operator $T:L^{m,p}(\R^n) + L^p(d\mu) \to L^{m,p}(\R^n)$ satisfying $ \|Tf\|_{L^{m,p}(\R^n)}+ \|f-Tf\|_{L^p(d\mu)} \leq C \|f\|_{L^{m,p}(\R^n) + L^p(d\mu)}$, where $C$ is independent of $f$? \label{q01}
\end{Question}
\begin{Question}
Can we estimate $\|f\|_{L^{m,p}(\R^n) + L^p(d\mu)}$? \label{q02}
\end{Question}
We answer these questions with our first theorem. Let $C^{m}_{loc}(\R^n)$ denote the space of all functions $F : \R^n \rightarrow \R$ with continuous derivatives up to order $m$. The function space $C^{m}(\R^n)$ consists of functions with continuous, bounded derivatives up to order $m$, with  norm: $\|F\|_{C^{m}(\R^n)} = \max_{ |\al| \leq m} \sup_{x \in \R^n} \left\{ | \p^\al F(x) | \right\}$. Let $\mpp$ denote the space of real-valued $(m-1)^{\text {rst}}$ degree polynomials on $\R^n$. For $K \subset \R^n$, a \emph{Whitney field} on $K$ is a tuple of polynomials $(P_x)_{x \in K}$ with $P_x \in \mpp$ for all $x \in K$. The space of Whitney fields on $K$ is denoted by
\[
Wh(K): = \big\{ \vec{P}: \vec{P} = (P_x)_{x \in K}, \; P_x \in \mpp \text{ for all } x \in K \big\}.
\]
If $F$ is a $C^{m-1}$-function defined on a neighborhood of a point $y \in \R^n$, then we write $J_y(F)$ (the ``jet" of $F$ at $y$ ) for the $(m-1)^{\text {rst}}$ degree Taylor polynomial of $F$ at $y$.

If $n < p < \infty$, then the Sobolev embedding theorem implies that $L^{m,p}(\R^n) \subset C_{loc}^{m-1}(\R^n)$. We define a semi-norm on Whitney fields, $\vec{S} \in Wh(K)$, 
\begin{align}
\|\vec{S}\|_{L^{m,p}(K)} = \inf \{ \|F\|_{L^{m,p}(\R^n)}: \; F \in L^{m,p}(\R^n), \; J_x(F) = S_x \text{ for all } x \in K \}.
\end{align}
By definition, if there does not exist $F \in L^{m,p}(\R^n)$ with $J_x (F) = S_x$ for all $x \in K$, then $\| \vec{S} \|_{L^{m,p}(K)} = + \infty$. 

Below, we write $\Cl(X)$ to denote the closure of a subset $X \subset \mathbb{R}^n$, and we write $\supp(\mu) \subset \R^n$ to denote the support of a measure $\mu$.

\begin{THM} Let $\mu$ be a compactly supported Borel regular measure on $\R^n$, $m \in \N$, and $n<p<\infty$. Then there exists a linear operator $T: L^{m,p}(\R^n) + L^p(d\mu) \to L^{m,p}(\R^n)$ and a map $M: L^{m,p}(\R^n) + L^p(d\mu) \to \R$ satisfying for all $f \in L^{m,p}(\R^n) + L^p(d\mu)$:
\begin{align}
& \|f\|_{L^{m,p}(\R^n) + L^p(d\mu)} \leq \|Tf\|_{L^{m,p}(\R^n)} + \|Tf-f\|_{L^p(d\mu)}  \leq C \cdot  \|f\|_{L^{m,p}(\R^n) + L^p(d\mu)}; \nonumber \\
& c \cdot Mf \leq \|Tf\|_{L^{m,p}(\R^n)}+\|Tf-f\|_{L^p(d\mu)} \leq C \cdot Mf; \text{ and}  \nonumber \\
&Mf = \Big(\|\vec{S}(f)\|_{L^{m,p}(K)}^p+\sum_{\ell \in \N} \big( \int_{A_\ell} |\zeta_\ell(f)-f |^p d\mu + |\psi_{\ell}(f) |^p \big) \Big)^{1/p}, \label{l3}
\end{align}
where $K \subset \Cl(\supp(\mu))$, $\vec{S}:L^{m,p}(\R^n) + L^p(d\mu) \to Wh(K)$ is a linear map, and for each $\ell \in \N$,  $A_\ell \subset \supp (\mu)$ is a Borel set, $\zeta_{\ell}: L^{m,p}(\R^n) + L^p(d\mu) \to L^p(d\mu)$ and $\psi_{\ell}:L^{m,p}(\R^n) + L^p(d\mu) \to \R$ are linear maps. The constants $c$ and $C$ depend on $m$, $n$, and $p$ but are independent of $f$ and $\mu$.
\label{amainthm}
\end{THM}
We introduce the notion of \emph{constructibility} in Theorems \ref{constructthm} and \ref{finitethm} to further describe the structure of the operator $T$ and map $M$. 

As an application of Theorem \ref{amainthm}, we construct a linear extension operator from the trace space $L^{m,p}(E)$ ($E\subset \R^n$ arbitrary) to $L^{m,p}(\R^n)$ for $p \in (n,\infty)$. Let $(\mathbb{X}(\R^n),\|\cdot\|_{\mathbb{X}(\R^n)})$ be a complete semi-normed linear space of continuous functions. For $E \subset \R^n$, let $\mathbb{X}(E)$ be the space of restrictions to $E$ of functions in $\mathbb{X}(\R^n)$, equipped with the trace semi-norm:
\begin{align*}
    &\mathbb{X}(E) := \{f:E \to \R: \exists F \in \mathbb{X}(\R^n), \; F|_E = f \}, \text{ with} \\
    &\|f\|_{\mathbb{X}(E)} := \inf \{\|F\|_{\mathbb{X}(\R^n)}: F \in \mathbb{X}(\R^n) \text{ and } F|_E=f \}. 
\end{align*}
A function $F : \R^n \rightarrow \R$ satisfying $F|_E = f$ is an \emph{extension} of $f$. A linear map $T: \mathbb{X}(E) \to \mathbb{X}(\R^n)$ satisfying $Tf|_E = f$ and $\|Tf\|_{\mathbb{X}(\R^n)} \leq C \|f\|_{\mathbb{X}(E)}$, where $C$ is independent of $f$, is a \emph{bounded linear extension operator}. We pose the following questions about the trace and extension problems in $\mathbb{X}(\R^n)$: 
\begin{Question}
Given $E \subset \R^n$, does there exist a bounded linear extension operator $T: \mathbb{X}(E) \to \mathbb{X}(\R^n)$ satisfying $Tf|_E = f|_E$ and $\|Tf\|_{\mathbb{X}(\R^n)} \leq C \|f\|_{\mathbb{X}(E)}$ for all $f:E \to \R$, where $C$ is independent of $E$ and $f$?
\label{niq1}
\end{Question}
\begin{Question}
Can we estimate $\|f\|_{\mathbb{X}(E)}$? \label{niq2}
\end{Question}
For $\mathbb{X}(\R^n)=L^{m,p}(\R^n)$, $n<p<\infty$, and $E$ arbitrary, C. Fefferman, A. Israel, and G.K. Luli prove the existence of a bounded linear extension operator, as in Question 3, in \cite{arie3} and \cite{arie4}. When  $E$ is finite, they introduce the concept of \emph{assisted bounded depth} to describe the structure of the operator $T$ and they give an approximate formula for the trace norm $\| f \|_{\mathbb{X}(E)}$ in this case. When $E$ is arbitrary, they prove the existence of a bounded linear extension operator $T$ by taking the Banach limit of operators extending $f$ from a sequence of finite subsets of $E$. Consequently, their extension operator $T$ loses all of its structural properties when $E$ is arbitrary. To prove Theorem \ref{introtracethm} below, we provide a direct construction of a bounded linear extension operator $T$, valid when $E$ is arbitrary.  In Theorem \ref{arbethm}, we describe the structure of the extension operator and an approximate formula for the trace norm through the notion of constructibility, which is a generalization of the notion of assisted bounded depth to the setting when $E$ is arbitrary.

The question of optimal decomposition in the sum space $L^{m,p}(\R^n) + L^p(d\mu)$ can be phrased as the problem of \emph{approximate extension and interpolation} in Sobolev spaces, a topic of independent interest. Suppose an experimenter collects data defining a function $f: E \to \R$ ($E$ finite). Rather than assume $f$ is the restriction to the set $E$ of a function in the space $L^{m,p}(\R^n)$, we assume $f$ lies \emph{near} a function in the space $L^{m,p}(\R^n)$. Let the function $\mu: E \to [0,\infty]$ represent the confidence in data collected; where the experimenter has high confidence in the data, we expect $\mu$ to be very large, and where the experimenter lacks confidence, we expect $\mu$ to approach $0$. Then we wish to estimate
\begin{align}
\inf_{F \in L^{m,p}(\R^n)} \big\{ \|F\|_{L^{m,p}(\R^n)}^p + \sum_{x \in E} |F(x)-f(x)|^p \mu(x) \big\} \label{insum}
\end{align}
(subject to the convention that if $|F(x) - f(x)| = 0$ and $\mu(x) = \infty$, then $|F(x)-f(x)|^p \mu(x) = 0 \cdot \infty =0$), and to construct a function $Tf \in L^{m,p}(\R^n)$ satisfying,
\begin{align*}
\|Tf\|_{L^{m,p}(\R^n)}^p + \sum_{x \in E} |Tf(x)-f(x)|^p \mu(x) \leq C \cdot \inf_{F \in L^{m,p}(\R^n)} \big\{ \|F\|_{L^{m,p}(\R^n)}^p + \sum_{x \in E} |F(x)-f(x)|^p \mu(x) \big\}.
\end{align*}
In this setting, we can interpret Theorem \ref{amainthm} as giving a construction for a near-optimal approximate extension $Tf$ of the function $f$ subject to the confidence, represented by $\mu$, and an approximate formula for the optimal value of (\ref{insum}). When $\mu: E \to \R$ is defined as $\mu(x) \equiv \infty$ for all $x \in E$, the expression (\ref{insum}) is finite only if $F|_E = f$ (i.e., $F$ is an exact interpolant of $f$ on $E$). This problem then encompasses the problem of interpolation in Sobolev spaces, which was studied by Israel, Luli, and Fefferman in \cite{arie6}, \cite{arie7}, and \cite{arie8}. 

We use Theorem \ref{amainthm} to construct a bounded linear extension operator for $L^{m,p}(\R^n)$ ($n<p<\infty$): Let $E \subset \R^n$ be a bounded Borel set, and $f: E \to \R$ be Borel measurable.  Define a Borel measure $\mu_E$ on $\R^n$ so that, for all Borel sets $A \subset \R^n$,
\begin{equation}\label{muE:defn}
    \mu_E(A) = \left\{
    \begin{aligned}
    &\infty \quad \mbox{if} \;  A \cap E \neq \emptyset, \\
    &0  \quad \; \mbox{if} \;  A \cap E = \emptyset. \\
    \end{aligned}
    \right.
\end{equation}
We apply Theorem \ref{amainthm} to the measure $\mu_E$ to produce a linear operator $T$ and map $M$. We make a few observations. We shall identify $f : E \rightarrow \R$ with a function on $\R^n$, using its extension by zero. First, note that $f \in L^{m,p}(E)$ if and only if $f \in L^{m,p}(\R^n) + L^p(d\mu_E)$; furthermore,  $\|f\|_{L^{m,p}(E)} = \|f\|_{L^{m,p}(\R^n) + L^p(d\mu_E)}$. We also have that $\|F\|_{L^{m,p}(\R^n)} + \| F - f \|_{L^p(d \mu_E)}$ is finite if and only if $F \in L^{m,p}(\R^n)$ and $F= f$ on $E$, and in this case we have $\| F - f \|_{L^p(d \mu_E)} = 0$. Therefore, the map $T:L^{m,p}(\R^n) + L^p(d\mu_E) \to L^{m,p}(\R^n)$, given in Theorem \ref{amainthm}, is a bounded linear extension operator on the trace space $L^{m,p}(E)$. Further, because $\mu_E(\{x\})= \infty$ for $x \in E$, $Mf$ is finite if and only if $\zeta_\ell(f) -f \equiv 0 \in L^p(d\mu_E)$ for all $\ell \in \N$. Consequently, by applying Theorem \ref{amainthm}, we obtain the following result:
\begin{THM}
\label{introtracethm}
Let $E \subset \R^n$ be a compact set, $m \in \N$, and $n<p<\infty$. There exists a bounded linear extension operator $T:L^{m,p}(E) \to L^{m,p}(\R^n)$ and a map $M: L^{m,p}(E) \to \R$ satisfying for all $f \in L^{m,p}(E)$, 
\begin{align*}
& Tf = f \mbox{ on } E\\
& \|f\|_{L^{m,p}(E)} \leq \|Tf\|_{L^{m,p}(\R^n)}  \leq C \cdot  \|f\|_{L^{m,p}(E)}; \\
& c \cdot Mf \leq \|Tf\|_{L^{m,p}(\R^n)} \leq C \cdot Mf; \text{ and}  \nonumber \\
&Mf = \big(\sum_{\ell \in \N} |\psi_{\ell}(f) |^p + \|\vec{S}(f)\|_{L^{m,p}(K)}^p\big)^{1/p}, 
\end{align*}
where $K\subset E$, $\vec{S}:L^{m,p}(E) \to Wh(K)$ is a linear map, and for each $\ell \in \N$, $\psi_\ell: L^{m,p}(E) \to \R$ is a bounded linear functional.  The constants $c$ and $C$ depend on $m$, $n$, and $p$ but are independent of $f$ and $E$.
\end{THM}
Given a real normed linear space $\mathbb{X}$, we denote the dual space of $\mathbb{X}$ by $\mathbb{X}^* := \{f:\mathbb{X}\to \R : \; f $ is a bounded linear functional $\}$.

When $E$ is finite, Fefferman, Israel, and Luli \cite{arie3, arie4} introduce the notion of assisted bounded depth to describe how the value of their extension and its derivatives up to order $m-1$ rely on a collection of linear functionals contained in $L^{m,p}(E)^*$. Below, in Theorems \ref{constructthm} and \ref{finitethm}, we express the operator $T$ and the map $M$ through a collection of functionals in $(L^{m,p}(\R^n) + L^p(d\mu) )^*$ whose supports have bounded overlap, generalizing the notion of assisted bounded depth to the setting of the optimal decomposition problem (Questions \ref{q01} and \ref{q02}). In Theorem \ref{arbethm}, we translate this to the setting of the extension problem (Questions \ref{niq1} and \ref{niq2}), where the operator $T$ and the map $M$ can be expressed through a collection of functionals in $L^{m,p}(E)^*$ whose supports have bounded overlap.

Recall the support of a functional $\omega \in (L^{m,p}(\R^n) + L^p(d\mu))^*$ is defined as
\begin{align*}
    \supp&(\omega):= \left( \bigcup \left\{ E \subset \R^n: \begin{aligned}
    &E \text{ is open, and for all } f \in L^{m,p}(\R^n) + L^p(d\mu) \\ &\text{satisfying } \supp(f) \subset E, \; \omega(f)=0. 
    \end{aligned}  \right\} \right)^c.
\end{align*}

We write $|S|$ for the cardinality of a set $S$. A collection $\Pi$ of sets has \emph{$A$-bounded overlap} ($A \geq 1$) provided that $|\{\pi \in \Pi: x \in \pi \} |\leq A$ for all $x$. With this notation, we now state the refined version of Theorem \ref{amainthm}:

\begin{THM}
The linear operator $T$ and map $M$ in Theorem \ref{amainthm} can be chosen to be \underline{${\Omega}$-constructible}, in the following sense: 

There exists a collection of linear functionals $\Omega = \{\omega_r\}_{r \in \Upsilon} \subset (L^{m,p}(\R^n) + L^p(d\mu))^*$, such that the collection of sets $\{\supp(\omega_r)\}_{r \in \Upsilon}$ has $C'$-bounded overlap, and for each $x \in \R^n$, there exists a finite subset $\Upsilon_x \subset \Upsilon$ and a collection of polynomials $\{v_{r,x}\}_{r \in \Upsilon_x} \subset \mpp$ such that $|\Upsilon_x| \leq C$ and
\begin{align}
    J_x Tf = \sum_{r \in \Upsilon_x}  \omega_r (f) \cdot v_{r,x}. \label{l4}
\end{align}

Recall that the map $M$ is defined in \eqref{l3} in terms of linear maps $(\zeta_\ell)_{\ell \in \mathbb{N}}$, $(\psi_\ell)_{\ell \in \mathbb{N}}$, and $\vec{S} = (S_x)_{x \in K}$. Then the following holds:

\textbf{(1)} For each $\ell \in \N$ and $y \in \supp(\mu)$, there exists a finite subset $\bar{\Upsilon}_{\ell,y} \subset {\Upsilon}$ and constants $\{\eta^{\ell}_{s,y}\}_{s \in \bar{\Upsilon}_{\ell,y}} \subset \R$ such that $|\bar{\Upsilon}_{\ell,y}| \leq C$, and the map $f \mapsto \zeta_\ell(f)(y)$ has the form
\begin{align*}
&\zeta_\ell(f)(y) = \sum_{s \in \bar{\Upsilon}_{\ell,y}} \eta^{\ell}_{s,y} \cdot \omega_{s}(f). 
\end{align*}

\textbf{(2)} For each $\ell \in \N$, there exists a finite subset $\bar{\Upsilon}_{\ell} \subset {\Upsilon}$ and constants $\{\eta^\ell_{s}\}_{s \in \bar{\Upsilon}_{\ell}} \subset \R$ such that $|\bar{\Upsilon}_{\ell}| \leq C$, and the map $\psi_\ell$ has the form 
\begin{align*}
&\psi_\ell(f) = \sum_{s \in \bar{\Upsilon}_{\ell}} \eta_{s}^\ell \cdot \omega_{s}(f).
\end{align*}

\textbf{(3)} For $y \in K$, there exist $\{\omega^\al_{y}\}_{\al \in \mm} \subset \Omega$ satisfying for all $\al \in \mm$, $\supp(\omega^\al_{y}) \subset \{y\}$, and the map $f \mapsto S_y(f)$ has the form
\begin{align*}
S_y(f) = \sum_{\al \in \mm} \omega^\al_{y}(f) \cdot v_{\al}, 
\end{align*}
where $\{v_\al\}_{\al \in \mm}$ is a basis for $\mpp$.

The constants $C$ and $C'$ depend on $m$, $n$, and $p$ but are independent of $f$ and $\mu$.
\label{constructthm}
\end{THM}

When $\mu$ is a finite measure (i.e., $\mu(\R^n) < \infty$) the formulas for $M$ and $T$ are simpler. Precisely,
\begin{THM} Let $\mu$ be a \underline{finite} Borel regular measure on $\R^n$ with compact support. Then the map $M: L^{m,p}(\R^n) + L^p(d\mu) \to \R$ in Theorem \ref{amainthm} satisfies: 
\begin{align*}
Mf = \big(\sum_{\ell = 1}^{K} \left( \int_{A_\ell} |\zeta_\ell(f)-f|^p d\mu + |\psi_{\ell}(f) |^p + |\psi_{\ell}(f) |^p \right) \big)^{1/p}, 
\end{align*}
where for each $\ell$, $A_\ell \subset \supp(\mu)$ is a Borel set, $\zeta_\ell: L^{m,p}(\R^n) + L^p(d\mu) \to L^p(d\mu)$, and $\psi_\ell: L^{m,p}(\R^n) + L^p(d\mu) \to \R$ are linear maps. 
Further the collection $\Omega \subset (L^{m,p}(\R^n) + L^p(d\mu))^*$ used to express $M$ and $T$ in Theorem \ref{constructthm} consists of a finite collection of linear functionals.\label{finitethm}
\end{THM}
As promised, we next adapt the notion of constructibility to describe the structure of the linear extension operator in Theorem \ref{introtracethm}.
\begin{THM}
\label{arbethm}
The linear operator $T$ in Theorem \ref{introtracethm} is $\Omega$-constructible. Precisely, there exists a collection of linear functionals $\Omega = \{\omega_r\}_{r \in \Upsilon} \subset L^{m,p}(E)^*$, such that the collection of sets $\{\supp(\omega_r)\}_{r \in \Upsilon}$ has $C'$-bounded overlap, and for each $x \in \R^n$, there exists a finite subset $\Upsilon_x \subset \Upsilon$ and a collection of polynomials $\{v_{r,x}\}_{r \in \Upsilon_x} \subset \mpp$ such that $|\Upsilon_x| \leq C$ and
\begin{align}
    J_x Tf = \sum_{r \in \Upsilon_x}  \omega_r (f) \cdot v_{r,x}. \label{forabdgen}
\end{align}
And the map $M$ in Theorem \ref{introtracethm} is $\Omega$-constructible:

\textbf{(1)} For each $\ell \in \N$, there exists a finite subset $\bar{\Upsilon}_{\ell} \subset {\Upsilon}$ and constants $\{\eta^\ell_{s}\}_{s \in \bar{\Upsilon}_{\ell}} \subset \R$ such that $|\bar{\Upsilon}_{\ell}| \leq C$, and the map $\psi_\ell$ has the form 
\begin{align*}
&\psi_\ell(f) = \sum_{s \in \bar{\Upsilon}_{\ell}} \eta_{s}^\ell \cdot \omega_{s}(f).
\end{align*}

\textbf{(2)} For $y \in K$, there exist $\{\omega^\al_{y}\}_{\al \in \mm} \subset \Omega$ satisfying for all $\al \in \mm$, $\supp(\omega^\al_{y}) \subset \{y\}$, and the map $f \mapsto S_y(f)$ has the form
\begin{align*}
S_y(f) = \sum_{\al \in \mm} \omega^\al_{y}(f) \cdot v_{\al}, 
\end{align*}
where $\{v_\al\}_{\al \in \mm}$ is a basis for $\mpp$. 

The constants $C$ and $C'$ depend on $m$, $n$, and $p$ but are independent of $f$ and $E$.
\end{THM}

Let $E \subset \R^n$ be finite. A linear map which is $\Omega$-constructible as defined in (\ref{forabdgen}) has $\Omega$-assisted bounded depth (see \cite{arie3} for the definition of assisted bounded depth operators). Consequently, the extension operator we construct in Theorem \ref{arbethm} has the same good structural property as the extension operator in Theorem 3 of \cite{arie3} when $E$ is finite. 

A pair of complete semi-normed spaces $(A,B)$ is said to be compatible if there exists a Hausdorff topological vector space $\mathscr{H}$ into which both $A$ and $B$ can be continuously embedded (see \cite{brbr}, p. 153). For a compatible couple $(A,B)$, the sum $A + B$ is the set of elements $f \in \mathscr{H}$ that can be represented as the sum, $f = f_1 + f_2$, of elements $f_1 \in A$ and $f_2 \in B$. Under the norm $\| f \|_{A+B} = \inf_{f_1+f_2 = f} \{ \|f_1\|_A + \|f_2\|_B \}$, $A+B$ is a complete semi-normed linear space. We can generate interpolation spaces for the couple $(A,B)$ via the real method, by calculating the $K-$functional:
\begin{align*}
K\big(t; f: (A,B) \big) := \inf \left\{ \|f_1\|_{A} + t \cdot \|f_2\|_{B}: f = f_1 + f_2 \right\}.
\end{align*}

The Banach couple, $(A,B)$, is \emph{$C,K$-linearized} if there exists a constant $C$ independent of $t$ such that for all $t>0$, for all $f \in A+B$, there exists an almost optimal decomposition $f = F_1^t + F_2^t$, where $F_1^t, F_2^t$ depend linearly on $f$:
\begin{align*}
\|F_1^t\|_{A}+ t\|F_2^t\|_{B} \leq C \cdot \|f \|_{K(t; f: (A,B) )}.
\end{align*}
The pair $(L^{m,p}(\R^n), L^p(d\mu))$ is a compatible couple, and 
\[
\|f\|_{L^{m,p}(\R^n) + L^p(d\mu)} \leq K\Big(1; f: (L^{m,p}(\R^n), L^p(d\mu)) \Big) \leq C \|f\|_{L^{m,p}(\R^n) + L^p(d\mu)} ,
\]
where $C=C(n,p)$.

Furthermore, 
\[
\|f\|_{L^{m,p}(\R^n) + L^p(t^pd\mu)} \leq K\Big(t; f: (L^{m,p}(\R^n), L^p(d\mu)) \Big) \leq C \|f\|_{L^{m,p}(\R^n) + L^p(t^pd\mu)},
\]
where $C=C(n,p)$.

Theorem \ref{amainthm} immediately implies: \\

\begin{THM} Let $\mu$ be a compactly supported Borel regular measure on $\R^n$. Then the Banach couple $(L^{m,p}(\R^n), L^p(d\mu))$ is $C,K$-linearized, with $C$ independent of $\mu$.\\ \label{bancouplethm}
\end{THM}

\subsection{Background}

P. Shvartsman considered Questions \ref{q01} and \ref{q02}, providing a solution when $\mu$ is a $\sigma$-finite Borel measure, $m=1$, and $ p \in (n,\infty)$ in \cite{shv4}. 

H. Whitney gave an answer to Questions \ref{niq1} and \ref{niq2} for the function space $\mathbb{X}(\R^n)=C^m(\R^n)$ in the case $n=1$ (\cite{wh1}). For $\omega \in (0,1]$, the space $C^{m, \omega}(\R^n)$ consists of $C^m(\R^n)$ functions with $\omega$-H\"older continuous $m^{th}$ order derivatives and finite norm: 
\begin{align*}
\|F\|_{C^{m,\omega}(\R^n)} = \|F\|_{C^{m}(\R^n)}+ \max_{ |\al| = m} \sup_{x, y \in \R^n, x \neq y} \left\{ \frac{| \p^\al F(x) - \p^\al F(y) |}{|x-y|^\omega} \right\}.
\end{align*}
Fefferman solved Questions \ref{niq1} and \ref{niq2} for $\mathbb{X}(\R^n)=C^m(\R^n)$ and $\mathbb{X}(\R^n)=C^{m,\omega}(\R^n)$ for $m,n \in \N$, and $\omega \in (0,1]$ in \cite{f1}, \cite{f2}, \cite{f3}, and \cite{f4}. His work built on theory developed by G. Glaeser, Y. Brudnyi, Shvartsman, E. Bierstone, P. Milman, and W. Paw\l{}ucki in \cite{gl}, \cite{brshv0}, \cite{brshv}, and \cite{bmp1}. Fefferman considered questions of approximate extension and utilized this relaxation of the extension problem in the spaces $C^{m}(\R^n)$ and $C^{m,\omega}(\R^n)$ in \cite{f6}-\cite{f2} and \cite{f5}, and in his work with B. Klartag in \cite{f7} and \cite{f8}.

For $\mathbb{X}(\R^n)=L^{m,p}(\R^n)$, Shvartsman used the classical Whitney extension operator to answer Questions \ref{niq1} and \ref{niq2} for $ p \in (n,\infty)$ and $m=1$ in \cite{shv2}. Israel \cite{arie5} and Shvartsman \cite{shv5} independently answered Questions \ref{niq1} and \ref{niq2} when $m=2$, $n<p<\infty$, and $E$ is finite.
Fefferman, Israel, and Luli extended the method of \cite{arie5} to $m \in \N$, $n<p<\infty$ in \cite{arie3} and \cite{arie4}. However, they describe structural properties of the extension operator $T$ only for finite $E \subset \R^n$.

\subsection{Overview of the Proof of Theorem 1}
Let $\mu$ be a Borel regular measure satisfying that  $\supp(\mu) \subset Q$ for a cube $Q \subset \R^n$. In order to construct a linear operator approximately extending a function from the support of $\mu$, we study the freedom we have to define a Sobolev function at each point of $Q$. In light of the Sobolev embedding theorem on $L^{m,p}(\R^n)$ for $p>n$, we can study the set of prospective $(m-1)$-jets of approximate extensions and utilize the inductive framework introduced by Fefferman in \cite{f1}. We define $\J$-functionals to encode how well a Sobolev function approximates a function $f:L^{m,p}(\R^n) + L^p(d\mu) \to \R$ in terms of the measure $\mu$. When the measure $\mu$ is finite, each step of the induction utilizes a \emph{finite} Calder\'on Zygmund (CZ) decomposition of $Q$ that identifies cubes where we can use the inductive hypothesis to solve the approximate extension problem locally. Then we can use the techniques of Israel in \cite{arie5} and Israel, Luli, and Fefferman in \cite{arie3} to ensure global compatibility of these local solutions and patch them together.

When the measure $\mu$ is not finite, the CZ decomposition need not be finite, implying there is a subset of $Q$ we call \emph{keystone points} where we cannot use an inductive hypothesis to produce a local solution. How can we define an approximate extension on this set so that we produce a Sobolev function on $Q$? This is a new problem for extension in Sobolev space, though Whitney navigates the boundary between an infinite decomposition and its complement in his $C^m$ extension theorem in \cite{wh1}. In Section \ref{subsec:kpt}, we identify the set of keystone points $K_p$, and consider its properties. In Section \ref{subsec:kpt_jet}, we show how the measure $\mu$ restricts a prospective approximate extension on $K_p$, relying on a new estimate (Lemma \ref{newestlm}). Then we use a Sobolev function to define an approximate extension on $K_p$ that is compatible with the extension defined on the CZ cubes. In Section \ref{sec:patchQ}, we show that this function defined piecewise on $K_p$ and the CZ cubes is, in fact, a function in $L^{m,p}(Q)$ through a characterization of Sobolev space due to Brudnyi in \cite{br1}. This follows the work of Shvartsman in \cite{shv2}, who used \cite{br1} to give an intrinsic characterization of the trace space $L^{1,p}(E)$ for $E \subset \R^n$. In Sections \ref{sec:decomp_func}-\ref{sec:proof_ml}, we establish that the constructed approximate extension is optimal using properties of the map $M: L^{m,p}(\R^n) + L^p(d\mu) \to \R$ in Theorem \ref{amainthm}.

\section{Notation and Further Theory}
\label{sec:notation}

\subsection{Notation} 
Fix integers $m, n \geq 1$ and a real number $p>n$. Unless we say otherwise, constants written $c, c', C, C'$, etc.\ depend only on $m, n$, and $p$. They are called ``universal" constants. The lower case letters denote small (universal) constants while the upper case letters denote large (universal) constants. Given a parameter $\xi$, we write $c(\xi)$, $C(\xi)$, etc., to denote constants depending only on $m,n,p,$ and $\xi$.

For non-negative quantities $A, B,$ we write $A \simeq B$, $A \lesssim B,$ or $A \gtrsim B$ to indicate that $c B \leq A \leq C B$, $A \leq C B$, or $A \geq c B$, respectively, for universal constants $0 < c < C$. Given a parameter $\xi$, we write $A \simeq_\xi B$, $A \lesssim_\xi B$, or $A \gtrsim_\xi B$ to indicate that $c(\xi) B \leq A \leq C(\xi) B$, $A \leq C(\xi) B$, or $A \geq c(\xi) B$, respectively, for constants $0 < c(\xi) < C(\xi)$.

When $\varphi_y$ is an indexed family of functions, an expression of the form $\partial^\alpha \varphi_y(y)$ will always mean $\partial^\alpha_z \varphi_y(z)|_{z=y}$, and never $\partial^\alpha_z \varphi_z(z)|_{z=y}$

A cube $Q \subset \R^n$ is a set of the form:
\[
Q=a+(-\delta, \delta]^n \quad\left(a \in \R^n, \delta>0\right).
\]
The sidelength of $Q$ is denoted $\delta_Q :=2 \delta$, while the center of $Q$ is denoted $\ctr(Q) := a$. For $\gamma>0$ let $\gamma Q$ be the cube having the same center as $Q$ but with sidelength $\gamma \delta_Q$. A \emph{dyadic cube} $Q \subset \R^n$ has the form:
\begin{align*}
Q=\left(j_{1} \cdot 2^{k},\left(j_{1}+1\right) \cdot 2^{k}\right] \times\left(j_{2} \cdot 2^{k},\left(j_{2}+1\right) \cdot 2^{k}\right] \times &\cdots \times\left(j_{n} \cdot 2^{k},\left(j_{n}+1\right) \cdot 2^{k}\right], \; \left(j_{1}, j_{2}, \ldots, j_{n} \in \mathbb{Z}, \; k \in \mathbb{Z}\right).
\end{align*}
To \emph{bisect} a cube $Q \subset \R^n$ is to partition it into $2^n$ disjoint subcubes of sidelength $\frac{1}{2} \delta_Q$. These subcubes are called the children of $Q$. If $Q \subsetneq Q'$ are dyadic cubes we say that $Q'$ is an ancestor of Q. Every dyadic cube $Q$ has a smallest ancestor called its parent, which we denote by $Q^+$.

A rectangular box in $\R^n$ is a set of the form $R = \prod_{j=1}^n I_j$, where each $I_j \subset \R$ is an interval of length $\delta_j > 0$. We refer to $\delta_1,\dots,\delta_n$ as the sidelengths of $R$. If the sidelengths of $R$ differ by at most a constant factor $\eta \geq 1$ (i.e., $\delta_j \leq \eta \delta_i$ for all $i,j$), then we say $R$ is \emph{$\eta$-non-degenerate}. 

We use the following notation:
\begin{align*}
& |x| := |x|_{\infty} = \max \{|x_1|,\dots,|x_n| \} &&(x = (x_1, \dots, x_n) \in \R^n ); \\
& \dist(x, \Omega) : = \inf \{ |x-y|: y \in \Omega \} &&( \Omega \subset \R^n, \; x \in \R^n); \\
& \dist (\Omega', \Omega):= \inf \{|x-y|: x \in \Omega', \; y \in \Omega \} &&(\Omega, \Omega' \subset \R^n); \\
& B(\Omega, R):= \{ x \in \R^n: \dist(x, \Omega) \leq R \} && (\Omega \subset \R^n, R>0);\\
& \diam(S) := \sup \{|x-y|: x,y \in S \} && (S \subset \R^n); \text{ and} \\
& |S| := \text{cardinality of } S && (S \subset \R^n).
\end{align*}
In particular, since we use the $\ell^\infty$ norm on $\R^n$, we have $\diam(Q) = \delta_Q$ for any cube $Q$.

The analogous metric quantities defined with respect to the Euclidean norm $|x|_2 = (|x_1|^2 + \dots + |x_n|^2)^{1/2}$ are denoted by $\dist_2(x, \Omega)$, $\dist_2(\Omega', \Omega)$, $B_2(\Omega, R)$, and $\diam_2(S)$.

Given a subset $K \subset \R^n$, we write $\inte(K)$ to denote the interior of $K$, and we write $\Cl(K)$ to denote the closure of $K$.

We write $\mm$ for the collection of all multi-indices $\al = (\al_1,\dots,\al_n) \in \Z^n$, with $\al_i \geq 0$ for all $i$, of order $|\al| := \al_1 + \dots + \al_n \leq m-1$. If $\al$ and $\beta$ are multi-indices, then $\delta_{\alpha \beta}$ denotes the Kronecker delta: $\delta_{\alpha \beta} = 1$ if $\alpha = \beta$; $\delta_{\alpha \beta} = 0$ if $\alpha \neq \beta$. If $\al = (\al_1,\dots,\al_n)$ is a multi-index, then $\al ! := \prod_{j=1}^n \al_j !$.

Let $\mpp$ denote the space of real-valued $(m-1)^{\text {rst}}$ degree polynomials on $\R^n$. Then $\mpp$ is a vector space of dimension $D:=\dim(\mpp)$. 

If $F$ is a $C^{m-1}$ function on a neighborhood of a point $y \in \R^n$, then we write $J_y(F) \in \mpp$ (the``jet" of $F$ at $y$) for the $(m-1)^{\text {rst}}$ degree Taylor polynomial of $F$ at $y$, given by
\[
J_y (F)(z) = \sum_{\al \in \mm} \frac{\partial^\al F(y)}{\al!} (z-y)^\al.
\]

For each $x \in \R^n$, the jet product $\odot_x$ on $\mpp$ is
defined by
\[
P \odot_{x} Q:=J_{x}(P \cdot Q) \quad(P, Q \in \mpp).
\]
For $x \in \R^n$, $\delta>0$, define a norm on $\mpp$:
\begin{align}
|P|_{x,\delta} = \Big( \sum_{\al \in \mm} |\p^\al P(x)|^p \cdot \delta^{n +(|\al| - m)p} \Big)^{1/p} \quad (P \in \mpp). \label{pnorm}
\end{align}
For $x' \in \R^n$, we have the Taylor expansion
\[
\p^{\alpha} P(x)=\sum_{ |\gamma| \leq m-1-|\alpha|} \frac{1}{\gamma !} \p^{\alpha+\gamma} P\left(x'\right) \cdot\left(x-x'\right)^{\gamma} \quad(|\alpha| \leq m-1).
\]
Thus, the norms defined in (\ref{pnorm}) satisfy the inequality
\begin{align}
|P|_{x, \delta} \leq C' |P|_{x', \delta} \quad\left(x, x'  \in \R^n,\left|x-x'\right| \leq C \delta\right). \label{movept}
\end{align}
And by computation, for $\delta'>\delta$,
\begin{align}
    |P|_{x, \delta'} \leq |P|_{x, \delta} \leq (\delta'/\delta)^{m-n/p} |P|_{x, \delta'}. \label{normdi}
\end{align}

The homogeneous H\"older space $C^{m-1,1-n/p}(\R^n)$ is the space of $(m-1)$-times differentiable functions $F:\R^n \to \R$, with finite semi-norm,
\begin{align*}
\|F\|_{C^{m-1, 1-n/p}(\R^n)}: = \max_{|\al|=m-1} \sup_{x,y \in \R^n, x \neq y} \frac{|\p^\al F(x) - \p^\al F(y)|}{|x-y|^{1-n/p}}.    
\end{align*}

Given a set $K \subset \R^n$, we let $Wh(K)$ be the space of Whitney fields on $K$, namely, the set of all collections of polynomials $\vec{P} = (P_x)_{x \in K}$, where $P_x \in \mpp$ for all $x \in K$.



Let $\mu$ be a Borel regular measure on $\R^n$.  For a Borel set $S \subset \R^n$, we define the restricted measure $\mu|_S$ by $\mu|_S(A) := \mu(A \cap S)$ for all Borel sets $A \subset \R^n$. 

Write $\mathrm{meas}_\mu(\R^n)$ for the vector space of all equivalence classes of $\mu$-measurable functions $f : \R^n \rightarrow \R$, with functions identified in an equivalence class if they agree on the complement of a set of $\mu$-measure $0$.

For a $\mu$-measurable function $f : \R^n \rightarrow \R$, we define: 
\begin{equation}\label{J_func:eqn}
\|F\|_{\J(f,\mu)} :=  \Big( \|F\|_{L^{m,p}(\R^n)}^p + \int_{\R^n} |F-f|^p d\mu \Big)^{1/p} \qquad F \in L^{m,p}(\R^n).
\end{equation}
We remark that $\| \cdot \|_{\J(f,\mu)}$ is not a norm on $L^{m,p}(\R^n)$. Rather, $\| \cdot \|_{\J(f,\mu)}$ is a $[0,\infty]$-valued functional on $L^{m,p}(\R^n)$, satisfying the convexity condition,
\[
\| \lambda_1 F_1 + \lambda F_2 \|_{\J(\lambda_1 f_1 + \lambda_2 f_2, \mu)} \leq \lambda_1 \|  F_1\|_{\J( f_1, \mu)} +  \lambda_2 \| F_2 \|_{\J( f_2, \mu)}.
\]

We write $\J(\mu): = L^{m,p}(\R^n) + L^p(d\mu)$ for the sum space defined in the introduction. Then $\J(\mu)$ is a complete seminormed vector space. We can characterize the seminorm\footnote{A seminorm on a vector space $X$ is a $[0,\infty)$-valued functional $ \| \cdot \|_X$ on $X$ satisfying the conditions: $\| f_1 + f_2 \|_X \leq \| f_1 \|_X + \| f_2 \|_X$, and $\| \lambda f \|_X = | \lambda| \| f \|_X$ for $\lambda \in \R$. There is no requirement that $\|f \|_X = 0 \implies f = 0$ for a seminorm.} on this space using the functionals in \eqref{J_func:eqn}:
\begin{align*}
&\J(\mu) := \big\{f: f \text{ is } \mu\text{-measurable, and} \; \|f\|_{\J(\mu)} < \infty \big\}, \mbox{ with} \\
&\|f\|_{\J(\mu)} := \inf \Big\{ \|F\|_{\J(f,\mu)}: F \in L^{m,p}(\R^n) \Big\}.
\end{align*}

Next, we introduce localized variants of the above functionals. 

Let $R \subset \R^n$ be a rectangular box, and let $\delta > 0$.

Given $F \in L^{m,p}(R)$, a $\mu$-measurable function $f$, and  $P \in \mpp$, we define
\begin{align}
&\|F,P\|_{\J_*(f, \mu;R)} := \Big( \|F\|_{L^{m,p}(R)}^p + \int_R |F-f|^p d\mu  + \|F-P\|_{L^p(R)}^p/\diam(R)^{mp} \Big)^{1/p}; \text{ and} \nonumber\\
&\|f,P\|_{\J_*(\mu;R)} := \inf \Big\{ \|F,P\|_{\J_*(f, \mu;R)} : F \in L^{m,p}(R) \Big\} \label{Jstar1A:defn}.
\end{align}
Given $F \in L^{m,p}(\R^n)$, a $\mu$-measurable function $f$, and $P \in \mpp$, we define:
\begin{align}
&\|F,P\|_{\J(f, \mu; \delta)} := \Big( \|F\|_{L^{m,p}(\R^n)}^p + \int_{\R^n} |F-f|^p d\mu  + \|F-P\|_{L^p(\R^n)}^p/\delta^{mp} \Big)^{1/p}; \text{ and} \nonumber\\
&\|f,P\|_{\J(\mu; \delta)} := \inf \Big\{ \|F,P\|_{\J(f, \mu; \delta)}: F \in L^{m,p}(\R^n) \Big\}. \label{Jstar1B:defn}
\end{align}
These are $[0,\infty]$-valued functionals on the spaces $L^{m,p}(R) \times \mpp$,  $\mathrm{meas}_\mu(\R^n) \times \mpp$, $L^{m,p}(\R^n)\times \mpp$, and $\mathrm{meas}_\mu(\R^n) \times \mpp$, respectively.  We make use of the fourth functional to define the seminormed vector space:
\begin{align} \label{Jstar_space_2A:defn}
   \J(\mu;\delta) =  \big\{ (f,P):  f \in \J(\mu), \; P \in \mathcal{P}, \; \|f,P\|_{\J(\mu; \delta)} < \infty \big\}.
\end{align}
Generally, $\J(\mu;\delta)$ is a subspace of $\J(\mu) \times \mpp$. Later, we will show that $\J(\mu; \delta) = \J(\mu) \times \mathcal{P}$ if $\supp(\mu)$ is compact (see Lemma \ref{J_space:lem}).

We can make comparisons between the different localized functionals. First, immediately from the definitions, if $\diam(R) \simeq \delta$ and $F \in L^{m,p}(\R^n)$, then
\begin{align}
\|F,P\|_{\J_*(f, \mu;R)} \lesssim \|F,P\|_{\J(f, \mu; \delta)}.\label{j5}
\end{align}
Furthermore, the $\J(\mu;\delta)$ and $\J(f,\mu;\delta)$ functionals are monotone in $\delta$ in the sense that:
\begin{align}
&\|F,P\|_{\J(f, \mu; \delta)} \leq \|F,P\|_{\J(f, \mu; \delta')}, \nonumber\\
&\|f,P\|_{\J(\mu; \delta)} \leq \|f,P\|_{\J(\mu; \delta')} \qquad (\delta \geq \delta').\label{j7}
\end{align}

\subsection{Further Elementary Inequalities} 

The next result is immediate from the definitions of the $\J$-functionals, due to the sublinearity of the $L^{m,p}$-norm, $L^p$-norm, and $L^p(d\mu)$-norm

\begin{lem}[Sublinearity of $\J$-functionals]
Let $\mu$ be a Borel regular measure on $\R^n$, and let $\delta>0$. Then $\| \cdot, \cdot \|_{\J(\cdot, \mu; \delta)}$ and $\| \cdot, \cdot \|_{\J(\mu; \delta)}$  are sublinear: Let $f_1, f_2$ be $\mu$-measurable functions. Given $F_1, F_2 \in L^{m,p}(\R^n)$, $\lambda>0$, $P_1, P_2 \in \mpp$,  
\begin{align}
&\| F_1+ \lambda F_2, P_1 + \lambda P_2\|_{\J(f_1 + \lambda f_2, \mu; \delta)} \leq \|F_1,P_1\|_{\J(f_1, \mu; \delta)} + \lambda \|F_2, P_2\|_{\J(f_2,\mu; \delta)}; \text{ and}  \label{sublinearj}\\
&\| f_1+ \lambda f_2, P_1 + \lambda P_2\|_{\J(\mu; \delta)} \leq \|f_1,P_1\|_{\J(\mu; \delta)} + \lambda \|f_2, P_2\|_{\J(\mu; \delta)}.\label{sublinearj2}
\end{align}
\end{lem}

Our assumption moving forward is that $n<p<\infty$, so we can apply the Sobolev inequality on $L^{m,p}(\R^n)$. We let $U \subset \R^n$ be a domain in $\R^n$. We shall consider the setting where $U$ is either all of $\R^n$, or the union of two $\eta$-non-degenerate rectangular boxes with a common interior point. This includes, for instance, the case when $U$ is a cube, when $\eta = 1$.

\noindent \textbf{Sobolev Inequality.} (see \cite{arie3}) For $F \in L^{m,p}(U)$, there exists a constant $C$ such that
\[
\left|\p^{\alpha}\left(J_{y}(F)-F\right)(x)\right| \leq C \cdot|x-y|^{m-|\alpha|-n / p} \cdot\|F\|_{L^{m,p} (U)} \quad(x, y \in U,|\alpha| \leq m-1).
\]
If $U = \R^n$, the constant $C$ is determined by $m$, $n$, and $p$ alone. If $U$ is the union of  two $\eta$-non-degenerate rectangular boxes with a common interior point, then $C$ is dependent on $\eta$ as well. As an immediate consequence, we have for $F \in L^{m,p}(U)$ and $y \in U$,
\begin{align}
\|F-J_y(F)\|_{L^p(U)}/\diam(U)^m \leq C \|F\|_{L^{m,p}(U)}. \label{SobLp}
\end{align}

Next we establish a relatinship between the $L^p$ norm and the $|\cdot|_{x,\delta}$ norm of polynomials.

\begin{lem}\label{polynorm:lem}
Let $R_1$ and $R_2$ be two $\eta$-non-degenerate rectangular boxes with a common interior point, such that $\diam(R_1) \simeq \diam(R_2) \simeq \delta$. Let $U = R_1 \cup R_2$. Then for all $y \in U$ and $P \in \mpp$,
\begin{align*}
    \|P \|_{L^p(U)}/\delta^m \simeq_{\eta} |P|_{y,\delta} .
\end{align*} \label{lplem}
\end{lem}

\begin{proof}
Observe $\diam(U) \leq \diam(R_1) + \diam(R_2) \leq C \delta$, so the Lebesgue measure of $U$ is at most $C \delta^n$. We begin by writing the polynomial $P$ in a Taylor expansion at the basepoint $y$; then because for all $x \in U$, we have $|x-y| \leq \diam(U) \leq C \delta$, we deduce
\begin{align*}
    \|P\|_{L^p(U)}^p/\delta^{mp} &= \int_U |P(x)|^p dx/\delta^{mp} \\
    &= \int_U \bigg|\sum_{\al \in \mm} \p^\al P(y) (x-y)^\al/\al!\bigg|^p dx /\delta^{mp} \\
    &\leq C \int_U \sum_{\al \in \mm} |\p^\al P(y)|^p \delta^{|\al| p} dx / \delta^{mp} \\
    & \leq C' \sum_{\al \in \mm} |\partial^\al P(y)|^p \delta^{n-(m-|\al|)p}\\
    &\leq C'' |P|_{y,\delta}^p,
\end{align*}
where $C, C', C''$ are universal constants. 

To show the reverse inequality, we may assume by rescaling and translation ($x \mapsto C\cdot \eta x/\delta+y'$) that $y \in Q_0 = (0,1]^n$, $Q_0 \subset R_1$ or $Q_0 \subset R_2$, and $\delta = \eta$ in the statement of the lemma. Without loss of generality, suppose $y \in R_1$. In light of (\ref{normdi}), we have $|P|_{y,\eta}  \simeq_{\eta} |P|_{y,1}$. Because $\mpp$ is a finite dimensional vector space, all norms on $\mpp$ are equivalent; in particular, we have $|P|_{y,1} \simeq \|P\|_{L^p(Q_0)}$. 
Consequently 
\begin{align*}
    |P|_{y,\eta} \simeq_{\eta} |P|_{y,1} \simeq  \|P\|_{L^p(Q_0)} \lesssim_{\eta} \|P\|_{L^p(U)} / \eta^m,
\end{align*}
proving the lemma.

\end{proof}

\begin{lem}
Let $U$ be the union of two $\eta$-non-degenerate rectangular boxes with a common interior point, $U =R_1 \cup R_2$, such that $\diam(R_1) \simeq \diam(R_2) \simeq \delta$. There exists a constant $C$ depending on $m,n,p,\eta$ such that the following holds. 

Let $x \in U$, $F \in L^{m,p}(U)$, and $P \in \mathcal{P}$. For $|\al| \leq m-1$,
\begin{align}
|\p^\al(F-P)(x)| \leq C \left(\|F\|_{L^{m,p}(U)}+\|F-P\|_{L^p(U)}/\delta^m \right)\delta^{m-|\al|-n/p}. \label{sob0}
\end{align}

If $\mu$ is a Borel regular measure on $\R^n$, and $f : \R^n \rightarrow \R$ is a $\mu$-measurable function, then for  $F \in L^{m,p}(U)$,
\begin{equation}
|\p^\al(F-P)(x)| \leq C (\|F,P\|_{\J_*(f,\mu;R_1)}+ \|F,P\|_{\J_*(f,\mu;R_2)})\delta^{m-|\al|-n/p},  \label{sob}
\end{equation}
while for  $F \in L^{m,p}(\R^n)$,
\begin{equation}
|\p^\al(F-P)(x)| \leq C \|F,P\|_{\J(f, \mu; \delta)}\delta^{m-|\al|-n/p}.  \label{sob2}
\end{equation}
\end{lem}

\begin{proof} Let $F \in L^{m,p}(U)$. To prove \eqref{sob0} we will show that
\begin{align}
\sup_{x \in U} \max_{\al \in \mm} |\p^\al F(x)| \delta^{|\al|+n/p-m} \leq C \left(\|F\|_{L^{m,p}(U)}+\|F\|_{L^p(U)}/\delta^m \right). \label{sob00}
\end{align}
By replacing $F$ by $F-P$ in \eqref{sob00}, we deduce \eqref{sob0}. 

By the Sobolev Inequality, and since $\diam(U) \lesssim \delta$, for any $x,y \in U$, we have
\[
| (F - J_y (F))(x)| \leq C_s' \|F\|_{L^{m,p}(U)} \delta^{m-n/p},
\]
and hence by integrating,
\[
\| F - J_y (F)\|_{L^p(U)} \leq C_s \|F\|_{L^{m,p}(U)} \delta^m, 
\]
where $C_s'$ and $C_s$ depend on $m$, $n$, $p$, and $\eta$. If $F|_U = 0$, then both sides of (\ref{sob00}) are zero, and the inequality holds true. Thus we may assume $F|_U \neq 0$.

For sake of contradiction, suppose (\ref{sob00}) does not hold with a constant $C = \Lambda$, for $\Lambda > 0$ to be determined momentarily. Let $y \in U$, $\beta \in \mm$ be the argument of 
\begin{align*}
\sup_{x \in U} \max_{\al \in \mm} |\p^\al F(x)|\delta^{|\al| + n/p -m}.
\end{align*}
Then 
\begin{align}
|\p^\beta F(y)| \delta^{|\beta|+n/p-m} \geq \Lambda \left(\|F\|_{L^{m,p}(U)}+\|F\|_{L^p(U)}/\delta^m \right). \label{sobcontra}
\end{align}
Applying Lemma \ref{lplem} to $J_y(F)$ and from the definition (\ref{pnorm}), we have
\begin{align*}
\|J_y (F) \|_{L^p(U)}/\delta^m \simeq |J_y(F)|_{y,\delta} \simeq |\p^\beta F(y)| \delta^{|\beta| + n/p-m}.
\end{align*}
Consequently, there exists a universal constant $c'$ such that 
\[
\|J_y(F)\|_{L^p(U)}/\delta^m \geq c'  |\p^\beta (F)(y)| \delta^{|\beta| + n/p-m}.
\]
Choose $\Lambda > \frac{1 + C_s}{c'}$. Then from (\ref{sobcontra}) and the Sobolev Inequality, we have
\begin{align*}
|\p^\beta (F)(y)|  \delta^{|\beta|+n/p-m} &> \frac{1 + C_s}{c'} \left(\|F\|_{L^{m,p}(U)}+\|F\|_{L^p(U)}/\delta^m \right) \\
&> \frac{1 + C_s}{c'}\|F\|_{L^{m,p}(U)}+ \frac{1}{c'}\|F\|_{L^p(U)}/\delta^m \\
&\geq \frac{1 + C_s}{c'}\|F\|_{L^{m,p}(U)}+ \frac{1}{c'} \Big(\|J_y(F)\|_{L^p(U)}/\delta^m - \|F-J_y(F)\|_{L^p(U)}/\delta^m \Big) \\
&\geq \frac{1 + C_s}{c'}\|F\|_{L^{m,p}(U)}+ \frac{1}{c'} \Big(\|J_y(F)\|_{L^p(U)}/\delta^m - C_s \|F\|_{L^{m,p}(U)} \Big) \\
&\geq \frac{1}{c'} \|F\|_{L^{m,p}(U)}+ |\p^\beta F(y)| \delta^{|\beta|+n/p-m}.
\end{align*}
This implies $ \|F\|_{L^{m,p}(U)} < 0$, a contradiction. So (\ref{sob00}) must hold, which implies (\ref{sob0}). 

Note that (\ref{sob}) follows from (\ref{sob0}) and the definition of the $\mathcal{J}_*(\cdots)$ functional. Finally,  (\ref{sob2}) follows from (\ref{j5}) and (\ref{sob}).
\end{proof}

\begin{lem}\label{lem:1}
Let $Q$ be a cube, let $\mu$ be a Borel regular measure on $\R^n$ with $\supp(\mu) \subset Q$, let $f : \R^n \to \R$ be $\mu$-measurable, let $P \in \mathcal{P}$, and let $\eta \in [.001,100]$. Let $\theta$ be a $C^\infty$ function satisfying $\supp(\theta) \subset (1+\eta) Q$, $\theta|_Q = 1$, and $|\partial^\alpha \theta(x)| \leq C \delta_Q^{-|\alpha|}$ for all $|\alpha| \leq m$. 

Let $F \in L^{m,p}((1+\eta)Q)$. Define $\bar{F}:= \theta \cdot F + (1 - \theta) P$. Then $\bar{F} \in L^{m,p}(\R^n)$, $\bar{F}|_Q = F|_Q$, and
\begin{align}
&\| \bar{F} \|_{L^{m,p}(\R^n)} \lesssim \| F \|_{L^{m,p}((1+\eta)Q)} + \| F - P \|_{L^p((1+\eta)Q)}/\delta_Q^m, \label{j6a} \\
&\|\bar{F},P\|_{\J(f,\mu;\delta_Q)} \lesssim \|F,P\|_{\J_*(f,\mu; (1+\eta)Q)}  \label{j1}.
\end{align}
\end{lem}

\begin{proof}
Proof of (\ref{j6a}): Note that $\bar{F}$ agrees with the $(m-1)^{\text {rst}}$ degree polynomial $P$ on $\R^n \setminus (1 +\eta) Q$. Thus, 
\begin{align*}
    \|\bar{F}\|_{L^{m,p}(\R^n)}^p &= \|\bar{F} \|_{L^{m,p}((1+\eta)Q)}^p \\
    &= \|\bar{F}- P \|_{L^{m,p}((1+\eta)Q)}^p\\
    &= \|(F - P)\theta\|_{L^{m,p}((1+\eta)Q)}^p \\
    &\lesssim \int_{(1+\eta)Q} \sum_{|\al|+|\beta|=m} |\p^\al(F-P)(x)|^p \cdot |\p^\beta \theta(x)|^p dx \\
    &\overset{(\ref{sob0})}{\lesssim} \|F\|_{L^{m,p}((1+\eta)Q)}^p + \|F-P\|_{L^p((1+\eta)Q)}^p/\delta_Q^{mp}.
\end{align*}

Proof of (\ref{j1}): Note that $\bar{F} - P = 0$ on $\R^n \setminus (1 +\eta) Q$, $\bar{F} - P = \theta(P-F)$ on $(1+\eta)Q$, and $\bar{F} = F$ on $Q \supset \supp(\mu)$. Thus, we can use \eqref{j6a} to bound
\begin{align*}
    \|\bar{F},P\|_{\J(f,\mu; \delta_Q)}^p &= \|\bar{F}\|_{L^{m,p}((1+\eta)Q)}^p + \int_{\R^n} |\bar{F} - f|^pd\mu + \|\bar{F} - P\|_{L^p((1+\eta)Q)}^p /\delta_Q^{mp} \\
    &\overset{\eqref{j6a}}{\lesssim} \|F\|_{L^{m,p}((1+\eta)Q)}^p +\|F-P\|_{L^p((1+\eta)Q)}^p/\delta_Q^{mp}+ \int_{\R^n}|F-f|^p d\mu \\
    & \qquad\qquad\qquad\qquad\qquad\qquad\qquad +\|\theta(F -P)\|_{L^p((1+\eta)Q)}^p/\delta_Q^{mp} \\
    &\lesssim \|F\|_{L^{m,p}((1+\eta)Q)}^p + \int_{\R^n}|F-f|^p d\mu +\|F -P\|_{L^p((1+\eta)Q)}^p/\delta_Q^{mp} \\
    &=\|F,P\|_{\J_*(f,\mu;(1+\eta)Q)}^p.
\end{align*}
\end{proof}

\begin{lem}
Suppose $R_1$ and $R_2$ are $\eta$-non-degenerate rectangular boxes, satisfying $R_1 \subset R_2$ and $\diam(R_2) \leq \eta \diam(R_1)$, for $\eta \geq 1$. For any $P \in \mpp$,
\begin{align}
\|P\|_{L^p(R_2)} \simeq_\eta \|P\|_{L^p(R_1)}. \label{cub1}
\end{align}
Here, the constants in $\simeq_\eta$  depend on $m$, $n$, $p$, and $\eta$.
\label{polycubelm}
\end{lem}

\begin{proof}
Let $x \in R_1$. By applying Lemma \ref{lplem} for both $U = R_2$ and $U=R_1$, and inequality (\ref{normdi}), for $P \in \mpp$,
\begin{align*}
    \|P\|_{L^p(R_2)} &\lesssim_\eta \diam(R_2)^m |P|_{x, \diam(R_2)} \lesssim_\eta  \diam(R_1)^m |P|_{x, \diam(R_1)} \lesssim_\eta \|P\|_{L^p(R_1)},
\end{align*}
and since $R_1 \subset R_2$, also $\|P\|_{L^p(R_1)} \leq \|P\|_{L^p(R_2)}$,  proving $\|P\|_{L^p(R_2)} \simeq_\eta \|P\|_{L^p(R_1)}$.
\end{proof}

\begin{lem}
Suppose $R_1$ and $R_2$ are $\eta$-non-degenerate rectangular boxes, satisfying $R_1 \subset R_2$, for $\eta \geq 1$. For any $H \in L^{m,p}(R_2)$ and $P \in \mpp$,
\begin{align}
    \|H-P\|_{L^p(R_2)}/ \diam(R_2)^{m} \lesssim_\eta \|H\|_{L^{m,p}(R_2)} + \|H-P\|_{L^p(R_1)} / \diam(R_1)^{m}. \label{cub2}
\end{align}
Here, the constants in $\lesssim_\eta$ depend on $m$, $n$, $p$, and $\eta$.
\label{polycubelm2}
\end{lem}

\begin{proof}

Let $x \in R_1$. Repeatedly applying the triangle inequality, (\ref{normdi}), (\ref{SobLp}) and Lemma \ref{polynorm:lem}, we have
\begin{align*}
\|H-P\|_{L^p(R_2)}/\diam(R_2)^{m} 
&\leq \|H-J_x H\|_{L^p(R_2)}/\diam(R_2)^{m} 
     + \|J_x H -P\|_{L^p(R_2)}/\diam(R_2)^{m}    \\
&\lesssim_\eta \|H\|_{L^{m,p}(R_2)} 
     + \|J_{x}H - P\|_{L^p(R_2)}/\diam(R_2)^{m}  \\
&\simeq_\eta \|H\|_{L^{m,p}(R_2)} + |J_{x}H - P|_{x,\diam(R_2)}  \\
&\leq \|H\|_{L^{m,p}(R_2)} + |J_{x}H - P|_{x,\diam(R_1)}   \\
&\simeq_\eta \|H\|_{L^{m,p}(R_2)} 
     + \|J_{x}H - P\|_{L^p(R_1)}/\diam(R_1)^m  \\
&\leq \|H\|_{L^{m,p}(R_2)} + \| J_x H  - H \|_{L^p(R_1)}/\diam(R_1)^m
     + \|H - P\|_{L^p(R_1)}/\diam(R_1)^m\\
&\lesssim_\eta \|H\|_{L^{m,p}(R_2)} + \| H \|_{L^{m,p}(R_1)}
     + \|H - P\|_{L^p(R_1)}/\diam(R_1)^m. 
\end{align*}
completing the proof of (\ref{cub2}).
\end{proof}

The next result is an immediate consequence of Stein's Sobolev Extension Theorem for minimally smooth domains (see \cite{st}). 

\noindent \textbf{Sobolev Extension Theorem for Cubes \cite{st}. } Let $Q \subset \R^n$ be a cube. Let $F \in L^{m,p}(Q)$. There exists a linear operator $E: L^{m,p}(Q) \to  L^{m,p}(\R^n) $ satisfyig $EF|_{Q} = F|_{Q}$, and $\|EF\|_{ L^{m,p}(\R^n)} \leq C \|F\|_{L^{m,p}(Q)}$, where $C$ is independent of $m$, $n$, $p$, and $Q$.

\begin{lem} \label{restwh}
Let $Q \subset \R^n$ be a cube and $K \subset Q$. Let $\vec{S} \in Wh(K)$ satisfy $ \|\vec{S} \|_{L^{m,p}(K)} < \infty$. Then $\|\vec{S}\|_{L^{m,p}(K)} \simeq \inf \{ \|F\|_{L^{m,p}(Q)}: \; F \in L^{m,p}(Q), \; J_x(F) = S_x \text{ for all } x \in K \}$. Hence, for all $F \in L^{m,p}(Q)$ satisfying $J_xF = S_x$ for all $x \in K$, we have
\begin{align*}
\|\vec{S}\|_{L^{m,p}(K)} \lesssim \|F\|_{L^{m,p}(Q)}. 
\end{align*}
\end{lem}

\begin{proof}
Let $\e_1>0$; then there exists $F_1 \in L^{m,p}(\R^n)$ satisfying $J_x F_1 = S_x$ for all $x \in K$ and $\|F_1\|_{L^{m,p}(\R^n)} \leq \|\vec{S}\|_{L^{m,p}(K)} + \e_1$. By restricting $F_1$ to $Q$ and letting $\e_1 \to 0$, we deduce $\inf \{ \|F\|_{L^{m,p}(Q)}: \; F \in L^{m,p}(Q), \; J_x(F) = S_x \text{ for all } x \in K \} \leq \|\vec{S}\|_{L^{m,p}(K)}$. Let $\e_2>0$; then there exists $G \in L^{m,p}(Q)$ satisfying $J_xG = S_x$ for all $x \in K$, and $\|G\|_{L^{m,p}(Q)} \leq \inf \{ \|F\|_{L^{m,p}(Q)}: \; F \in L^{m,p}(Q), \; J_x(F) = S_x \text{ for all } x \in K \} + \e_2$. Applying Stein's extension operator to $G$, we deduce
\begin{align*}
\|\vec{S}\|_{L^{m,p}(K)} \leq \|EG\|_{L^{m,p}(\R^n)} &\leq C \|G\|_{L^{m,p}(Q)} \\
&\leq \inf \{ \|F\|_{L^{m,p}(Q)}: \; F \in L^{m,p}(Q), \; J_x(F) = S_x \text{ for all } x \in K \} + \e_2.
\end{align*}
Letting $\e_2 \to 0$, we have $\|\vec{S}\|_{L^{m,p}(K)} \lesssim \inf \{ \|F\|_{L^{m,p}(Q)}: \; F \in L^{m,p}(Q), \; J_x(F) = S_x \text{ for all } x \in K \}$, proving the lemma.
\end{proof}

\begin{lem}
Let $Q\subset \R^n$ be a cube, let $\mu$ be a Borel regular measure on $\R^n$ with $\supp(\mu) \subset Q$, let $f : \R^n \to \R$ be $\mu$-measurable, and let $P \in \mathcal{P}$. Then
\begin{align}
& \|f,P\|_{\J_*(\mu;Q)} \simeq  \|f,P\|_{\J(\mu;\delta_Q)}; \text{ and} \label{j3}\\
&  \|f\|_{\J(\mu)} \simeq \inf_{P \in \mathcal{P}} \|f,P\|_{\J_*(\mu; Q)} \label{j4}.
\end{align}
\end{lem}

\begin{proof} Proof of (\ref{j3}): Given $F \in L^{m,p}(Q)$, we will show there exists $\bar{F} \in L^{m,p}(\R^n)$ such that $\bar{F}|_{Q}=F|_{Q}$, and
\begin{align}
    \|\bar{F},P\|_{\J(f,\mu;\delta_Q)} \lesssim \|F,P\|_{\J_*(f,\mu;Q)}. \label{j2}
\end{align} 
Define $\widetilde{F} := E(F)$, where $E:L^{m,p}(Q) \to L^{m,p}(\R^n)$ is the Sobolev extension operator for cubes, so $\widetilde{F}|_{Q}=F|_{Q}$, and 
\begin{align}
\|\widetilde{F}\|_{L^{m,p}(1.1Q)}\leq \|\widetilde{F}\|_{L^{m,p}(\R^n)} \lesssim \|F\|_{L^{m,p}(Q)}. \label{extf}    
\end{align}
Let $\bar{F}:= \theta \cdot \widetilde{F} + (1 - \theta) P$, where $\theta$ is a $C^\infty$ function satisfying $\supp(\theta) \subset 1.1 Q$, $\theta|_Q = 1$, and $|\partial^\alpha \theta(x)| \leq C \delta_Q^{-|\alpha|}$ for all $|\alpha| \leq m$. Then from Lemma \ref{lem:1}, we have $\bar{F} \in L^{m,p}(\R^n)$ and $\bar{F} = \widetilde{F}=F$ on $Q \supsetneq \supp(\mu)$. Applying this, (\ref{j1}), (\ref{cub2}) (with $R_1=Q$, $R_2 = 1.1Q$), and (\ref{extf}), we have
\begin{align*}
\|\bar{F},P\|_{\J(f,\mu;\delta_Q)}^p & \overset{(\ref{j1})}{\lesssim} \|\widetilde{F},P\|_{\J_*(f,\mu;1.1Q)}^p\\
&= \| \widetilde{F} \|_{L^{m,p}(1.1Q)}^p + \int_{\R^n}|\widetilde{F}-f|^p d\mu + \| \widetilde{F} - P \|_{L^p(1.1Q)}^p/\delta_Q^{mp}\\ 
& \overset{(\ref{cub2})}{\lesssim} \| \widetilde{F} \|_{L^{m,p}(1.1Q)}^p + \int_{\R^n}|\widetilde{F}-f|^p d\mu + \| \widetilde{F} - P \|_{L^p(Q)}^p/\delta_Q^m \\
& = \| \widetilde{F} \|_{L^{m,p}(1.1Q)}^p + \int_{\R^n}|F-f|^p d\mu + \| F - P \|_{L^p(Q)}^p/\delta_Q^m \\
&\overset{(\ref{extf})}{\lesssim} \| F \|^p_{L^{m,p}(Q)} + \int_{\R^n}|F-f|^p d\mu + \| F - P \|^p_{L^p(Q)}/\delta_Q^m\\
&= \|F,P\|_{\J_*(f,\mu; Q)}^p.
\end{align*}
This proves (\ref{j2}). 

By definition of the $\J(\mu;\delta_Q)$ functional, as an infimum, we have that $\|\bar{F},P\|_{\J(f,\mu;\delta_Q)} \geq \| f, P \|_{\J(\mu;\delta_Q)}$ for any $\bar{F} \in L^{m,p}(\R^n)$. Thus, from (\ref{j2}), we have for any $F \in L^{m,p}(Q)$,
\[
\| f, P \|_{\J(\mu;\delta_Q)} \lesssim  \| F,P \|_{\J_*(f,\mu;Q)}.
\]
Taking the infimum over functions  $F \in L^{m,p}(Q)$, we deduce that
\[
\|f,P\|_{\J(\mu;\delta_Q)} \lesssim \|f,P\|_{\J_*(\mu; Q)}.
\]
On the other hand, using \eqref{j5}, we take the infimum over functions  $F \in  L^{m,p}(\R^n)$, implying that
\[
\inf_{F \in L^{m,p}(\R^n)} \|F,P\|_{\J(f,\mu;\delta_Q)}  \gtrsim \inf_{F \in L^{m,p}(\R^n)} \|F,P\|_{\J_*(f,\mu; Q)},
\]
deducing $\|f,P\|_{\J(\mu;\delta_Q)} \gtrsim \|f,P\|_{\J_*(\mu; Q)}$.

Proof of (\ref{j4}): Let $x \in Q$; by (\ref{SobLp}), we have
\begin{align*}
    \big\{ \inf_{P \in \mpp} \|f,P\|_{\J_*(\mu;Q)}\big\}^p &=  \inf_{P \in \mpp} \big\{ \inf_{F \in L^{m,p}(Q)}  \{ \|F,P\|_{\J_*(f,\mu; Q)} \}^p\big\} \\
    &= \inf_{F \in L^{m,p}(Q)} \big\{ \inf_{P \in \mpp} \{ \|F,P\|_{\J_*(f,\mu; Q)}^p \}\big\} \\
    &\leq \inf_{F \in L^{m,p}(Q)} \{ \|F,J_x(F)\|_{\J_*(f,\mu; Q)}^p\} \\
    &= \inf_{F \in L^{m,p}(Q)} \big\{ \|F\|_{L^{m,p}(Q)}^p + \int_{Q} |F-f|^p d \mu + \| F-J_x(F) \|_{L^p(Q)}^p/\delta_Q^{mp} \big\} \\
    &\lesssim \inf_{F \in L^{m,p}(Q)} \big\{ \| F \|_{L^{m,p}(Q)}^p + \int_{Q} |F-f|^p d \mu \big\}\\
    &\lesssim \inf_{F \in L^{m,p}(\R^n)} \big\{ \| F \|_{L^{m,p}(Q)}^p + \int_{Q} |F-f|^p d \mu \big\}\\
    &\lesssim\|f\|_{\J(\mu)}^p.
\end{align*}

Let $F \in L^{m,p}(Q)$ and $P \in \mpp$. Define $\bar{F}= EF$, where $E:L^{m,p}(Q) \to L^{m,p}(\R^n)$ is the Sobolev extension for cubes, so $\bar{F}|_{Q}=F|_{Q}$ and $\|\bar{F}\|_{L^{m,p}(\R^n)} \lesssim \|F\|_{L^{m,p}(Q)}$. Because $\supp(\mu) \subsetneq Q$, we have
\begin{align*}
\|\bar{F}\|_{\J(f,\mu)}^p &= \|\bar{F}\|_{L^{m,p}(\R^n)}^p + \int_{\R^n} |\bar{F}-f|^p d\mu \\
&\lesssim \|F\|_{L^{m,p}(Q)}^p+\int_Q |F-f|^p d\mu \\
&\lesssim \|F,P\|_{\J(f,\mu;Q)}^p.
\end{align*}
Taking the infimum over $F \in L^{m,p}(Q)$ and $P \in \mpp$, we have $\| f \|_{\J(\mu)} \lesssim \inf_{P \in \mpp} \|f,P\|_{\J_*(\mu; Q)}$.

This completes the proof of (\ref{j4}).
\end{proof}
\begin{cor}
Let $R$ be a $\eta$-non-degenerate rectangular box, let $\mu$ be a Borel measure with $\supp(\mu) \subset R$, let $f : \R^n \to \R$, and let $P \in \mathcal{P}$. Then
\begin{align}
& \|f,P\|_{\J_*(\mu;R)} \simeq_{\eta}  \|f,P\|_{\J(\mu;\diam(R))}. \label{j8}
\end{align}
\end{cor}

\begin{lem}\label{J_space:lem}
If $\supp(\mu)$ is compact,
\[
    \J(\mu;\delta) = \J(\mu) \times \mathcal{P}. 
\]
\end{lem}
\begin{proof}
Suppose $\supp(\mu)$ is compact. Then there exists $Q \subset \R^n$ such that $\supp(\mu) \subset Q$. By definition, if $(f,P) \in \J(\mu; \delta)$ then $(f,P) \in \J(\mu) \times \mpp$. Suppose $(f,P)\in \J(\mu) \times \mpp$; we will show there exists $\bar{F} \in L^{m,p}(\R^n)$ such that $\|\bar{F},P\|_{\J(f,\mu;\delta)}<\infty$, implying $(f,P) \in \J(\mu; \delta)$. Because $f \in \J(\mu)$, there exists $F \in L^{m,p}(\R^n)$ such that $\|F\|_{\J(f,\mu)} < \infty$. Let $\theta$ be a $C^\infty$ function satisfying $\supp(\theta) \subset (1.1) Q$, $\theta|_Q = 1$, and $|\partial^\alpha \theta(x)| \leq C \delta_Q^{-|\alpha|}$ for all $|\alpha| \leq m$. Define $\bar{F}:= \theta \cdot F + (1 - \theta) P$. Then from (\ref{j6a}), $\bar{F} \in L^{m,p}(\R^n)$, $\bar{F}|_Q = F|_Q$,
\begin{align*}
\|\bar{F},P\|_{\J(f,\mu;\delta_Q)} &\lesssim \|F,P\|_{\J_*(f,\mu; 1.1Q)}.
\end{align*}
Consequently, for $x \in Q$, we apply (\ref{SobLp}) to deduce
\begin{align*}
    \|\bar{F},P\|_{\J(f,\mu;\delta)}^p &\leq (1+ \delta_Q^{mp}/\delta^{mp})\|\bar{F},P\|_{\J(f,\mu;\delta_Q)}^p \\
    &\lesssim (1+ \delta_Q^{mp}/\delta^{mp}) \|F,P\|_{\J_*(f,\mu; 1.1Q)}^p.\\
    &=(1+ \delta_Q^{mp}/\delta^{mp}) \Big( \|F\|_{L^{m,p}(1.1Q)}^p + \int_{1.1Q} |F-f|^pd\mu + \|F-P\|_{L^p(1.1Q)}^{p}/\delta_Q^{mp} \Big) \\
    &\lesssim (1+ \delta_Q^{mp}/\delta^{mp}) \Big( \|F\|_{L^{m,p}(1.1Q)}^p + \int_{1.1Q} |F-f|^pd\mu + \|F-J_x(F)\|_{L^p(1.1Q)}^{p}/\delta_Q^{mp} \\ &\qquad + \|J_x(F)-P\|_{L^p(1.1Q)}^{p}/\delta_Q^{mp} \Big)\\
    &\lesssim (1+ \delta_Q^{mp}/\delta^{mp}) \Big(\|F\|_{L^{m,p}(1.1Q)}^p + \int_{1.1Q} |F-f|^pd\mu + \|J_x(F)-P\|_{L^p(1.1Q)}^{p}/\delta_Q^{mp} \Big)\\
    &\lesssim (1+ \delta_Q^{mp}/\delta^{mp}) \Big( \|F\|_{\J(f,\mu)}^p + \|J_x(F)-P\|_{L^p(1.1Q)}^{p}/\delta_Q^{mp} \Big) <\infty.
\end{align*}

\end{proof}

\subsubsection{Characterization of Sobolev Space by Local Polynomial Approximation}\label{subsec:char_ss}

\begin{define}[Packing]
Let $Q_0 \subset \R^n$ be a cube. A collection of cubes $\pi$ is a packing of $Q_0$ provided the following conditions hold:
\begin{enumerate}
    \item $Q \subset Q_0$ for all $Q \in \pi$.
    \item $\inte(Q) \cap \inte(Q') = \emptyset$ for all distinct $Q, Q' \in \pi$.
\end{enumerate}
We write $\Pi(Q_0)$ to denote the collection of all packings of $Q_0$. 
\end{define}

\begin{define}[Congruent $\delta$-packing] Given a cube $Q \subset \R^n$, we say $\pi$ is a congruent $\delta$-packing of $Q$ if it is a finite set of disjoint cubes of equal sidelength, $\delta$, contained in $Q$. Let
\begin{align*}
\Pi_{\simeq}(Q) : = \{\pi: \pi \text{ is a congruent } \delta \text{-packing of } Q, \; \delta \leq \delta_Q \}, 
\end{align*}
\end{define}

\begin{define}[Polynomial Approximation Error]
Given $F \in L^p_{loc}(\R^n)$, and a measurable set, $S \subset \R^n$, we define the local approximation error of $F$ for $S$ as
\begin{align*}
    E(F,S) := \inf_{P \in \mpp} \|F - P\|_{L^p(S)}.
\end{align*}
\end{define}

In \cite{br1}, A. Brudnyi characterizes Sobolev Space with the following result (Theorem 4 of Section 4):

\begin{prop}[Brudnyi, A. Yu.] 
Let $F \in L^p(Q)$. Suppose there exists $\lambda>0$ such that
\begin{align*}
    \sup_{\pi \in \Pi_{\simeq}(Q)} \delta_{Q}^{-m} \Big\{ \sum_{\widehat{Q} \in \pi} \big(E(F,\bar{Q})\big)^p \Big\}^{1/p} \leq \lambda.
\end{align*}
Then $F \in L^{m,p}(Q)$, and $\|F\|_{L^{m,p}(Q)} \lesssim \lambda$.
\end{prop}

If $\pi$ is a congruent $\delta$-packing of $Q$ then $\delta_{\bar{Q}} \leq \delta_Q$ for all $\bar{Q} \in \pi$. Therefore, we obtain the following corollary of Brudnyi's result:
\begin{cor} 
Let $F \in L^p(Q)$. Suppose there exists $\lambda>0$ such that
\begin{align*}
    \sup_{\pi \in \Pi_{\simeq}(Q)} \Big\{ \sum_{\bar{Q} \in \pi} \big(E(F,\bar{Q}\big)/ \delta_{\bar{Q}}^m)^p \Big\}^{1/p} \leq \lambda.
\end{align*}
Then $F \in L^{m,p}(Q)$, and $\|F\|_{L^{m,p}(Q)} \lesssim \lambda$. \label{brlm}
\end{cor}

\subsubsection{Consequence of the classical Whitney extension theorem}

\begin{lem}
Let $F : Q \rightarrow \R$, and suppose there exists a Whitney field $\vec{P} \in Wh(Q)$ and a constant $A \geq 0$ satisfying $P_x(x) = F(x)$ for all $x \in Q$ and $|P_x - P_y|_{x,|x-y|} \leq A$ for all $x,y \in Q$.

Then $F \in C^{m-1,1-n/p}(Q)$, $J_x F = P_x$ for all $x \in Q$, and 
\[
\| F \|_{C^{m-1,1-n/p}(Q)} \leq C A,
\]
for a constant $C$ determined by $m,n,p$.
\label{whitlm}

\end{lem}
\begin{proof}
Observe that the condition $|P_x - P_y|_{x,|x-y|} \leq A$ ($x,y \in Q$) implies that the Whitney field $\vec{P}$ satisfies the hypothesis of the classical Whitney extension theorem (see \cite{gl}, \cite{st}) for $C^{m-1,\alpha}$, $\alpha = 1-n/p$. Thus, there exists a function $G : \R^n \rightarrow \R$ such that $\| G \|_{C^{m-1,\alpha}(\R^n)} \leq C A$, $J_x G = P_x$ for all $x \in Q$. In particular, $G(x) = J_x G(x) = P_x(x) = F(x)$ for all $x \in Q$. Hence, $F = G|_Q \in C^{m-1,\alpha}(Q)$ and $\|F \|_{C^{m-1,\alpha}(Q)} \leq \| G \|_{C^{m-1,\alpha}(\R^n)} \leq C A$. Meanwhile, because $F=G$ on $Q$, we have $J_x F = J_x G = P_x$ for all $x \in Q$.
\end{proof}

\section{Structure of the Proof}
\label{sec:structure_pf}

Here we will take first steps toward the proof of Theorems \ref{amainthm} and \ref{constructthm}. We will state the Extension Theorem for $(\mu,\delta)$, whose proof will occupy much of the remainder of this paper. In Section \ref{sec:proof_maintheorems}. we will show that Theorems \ref{amainthm} and \ref{constructthm} follow from this result.

\subsection{Plan for the Proof}

Let $\mu$ be a Borel measure on $\R^n$ with compact support and let $\delta>0$. We will prove the following theorem:

\begin{prop} [Extension Theorem for $(\mu,\delta)$] \label{aextthmeq}
Suppose $\diam(\supp(\mu)) < \delta$. There exist a linear map $T: \J(\mu;\delta) \to L^{m,p}(\R^n)$, a map $M: \J(\mu;\delta) \to \R_+$, $K \subset \Cl(\supp(\mu))$, a linear map $\vec{S}: \J(\mu) \to Wh(K)$, and countable collections  of Borel sets $\{A_\ell\}_{\ell \in \mathbb{N}} \subset \supp(\mu)$, and of linear maps $\{\phi_\ell\}_{\ell \in \mathbb{N}}$, $\phi_\ell: \J(\mu;\delta) \to \R$, and $\{\lambda_\ell\}_{\ell \in \mathbb{N}}$, $\lambda_\ell: \J(\mu;\delta) \to L^p(d\mu)$, that satisfy for each $(f,P_0) \in \J(\mu;\delta)$,  
\begin{align}
&\textbf{(i) }\|f,P_0\|_{\J(\mu; \delta)} \leq  \|T(f,P_0), P_0\|_{\J(f, \mu; \delta)} \leq C \cdot \|f,P_0\|_{\J(\mu; \delta)};  \label{aet1} \\
&\textbf{(ii) } c \cdot M(f,P_0) \leq \|T(f,P_0), P_0\|_{\J(f, \mu; \delta)} \leq C \cdot M(f,P_0); \text{ and} \label{aet2} \\
&\textbf{(iii) }M(f,P_0) = \Big(\sum_{\ell \in \mathbb{N}} \int_{A_\ell} |\lambda_\ell(f,P_0)-f|^p d\mu + \sum_{\ell \in \mathbb{N}} |\phi_\ell(f,P_0)|^p + \|\vec{S}(f)\|_{L^{m,p}(K)}^p \Big)^{1/p}. \label{aet3}
\end{align} 

The map $T$ is $\Omega'$-constructible in the following sense:

There exists a collection of linear functionals $\Omega'=\{\omega_{s}\}_{s \in \Upsilon} \subset \J(\mu)^*$ such that collection of sets $\{\supp(\omega_s)\}_{s \in \Upsilon}$ has $C$-bounded overlap,  with $\supp(\omega_s) \subset \supp(\mu)$ for each $s \in \Upsilon$. Further, for each $y \in \R^n$, there exists a finite subset $\Upsilon_y \subset \Upsilon$ and a collection of polynomials $\{v_{s,y} \}_{s \in \Upsilon_y} \subset \mpp$ such that $|\Upsilon_y| \leq C$ and 
    \begin{align}
    J_y T(f,P_0) =  \sum_{s \in \Upsilon_y}  \omega_{s} (f) \cdot v_{s,y} +  \ot_y (P_0),
    \label{constructlm}
    \end{align}
where $\ot_y: \mpp \to \mpp$ is a linear map.

Further, the map $M$ is $\Omega'$-constructible:

\textbf{(1)} For each $\ell \in \N$ and $y \in \supp(\mu)$, there exists a finite subset $\bar{\Upsilon}_{\ell,y} \subset {\Upsilon}$ and constants $\{\eta^{\ell}_{s,y}\}_{s \in \bar{\Upsilon}_{\ell,y}} \subset \R$ such that $|\bar{\Upsilon}_{\ell,y}| \leq C$, and the map $(f,P_0) \mapsto \lambda_\ell(f,P_0)(y)$ has the form
\begin{align*}
&\lambda_\ell(f,P_0)(y) = \sum_{s \in \bar{\Upsilon}_{\ell,y}} \eta^{\ell}_{s,y} \cdot \omega_{s}(f) + \widetilde{\lambda}_{y,\ell}(P_0), 
\end{align*}
where $\widetilde{\lambda}_{y,\ell} : \mpp \rightarrow \R$ is a linear functional. 

\textbf{(2)} For each $\ell \in \N$, there exists a finite subset $\bar{\Upsilon}_{\ell} \subset {\Upsilon}$ and constants $\{\eta^\ell_{s}\}_{s \in \bar{\Upsilon}_{\ell}} \subset \R$ such that $|\bar{\Upsilon}_{\ell}| \leq C$, and the map $\phi_\ell$ has the form 
\begin{align*}
&\phi_\ell(f,P_0) = \sum_{s \in \bar{\Upsilon}_{\ell}} \eta_{s}^\ell \cdot \omega_{s}(f) + \widetilde{\lambda}_{\ell}(P_0), 
\end{align*}
where $\widetilde{\lambda}_{\ell} : \mpp \rightarrow \R$ is a linear functional.

\textbf{(3)} For $y \in K$, there exist $\{\omega^\al_{y}\}_{\al \in \mm} \subset \Omega'$ satisfying for all $\al \in \mm$, $\supp(\omega^\al_{y}) \subset \{y\}$, and the map $f \mapsto S_y(f)$ has the form
\begin{align*}
S_y(f) = \sum_{\al \in \mm} \omega^\al_{y}(f) \cdot v_{\al}, 
\end{align*}
where $\{v_\al\}_{\al \in \mm}$ is a basis for $\mpp$.

\end{prop}

\subsubsection{Order Relation on Labels} 
To prove Proposition \ref{aextthmeq}, we study the shape of symmetric, convex subsets of $\mpp$ that vary as we restrict the domain of the measure $\mu$. The shape of $\sigma \subset \mpp$ will be defined by a multi-index set $\ma \subset \mm$. We will sometimes refer to multi-index sets as \emph{labels}.

Given distinct elements $\al=\left(\al_1, \ldots, \al_n\right), \beta=\left(\beta_1, \ldots, \beta_n \right) \in \mm$, let $k \in \{0, \ldots, n\}$ be maximal subject to the condition $\sum_{i=1}^k \al_i \neq \sum_{i=1}^k \beta_i$. We write $\al<\beta$ if
\[
\sum_{i=1}^k \al_i<\sum_{i=1}^k \beta_i.
\]
Given distinct multi-index sets $\ma, \mab \subset \mm$, we write $\ma<\mab$ if the minimal element of the symmetric difference $\ma \Delta \mab$ (under the order $<$ on elements) lies in $\ma$. The following properties hold:
\begin{itemize}
\item If $\al, \beta \in \mm$ and $|\al|<|\beta|$ then $\al<\beta$
\item If $\ma \subsetneq \mab \subset \mm$ then $\mab < \ma$. In particular, the empty set is maximal and the whole set $\mm$ is minimal under the order on multi-index sets.
\end{itemize}

\subsubsection{Polynomial Bases}

A subset $\sigma$ of a vector space $V$ is symmetric provided that $v \in \sigma \implies - v \in \sigma$.

\begin{define}
Given a symmetric, convex set $\sigma \subset \mathcal{P}$, $\ma \subset \mm$, $\e > 0$, $x \in \mathbb{R}^n$, and $\delta>0$, we say \emph{$(P^\al)_{\al \in \ma} \subset \mpp$ forms an $(\ma, x, \e, \delta)$-basis for $\sigma$} if the following are satisfied:
\begin{align*}
&\textbf{(i) } P^\al \in \e \delta ^{n/p + |\al| -m} \sigma \text{ for } \al \in \ma; \\
&\textbf{(ii) } \p^\beta P^\al (x) = \delta_{\al \beta} \text{ for } \al,\beta \in \ma; \text{ and} \\
&\textbf{(iii) } \left| \p^\beta P^\al (x) \right| \leq \e \delta^{|\al| - |\beta| } \text{ for } \al \in \ma, \beta \in \mm \text{ s.t. } \beta > \al.
\end{align*} \label{basisdef}
\end{define}

Evidently, the basis property for symmetric convex sets is monotone in $\sigma, \e, \delta$ in the following sense:
\begin{align}
    &\text{Suppose that } (P^\al)_{\al \in \ma} \text{ forms an } (\ma, x, \e, \delta)\text{-basis for } \sigma. \text{ Then } (P^{\al})_{\al \in \ma} \nonumber \\
    &\text{forms an } (\ma, x, \e', \delta')\text{-basis for } \sigma' \text{ for } \e'\geq \e, \; \sigma' \supset \sigma, \text{ and } 0 < \delta' \leq \delta. \label{monotone} 
\end{align}
\begin{lem}
Suppose $(P^\al)_{\al \in \ma}$ forms an $(\ma, x, \e, \delta)$-basis for a convex set $\sigma \subset \mpp$. Then for $k>1$, $(P^\al)_{\al \in \ma}$ forms an $(\ma, x, k^m\cdot \e, k\cdot \delta)$-basis for $\sigma$. \label{edlm}
\end{lem}
\begin{proof}
We verify Property \textbf{(i)} from Definition \ref{basisdef} first: For $\al \in \mm$, 
\[
\e \delta^{n/p+|\al|-m} = (k^m \cdot \e)\delta^{n/p+|\al|} (k \cdot \delta)^{-m} \leq (k^m \cdot \e) (k \cdot \delta)^{n/p+|\al|-m},
\]
implying if $P^\al \in \e \delta ^{n/p + |\al| -m} \sigma$ then $P^\al \in (k^m \cdot \e) (\delta \cdot k)^{n/p+|\al|-m} \sigma$. Property \textbf{(ii)} is independent of $\e$ and $\delta$. Fix $\al, \beta \in \mm$ such that $\beta>\al$. Then we see Property \textbf{(iii)} holds:
\[
|\p^\beta P^\al(x)| \leq \e \delta^{|\al|-|\beta|} = (k^m \cdot \e) (k \cdot \delta)^{|\al| - |\beta|} k^{|\beta|-|\al|-m} \leq (k^m \cdot \e) (k \cdot \delta)^{|\al| - |\beta|},
\]
so $(P^\al)_{\al \in \ma}$ forms an $(\ma, x, k^m\cdot \e, k\cdot \delta)$-basis for $\sigma$.
\end{proof}

Given $x \in \mathbb{R}^n$, $\mu$, and $\delta > 0$, define symmetric convex subsets of $\mathcal{P}$:
\begin{align*}
&\sigma_J(x,\mu) = \left\{ P \in \mathcal{P}: \exists F \in L^{m,p}(\R^n) \text{ s.t. } J_x(F) = P \text{, }\|F\|_{\J(0,\mu)} \leq 1 \right\}.\\
&\sigma(\mu, \delta) = \left\{ P \in \mathcal{P}: \exists F \in L^{m,p}(\R^n) \text{ s.t. } \|F,P\|_{\J(0,\mu; \delta)} \leq 1 \right\}.
\end{align*}

\begin{lem} 
Suppose $\supp(\mu) \subset Q$. Then there exists $C_0>0$ such that for all $x \in Q$, $\sigma_J(x,\mu) \subset C_0 \cdot \sigma(\mu, \delta_Q)$. In particular, if $(P^\al)_{\al \in \ma}$ forms an $(\ma, x, \e/C_0, \delta_Q)$-basis for $\sigma_J(x,\mu)$, then $(P^\al)_{\al \in \ma}$ forms an $(\ma, x, \e, \delta_Q)$-basis for $\sigma(\mu, \delta_Q)$. \label{sjins}
\end{lem}

\begin{proof}
Let $P \in \sigma_J(x,\mu)$. Then there exists $F \in L^{m,p}(\R^n)$ such that $J_x(F) = P$ and 
\begin{align*}
\|F\|_{\J(0,\mu)}^p = \|F\|_{L^{m,p}(\R^n)}^p + \int_Q |F|^p d\mu \leq 1.
\end{align*}
By (\ref{j1}) there exists $F' \in L^{m,p}(\R^n)$ satisfying $F'|_Q = F|_Q$ and 
\begin{align*}
\|F',P\|_{\J(0,\mu;\delta_Q)} \leq C \|F, P\|_{\J_*(0, \mu;1.1Q)}.
\end{align*}
By (\ref{SobLp}), we have $\|F-P\|_{L^p(1.1Q)}/\delta_Q^m \leq C \|F\|_{L^{m,p}(1.1Q)}$, and thus
\[
\|F, P\|_{\J_*(0, \mu;1.1Q)} \leq C' \|F\|_{\J(0,\mu)} \leq C'.
\]
Combining the above inequalities, we have $\|F',P\|_{\J(0,\mu;\delta_Q)} \leq C_0$. Therefore, $P \in C_0 \cdot \sigma(\mu, \delta_Q)$. This proves $\sigma_J(x,\mu) \subset C_0 \cdot \sigma(\mu, \delta_Q)$. 

The basis result follows, for if $P^\alpha \in \frac{\epsilon}{C_0} \delta_Q^{m-n/p} \sigma_J(x,\mu)$ ($\alpha \in \ma$) then $P^\alpha \in \epsilon \delta_Q^{m-n/p} \sigma(\mu, \delta_Q)$.
\end{proof}

\subsection{The Induction}

\subsubsection{The Main Lemma}

Fix a collection of multi-indices $ \ma \subset \mm$. Let $C_0 > 0$ be the universal constant defined in Lemma \ref{sjins}. We prove the following by induction with respect to the multi-index set $\ma$.

\begin{lem} [Main Lemma for $\ma$] Fix $\ma \subset \mm$. There exists a constant $\e= \e(\ma)>0$, depending on $\ma, m, n$, and $p$, such that the following holds: Suppose $\mu$ is a Borel regular measure on $\R^n$ with compact support, satisfying $\diam(\supp(\mu)) < \delta$, and suppose:
\begin{align}
\text{For all } x \in \supp(\mu), \; \sigma_J(x,\mu) \text{ has an } (\ma,x,\e/C_0,10\delta)\text{-basis}. \label{amainlm}
\end{align}
Then the Extension Theorem for $(\mu,\delta)$ (Proposition \ref{aextthmeq}) is true.

Furthermore, for $\ma \neq \emptyset$, in the Extension Theorem for $(\mu, \delta)$, one can take $K = \emptyset$, and so the functional $M: \J(\mu;\delta) \to \R_+$ has the form
\begin{align}
M(f,P_0) = \Big(\sum_{\ell \in \mathbb{N}} \int_{A_\ell} |\lambda_\ell(f,P_0)-f|^p d\mu +\sum_{\ell \in \N}|\phi_\ell(f,P_0)|^p \Big)^{1/p} \label{aextthm3plus}.
\end{align}
\label{amainlma}
\end{lem}

Note that condition (\ref{amainlm}) in the Main Lemma for $\ma= \emptyset$ holds vacuously. Thus the Main Lemma for $\ma=\emptyset$ implies the Extension Theorem for $(\mu,\delta)$. Thus we have reduced the proof of the Extension Theorem for $(\mu,\delta)$ to the task of proving the Main Lemma for $\ma$, for each $\ma \subset \mm$. We proceed by induction and establish the following:

\noindent \underline{Base Case: } The Main Lemma for $\mm$ holds. \\
\underline{Induction Step:} Let $\ma \subset \mm$ with $\ma \neq \mm$. Suppose that the Main Lemma for $\ma'$ holds for
each $\ma'<\ma$. Then the Main Lemma for $\ma$ holds.

\subsubsection{Proof of the Base Case}
\begin{proof}
Fix $(\mu,\delta)$ as in the Main Lemma for $\mm$, satisfying (\ref{amainlm}) for $\ma = \mm$. Given that $\diam(\supp(\mu)) < \delta$, we can fix a cube $Q \subset \R^n$ with $\delta_Q = \delta$ and $\supp(\mu) \subset Q$. By (\ref{amainlm}), $\sigma_J(x,\mu)$ has an $(\mm,x,\e/C_0,10 \delta_Q)$-basis for all $x \in \supp(\mu)$. We will fix the choice of $\e(\mm)$ momentarily. By Lemma \ref{sjins},  $\sigma(\mu, 10 \delta_Q)$ has an $(\mm,x,\e,10 \delta_Q)$-basis for all $x \in \supp(\mu)$. Specifically, $1 \in \e (10\delta_Q)^{n/p-m} \cdot \sigma(\mu, 10 \delta_Q)$, implying there exists $G \in L^{m,p}(\R^n)$ satisfying
\begin{align}
\Big(\int_Q|G|^p d\mu \Big)^{1/p} \leq \|G, 1 \|_{\J(0,\mu; 10 \delta_Q)} \leq \e (10\delta_Q)^{n/p-m} \label{base1}.
\end{align}
By (\ref{sob2}), for $x \in Q$, 
\begin{align*} 
|G(x) - 1| &\leq C \|G, 1 \|_{\J(0,\mu; \delta_Q)} \delta_Q^{m-n/p} \\
&\lesssim \|G, 1 \|_{\J(0,\mu; 10 \delta_Q)} \delta_Q^{m-n/p} \\
&\lesssim \e.
\end{align*}
We now fix $\e$ small enough so that the previous inequality implies $|G(x)-1| <1/2$ for $x \in Q$. Then $|G(x)|>1/2$ for $x \in Q$, and $\left(\int_Q|G|^p d\mu \right)^{1/p} \geq C' \mu(Q)^{1/p}$, and by (\ref{base1})
\begin{align}
\mu(Q)^{1/p} \leq C \e \delta_Q^{n/p-m} \label{base2}.
\end{align}

Let $(f,P_0) \in \J(\mu; \delta_Q)$. Define $T(f,P_0) := P_0$. Then for any $F \in L^{m,p}(\R^n)$, we use the assumption $\supp(\mu) \subset Q$ to deduce that
\begin{align*}
\|T(f,P_0),P_0\|_{\J(f, \mu; \delta_Q)} &= \left( \int_Q|P_0 - f|^p d\mu \right)^{1/p}\\
& \leq \left( \int_Q|P_0 - F|^p d\mu \right)^{1/p} + \left( \int_Q|F- f|^p d\mu \right)^{1/p} \\
& \overset{(\ref{sob2})}{\lesssim} \left( \int_Q\|F,P_0\|_{\J(f, \mu; \delta_Q)}^p\delta_Q^{mp-n} d\mu \right)^{1/p} + \left( \int_Q|F- f|^p d\mu \right)^{1/p} \\
& \leq \|F,P_0\|_{\J(f, \mu; \delta_Q)} \delta_Q^{m-n/p} \mu(Q)^{1/p} + \left( \int_Q|F- f|^p d\mu \right)^{1/p} \\
& \overset{(\ref{base2})}{\leq} (1+C \e) \|F,P_0\|_{\J(f, \mu; \delta_Q)}.
\end{align*}
Because $F$ is arbitrary, we can conclude $\|T(f,P_0),P_0\|_{\J(f, \mu; \delta_Q)} \leq C \|f,P_0\|_{\J(\mu; \delta_Q)}$. When $\ma = \mm$, we take $K = \emptyset$, as promised in Lemma \ref{amainlma}, so the functional $M$ has the form (\ref{aextthm3plus}). Define the family $\{\phi_\ell\}_{\ell \in \N}$ by $\phi_\ell(f,P_0) = 0$ for $\ell \geq 1$. 

Define the family $\{\lambda_\ell\}_{\ell \in \N}$, by $\lambda_\ell(f,P_0) = 0$ for $\ell \geq 2$, and $\lambda_1 (f, P_0) =  P_0$, and define the sets $A_\ell = \emptyset$ for $\ell \geq 2$, and $A_1 = \supp (\mu)$. For these families of linear maps, the functional $M$ in (\ref{aextthm3plus}) is given by $M(f,P_0) = \| P_0 -f \|_{L^p(d \mu)}$. Note the above chain of inequalities implies that $P_0 - f \in L^p(d \mu)$, and 
\[
M(f,P_0) = \left( \int_{\R^n}|P_0-f|^p d\mu \right)^{1/p} = \|T(f, P_0),P_0\|_{\J(f,\mu;\delta_Q)}.
\]
Let $\Omega' = \emptyset$; immediately, the collection $\{\supp(\omega)\}_{\omega \in \Omega'}$ has bounded (empty) overlap. For $y \in \supp(\mu)$, define $\ot_y(P):=P(y)$; then $\lambda_1(f,P_0)(y) = \ot_y(P_0)$, indicating $M$ is $\Omega'$-constructible. And $J_yT(f,P_0) = \ot_y(P_0) = P_0$, so $T$ is $\Omega'$-constructible.

\end{proof}
\label{pfbasecase}

\subsubsection{Technical Lemmas} 
The following linear algebra lemmas are adapted from Sections 3 and 4 of \cite{arie3}, relying on (\ref{sublinearj}) and the inequality $\|F\|_{L^{m,p}(\R^n)} \leq \|F,P\|_{\J(0,\mu)}$. \\

\begin{lem}
There exist universal constants $c_1 \in (0,1)$ and $C_1 \geq 1$ so that the following holds. \\
Suppose we are given the following: 
\begin{description}
\setlength{\itemindent}{0em}
\item[(D1)]  Real numbers $\e_1 \in (0,c_1]$ and $\e_2 \in (0,\e_1^{2D+2}]$.
\item[(D2)]  A lengthscale $\delta>0$.
\item[(D3)]  A collection of multi-indices $\ma \subset \mm$.
\item[(D4)]  A Borel regular measure $\mu$ and a bounded, non-empty set $E\subset \R^n$, satisfying $\supp(\mu) \subset E$ and $\diam(E) \leq 10 \delta$.
\item[(D5)]  A family of polynomials $(\pt_x^\al)_{\al \in \ma}$ that forms an $(\ma, x, \e_2, \delta)$-basis for $\sigma_J(x,\mu)$ for each $x \in E$.
\item[(D6)]  A point $x_0 \in E$, a multi-index $\bb \in \mm \setminus \ma$, and $\pt_{x_0}^{\bb} \in \mpp$ satisfying
\begin{align}
&\pt_{x_0}^{\bb} \in \e_2 \delta^{|\bb|+n/p-m} \sigma_J(x_0, \mu) \label{nd4}; \\
&\p^{\bb} \pt_{x_0}^{\bb}(x_0) = 1 \label{nd5}; \\
&|\p^\beta \pt_{x_0}^{\bb}(x_0)| \leq \e_1 \delta^{|\bb|-|\beta|} &&(\beta \in \mm, \; \beta> \bb) \label{nd6}; \text{ and} \\
&|\p^\beta \pt_{x_0}^{\bb}(x_0)| \leq \e_1^{-D} \delta^{|\bb|-|\beta|} &&(\beta \in \mm) \label{nd7}.
\end{align}
\item[(D7)] For all $x \in E$, the basis $(\pt_x^\al)_{\al \in \ma}$ in \textbf{(D5)} satisfies
\begin{align}
|\p^\beta \pt_x^\al(x)| \leq \e_1^{-D-1} \delta^{|\al|-|\beta|} \quad \quad \quad \quad (\al \in \ma, \; \al < \bb, \; \beta \in \mm) \label{nd8}.
\end{align}
\end{description}
Then there exists $\mab<\ma$ so that  for every $x \in E$, $\sigma_J(x,\mu)$ contains an $(\mab, x, C_1 \e_1, \delta)$-basis. \label{newdir} 
\end{lem}

\begin{proof}
Let $c_1$ be a sufficiently small constant, to be determined later. By rescaling it suffices to assume that $\delta=1$. Our hypothesis tells us that $\e_1 \leq c_1$, $\e_2 \leq \e_1^{2 D+2}$, and that $(\pt_x^\al)_{\al \in \ma}$ forms an $(\ma, x, \e_2, 1)$-basis for $\sigma_J(x, \mu),$ for each $x \in E$. That is,
\begin{align}
&\pt_x^\al \in \e_2 \cdot \sigma_J(x, \mu); \label{nd1} \\
&\p^{\beta} \pt_x^\al(x)=\delta_{\al \beta} && (\al, \beta \in \ma) ; \text { and } \label{nd2} \\
&|\p^{\beta} \pt_x^\al(x)| \leq \e_2 && (\al \in \ma, \beta \in \mm, \; \beta>\al \label{nd3}).
\end{align}
By (\ref{nd4}) (for $\delta=1$), there exists $\varphi^{\bar{\beta}} \in L^{m,p}(\R^n)$ satisfying 
\begin{align}
&J_{x_0} \varphi^{\bar{\beta}} = \pt_{x_0}^{\bar{\beta}} \label{la15}; \text{ and} \\
&\| \varphi^{\bar{\beta}} \|_{L^{m,p}(\R^n)} \leq \|\varphi^{\bar{\beta}}\|_{\J(0, \mu)} \leq \e_2 \label{la16}.
\end{align}
Fix $y \in E$ and define $\pt_y^{\bar{\beta}} := J_y \varphi^{\bar{\beta}}$. Then the definition of $\sigma_J(\cdot, \cdot)$ and (\ref{la16}) imply
\begin{align}
\pt_y^{\bar{\beta}} \in \e_2 \cdot \sigma_J(y, \mu) \label{la17}.
\end{align}
We have $|x_0 - y| \leq \diam(E) \leq 10$, so by the Sobolev Inequality, (\ref{la15}), (\ref{la16}), (\ref{nd5}), and (\ref{nd6}),
\begin{align}
|\p^\beta \pt_y^{\bar{\beta}}(y) - \delta_{\beta \bar{\beta}}| &\leq |\p^\beta \varphi^{\bar{\beta}}(y) - \p^\beta J_{x_0}\varphi^{\bar{\beta}}(y)| + |\p^\beta \pt_{x_0}^{\bar{\beta}}(y) - \delta_{\beta \bar{\beta}}| \nonumber \\
&\leq C \e_2 + C\e_1 <C\e_1 \quad \quad (\beta \in \mm, \; \beta \geq \bar{\beta}) \label{la18}.
\end{align}
Similarly, from (\ref{nd7}), we estimate:
\begin{align}
|\p^\beta \pt_y^{\bar{\beta}}(y)| &\leq |\p^\beta \pt_y^{\bar{\beta}}(y) - \delta_{\beta \bar{\beta}}| +1 \nonumber \\
& \leq C \e_1^{-D} \quad \quad (\beta \in \mm) \label{la19}.
\end{align}
Define $\mab = \{ \al \in \ma: \al < \bar{\beta} \} \cup \{ \bar{\beta} \}$. Then because $\bb \notin \ma$, the minimal element of $\ma \D \mab$ is $\bar{\beta}$. Thus, we have $\mab < \ma$, by definition of the order relation on multiindex set. Define 
\[
P_y^{\bar{\beta}} :=\pt_y^{\bar{\beta}} - \sum_{\gamma \in \mab \setminus \{\bar{\beta} \} } \p^\gamma \pt_y^{\bar{\beta}}(y) \pt_y^\gamma.
\]
Notice that $\mab \setminus \{ \bar{\beta} \} \subset \ma$; thus (\ref{nd2}) (for $x=y$) implies that 
\[
\p^\al P_y^{\bar{\beta}} (y) = 0 \quad \quad (\al \in \mab \setminus \{ \bar{\beta} \} ).
\]
And (\ref{la17}), (\ref{la19}), and (\ref{nd1}) imply that 
\[
P_y^{\bar{\beta}} \in (\e_2 + C \e_1^{-D} \e_2) \cdot  \sigma_J (y, \mu) \subset C\e_1^{-D} \e_2 \cdot \sigma_J (y, \mu).
\]
Since $\bar{\beta}$ is the maximal element of $\mab$, it follows that for any $\beta \geq \bar{\beta}$ and any $\gamma \in \mab \setminus \{ \bar{\beta} \}$, we have $\beta > \gamma$. Thus (\ref{la18}), (\ref{la19}), and (\ref{nd3}) imply that
\begin{align}
|\p^\beta P_y^{\bar{\beta}}(y) - \delta_{\beta \bar{\beta}}| \leq C(\e_1 + \e_1^{-D}\e_2) \quad \quad (\beta \in \mm, \; \beta \geq \bar{\beta}).\label{la21}
\end{align}
And (\ref{la19}) and (\ref{nd8}) imply that 
\[
|\p^\beta P_y^{\bar{\beta}}(y)| \leq C \e_1^{-2D-1} \quad \quad (\beta \in \mm).
\]
Recall that $\e_2 \leq \e_1^{D+1}$ and $\e_1 \leq c_1$. We now fix $c_1$ to be a small universal constant, so that (\ref{la21}) yields $\p^{\bar{\beta}} P_y^{\bar{\beta}}(y) \in [1/2,2]$. We then define $\hat{P}_y^{\bar{\beta}} = P_y^{\bar{\beta}} / \p^{\bar{\beta}} P_y^{\bar{\beta}}(y)$. The above properties of $ P_y^{\bar{\beta}}$ give that
\begin{align}
&\hat{P}_y^{\bar{\beta}} \in C \e_1^{-D} \e_2 \cdot \sigma_J(y, \mu); \label{la22} \\
&\p^\beta \hat{P}_y^{\bar{\beta}} (y) = \delta_{\beta \bar{\beta}} &&(\beta \in \mab); \label{la23} \\
&|\p^\beta \hat{P}_y^{\bar{\beta}}(y)| \leq C(\e_1 + \e_1^{-D} \e_2) &&(\beta \in \mm, \; \beta>\bar{\beta}); \text{ and} \label{la24} \\
&|\p^\beta \hat{P}_y^{\bar{\beta}}(y)| \leq C\e_1^{-2D-1} && (\beta \in \mm) \label{la25}.
\end{align}

For each $\al \in \mab \setminus \{ \bar{\beta} \}$, we define $\hat{P}_y^{\al} = \pt_y^\al - \left( \p^{\bar{\beta}} \pt_y^\al(y) \right) \hat{P}_y^{\bar{\beta}}$. Note that $\al < \bar{\beta}$, and hence $|\p^{\bar{\beta}} \pt_y^\al(y)| \leq \e_2 \leq 1$, thanks to (\ref{nd3}). From (\ref{nd1}) and (\ref{la22}), we now obtain
\begin{align}
\hat{P}_y^{\al} \in C' \e_1^{-D} \e_2 \cdot \sigma_J(y, \mu) \label{la26}.
\end{align}
From (\ref{nd3}) and (\ref{la25}),
\begin{align}
|\p^\beta \hat{P}_y^{\al} (y)| \leq |\p^\beta \pt_y^\al(y)| + |\p^{\bar{\beta}} \pt_y^\al(y)| \cdot |\p^\beta \hat{P}_y^{\bar{\beta}}(y)| &\leq \e_2 + \e_2 \cdot C \e_1^{-2D-1} \nonumber \\
&\leq C' \e_1^{-2D-1} \e_2, \; (\beta \in \mm, \; \beta > \al) \label{la27}.
\end{align}
Recall that $\mab = \{ \al \in \ma: \al < \bar{\beta} \} \cup \{ \bar{\beta} \}$. From (\ref{nd2}) and (\ref{la23}), we have
\begin{align}
\p^\beta \hat{P}_y^{\al} (y) = \delta_{\al \beta} \quad \quad (\beta \in \mab) \label{la28}.
\end{align}
By now varying the point $y \in E$, we deduce from (\ref{la22})-(\ref{la24}) and (\ref{la26})-(\ref{la28}) that $\sigma_J(y, \mu)$ contains an $\left(\mab, y, C \cdot (\e_1 + \e_1^{-2D-1} \e_2),1 \right)$-basis for each $y \in E$. Since $\e_2 \leq \e_1^{2D+2}$, the conclusion of the Lemma is immediate.
\end{proof}

\begin{lem}
Let $c_1$ and $C_1$ be the constants from Lemma \ref{newdir}. Suppose we are given data 
\begin{align*}
\Big(\e_1,\e_2,\delta,\ma,\mu, E, \left(\pt_x^\al \right)_{\al \in \ma, x \in E} \Big),
\end{align*}
satisfying \textbf{(D1)-(D5)} of Lemma \ref{newdir}, and the family of polynomials $(\pt_x^\al)_{\al \in \ma, x \in E}$ also satisfies
\begin{align}
\max \left\{|\p^\beta \pt_x^\al (x)|\delta^{|\beta|-|\al|}: x \in E \text{, } \al \in \ma \text{, } \beta \in \mm \right\} \geq \e_1^{-D-1} \label{bigdir1}.
\end{align}
Then there exists $\bb \in \mm \setminus \ma$ so that \textbf{(D7)} is satisfied, and additionally there exist $x_0 \in E$ and $\pt_{x_0}^{\bb} \in \mpp$ so that \textbf{(D6)} is satisfied. Hence, there exists $\mab<\ma$ so that for every $x \in E$, $\sigma_J(x,\mu)$ contains an $(\mab, x, C_1 \e_1, \delta)$-basis. \label{bigdir}
\end{lem}

\begin{proof}
By rescaling it suffices to assume that $\delta=1$. Our hypotheses tell us that $\e_1 \leq c_1$, $\e_2 \leq \e_1^{2 D+2}$, and that $(\pt_x^\al)_{\al \in \ma}$ forms an $(\ma, x, \e_2, 1)$-basis for $\sigma_J(x, \mu)$, for each $x \in E$. That is,
\begin{align}
&\pt_x^\al \in \e_2 \cdot \sigma_J(x, \mu); \label{la1} \\
&\p^{\beta} \pt_x^\al(x)=\delta_{\al \beta} && (\al, \beta \in \ma) ; \text { and } \label{la2} \\
&|\p^{\beta} \pt_x^\al(x)| \leq \e_2 && (\al \in \ma, \beta \in \mm, \; \beta>\al \label{la3}).
\end{align}
For each $\al \in \ma$, we define $Z_\al=\max \left\{ | \p^\beta \pt_x^\al(x) |: x \in E, \beta \in \mm \right\}$. Then hypothesis (\ref{bigdir1}) is equivalent to $\max \left\{Z_\al: \al \in \ma\right\} \geq \e_1^{-D-1}$. Let $\bar{\al} \in \ma$ be the minimal index with $Z_{\bar{\al}} \geq \e_1^{-D-1}$. Thus,
\begin{align}
Z_{\al}<\e_1^{-D-1}, \text { for all } \al \in \ma \text { with } \al<\bar{\al} \label{la4},
\end{align}
and there exist $x_{0} \in E$ and $\beta_0 \in \mm$ with
\begin{align}
\e_{1}^{-D-1} \leq Z_{\bar{\al}}=|\p^{\beta_0} \pt_{x_0}^{\bar{\al}}(x_0)| \label{la5}.
\end{align}
Thus, (\ref{la2}) and (\ref{la3}) imply that $\beta_0 \neq \bar{\al}$ and $\beta_0 \leq \bar{\al},$ respectively. Therefore, $\beta_0<\bar{\al} .$ By definition of $Z_{\bar{\al}}$ we also have
\begin{align}
|\p^{\beta} \pt_y^{\bar{\al}}(y)| \leq |\p^{\beta_0} \pt_{x_0}^{\bar{\al}}(x_{0})|, \text { for all } y \in E_{2} \text { and } \beta \in \mm \label{la6}.
\end{align}
Let the elements of $\mm$ between $\beta_0$ and $\bar{\al}$ be ordered as follows:
\[
\beta_0 < \beta_1< \dots <\beta_k = \bar{\al}.
\]
Note that $k+1 \leq |\mm| = D$. Define
\[
a_i = |\p^{\beta_i} \pt_{x_0}^{\bar{\al}}(x_0)|, \text{    for } i = 0,\dots,k.
\]
Then (\ref{la2}) and (\ref{la5}) imply that $a_k = 1$ and $a_0 \geq \e_1^{-D-1}$. Choose $r \in \{0,\dots,l\}$ with $a_r\e_1^{-r} = \max \{a_l \e_1^{-l}: 0 \leq l \leq k \}$. Note that $a_0 \geq \e_1^{-D-1} > a_k\e_1^{-k}$ which implies $r \neq k$. Moreover, we have
\begin{align}
a_r \geq \e_1^D a_0 \text{ and } a_r \geq \e_1^{-1} a_i \text{ for } i=r+1, \dots, k \label{la7}.
\end{align}
Define $\bar{\beta} = \beta_r \in \mm$. Then (\ref{la7}) states that 
\begin{align}
|\p^{\bar{\beta}}\pt_{x_0}^{\bar{\al}}(x_0)| \geq \e_1^D |\p^{\beta_0}\pt_{x_0}^{\bar{\al}}(x_0)| \geq \e_1^{-1} \label{la8}.
\end{align}
Also, (\ref{la3}) and (\ref{la8}) imply that
\begin{align*}
|\p^{\beta}\pt_{x_0}^{\bar{\al}}(x_0)| \leq \e_2 \leq 1 \leq \e_1 |\p^{\bar{\beta}}\pt_{x_0}^{\bar{\al}}(x_0)|  \quad \quad (\beta \in \mm, \; \beta>\bar{\beta}).
\end{align*}
For $\bar{\beta} < \beta \leq \bar{\al}$, (\ref{la7}) implies that $|\p^{\beta}\pt_{x_0}^{\bar{\al}}(x_0)| \leq \e_1 |\p^{\bar{\beta}}\pt_{x_0}^{\bar{\al}}(x_0)|$. Thus we have,
\begin{align}
|\p^{\beta}\pt_{x_0}^{\bar{\al}}(x_0)| \leq \e_1 |\p^{\bar{\beta}}\pt_{x_0}^{\bar{\al}}(x_0)| \quad \quad (\beta \in \mm, \;\beta>\bar{\beta}).\label{la9}
\end{align}
By (\ref{la8}) we have $|\p^{\bar{\beta}}\pt_{x_0}^{\bar{\al}}(x_0)| > 1$. Hence, (\ref{la2}) and (\ref{la3}) imply that 
\begin{align}
\bar{\beta} < \bar{\al} \text{ and } \bar{\beta} \notin \ma \label{la10}.
\end{align}
Define $\pt_{x_0}^{\bar{\beta}} = \pt_{x_0}^{\bar{\al}} / \p^{\bar{\beta}}\pt_{x_0}^{\bar{\al}}(x_0)$, which satisfies
\begin{align}
&\pt_{x_0}^{\bar{\beta}} \in \e_2 \sigma_J(x_0, \mu) &&\text{from (\ref{la1}) and } |\p^{\bar{\beta}}\pt_{x_0}^{\bar{\al}}(x_0)| \geq 1 \label{la11}; \\
&|\p^{\beta}\pt_{x_0}^{\bar{\beta}}(x_0)| \leq \e_1 &&(\beta \in \mm, \; \beta> \bar{\beta}), \text{ from (\ref{la9})} \label{la12}; \\
&|\p^{\beta}\pt_{x_0}^{\bar{\beta}}(x_0)| \leq \e_1^{-D} && (\beta \in \mm), \text{ from (\ref{la6}) and (\ref{la8}); and } \label{la13} \\
&|\p^{\bar{\beta}}\pt_{x_0}^{\bar{\beta}}(x_0)| = 1 \label{la14}.
\end{align}
For $x \in E$, from (\ref{la4}) and the definition of $Z_\al$, 
\begin{align}
|\p^\beta \pt_x^\al(x)| \leq \e_1^{-D-1} \quad \quad (\al \in \ma, \; \al < \bb, \; \beta \in \mm) \label{la20}
\end{align}
Now the hypotheses of Lemma \ref{newdir} hold with 
\[
\Big(\e_1,\e_2,\delta,\ma,\mu, E, \left(\pt_x^\al \right)_{\al \in \ma, x \in E}, x_0, \bb, \pt_{x_0}^{\bb} \Big)
\]
satisfying \textbf{(D1)-(D7)} due to (\ref{la11})-(\ref{la20}). Hence, there exists $\mab<\ma$ so that for every $x \in E$, $\sigma_J(x,\mu)$ contains an $(\mab, x, C_1 \e_1, \delta)$-basis.
\end{proof}

\begin{define}
Let $S \geq 1$, $\e \in (0,1)$, and $\ma \subset \mm$ be given. A matrix $(B_{\al \beta})_{\al, \beta \in \ma}$ is called $(S, \e)$ near-triangular if 
\begin{align*}
&|B_{\al \beta} - \delta_{\al \beta}| \leq \e && \al, \beta \in \mm, \; \al \leq \beta \text{; and} \\
&|B_{\al \beta}| \leq S && \al \in \ma, \; \beta \in \mm. \\
\end{align*}
\end{define}

\begin{lem}[Lemma 3.4 of \cite{arie3}]
Given $R \geq 1$, there exist constants $c_2 \geq 0$, $C_2 \geq 1$ depending only on $R, m, n$, so that the following holds. Suppose we are given $\e_2 \in (0,c_2]$, $x \in \R^n$, a symmetric convex subset $\sigma \subset \mathcal{P}$ and a family of polynomials $(P^\al)_{\al \in \ma} \subset \mathcal{P}$, such that
\begin{align}
&P^\al \in \e_2 \sigma && \al \in \ma;  \label{lna1} \\
&|\p^\beta P^\al(x) - \delta_{\al \beta}| \leq \e_2 &&\al \in \ma, \beta \in \mm, \beta \geq \al; \text{ and} \\
&|\p^\beta P^\al(x)| \leq R && \al \in \ma, \beta \in \mm. \label{lna3}
\end{align}
Then there exists a $(C_2, C_2\e_2)$ near-triangular matrix $B = (B_{\al \beta})_{\al, \beta \in \ma}$, so that if we define
\begin{align*}
\pt^\al := \sum_{\beta \in \ma} B_{\al \beta} P^\beta \text{, } \al \in \ma,
\end{align*}
then $(\pt^\al)_{\al \in \ma}$ forms an $(\ma, x, C_2\e_2,1)$-basis for $\sigma$. Furthermore $|\p^\beta \pt^\al(x) |  \leq C_2$ for every $\al \in \ma$ and every $\beta \in \mm$. \label{linalg}
\end{lem}

\subsection{The Inductive Hypothesis} \label{sec:ind_hyp}

Fix $\ma \subset \mm$, $\ma \neq \mm$. We will  impose the inductive hypothesis that the Main Lemma holds for all $\ma' < \ma$. Our task is to prove the Main Lemma for $\ma$. The inductive hypothesis will be a standing assumption until we complete the proof of the Main Lemma for $\ma$ in Section \ref{sec:proof_ml}.

We will assume the value of $\e = \e(\ma)$ in the Main Lemma for $\ma$ is less than a small enough constant determined by $m$ and $n$ and eventually determine such a constant. Let $(\mu,\delta)$ be as in the statement of the Main Lemma for $\ma$ (Lemma \ref{amainlma}). By rescaling and translating, we may assume without loss of generality that
\begin{align}
&\diam(\supp(\mu))< \delta = 1/10, \nonumber \\
&\supp(\mu) \subset \frac{1}{10} Q^{\circ},  \quad Q^\circ = (0,1]^n \label{aqsubset}.
\end{align}
From  hypothesis (\ref{amainlm}) in the Main Lemma for $\ma$, with $\delta = 1/10$, we have:
\begin{align}
    \text{For every }x \in \supp(\mu), \;\sigma_J(x,\mu) \text{ has an }(\ma,x, \e/C_0, 1) \text{-basis}. \label{amainlmi}
\end{align}

The inductive hypothesis states that the Main Lemma for $\ma'$ is true for every $\ma' < \ma$. Let $\epsilon(\ma')$ be the constants arising in the Main Lemma for $\ma'$ ($\ma' < \ma$). Define
\begin{align}
\e_0 := \min \{\e(\ma'): \ma' < \ma \}.   \label{e0def}
\end{align}
By the inductive hypothesis, for $\ma' < \ma$ and any Borel regular measure $\widehat{\mu}$ on $\R^n$,
\begin{align}\nonumber
& \mbox{if } \sigma_J(x,\widehat{\mu}) \mbox{ admits an } (\ma', x, \e_0/C_0, 10 \widehat{\delta})\mbox{-basis}, \mbox{for each } x \in \supp(\widehat{\mu}), \nonumber\\
&\mbox{for some } \widehat{\delta} \geq \diam(\supp(\widehat{\mu})),
\mbox{then the Extension Theorem for }(\widehat{\mu}, \widehat{\delta}) \mbox{ holds}.\label{indhyp} 
\end{align}

We will assume that $\e < \e_0$. 

Suppose that there exists $\mab<\ma$ such that $\sigma_J(x, \mu)$ contains an $(\mab, x, \e_0/C_0, 1)-$basis for every $x \in \supp(\mu)$. Then by the validity of the Main Lemma for $\mab$, the Extension Theorem for $(\mu, \delta)$ holds (see Proposition \ref{aextthmeq}). Note that $\mab \neq \emptyset$ because $\emptyset$ is maximal under the order on multiindex sets. Thus, by the Main Lemma for $\mab$, in the conclusion of the Extension Theorem for $(\mu, \delta)$ one can take $K = \emptyset$, and so the functional $M: \J(\mu;\delta) \to \R_+$ has the form \eqref{aextthm3plus}. Therefore, we have proven the Main Lemma for $\ma$ in the case that there exists $\mab$ as above. Therefore, we may now assume:
\begin{align}
\text{For every }&\mab<\ma,\text{ there exists }x \in \supp(\mu)\text{ such that } \nonumber \\
&\sigma_J(x, \mu) \text{ does not contain an }(\mab, x, \e_0/C_0, 1)\text{-basis} \label{amaxa}.
\end{align}

\subsubsection{Auxiliary Polynomials}
\label{subsubsec:aux_poly}

\begin{lem}
Let $\mu$ be a Borel regular measure on $\R^n$ satisfying \eqref{aqsubset}, \eqref{amainlmi}, \eqref{amaxa}. Then for all $x \in 100 Q^\circ$, there exists a family of polynomials  $(P_x^\al)_{\al \in \ma}$ such that
\begin{align}
(P_x^\al)_{\al \in \ma} &\text{ forms an } (\ma,x, C\e/C_0, 1) \text{-basis for } \sigma_J(x,\mu) \text{, and} \label{basis} \\
&|\p^\beta P_x^\al(x)| \leq C \quad (\al \in \ma, \; \beta \in \mm) \label{bound}.
\end{align}
\label{lem:polybasis}
\end{lem}
\begin{proof}
From (\ref{amainlmi}), for $y \in \supp(\mu)$, there exists $(\pt_y^\al )_{\al \in \ma}$ that form an $(\ma,y,\e/C_0,1)$-basis for $\sigma_J(y,\mu)$. As a consequence of (\ref{amaxa}), we have
\begin{align}
|\p^\beta \pt_y^\al(y)| \leq C  \quad \al \in \ma, \; \beta \in \mm, \; y \in \supp(\mu). \label{bound0}
\end{align}
To see this, suppose (\ref{bound0}) fails for the constant $C = \e_1^{-D-1}$, where \\
$\e_1 \leq \min \{ c_1, \e_0/(C_0C_1) \}$ (and $c_1, C_1$ are the constants from Lemma \ref{bigdir}). That is, suppose
\begin{align*}
\max \left\{ |\p^\beta \pt_y^\al(y)|: y \in \supp(\mu), \; \al \in \ma, \; \beta \in \mm \right\} > \e_1^{-D-1}.
\end{align*}
We may assume that $\e \leq \e_1^{2D+2}$. Then the hypotheses of Lemma \ref{bigdir} hold with parameters
\[
\Big(\e_1,\e_2,\delta,\ma,\mu, E, (\pt_x^\al )_{\al \in \ma, x \in E} \Big): = \Big(\e_1, \e, 1, \ma, \mu, \supp(\mu), (\pt_y^\al)_{\al \in \ma, y \in \supp(\mu)} \Big).
\]
Thus we find $\mab < \ma$ so that $\sigma_J(y, \mu)$ contains an $(\mab, y, C_1\e_1,1)$-basis for each $y \in \supp(\mu)$. Since $C_1\e_1 \leq \e_0/C_0$, this contradicts (\ref{amaxa}), which concludes our proof of (\ref{bound0}).

Now fix $x_0 \in \supp(\mu)$. Then we have shown $\sigma_J(x_0, \mu)$ has an $(\ma,x_0, \e/C_0, 1)$-basis, $\{\pt_{x_0}^\al\}_{\al \in \ma}$ satisfying the inequalities (\ref{bound0}) for $y=x_0$. That is, for $\al \in \ma$,
\begin{align}
&\p^\beta \pt_{x_0}^\al(x_0) = \delta_{\al \beta} &&(\beta \in \ma); \label{aux1}\\
&|\p^\beta \pt_{x_0}^\al(x_0)| \leq \e/C_0 &&(\beta \in \mm, \; \beta>\al); \text{ and} \label{aux2} \\
&|\p^\beta \pt_{x_0}^\al(x_0)| \leq C &&(\beta \in \mm). \label{aux3}
\end{align}
Furthermore, since $\pt_{x_0}^\al \in (\epsilon/C_0) \sigma_J(x_0,\mu)$ for each $\al \in \ma$ there exists
\begin{align}
\varphi^\al \in L^{m,p}(\R^n) \text{ with }J_{x_0}\varphi^\al = \pt_{x_0}^\al \text{ and }\|\varphi^\al\|_{\J(0,\mu)} \leq \e/C_0. \label{mat0}
\end{align}
For $x \in 100 Q^{\circ}$, define 
\begin{equation}\label{mat0.1}
    \hat{P}_x^\al := J_x \varphi^\al.
\end{equation}
Because $\|\varphi^\al\|_{\J(0,\mu)} \leq \e/C_0$, we have
\begin{align}
     \hat{P}_x^\al \in \e/C_0 \cdot \sigma_J(x, \mu) \subset C \epsilon \sigma_J(x,\mu). \label{mat1}
\end{align}
From the Sobolev Inequality, since $\hat{P}_x^\al = J_x\varphi^\alpha$ and $\pt^\alpha_{x_0} = J_{x_0} \varphi^\alpha$, and since $|x-x_0| \leq 100 \delta_{Q^\circ} = 100$, 
\begin{align}
|\p^\beta \hat{P}_x^\al(x) - \delta_{\al \beta}| &\leq |\p^\beta \pt_{x_0}^\al(x) - \delta_{\al \beta}| + |\p^\beta J_{x_0}\varphi^\al(x) - \p^\beta \hat{P}_x^\al(x)| \nonumber \\
& \leq \left|\sum_{0 \leq \gamma \leq m-1-|\beta|} \p^{\beta + \gamma}  \pt_{x_0}^\al(x_0) (x-x_0)^\gamma/\gamma! - \delta_{\al \beta} \right|  + C \| \varphi^\al\|_{L^{m,p}(Q^\circ)} \delta_{Q^\circ}^{m - |\beta| - n/p} \label{aux}.
\end{align}
From (\ref{aux1}), (\ref{aux2}), (\ref{mat0}), (\ref{aux}), and since $|x-x_0| \leq 100$, for $\beta \in \mm$ satisfying $\beta \geq \al$, we have 
\begin{align}
|\p^\beta \hat{P}_x^\al(x) - \delta_{\al \beta}| \leq C \e. \label{mat2}
\end{align}
From (\ref{aux3}), (\ref{mat0}), (\ref{aux}),  and since $|x-x_0| \leq 100$, for $\beta \in \mm$,
\begin{align}
|\p^\beta \hat{P}_x^\al(x) - \delta_{\al \beta}| &\leq C, \mbox{ hence }  |\p^\beta \hat{P}_x^\al(x)| \leq C + 1. \label{mat3}
\end{align}
Fix a constant $R$, determined by $m$, $n$, and $p$, so that $R=C+1$, with $C$ from (\ref{mat3}). We  fix $c_2$ and $C_2$ for the universal constants in Lemma \ref{linalg} determined by this choice of $R$. By taking $\epsilon$ small enough, we may assume $\epsilon_2 := C\e$ in \eqref{mat1}, \eqref{mat2} satisfies $\epsilon_2  < c_2$. Then (\ref{mat1}), (\ref{mat2}), and (\ref{mat3}) allow us to apply Lemma \ref{linalg} to the family of polynomials $(\hat{P}_x^\al)_{\al \in \ma}$, with $\sigma = \sigma_J(x,\mu)$, $\epsilon_2 = C \epsilon$ and $R= C +1$. Thus, there exists a $(C_2, C_2 \cdot C \e)$-near triangular matrix $A^x = (A^x_{\al \beta})_{\al, \beta \in \ma}$ such that if we define $P_x^\al \in \mpp$ as
\begin{align} 
P_x^\al := \sum_{\beta \in \ma} A^x_{\al \beta} \cdot \hat{P}_x^\beta \quad (\al \in \ma), \label{polydef}
\end{align}
then $(P_x^\al)_{\alpha \in \ma}$ forms an $(\ma, x , C_2 C \epsilon,1)$-basis for $\sigma_J(x,\mu)$ and $| \partial^\beta P^\al_x(x)| \leq C_2$ for all $\beta \in \ma$. Thus, the family of polynomials $(P_x^\al)_{\al \in \ma}$ satisfies (\ref{basis})-(\ref{bound}).
\end{proof}

For $\al \in \ma$, $\beta \in \mm$ satisfying $\beta>\al$, $x,y \in 100Q^\circ$, and $P_y^\al \in \mpp$ satisfying (\ref{basis}) and (\ref{bound}), we use the Taylor expansion to bound
\begin{align}
    \big|\p^\beta P^\al_y(x)\big| &= \bigg|\p^\beta \sum_{\gamma \in \mm} \p^\gamma P^\al_y(y)(x-y)^\gamma/\gamma! \bigg| \nonumber \\
    &=\bigg| \sum_{\eta + \beta \in \mm} \p^{\eta+\beta}P_y^\al(y)(x-y)^\eta/\eta! \bigg | \nonumber \\
    & \leq C \e/C_0. \label{movepoly}
\end{align}

\begin{prop}
For each $x \in 100 Q^\circ$, there exists a $(C, C\e)$-near triangular matrix $A^x = (A^x_{\al \beta})_{\al, \beta \in \ma}$, and there exists a corresponding family of functions $(\varphi_x^\al)_{\al \in \ma} \subset L^{m,p}(\R^n)$ given by
\begin{equation} 
\varphi_x^\al := \sum_{\beta \in \ma} A^x_{\al \beta} \cdot \varphi^\beta; \label{phi1}
\end{equation}
where $(\varphi^\al)_{\al \in \ma}$ is a family of functions satisfying
\begin{equation}
\|\varphi^\al\|_{\J(0,\mu)} \leq \e/C_0,
    \label{mat0_new}
\end{equation}
and the family $(\varphi_x^\al)_{\al \in \ma}$ also satisfies
\begin{align}
&J_x\varphi_x^\al = P_x^\al; \text{ and} \label{phi2} \\
&\|\varphi_x^\al\|_{\J(0,\mu)} \leq C\e/C_0. \label{phi3}
\end{align}
\label{phiprop}
\end{prop}
\begin{proof}
We follow the proof of Lemma \ref{lem:polybasis}, and let $(\varphi^\al)_{\al \in \ma}$ be the family of functions satisfying (\ref{mat0}). Accordingly, we have $\|\varphi^\al\|_{\J(0,\mu)} \leq \e/C_0$, giving \eqref{mat0_new}.

For $x \in 100 Q^\circ$, let $A^x = (A^x_{\al \beta})_{\al, \beta \in \ma}$ be the $(C, C\e)$-near triangular matrix defined just before (\ref{polydef}). Then define $\varphi_x^\al$ ($\al \in \ma$) as in \eqref{phi1}. Then \eqref{phi2} follows by applying $J_x$ to both sides of \eqref{phi1} and applying \eqref{mat0.1} and (\ref{polydef}). 

Note $|A^x_{\al \beta}| \leq C$ for $\al,\beta \in \ma$, since $A^x$ is $(C,C\epsilon)$-near triangular. Thus, \eqref{phi3} follows by sublinearity of the $\J(0,\mu)$ functional; indeed, $\| F \|_{\J(0,\mu)} = ( \| F \|_{L^{m,p}(\R^n)}^p + \| F \|_{L^p(d \mu)}^p)^{1/p}$ is a norm on $L^{m,p}(\R^n) \cap L^p(d \mu)$.
\end{proof}


\subsubsection{Reduction to Monotonic $\ma$} \label{subsubsec:monotonic}
\begin{define} [Monotonic labels] A collection of multi-indices $\ma \subset \mm$ is monotonic if
\[
\al \in \ma \text{ and } |\gamma| \leq m-1-|\al| \text{ implies } \al + \gamma \in \ma.
\]
If the above property fails, we say that $\ma$ is non-monotonic.
\end{define}

In this section we follow Section 4.2 of \cite{arie3} to deduce the monotonicity of $\ma$ using assumption (\ref{amaxa}) and condition (\ref{amainlmi}) for $\ma$. 

For the sake of contradiction, we assume that $\ma$ is non-monotonic. We will show there exists $\mab<\ma$ such that for all $x \in \supp(\mu)$, $\sigma_J(x, \mu)$ contains an $(\mab,x,\e_0/C_0, 1)$-basis, contradicting (\ref{amaxa}). Thus our proof of the Induction Step is reduced to the case of monotonic $\ma$.

Let $\al_0 \in \ma$, $\gamma \in \mm$ satisfy
\begin{align}
0 < |\gamma| \leq m-1-|\al_0| \text{ and } \bar{\al}:= \al_0 + \gamma \in \mm \setminus \ma \label{alchoice}.
\end{align}
Define $\mab = \ma \cup \{\bar{\al} \}$. Note that $\al_0 < \bar{\al}$, and the only element of $\mab \Delta \ma$ is $\bar{\al}$, which is in $\mab$, so $\mab < \ma$ by definition of the order on multiindex sets.

For fixed $y \in \supp(\mu) \subset \frac{1}{10} Q^{\circ}$, we let $(P_y^\al)_{\al \in A}$ satisfy (\ref{basis}) and (\ref{bound}). We now define
\begin{align}
P_y^{\bar{\al}}:=P_y^{\al_0} \odot_y q^y, \quad \text { where } q^y(x):=\frac{\al_0 !}{\bar{\al} !}(x-y)^\gamma \label{pbar}.
\end{align}
Expanding out this product, we have
\[
P_y^{\bar{\al}}(x) = \frac{\al_0 !}{\bar{\al} !} \sum_{|\omega| \leq m-1-|\gamma|} \frac{1}{\omega !} \p^{\omega} P_y^{\al_0}(y)(x-y)^{\omega+\gamma}.
\]
Note that $\omega=\al_0$ arises in the sum above, thanks to (\ref{alchoice}). Also, the terms with $\omega+\gamma>\bar{\al}=\alpha_0+\gamma$ correspond precisely to $\omega>\al_0$, by definition of the order. The following properties are now immediate from (\ref{basis}) and (\ref{bound}):
\begin{align}
&\p^{\bar{\al}} P_y^{\bar{\al}}(y) = 1; \label{pbar1} \\
&|\p^\beta P_y^{\bar{\al}}(y)| \leq C' \e &&(\beta \in \mm, \; \beta>\bar{\al}); \text{ and} \\
&|\p^\beta P_y^{\bar{\al}}(y)| \leq C' &&(\beta \in \mm). \label{pbar3}
\end{align}
From (\ref{basis}), we have $P_y^{\al_0} \in C \e \cdot \sigma_J(y, \mu)$, so there exists $\varphi \in L^{m,p}(\R^n)$ satisfying  
\begin{align}
&\|\varphi\|_{\J(0,\mu)}  = \left(\|\varphi\|_{L^{m,p}(\R^n)}^p + \int_{\R^n}|\varphi|^p d\mu \right)^{1/p} \leq C \e ; \text{ and} \label{fibbd0} \\
&J_y\varphi = P_y^{\al_0}. \label{fijet}
\end{align}
Let $\theta \in C_0^\infty(Q^\circ)$ satisfy $\theta|_{0.99Q^\circ} = 1$ and $|\p^\al \theta (x)| \leq C$ for $x \in Q^\circ$ and $|\al| \leq m$. Define $\fib: \R^n \to \R$ as $\fib:= \theta \cdot (\varphi q^y) + (1-\theta) \cdot P_y^{\bar{\al}}$. Then $J_y\fib = J_y( \varphi q^y) = P_y^{\al_0} \odot_y q^y = P_y^{\bar{\al}}$, since $y \in \frac{1}{10}Q^\circ$. Note that $\| q^y \|_{L^\infty(Q^\circ)} \leq C$. Thus, from (\ref{aqsubset}) and (\ref{fibbd0}) we deduce
\begin{align*}
\int_{\R^n} |\fib|^p d\mu \leq \int_{\frac{1}{10} Q^\circ} |\varphi q^y|^p d\mu \leq C \int_{\frac{1}{10} Q^\circ} |\varphi|^p d\mu \leq C (C \e)^p.
\end{align*}
We will show $\| \fib \|_{L^{m,p}(\R^n)} \leq C \e$, implying that $\| \fib \|_{\J(0, \mu)} \leq C' \e$ and $P_y^{\bar{\al}} \in C' \e \cdot \sigma_J (y, \mu)$. 

Since $\bar{\varphi}$ agrees with an $(m-1)^{\text {rst}}$ degree polynomial on $\R^n \setminus Q^\circ$, and since the partial derivatives of $\theta$ are uniformly bounded, we have
\begin{align}
\| \fib \|_{L^{m,p}(\R^n)} &= \| \fib \|_{L^{m,p}(Q^\circ)}  \nonumber \\
&= \| \fib -  P_y^{\bar{\al}}\|_{L^{m,p}(Q^\circ)} \nonumber \\ 
&= \| \theta \cdot (\varphi q^y - P_y^{\bar{\al}}) \|_{L^{m,p}(Q^\circ)} \nonumber \\
& \leq C \sum_{|\beta| \leq m} \| \p^\beta (\varphi q^y - P_y^{\bar{\al}}) \|_{L^p(Q^\circ)}. \label{fibbd1}
\end{align}
As a consequence of (\ref{pbar}) and (\ref{fijet}), $P_y^{\bar{\al}} = J_y( \varphi q^y)$. Hence, by the Sobolev Inequality, for $x \in Q^\circ$, $|\beta| \leq m$,
\begin{align}
|\p^\beta (\varphi q^y - P_y^{\bar{\al}})(x)| \leq C \| \varphi q^y \|_{L^{m,p}(Q^{\circ})}. \label{fibbd2}
\end{align}
For $x \in Q^\circ$, $|\beta|=m$, because $q^y$ is a degree $|\gamma|$ polynomial, we have
\begin{align}
    |\p^\beta (\varphi q^y)(x)| &\leq \sum_{\omega + \omega' = \beta} |\p^{\omega'} \varphi (x)| \cdot |\p^\omega q^y(x)| \nonumber \\
    &= |\p^\beta \varphi(x)||q^y(x)| +  \sum_{\substack{\omega + \omega' = \beta,\\ 0<|\omega| \leq |\gamma|}} \left( |\p^{\omega'} ( \varphi -P_y^{\al_0})(x)| + |\p^{\omega'}P_y^{\al_0}(x)| \right) \cdot |\p^\omega q^y(x)|. \label{fibbd3}
\end{align}
Furthermore, for $(\omega,\omega')$ as in the sum above, $0 < |\omega| \leq |\gamma|$ implies $m-1 \geq |\omega'| \geq m - |\gamma| > |\al_0|$ (see (\ref{alchoice})). Recall estimate (\ref{movepoly}): $|\p^{\omega'} P_y^{\al_0}(x)| \leq C \e$, and note that by construction, we have $\| \partial^\omega q^y \|_{L^\infty(Q^\circ)} \leq C$. This, in combination with the Sobolev Inequality, (\ref{fibbd0}), and (\ref{fijet}) allows us to further reduce the sum on the right hand side of inequality (\ref{fibbd3}). When $x \in Q^\circ$ and $|\beta| =m$:
\begin{align*}
    \sum_{\substack{\omega + \omega' = \beta,\\ 0<|\omega| \leq |\gamma|}} \left( |\p^{\omega'} ( \varphi -P_y^{\al_0})(x)| + |\p^{\omega'}P_y^{\al_0}(x)| \right) \cdot |\p^\omega q^y(x)| \leq C \cdot \left( \| \varphi \|_{L^{m,p}(Q^{\circ})} + C \e \right) \leq C' \e.
\end{align*}
Substituting this into (\ref{fibbd3}), we have for $x \in Q^\circ$ and $|\beta|=m$, 
\begin{align*}
    |\p^\beta (\varphi q^y)(x)| &\leq |\p^\beta \varphi(x)||q^y(x)| +  C' \e.
\end{align*}
By integrating over $x \in Q^\circ$ and again applying (\ref{fibbd0}) and $\| q^y \|_{L^\infty(Q^\circ)} \leq C$, we deduce that $\|\varphi q^y \|_{L^{m,p}(Q^\circ)} \leq C \e $. Thus, keeping in mind (\ref{fibbd1}) and (\ref{fibbd2}), we have
\begin{align*}
\| \fib \|_{L^{m,p}(\R^n)} &\leq C \| \varphi q^y \|_{L^{m,p}(Q^{\circ})} \nonumber \\
&\leq C'\e.
\end{align*}
Because $J_y \bar{\varphi} = P^{\bar{\al}}_y$, this completes our proof that 
\begin{align}
P_y^{\bar{\al}} \in C' \e \cdot \sigma_J (y, \mu). \label{pbar4}
\end{align}

Due to (\ref{basis}), (\ref{bound}), (\ref{pbar1})-(\ref{pbar3}), and (\ref{pbar4}), the family $(P_y^\al)_{\al \in \mab}$ satisfies (\ref{lna1})-(\ref{lna3}) with $R$ equal to a universal constant, $\e_2 = C' \e$, and $\sigma = \sigma_J(y, \mu)$. We may assume $C' \e < c_2$, where $c_2$ comes from Lemma \ref{linalg}. Then the hypotheses of Lemma \ref{linalg} hold; hence, $\sigma_J(y, \mu)$ contains an $(\mab, y, C_2 C' \e, 1)$-basis for $C_2 \geq 1$, a universal constant. Recall our assumption that $\e$ is less than a small enough constant determined by $m$ and $n$. We may assume $\e < \e_0/(C'C_0C_2)$, so $\sigma_J(y, \mu)$ contains an $(\mab, y, \e_0/C_0, 1)$-basis. Since $y \in \supp(\mu)$ is arbitrary, this contradicts (\ref{amaxa}). 

This completes the proof by contradiction, and establishes that $\ma \subset \mm$ is monotonic.

\section{The Calder\'{o}n-Zygmund Decomposition}
\label{sec:cz_decomp}

Recall we have fixed a multi-index set $\ma \subset \mm$, $\ma \neq \mm$ and a Borel regular measure $\mu$ on $\R^n$ satisfying the conditions (\ref{aqsubset}), (\ref{amainlmi}), and (\ref{amaxa}). By the results of Section \ref{subsubsec:monotonic}, we deduce that $\ma$ is monotonic. Our goal is to prove the Extension Theorem for $(\mu, \delta = 1/10)$. Recall we have defined constants $\e_0$ in (\ref{e0def}) and $C_0$ in Lemma \ref{sjins}.

\subsection{Defining the Decomposition}

\begin{define}[OK Cubes] A dyadic cube $Q \subset Q^{\circ} = (0,1]^n$ is OK if there exists $\mab < \ma $ such that for every $x \in \supp(\mu) \cap 3Q$, $\sigma_J(x,\mu|_{3Q})$ contains an $(\mab,x, \e_0/C_0, 30 \delta_Q)$-basis. \label{ok_defn}
\end{define}

\begin{define}[Calder\'{o}n-Zygmund Cubes]
A dyadic cube $Q \subset Q^\circ$ is CZ if $Q$ is OK and every dyadic cube $Q' \subset Q^\circ$ that properly contains $Q$ is not OK. 
\end{define}

If $Q,Q' \subset Q^\circ$ are CZ and $Q \neq Q'$ then $Q \cap Q' = \emptyset$. Indeed, since CZ cubes are dyadic cubes, either $Q \subsetneq Q'$, $Q' \subsetneq Q$, or $Q \cap Q' \neq \emptyset$. The first case is impossible, since, according to the definition of CZ cubes if $Q$ is CZ and $Q \subsetneq Q'$ then $Q'$ is not OK, hence $Q'$ is not CZ. Similarly, the second case is impossible. Therefore, $Q \cap Q' \neq \emptyset$, as claimed.

We write $CZ^\circ  = \{Q_i\}_{i \in I}$ to denote the collection of all CZ cubes. The index set $I$ may be countable or finite. In contrast to \cite{arie3}, the CZ cubes do not necessarily form a partition of $Q^\circ$; however, if $\mu$ is a finite measure the CZ cubes do form a partition of $Q^\circ$ (see Lemma \ref{nottrivial} below).

\begin{lem} \label{okmonlem}
Suppose $Q, Q' \subset Q^\circ$ are dyadic cubes. Suppose $Q'$ is OK and $3 Q \subset 3 Q'$. Then $Q$ is OK. 
\end{lem}
\begin{proof}
We have $\sigma_J(x, \mu|_{3Q'}) \subset \sigma_J(x, \mu|_{3Q})$; indeed, this follows by the definition of $\sigma_J(\cdot, \cdot)$, and because $\|F \|_{\J(0,\mu|_{3Q})} \leq \|F \|_{\J(0,\mu|_{3Q'})}$ for $3Q \subset 3 Q'$ and $F \in L^{m,p}(\R^n)$. Because $Q'$ is OK, there exists $\mab<\ma$ such that for all $x \in \supp(\mu) \cap 3Q'$, $\sigma_J(x, \mu|_{3Q'})$ contains an $(\mab, x, \e_0/C_0, 30 \delta_{Q'})-$basis. Since $\delta_{Q} \leq \delta_{Q'}$, and thanks to (\ref{monotone}), we have that $\sigma_J(x, \mu|_{3Q})$ contains an $(\mab, x, \e_0/C_0, 30 \delta_{Q})-$basis for any $x \in 3Q$. This indicates $Q$ is OK.
\end{proof}

Because of Lemma \ref{sjins}, for every $Q \in CZ^\circ$, there exists $\mab < \ma $ such that for every $x \in \supp(\mu) \cap 3Q$, $\sigma(\mu|_{3Q}, \delta_Q)$ contains an $(\mab,x, \e_0, \delta_Q)$-basis.

\begin{lem}
Let $Q \subset Q^\circ$ be a dyadic cube. If $\mu(3Q)^{1/p} \leq \frac{\e_0}{C_0} (30 \delta_Q)^{n/p-m}$ then $Q$ is OK.  \label{smallok}
\end{lem}

\begin{proof}
Suppose $Q \subset Q^\circ$ satisfies $\mu(3Q)^{1/p} \leq \e_0/C_0 (30 \delta_Q)^{n/p-m}$. Let $F^0 = P^0 = 1$. Then for $y \in \supp(\mu) \cap 3Q$, $J_yF^0 = P^0$, and 
\[
\|F^0\|_{\J(0,\mu|_{3Q})} = \left(\| F^0 \|_{L^{m,p}(\R^n)}^p + \| F^0 \|_{L^p(d \mu|_{3Q})}^{p} \right)^{1/p}= \mu(3Q)^{1/p} \leq \frac{\e_0}{C_0} (30\delta_Q)^{n/p-m}.
\]
So for all $y \in \supp(\mu) \cap 3Q$, $P^0 \in \frac{\e_0}{C_0}  (30 \delta_Q)^{n/p -m} \cdot \sigma_J (y, \mu|_{3Q})$, and $(P^0)$ forms a $(\{0\}, y, \e_0/C_0, 30 \delta_Q)$-basis for $\sigma_J(y, \mu|_{3Q})$. Because $0$ is the minimal element of $\mm$ and $0 \notin \ma$ ($\ma$ is monotonic and $\ma \neq \mm$ because $\mm$ is minimal in the order), $\{0\} < \ma$. Hence, $Q$ is OK. 
\end{proof}

\begin{lem}
The decomposition $CZ^\circ$ is not equal to $\{Q^\circ\}$. In particular, each $Q \in CZ^\circ$ has a (unique) dyadic parent $Q^+ \subset Q^\circ$. 

Furthermore, if $\mu$ is finite, the CZ cubes form a non-trivial, finite partition of $Q^{\circ}$ into dyadic cubes. \label{nottrivial}
\end{lem}

\begin{proof}
If $Q^{\circ}$ is OK, then there exists $\mab < \ma $ such that for every $x \in \supp(\mu)$, $\sigma_J(x,\mu)$ contains an $(\mab,x, \e_0/C_0, 1)$-basis, contradicting (\ref{amaxa}). So $Q^\circ \notin CZ^\circ$.

Suppose $\mu$ is finite. Define $H: \R_+ \to \R$ as 
\begin{align*}
H(\delta) := \max_{x \in Q^{\circ}}\{ \mu\left(B(x, 3\delta) \right)^{1/p} \cdot \delta^{m-n/p} \}.
\end{align*}
Because $\mu$ is finite, $\lim_{\delta \rightarrow 0^+} H(\delta) = H(0) = 0$. Further, H is non-increasing, so there exists $\delta_1>0$ such that if $\delta < \delta_1$ then $H(\delta) \leq H(\delta_1) < \frac{\e_0}{30^{m-n/p} C_0}$. Let $Q \subset Q^\circ$ satisfy $\delta_Q < \delta_1$, and fix $x \in Q$. Then 
\begin{align*}
\mu(3Q)^{1/p} &\leq \mu(B(x, 3\delta_Q))^{1/p} \\
&\leq H(\delta_Q) \cdot \delta_Q^{n/p-m}\\
&\leq \frac{\e_0}{C_0} (30\delta_Q)^{n/p-m}.
\end{align*}
We apply Lemma \ref{smallok} to deduce that $Q$ is OK, so every dyadic subcube of $Q^{\circ}$ of sidelength smaller than $\delta_1$ is OK. Therefore, $CZ^\circ$ is a finite partition of $Q^\circ$. 
\end{proof}

As a corollary of Lemma \ref{smallok}, we can bound below the measure of the dyadic parent of any cube in $CZ^\circ$:

\begin{cor}
For $Q \in CZ^\circ$, we have 
\begin{align*}
\mu(3Q^+)^{1/p} > (\e_0/C_0) \cdot (30 \delta_Q^+)^{n/p-m}. 
\end{align*}
\end{cor}

Two cubes, $Q_i, Q_j \in CZ^\circ$ are called ``neighbors" if their closures satisfy $\Cl(Q_i) \cap \Cl(Q_j) \neq \emptyset$. (In particular, any CZ cube is neighbors with itself.) We denote this relation by $Q_i \lra Q_j$, or $i \lra j$ ($i,j \in I$).

\begin{lem}[Good Geometry]
If $Q,Q' \in CZ^\circ$ satisfy $Q \lra Q'$, then 
\begin{align}
\frac{1}{2} \delta_Q \leq \delta_{Q'} \leq 2 \delta_Q. \label{goodgeo}
\end{align}
\end{lem}

\begin{proof}
For the sake of contradiction, suppose that $Q,Q' \in CZ^\circ$ satisfy
\[
\Cl(Q) \cap \Cl(Q') \neq \emptyset \text{ and } 4 \delta_Q \leq \delta_{Q'}.
\]
Since $Q^+,Q'$ are dyadic cubes and $\delta_{Q^+} = 2 \delta_Q \leq \frac{1}{2} \delta_{Q'}$, we have $3Q^+ \subset 3Q'$. Because $Q'$ is OK, by Lemma \ref{okmonlem}, we have that $Q^+$ is OK, contradicting $Q \in CZ^\circ$.
\end{proof}

\begin{lem}[More Good Geometry]
For each $Q \in CZ^{\circ}$, the following properties hold.
\begin{align}
    &\text{If } Q' \in CZ^{\circ} \text{ is such that }(1.3) Q' \cap(1.3) Q \neq \emptyset, \text{ then }Q \lra Q'. \text{ Consequently:}\nonumber \\ 
    &\text{Each point } x \in \R^n \text{ belongs to at most } C(n) \text{ of the cubes } 1.3Q \text{ with }Q \in CZ^{\circ}. \label{g1} \\
    &\text{If } \Cl(Q) \cap \partial Q^{\circ} \neq \emptyset, \text{ then } \delta_Q \geq \frac{1}{20} \delta_{Q^\circ}. \label{g2}
\end{align}

\label{goodgeo2}
\end{lem}

\begin{proof} Proof of (\ref{g1}): Fix $Q \in CZ^\circ$. Suppose that $Q' \in CZ^\circ$ does not neighbor $Q$. Let $\delta = \max \{ \delta_Q/2, \delta_{Q'}/2 \}$. Then (\ref{goodgeo}) implies that $\dist(Q, Q') \geq \delta$. Thus,
\[
(1.3)Q \cap (1.3)Q' \subset B(Q, (0.3)\delta) \cap B(Q', (0.3)\delta) = \emptyset.
\]

Proof of (\ref{g2}): For the sake of contradiction, suppose $\Cl(Q) \cap \p Q^\circ \neq \emptyset$ and $\delta_Q < \frac{1}{20}\delta_{Q^\circ}$. Therefore, $9Q \subset \R^n \setminus \frac{1}{10}Q^\circ$. Note that $3Q^+ \subset 9Q$. Because $\supp(\mu) \subset \frac{1}{10} Q^\circ$, we have $\mu(3Q^+)  = 0$, so by Lemma \ref{smallok}, $Q^+$ is OK. This is a contradiction.
\end{proof}

\noindent \textbf{Remark. } For $\eta \in [1,100]$ and $Q_i \in CZ^\circ$, $\eta Q_i \cap Q^\circ$ is a $C$-non-degenerate rectangular box, for $C=C(n)$, satisfying 
\begin{align*}
\diam(\eta Q_i \cap Q^\circ) \simeq \delta_{Q_i}.
\end{align*}

\begin{define} \label{keycuberef}
A cube $Q \in CZ^\circ$ is called a keystone cube if for any $Q' \in CZ^\circ$ with $Q' \cap 100Q \neq \emptyset$, we have $\delta_Q \leq \delta_{Q'}$.
\end{define}

We will use the following notation: 
\begin{align*}
    & CZ^\circ = \{Q_i\}_{i \in I} := \{ \text{the collection of CZ cubes} \} \in \Pi(Q^\circ);\\
    &K_{CZ} : = \bigcup_{Q \in CZ^\circ} Q ; \text{ and}\\
    & CZ_{key}= \{Q_s\}_{s \in \bar{I}} := \{ Q_s \in CZ^\circ: Q_s \text{ is a keystone cube} \}.
\end{align*}
For $i \in I$ ($I$ is the indexing set of $CZ^\circ$ as above), let $x_i$ denote the center of $Q_i$, i.e., $x_i := \ctr(Q_i)$. We denote the set of CZ basepoints by
\begin{align*}
\mathfrak{B}_{CZ} = \{ x_i \}_{i \in I}.    
\end{align*}

\begin{lem}
Let $Q_s$ be a keystone cube. Then
\[
\left| \{ s' \in \bar{I}: 10Q_s \cap 10Q_{s'} \neq \emptyset \} \right| \leq C.
\]
\label{10q}
\end{lem}

\begin{proof}
Let $s, s' \in \bar{I}$ with $10Q_s \cap 10 Q_{s'} \neq \emptyset$. Suppose, without loss of generality, $\delta_{Q_s} \leq \delta_{Q_{s'}}$. Then $10Q_s \cap 10Q_{s'} \neq \emptyset$ implies $100 Q_{s'} \cap Q_s \neq \emptyset$. But because $Q_{s'}$ is keystone, this implies $\delta_{Q_{s'}} \leq \delta_{Q_s}$. So $\delta_{Q_s} = \delta_{Q_{s'}}$. There are at most $C$ dyadic cubes $Q_{s'}$ that satisfy this condition and $10Q_s \cap 10Q_{s'} \neq \emptyset$ for fixed $s$. This completes the proof.
\end{proof}

\subsection{Keystone Points}
\label{subsec:kpt}
\begin{define}\label{keypt:def}
We define the set of keystone points as $K_p:= Q^\circ \setminus K_{CZ}$.
\end{define}

\begin{lem}
The set $K_p$ is closed. 
\label{kpclosed}
\end{lem}

\begin{proof}
Let $x \in K_p$. Then every dyadic cube containing $x$ is \underline{not OK}. As a consequence of Lemma \ref{smallok}, 
\begin{align}
\mu(B(x, \eta)) = \infty \quad \quad (\mbox{for all } \eta>0, \; x \in K_p). \label{kpinf} 
\end{align}
To see this, suppose for sake of contradiction that $\mu(B(x, \eta)) \leq A < \infty$ for some $\eta > 0$. Then, consider a dyadic cube $Q \subset Q^\circ$ with $x \in Q$ and $3Q \subset B(x,\eta)$. Then $\mu(3Q) \leq A$. If $\delta_Q$ is sufficiently small, depending on $A,m,n,p$, it follows that $\mu(3Q) \leq (\epsilon_0/C_0)^p (30 \delta_Q)^{n-mp}$, since $n-mp < 0$, and hence $Q$ is OK by Lemma \ref{smallok}. Thus, $x$ belongs to an OK cube, so $x$ belongs to a CZ cube, hence $x \in K_{CZ}$, contradicting that $x \in K_p$.

In combination with (\ref{aqsubset}), (\ref{kpinf}) implies $K_p \subset \Cl( \frac{1}{10} Q^\circ) \subset \frac{1}{9} Q^\circ$. 

Let $Q \in CZ^\circ$. We claim that for any $x \in 2Q \cap Q^\circ$, there exists $Q' \in CZ^\circ$ satisfying $x \in Q'$ -- in particular, $2 Q \cap Q^\circ \subset K_{CZ}$. To see this, let $Q'' \subset Q^\circ$ be a dyadic cube satisfying $x \in Q''$ and $\delta_{Q''}< \delta_Q/100$. Then $3Q'' \subset 3Q$. Since $Q$ is OK, we have that $Q''$ is OK, thanks to Lemma \ref{okmonlem}. Hence, by definition of the CZ decomposition, $Q''$ must have a dyadic ancestor $Q'$ with $Q ' \in CZ^\circ$. Then $x \in Q'' \subset Q'$, completing the proof of the claim.

Let $x \in K_{CZ}$ be arbitrary. Then $x \in Q$ for some $Q \in CZ^\circ$. Note that
\[
B(x,\delta_Q/3) \cap Q^\circ \subset 2 Q \cap Q^\circ  \subset K_{CZ}.
\]
Hence, $K_{CZ}$ is relatively open in $Q^\circ$. Thus, $K_p = Q^\circ \setminus K_{CZ}$ is relatively closed in $Q^\circ$. Since $K_p \subset \frac{1}{9} Q^\circ$, we have that $K_p$ is closed. 

\end{proof}

Below, we write $|S|$ for the Lebesgue measure of a measurable set $S \subset \R^n$.
\begin{lem}
If $\ma \neq \emptyset$ then $|K_p| = 0$. \label{measkplem}
\end{lem}

\begin{proof}
Suppose for sake of contradiction that $|K_p|>0$ and $\mathcal{A} \neq \emptyset$. Recall from (\ref{basis}) and (\ref{bound}) in Section \ref{subsubsec:aux_poly} that for all $x \in Q^\circ$, there exists a family of polynomials,  $(P_x^\al)_{\al \in \ma}$ such that $(P_x^\al)_{\al \in \ma}$ forms an $(\ma,x, C\e/C_0, 1)$-basis for $\sigma_J(x,\mu)$ satisfying 
\begin{align}
    |\p^\beta P_x^\al(x)| \leq C  \quad \quad (\al \in \ma, \; \beta \in \mm). \label{bdkp}
\end{align} 
Let $x_0 \in K_p$ be a Lebesgue point of $\goodchi_{K_p}$; then $\lim_{r \to 0} \frac{|K_p \cap B(x_0,r)|}{|B(x_0,r)|} = 1$. We will fix universal constants $c_1, c_2 \in (0,1)$ (determined only by $m$, $n$, and $p$) momentarily. Choose any $r \in (0,1)$ satisfying $r^{1-n/p} < c_2/(2 \e)$ and $\frac{|K_p \cap B(x_0,r)|}{|B(x_0,r)|} > 1 - c_1^n$. Since $\mathcal{A} \neq \emptyset$, we may fix $\al \in \ma$. Let 
\begin{align*}
    &M: = \max_{\beta \in \mm} \max_{z \in B(x_0,r)}  |\p^\beta P_{x_0}^\al(z)|r^{|\beta|},\\
    &y := \arg \max_{z \in B(x_0,r/2)} |P_{x_0}^\al(z)|.
\end{align*}
By Bernstein's Inequality for polynomials, there exists a universal constant $C>0$, such that \\$\| \partial^\beta P(z) \|_{L^\infty(B(x_0,r))} \leq C r^{-|\beta|} \| P \|_{L^\infty (B(x_0,r/2))}$ for all multi-indices $\beta \in \mm$ and $(m-1)^{\text {rst}}$ degree polynomials $P \in \mathcal{P}$. Thus, by the definition of $M$, 
\begin{equation}\label{eqn:lowerbd1}
|P_{x_0}^\al(y)| = \| P^\alpha_{x_0} \|_{L^\infty(B(x_0,r/2))} \geq c'M
\end{equation}
for a universal constant $c' \in (0,1)$. For any $\eta \in (0,1/2)$ we have $ B(y,\eta r) \subset B(x_0,r)$. Thus, since $|\nabla P_{x_0}(z)| \leq \sqrt{n} M/r$ for $z \in B(x_0,r)$ (see the definition of $M$), by the mean value theorem,
\[
|P_{x_0}^\al(y) - P_{x_0}^\al(x)| \leq ( \sqrt{n} M/r) |y-x| \leq (\sqrt{n} M/r) \eta r  = \sqrt{n} M \eta \;\; \mbox{for }   x \in B(y,\eta r).
\]
We fix the universal constant $c_1 \in (0,1/2)$ given by $c_1 = \frac{c'}{4\sqrt{n}}$. Taking $\eta = c_1$ in the above, and using \eqref{eqn:lowerbd1}, we learn that $B(y,c_1 r) \subset B(x_0,r)$, and
\[
|P_{x_0}^\al(x)| \geq |P_{x_0}^\al(y)| - | P_{x_0}^\al(x) - P_{x_0}^\al(y)| >c' M/2 \;\; \mbox{for }  x \in B(y, c_1 r).
\]
Furthermore, by the basis property of $(P^\al_{x_0})$, we have $\p^\al P_{x_0}^\al(x_0) = 1$, and so $M \geq r^{|\al|} \geq r^{m-1}$ (see the definition of $M$). Hence,
\begin{align}
    |P_{x_0}^\al(x)|>c' r^{m-1} /2 \;\; \text{for }x \in B(y, c_1 r). \label{ballbd}
\end{align}

Because $P_{x_0}^\al \in C \e \cdot \sigma_J(x_0,\mu)$, there exists $\varphi_{x_0}^\al \in L^{m,p}(\R^n)$ satisfying $J_{x_0}\varphi_{x_0}^\al = P_{x_0}^\al$ and $\|\varphi_{x_0}^\al\|_{L^{m,p}(\R^n)} \leq \|\varphi_{x_0}^\al\|_{\J(0, \mu)} \leq C \e$. If $x \in B(y,c_1r)$  then $x \in B(x_0,r)$, so $|x-x_0| \leq r$. By the Sobolev Inequality and (\ref{ballbd}),
\begin{align}
    |\varphi_{x_0}^\al(x)| &\geq |P_{x_0}^\al(x)| - C \|\varphi_{x_0}^\al\|_{L^{m,p}(\R^n)} r^{m-n/p} \nonumber \\
    &> c'r^{m-1}/2 - C' \e r^{m-n/p} \nonumber \\
    &= C'r^{m-1} (c_2 - \e r^{1-n/p}) \geq C' c_2 r^{m-1}/2, \;\; \mbox{for } x \in B(y,c_1 r),\label{posbd}
\end{align}
where we now fix the universal constant $c_2 := c'/(2C')$ so that the equality in the last line is valid, and we recall our choice of $r$ satisfying $r^{1-n/p} <  c_2/(2 \e)$ to justify the inequality in the last line.

Because $\frac{|K_p \cap B(x_0,r)|}{|B(x_0,r)|} > 1 - c_1^n$ and $B(y,c_1 r) \subset B(x_0,r)$, we have that $K_p \cap \inte(B(y,c_1 r)) \neq \emptyset$. Thus, by (\ref{kpinf}) we learn that $\mu(B(y, c_1 r))=\infty$. In combination with (\ref{posbd}), this implies
\begin{align*}
    \int_{B(y,c_1 r)} |\varphi_{x_0}^\al|^p d \mu = \infty,
\end{align*}
contradicting $\|\varphi_{x_0}^\al\|_{\J(0,\mu)} \leq C \e$. This completes the proof by contradiction, and the lemma follows.
\end{proof}

\begin{lem}
For $Q_i \in CZ^\circ$, we have $\inte(3Q_i) \cap K_p = \emptyset$, and consequently $\delta_{Q_i} \leq \dist(Q_i,K_p)$. \label{kpdist}
\end{lem}
\begin{proof}
For sake of contradiction suppose there exists $Q_i \in CZ^\circ$ with $\inte(3 Q_i) \cap K_p \neq \emptyset$. Then there exists a dyadic cube $Q$ satisfying $3Q \subset 3Q_i$ and $Q \cap K_p \neq \emptyset$. The cube $Q_i$ is OK, thus $Q$ is OK, thanks to Lemma \ref{okmonlem}. Thus, $Q \subset K_{CZ}$, so $Q \cap K_p$ must be empty, a contradiction. 
\end{proof}

\begin{lem}
There exist universal constants $C \geq 1$ and $c \in (0,1)$ such that the following holds. Let $Q \in CZ^\circ$. Then there exists a sequence of cubes, $\Se(Q) \subset CZ^\circ$, that either is (i) finite, satisfying, $\Se(Q) = \{Q^k\}_{k=1}^L$, where $Q^L$ is a keystone cube, or (ii) infinite, satisfying $\Se(Q) = \{Q^k\}_{k \in \N}$ and for $\ctr(Q^k) = x^k$, $\lim_{k \to \infty} x^k = x' \in K_p$. Regardless, the sequence $ \Se(Q) = \{Q^k\}$ satisfies
\begin{align}
&Q^{1} = Q, \label{ase0} \\
&Q^k \lra Q^{k+1} && (\mbox{for } Q^k, Q^{k+1} \in \Se(Q)); \text{ and} \label{ase1}\\
&\delta_{Q^k} \leq C\cdot c^{k-\ell} \delta_{Q^\ell} &&(\mbox{for } Q^k, Q^{\ell} \in \Se(Q), \; k \geq \ell). \label{ase2}
\end{align} \label{ageolm}
\end{lem}

\begin{proof}
Let $Q \in CZ^\circ$. If $Q$ is a keystone cube, then $\Se(Q) = \{Q^1 = Q\}$ satisfies the conclusion of the lemma.

Suppose $Q$ is not a keystone cube; then there exists $Q' \in CZ^\circ$ satisfying $100Q \cap Q' \neq \emptyset$ and $\delta_{Q'} \leq \frac{1}{2} \delta_Q$. Let $Q^{1,0}:=Q$ and
\begin{align*}
Q^{2,0} \in \arg \min_{Q'} \left\{\dist(Q^{1,0}, Q'):\begin{aligned} &Q' \in CZ^\circ, \; 100Q^{1,0} \cap Q' \neq \emptyset,& \\ & \text{and } \delta_{Q'} \leq (1/2) \delta_{Q^{1,0}} \end{aligned} \right\}.
\end{align*}
We call $Q^{2,0}$ constructed by this procedure a ``junior partner'' of $Q^{1,0}$.

Let $s: [0,1] \to \R^n$ be an affine map with $s(0) \in \Cl(Q^{1,0})$, $s(1) \in \Cl(Q^{2,0})$, and $|s(1) - s(0)|_2 = \dist_2(Q^{1,0}, Q^{2,0})$. Then $s((0,1))$ meets finitely many CZ cubes, $Q^{1,1}, \dots , Q^{1,K_1}$, where $K_1<C$, $Q^{1,k} \lra Q^{1,k+1}$, $Q^{1,K_1} \lra Q^{2,0}$, and $\delta_Q \leq \delta_{Q^{1,i}} \leq C \delta_Q$ for $i \in 1,\dots,K_1$, and $\delta_{Q^{2,0}} \leq \frac{1}{2} \delta_{Q^{1,0}}$. Hence, after removing repeated cubes, $\{Q=Q^{1,0}, Q^{1,1}, \dots , Q^{1,K_1}, Q^{2,0}\}$ satisfies $Q^{1,k} \lra Q^{1,k+1}$, and $\delta_{Q^{1,k}} \leq C\cdot c^{k-l} \delta_{Q^{1,l}}$ for $0 \leq l \leq k$. If $Q^{2,0}$ is a keystone cube, we stop, producing a finite sequence $\Se(Q)$ (which we relabel $\{Q^k\}_{k = 1}^L$). Otherwise, we repeat this process. 

We construct $\Se(Q)$ by concatenating the sequences of cubes connecting successive junior partners, $\Se(Q) := \{ Q^{1,0}, \dots, Q^{1, K_1}, Q^{2,0},\dots,Q^{2,K_2}, Q^{3, 0}, \dots \}$. Relabel $\Se(Q) = \{Q^k\}_{k=1}^L$ if a junior partner is keystone, stopping the process and producing a finite sequence that terminates at a keystone cube. Otherwise, if no junior partner is keystone, relabel $\Se(Q)= \{Q^k\}_{k \in \N}$. By construction, $\Se(Q)$ satisfies (\ref{ase0}) and (\ref{ase1}), Also, $\Se(Q)$ satisfies  (\ref{ase2}) because the subsequence of junior partners in $\Se(Q)$ satisfies $\delta_{Q^{k,0}} \leq 2^{j-k} \delta_{Q^{j,0}}$ for $k > j$, and because there are at most $C$ many cubes in $\Se(Q)$ connecting consecutive junior partners.

It remains to verify conclusion (ii) of the lemma in the event that $\Se(Q)$ is infinite. Suppose $\Se(Q)$ is infinite. Let $\Se(Q)= \{Q^k\}_{k \in \N}$. For each $k \in \N$, define $x^k := \ctr(Q^k)$. Because of (\ref{ase2}), 
\begin{align}
    \lim_{k \to \infty} \delta_{Q^k} = 0. \label{deltashrink}    
\end{align}
Because of (\ref{ase1}) and (\ref{ase2}), we have for $j>k$,
\begin{align*}
    |x^k-x^j| &\leq 3 \sum_{i =k}^{j} \delta_{Q^i} \\
    & \leq 3 \sum_{i=k}^j C c^{i-k} \delta_{Q^k} \\
    &\lesssim \delta_{Q^k}.
\end{align*}
In light of (\ref{deltashrink}), this proves the sequence $\{x^k\}_{k \in \N}$ is Cauchy, and there exists $x' \in \Cl(Q^\circ)$ such that $x^k \overset{k \to \infty}{\longrightarrow} x'$. Because of (\ref{g2}), $x' \notin \p Q^\circ$. 

It remains to show that $x' \in K_p$. Suppose for sake of contradiction that $x' \notin K_p$; then $x' \in Q'$ for some $Q' \in CZ^\circ$. Since $x^k = \ctr(Q^k) \rightarrow x' \in Q'$ as $k \rightarrow \infty$, there exists $K_0 \in \N$ such that for $k>K_0$, $Q^k$ belongs to the neighbor set of $Q'$, $Q^k \in \mathcal{N}(Q') := \{ Q'' \in CZ^\circ : Q'' \leftrightarrow Q'\}$. But by Good Geometry of $CZ^\circ$, any cube $Q'' \in \mathcal{N}(Q')$ has sidelength $\delta_{Q''} \geq \frac{1}{2} \delta_{Q'}$, contradicting (\ref{deltashrink}).  Hence, $x' \in K_p$, as desired.

\end{proof}

\subsection{Partition of Unity}\label{subsec:pou}
We have defined the collection of CZ cubes, $CZ^\circ = \{Q_i\}_{i \in I}$, such that $K_{CZ} = \bigcup_{i \in I} Q_i$. Due to the Good Geometry of $CZ^\circ$, (\ref{goodgeo}), (\ref{g1}), and (\ref{g2}), there exists a partition of unity, $\{\theta_i\}_{i \in I} \subset C^\infty(K_{CZ})$, satisfying 
\begin{description}
\setlength{\itemindent}{0em}
\item[(POU1)] $0 \leq \theta_i \leq 1$;
\item[(POU2)] $\theta_i$ vanishes on $K_{CZ} \setminus (1.1)Q_i$;
\item[(POU3)] $\left|\p^{\al} \theta_i\right| \leq C \delta_{Q_i}^{-|\al|}$ whenever $|\al| \leq m$; and
\item[(POU4)] $\sum_{i \in I}  \theta_i \equiv 1$ on $K_{CZ}$.
\end{description}
For the construction of such a partition of unity, see \cite{st}, page 170.

\subsection{Local Extension Operators} \label{subsec:loc_ext_op}
In this section, we apply the inductive hypothesis to construct local extension operators for functions defined on Borel subsets $E_i$ of $3Q_i$. 

Let $Q_i \in CZ^\circ$ and $E_i \subset 3Q_i$ be Borel. Because $Q_i$ is OK, there exists $\ma_i < \ma$ such that for all $x \in \supp(\mu) \cap 3Q_i$, $\sigma_J(x, \mu|_{3Q_i})$ contains an $(\ma_i, x, \e/C_0, 30 \delta_{Q_i})$-basis. Because $E_i \subset 3Q_i$ we have $\sigma_J(x, \mu|_{3Q_i}) \subset \sigma_J(x, \mu|_{E_i})$. Since $\e < \e_0$, for all $x \in \supp(\mu|_{E_i})$, 
\[
\sigma_J(x,\mu|_{E_i}) \mbox{ contains an } \left(\ma_i, x, \e_0/C_0, 10\cdot(3\delta_{Q_i})\right)\mbox{-basis}.
\]
The restricted measure $\mu|_{E_i}$ is Borel regular, and $\diam(\supp(\mu|_{E_i})) \leq 3 \delta_{Q_i}$ (since $E_i \subset 3 Q_i$), so by the consequence of the inductive hypothesis, \eqref{indhyp}, the Extension Theorem for $(\mu|_{E_i}, 3 \delta_{Q_i})$ holds. Because the seminorms for the spaces $\J(\mu|_{E_i}; 3\delta_{Q_i})$ and $\J(\mu|_{E_i}; \delta_{Q_i})$ are equivalent up to universal constant factors, we have that the Extension Theorem for $(\mu|_{E_i}, \delta_{Q_i})$ holds. 

Thus, there exist a linear map $T_i: \J(\mu|_{E_i};\delta_{Q_i}) \to L^{m,p}(\R^n)$, a map $M_i: \J(\mu|_{E_i};\delta_{Q_i}) \to \R_+$, and countable collections of Borel sets $\{ A_\ell^i \}_{\ell \in \N}, A_\ell^i \subset \supp(\mu|_{E_i})$, and of linear maps $\{\phi_\ell^i : \J(\mu|_{E_i};\delta_{Q_i}) \to \R\}_{\ell \in \N}$, and $\{\lambda_\ell^i : \J(\mu|_{E_i};\delta_{Q_i}) \to L^p(d\mu)\}_{\ell \in \N}$, that satisfy for each $(f,P) \in \J(\mu|_{E_i}; \delta_{Q_i})$:
\begin{align*}
&\textbf{(AL1) }\|f,P\|_{\J(\mu|_{E_i}; \delta_{Q_i})} \leq  \|T_i(f,P), P\|_{\J(f, \mu|_{E_i}; \delta_{Q_i})} \leq C \cdot \|f,P\|_{\J(\mu|_{E_i}; \delta_{Q_i})}; \\
&\textbf{(AL2) } c \cdot M_i(f,P) \leq \|T_i(f,P), P\|_{\J(f, \mu|_{E_i}; \delta_{Q_i})} \leq C \cdot M_i(f,P); \text{ and}\\
& \textbf{(AL3) }M_i(f,P) = \Big(\sum_{\ell \in \mathbb{N}} \int_{A_\ell^i} |\lambda_\ell^i(f,P)-f|^p d\mu + \sum_{\ell \in \mathbb{N}}|\phi_\ell^i(f,P)|^p \Big)^{1/p}.
\end{align*}
We obtain the particular form \textbf{(AL3)} for $M_i$ because $\ma_i \neq \emptyset$, a consequence of the fact that $\ma_i < \ma$ and $\emptyset$ is maximal under the order on multi-index sets. Thus, the map $M_i$ has the form  (\ref{aextthm3plus}).

Further, from the Extension Theorem for $(\mu|_{E_i}, \delta_{Q_i})$ we know that the maps $T_i$ and $M_i$ are $\Omega_i'$-constructible. Thus there exists a collection of linear functionals $\Omega_i' = \{\omega^i_{s}\}_{s \in \Upsilon^i} \subset \J(\mu|_{E_i})^*$, satisfying that the collection of sets $\{\supp(\omega^i_s)\}_{s \in \Upsilon^i}$ has $C$-bounded overlap, and for each $x \in \R^n$, there exists a finite subset $\Upsilon^i_x \subset \Upsilon^i$ and a collection of polynomials $\{v^i_{s,x} \}_{s \in \Upsilon^i_x} \subset \mpp$ such that $|\Upsilon^i_x| \leq C$ and 
\begin{align*}
\textbf{(AL4) }J_x T_i(f,P) = \sum_{s \in \Upsilon^i_x} \omega^i_{s}
(f) \cdot v^i_{s,x} +\ot^i_x(P),
\end{align*}
where $\ot_x^i: \mpp \to \mpp$ is a linear map.

Similarly, for each $\ell \in \N$ and $y \in \supp(\mu)$, there exists a finite subset $\bar{\Upsilon}^i_{\ell,y} \subset {\Upsilon}^i$ and constants $\{\eta^{\ell,i}_{s,y}\}_{s \in \bar{\Upsilon}^i_{\ell,y}} \subset \R$ such that $|\bar{\Upsilon}^i_{\ell,y}| \leq C$, and the map $(f,P) \mapsto \lambda_\ell(f,P)(y)$ has the form
\begin{align*}
\textbf{(AL5) } \lambda_\ell^i(f,P)(y) = \sum_{s \in \bar{\Upsilon}^i_{\ell,y}} \eta^{\ell,i}_{s,y} \cdot \omega^i_{s}(f) + \widetilde{\lambda}^i_{y,\ell}(P), 
\end{align*}
where $\widetilde{\lambda}^i_{y,\ell} : \mpp \rightarrow \R$ is a linear functional. 

And for each $\ell \in \N$, there exists a finite subset $\bar{\Upsilon}^i_{\ell} \subset {\Upsilon}^i$ and constants $\{\eta^{\ell,i}_{s}\}_{s \in \bar{\Upsilon}^i_{\ell}} \subset \R$ such that $|\bar{\Upsilon}^i_{\ell}| \leq C$, and the map $\phi^i_\ell$ has the form 
\begin{align*}
\textbf{(AL6) }\phi^i_\ell(f,P) = \sum_{s \in \bar{\Upsilon}^i_{\ell}} \eta_{s}^{\ell,i} \cdot \omega^i_{s}(f) + \widetilde{\lambda}^i_{\ell}(P), 
\end{align*}
where $\widetilde{\lambda}^i_{\ell} : \mpp \rightarrow \R$ is a linear functional.

\begin{remark}\label{rem:J_space}
Notice that $\J(\mu) \subset \J(\mu|_{E_i})$ with an inequality of norms, i.e., $\| f \|_{\J(\mu|_{E_i})} \leq \| f \|_{\J(\mu)}$ for any Borel measurable $f : \R^n \rightarrow \R$. Thus, $\J(\mu|_{E_i})^*$ naturally embeds in $\J(\mu)^*$. In particular, the $\Omega_i'$ may be regarded as a family of functionals in $\J(\mu)^*$.
\end{remark}

If $\omega \in \J(\mu)^*$ then $\supp(\omega) \subset \supp(\mu)$. Therefore, for any choice of $E_i \subset 3 Q_i$, the collection of functionals $\Omega_i' = \{ \omega^i_s\}_{s \in \Upsilon^i}$ satisfies
\begin{equation}\label{supp_cond_Omega:eqn}
\supp(\omega^i_s) \subset \supp(\mu|_{E_i}) \subset \supp(\mu) \cap \Cl(E_i) \quad \mbox{ for all } i \in I, s \in \Upsilon^i.
\end{equation}

\section{Preliminary Estimates and Technical Tools}
\label{sec:prel_est}

\subsection{Estimates for Auxiliary Polynomials}
\label{subsec:est_aux_poly}
Recall from (\ref{basis}) and (\ref{bound}), we have constructed $(P_x^\al)_{\al \in \ma}$, for all $x \in 100 Q^\circ$, such that 
\begin{equation}\label{basis2}
\begin{aligned}
&(P_x^\al)_{\al \in \ma} \mbox{ forms an } (\ma, x, \bar{C} \e, 1)\mbox{-basis for } \sigma_J(x,\mu), \mbox{ with } \\
&|\p^\beta P_x^\al(x)| \leq \bar{C} \mbox{ for } \al \in \ma, \beta \in \mm. 
\end{aligned} 
\end{equation}
Here, $\bar{C}$ is a universal constant, determined only by $m$, $n$, $p$. As in Section 12 of \cite{arie3}, we have:

\begin{lem} There exists a universal constant $C_3$ such that the following holds. Let $Q \in CZ^\circ$, and let $\widehat{\delta} \in [\delta_Q/2,1]$. Then
\begin{align}
|\p^\beta P_y^\al (y)| \leq C_3 \widehat{\delta}^{|\al|- |\beta|} \; \; (\al \in \ma, \: \beta \in \mm, \: y \in 3Q^+). \label{alocalbound}
\end{align}
In particular,
\begin{align}
|\p^\beta P_y^\al (y)| \leq C_3  \delta_Q^{|\al|- |\beta|} \; \; (\al \in \ma, \: \beta \in \mm, \: y \in 3Q^+). \label{localbound}
\end{align} \label{localbdlm}
\end{lem}

\begin{proof}
Inequality (\ref{localbound}) follows from (\ref{alocalbound}) by setting $\widehat{\delta} = \delta_Q$.

We now prove inequality (\ref{alocalbound}).

If $\widehat{\delta} \in [1/4, 1]$ then (\ref{alocalbound}) follows from (\ref{basis2}) because $3 Q^+ \subset 10Q \subset 100 Q^\circ$ for $Q \in CZ^\circ$. So we may assume $\widehat{\delta} \in [\delta_Q/2,1/4]$. 

Let $C' := 30^m \bar{C}$. By \eqref{basis2} and Lemma \ref{edlm}, $(P_x^\al)_{\al \in \ma}$ forms an $(\ma, x, C' \e, 30)$-basis for $\sigma_J(x,\mu)$ for $x \in 100Q^\circ$. By \eqref{monotone}, since $120 \widehat{\delta} \leq  30$, $\sigma_J(x,\mu) \subset \sigma_J(x,\mu|_{3Q^+})$, and $10Q \subset 100 Q^\circ$, we have that 
\begin{equation}\label{basisprop}
(P_x^\al)_{\al \in \ma} \mbox{ forms an }(\ma, x, C' \e, 120 \widehat{\delta})\mbox{-basis for }\sigma_J(x,\mu|_{3Q^+})\mbox{ for all }x \in 10Q.
\end{equation}

Define $\e_1 = \min \{c_1, \e_0/(C_0C_1) \}$, a universal constant, where $c_1, C_1$ are the constants from Lemma \ref{bigdir}, and $C_0$ is the constant from Lemma \ref{sjins}. For the sake of contradiction, suppose (\ref{alocalbound}) fails to hold for a sufficiently large constant $C_3$. Thus, we may assume
\begin{align}\label{basisprop2}
\max \{ |\p^\beta P_z^\al(z)| \cdot (120 \widehat{\delta})^{|\beta| -|\al|}: z \in 3 Q^+, \; \al \in \ma, \; \beta \in \mm \} > \e_1^{-D-1}.
\end{align}

By taking $\e$ small enough, we may assume that $C' \e \leq \e_1^{2D+2}$. We claim that the hypotheses of Lemma \ref{bigdir} hold with the parameters
\begin{align*}
\Big(\e_1,\e_2,\delta,\ma,\mu, E, &\left(\pt_x^\al \right)_{\al \in \ma, x \in E} \Big): = \Big(\e_1, C'\e, 120 \widehat{\delta}, \ma, \mu|_{3Q^+}, \Cl(3Q^+), \big(P_x^\al\big)_{\al \in \ma, x \in 3Q^+} \Big).
\end{align*}
Specifically, we have to check the conditions \textbf{(D1)-(D5)} in Lemma \ref{newdir} and condition \eqref{bigdir1} for the above choice of parameters. Note that $\textbf{(D1)}$ is satisfied because $\epsilon_1 \leq c_1$ and $C' \epsilon \leq \epsilon_1^{2D+2}$, while \textbf{(D2)} and \textbf{(D3)} do not mention any conditions on the parameters, hence are trivially satisfied. Further, \textbf{(D4)} is satisfied because $\supp(\mu|_{3Q^+}) \subset \Cl(3 Q^+)$ and $\diam(3Q^+) = 6 \delta_Q \leq 120 \widehat{\delta}$. Note \textbf{(D5)} is satisfied due to \eqref{basisprop}, since $\Cl(3 Q^+) \subset 10Q$. Finally, \eqref{bigdir1} is satisfied thanks to \eqref{basisprop2}.

Thus, we may apply Lemma \ref{bigdir} to deduce that there exists $\mab < \ma$ such that $\sigma_J(x, \mu|_{3Q^+})$ contains an $(\mab, x, C_1 \e_1, 120\widehat{\delta})$-basis for all $x \in 3Q^+$. We apply (\ref{monotone}), using that $\delta_{Q^+} = 2 \delta_Q \leq 4 \widehat{\delta}$ and $C_1 \e_1 \leq \e_0/C_0$, to deduce that $\sigma_J(x, \mu|_{3Q^+})$ contains an $(\mab, x,  \e_0/C_0, 30\delta_{Q^+})$-basis for all $x \in 3 Q^+$. Therefore, the cube $Q^+$ is OK (see Definition \ref{ok_defn}), contradicting that $Q \in CZ^\circ$. This completes the proof of (\ref{alocalbound}).

\end{proof}

\begin{lem}
\label{newestlm}
There exists a universal constant $C_4$ such that the following holds. Let $x \in K_p$. For $\delta \in (0,1)$,
\begin{align}
    |\p^\beta P_x^\al(x)| \leq C_4 \delta^{|\al| - |\beta|} \quad \quad (\al \in \ma, \beta \in \mm). \label{acobd}
\end{align}
In particular $\p^\beta P_x^\al(x) = 0$ for $|\al|>|\beta|$.
\end{lem}

\begin{proof}
Let $x \in K_p$ and $\delta\in (0,1)$. Recall from (\ref{basis2}), we have
\begin{align}
 |\p^\beta P_{x}^\al (x)| \leq C  \quad \quad (\al \in \ma, \; \beta \in \mm). \label{bound_1} 
\end{align}
For $|\beta| \geq |\al|$, (\ref{bound_1}) implies (\ref{acobd}). So it suffices to prove (\ref{acobd}) for $|\al| > |\beta|$. We will prove that $\p^\beta P_x^\al(x) = 0$ for $|\al|>|\beta|$. This will complete the proof of \eqref{acobd}, and with it, the proof of the lemma.

Suppose for sake of contradiction that there exist $\beta \in \mm$, $\al \in \ma$, $|\al| > |\beta|$ such that $|\p^\beta P_x^\al(x)|>\eta>0$. Fix $\e_1 < \min\{ c_1, \e_0/(30^m C_0 C_1)\}$, for $c_1$ and $C_1$ as in Lemma \ref{newdir}, and $C_0$ is the constant from Lemma \ref{sjins}. Fix $\delta_1<\eta \e_1^{D+1}$ with $\delta_1 < 1$. Fix a dyadic cube $Q \subset Q^\circ$ with $x \in Q$ and $\delta_Q < \delta_1$. Observe that $\delta_Q^{|\beta| - |\al|} > \delta_1^{|\beta|-|\al|} \geq \delta_1^{-1}$. So,
\begin{align}
\max \left\{|\p^{\hat{\beta}} P_y^{\hat{\al}} (y)|\delta_Q^{|\hat{\beta}|-|\hat{\al}|}: y \in 3Q  \text{, } \hat{\al} \in \ma \text{, } \hat{\beta} \in \mm \right\} &\geq |\p^\beta P_x^\al(x)|\delta_Q^{|\beta|- |\al|} \nonumber \\
&>\eta \delta_1^{-1} >\e_1^{-D-1}. \label{acohbigdir}
\end{align}
We fix $\e_2 < \e_1^{2D+2}$.  From  (\ref{basis2}), and since $5Q \subset 100 Q^\circ$, $\left(P_y^\al \right)_{\al \in \ma}$ forms an $(\ma, y, C\e, 1)$-basis for $\sigma_J(y, \mu)$ for all $y \in 5Q$. We can assume $C\e < \e_2$. Note that $\delta_Q < \delta_1 < 1$. From (\ref{monotone}), $\left(P_y^\al \right)_{\al \in \ma}$ forms an $(\ma, y, \e_2, \delta_Q)$-basis for $\sigma_J(y, \mu|_{3Q})$ for all $y \in 5Q$. In combination with (\ref{acohbigdir}), we see the hypotheses of Lemma \ref{bigdir} hold with parameters
\begin{align*}
\Big(\e_1,\e_2,\delta,\ma,\mu, E, \left(\pt_x^\al \right)_{\al \in \ma, x \in E} \Big) = \Big(\e_1,\e_2,\delta_Q,\ma,\mu|_{3Q}, \Cl(3Q), \left(P_x^\al \right)_{\al \in \ma, x \in 3Q} \Big).
\end{align*}
Hence there exists $\mab<\ma$ so that for every $y \in 3Q$, $\sigma_J(y,\mu|_{3Q})$ contains an $(\mab, y, C_1 \e_1, \delta_Q)$-basis. Because $C_1\e_1 < \e_0 /(30^mC_0 )$, we apply Lemma \ref{edlm} to deduce that for every $y \in 3Q$, $\sigma_J(y,\mu|_{3Q})$ contains an $(\mab, y, \e_0/C_0, 30 \delta_Q)$-basis, indicating that $Q$ is OK, thus $Q$ is contained in a CZ cube. But $x \in Q$, so $x \in K_{CZ}$. This contradicts that $x \in K_p$, completing the proof of \eqref{acobd} by contradiction.
\end{proof}

In Proposition \ref{phiprop} of Section \ref{subsubsec:aux_poly}, we defined a family of functions $(\varphi_x^\al)_{\al \in \ma, x \in Q^\circ} \subset L^{m,p}(\R^n)$, related to the $(P^\al_x)_{\al \in \ma, x \in Q^\circ}$, satisfying (\ref{phi1})-(\ref{phi3}). In particular, $J_x \varphi_x^\al = P_x^\al$ for $x \in Q^\circ$, $\al \in \ma$. 

\begin{lem}\label{lem:auxpolyzeros}
We have
\[
P^\al_x(x) = \varphi^\al_x(x) = 0 \quad \mbox{for } x \in K_p, \; \al \in \ma.
\]
\end{lem}
\begin{proof}
Recall that $\ma \subset \mm$, $\ma \neq \mm$, and $\ma$ is monotonic. Therefore, $\ma$ does not contain the zero multi-index. Hence, $|\al| > 0$ for $\al \in \ma$. Due to \eqref{acobd}, we have $|P_x^\al(x)| \leq C \delta^{|\al|}$ for all $\delta \in (0,1)$. Hence, $P_x^\al(x) = 0$. The result follows.
\end{proof}

\subsection{Estimates for Local Solutions}
\label{sec:est_ext}

Recall the norm defined on $\mpp$:
\begin{align*}
|P|_{x,\delta} = \Big( \sum_{|\al| \leq m-1} |\p^\al P(x)|^p \cdot \delta^{n +(|\al| - m)p} \Big)^{1/p}.
\end{align*}
Recall $I$ is the indexing set for the CZ decomposition $CZ^\circ = \{Q_i\}_{i \in I}$. For $P \in \mpp$ and $i \in I$, define 
\begin{align}
|P|_i := |P|_{x_i,\delta_{Q_i}}. \label{inorm}  
\end{align}
By applying (\ref{sob}) to the measure $\mu|_{Q_i}$ and domain $U= Q_i$, for $F \in L^{m,p}(\R^n)$ and $x \in Q_i$,
\begin{align}
|J_{x}F-P|_{x,\delta_{Q_i}} \leq C \|F,P\|_{\J_*(f,\mu|_{Q_i};Q_i)}. \label{normest}
\end{align}

\begin{lem} Let $Q_i \in CZ^\circ$. Then for $P \in \mpp$,
\begin{align}
\|0, P\|_{\J(\mu|_{1.1Q_i}; \delta_{Q_i})} \lesssim |P|_i =  |P|_{x_i, \delta_{Q_i}}. \label{locest1}
\end{align}
Moreover, if $P \in \mpp$ satisfies $\p^\al P(x_i) = 0$ for all $\al \in \ma$, then
\begin{align}
\|0, P\|_{\J(\mu|_{9{Q_i}}; \delta_{Q_i})} \simeq  |P|_i =  |P|_{x_i, \delta_{Q_i}}. \label{locest2}
\end{align}
\end{lem}

\begin{proof}
For the proof of (\ref{locest1}), let $F_0 = (1-\theta) P$, where $\theta \in C^\infty(\R^n)$, $\theta|_{9{Q_i}} = 1$, $\supp(\theta) \subset 10{Q_i}$, and $|\p^\al \theta(x)| \leq C \delta_{Q_i}^{-|\al|}$. Then
\begin{align*}
\|0, P\|_{\J(\mu|_{1.1{Q_i}}; \delta_{Q_i})}^p &\leq \|0, P\|_{\J(\mu|_{9{Q_i}}; \delta_{Q_i})}^p \\
&\leq \|F_0, P\|_{\J(0,\mu|_{9{Q_i}}; \delta_{Q_i})}^p \\
&= \|F_0 \|_{L^{m,p}(\R^n)}^p + \int_{9 Q_i} | F_0 |^p d \mu + \| F_0 - P  \|_{L^p(\R^n)}^p/\delta_{Q_i}^{mp} \\
&= \|F_0 - P \|_{L^{m,p}(\R^n)}^p + \int_{9 Q_i} | F_0  |^p d \mu + \| F_0 - P  \|_{L^p(\R^n)}^p/\delta_{Q_i}^{mp}.
\end{align*}
Note $F_0 = 0$ on $9Q_i$, and $F_0 = P$ on $\R^n \setminus 10 Q_i$. Thus, continuing from the above bounds, and using the Taylor expansion $P(x) = \sum_{|\al| \leq m-1} P(x_i) \cdot (x-x_i)^\al/\al!$,
\begin{align*}
\|0, P\|_{\J(\mu|_{1.1{Q_i}}; \delta_{Q_i})}^p & \leq \| \theta P \|_{L^{m,p}(10 Q_i)}^p + \| \theta P  \|_{L^p(10{Q_i})}^p/\delta_{Q_i}^{mp} \\
&\leq C |P|_i^p.
\end{align*}
This completes the proof of (\ref{locest1}).

We now prove (\ref{locest2}). Note, the proof of \eqref{locest1} above shows that $\| 0, P \|_{\J(\mu|_{9 Q_i}; \delta_{Q_i})} \lesssim | P |_i$. Thus, it suffices to establish the reverse inequality.  We assume for contradiction that there exists $P' \in \mpp$, $P' \neq 0$ satisfying $\p^\al P'(x_i) = 0$ for all $\al \in \ma$ and
\begin{align}
\|0,P'\|_{\J(\mu|_{9{Q_i}}; \delta_{Q_i})} \leq \e_1^{D+1} |P'|_i, \label{abcont}
\end{align}
for a universal constant $\epsilon_1$. We will later choose $\e_1>0$ small enough so that we reach a contradiction.  Define 
\[
P:=P' \cdot \Big(\max_{\beta \in \mm} \big\{ |\p^\beta P'(x_i)|\delta_{Q_i}^{|\beta| + n/p -m} \big\} \Big)^{-1}.
\]
Note that 
\begin{equation}\label{stuff}
\max_{\beta \in \mm} \big\{ |\p^\beta P(x_i)|\delta_{Q_i}^{|\beta|+ n/p-m} \big\} =1,
\end{equation}
and thus $|P|_i = |P|_{x_i, \delta_{Q_i}} \leq C$.  Further, since $P = \gamma P'$ for some $\gamma \in \R$, from \eqref{abcont}, $\|0,P\|_{\J(\mu|_{9{Q_i}}; \delta_{Q_i})} \leq \e_1^{D+1} |P|_i \leq C \e_1^{D+1}$, and so
\begin{align}
P \in C\e_1^{D+1} \sigma(\mu|_{9{Q_i}}, \delta_{Q_i}). \label{albar0}
\end{align}
Also, $P$ satisfies $\p^\al P(x_i) = 0$ for all $\al \in \ma$. For each integer $\ell \geq 0,$ we define $\Delta_\ell \subset \mm$ by
\[
\D_{\ell}=\left\{\al \in \mm: \left|\p^{\al} P (x_i) \right| \delta_{Q_i}^{ |\al|+ n/p-m} \in (\e_1^{\ell}, 1]\right\}.
\]
Note that $\D_\ell \subset \D_{\ell+1}$ for each $\ell \geq 0$ and that $\D_{\ell} \neq \emptyset$ for $\ell \geq 1$. Since $\mm$ contains $D$ elements, there exists $0 \leq \ell_* \leq D$ with $\D_{\ell_{*}}=\Delta_{\ell_*+1} \neq \emptyset .$ Let $\bar{\al} \in \mm$ be the maximal element of $\D_{\ell_*}$. Because $P$ satisfies $\p^\al P(x_i) = 0$, we have $\bar{\al} \notin \ma$. Further, because $\bar{\al} \in \Delta_{\ell_*}$, we have
\begin{align}
& \left|\p^{\bar{\al}} P(x_i)\right| \delta_{Q_i}^{|\bar{\al}| + n/p-m} > \e_1^{\ell_*}; \text{ and} \label{albar1} \\ 
& \left|\p^{\beta} P(x_i)\right| \delta_{Q_i}^{ |\beta|+ n/p - m} \leq \e_1^{\ell_*+1} \quad (\beta \in \mm, \beta>\bar{\al}), \label{albar2}
\end{align}
where \eqref{albar2} follows because for every $\beta \in \mm$ with $\beta>\bar{\al}$, we have $\beta \notin \D_{\ell_*}=\Delta_{\ell_*+1}$. Define $P^{\bar{\al}}:=P \cdot \left(\p^{\bar{\al}} P(x_i)\right)^{-1}$. Then because $\ell_* \leq D$, from (\ref{stuff}), (\ref{albar0}), (\ref{albar1}), and (\ref{albar2}), we have 
\begin{align}
&P^{\bar{\al}} \in C \e_1 \delta_{Q_i}^{ |\bar{\al}|+ n/p-m} \cdot \sigma(\mu|_{9{Q_i}}, \delta_{Q_i}); \label{ab0} \\
&\p^{\bar{\al}}P^{\bar{\al}}(x_i) = 1 \label{ab1};\\
&|\p^\beta P^{\bar{\al}}(x_i)| \leq \e_1 \delta_{Q_i}^{|\bar{\al}| - |\beta|} && (\beta \in \mm, \beta>\bar{\al}); \text{ and} \label{ab2} \\
&|\p^\beta P^{\bar{\al}}(x_i)| \leq \e_1^{-D} \delta_{Q_i}^{|\bar{\al}| - |\beta|} && (\beta \in \mm). \label{ab3}
\end{align}
From (\ref{ab0}), there exists $\varphi^{\bar{\al}} \in L^{m,p}(\R^n)$ satisfying
\begin{align}
\|\varphi^{\ab}, P^{\ab}\|_{\J(0,\mu|_{9{Q_i}}; \delta_{Q_i})} &= \Big( \|\varphi^{\ab}\|_{L^{m,p}(\R^n)}^p + \int_{9{Q_i}}|\varphi^{\ab}|^p d\mu  + \|\varphi^{\ab} - P^{\ab} \|_{L^p(\R^n)}^p/\delta_{Q_i}^{mp} \Big)^{1/p} \nonumber \\
&\leq C \e_1 \delta_{{Q_i}}^{ |\ab| + n/p-m}. \label{ab4}
\end{align}
For $x \in \supp(\mu) \cap 9{Q_i}$, set $P_x^{\ab}:= J_x\varphi^{\ab}$. 

Because $\|\varphi^{\ab}\|_{\J(0,\mu|_{9{Q_i}})} \leq \|\varphi^{\ab}, P^{\ab}\|_{\J(0,\mu|_{9{Q_i}}; \delta_{Q_i})}$, we have
\begin{align}
P_x^{\bar{\al}} \in C \e_1 \delta_{{Q_i}}^{|\bar{\al}|+n/p-m} \cdot \sigma_J(x, \mu|_{9{Q_i}}). \label{AB0}
\end{align}
Also, $|x - x_i| \leq C \delta_{Q_i}$, so for $\beta \in \mm$ we have
\begin{align}
|\p^\beta P_x^{\ab} (x) - & \p^\beta P^{\ab} (x_i)| = |\p^\beta \varphi^{\ab}(x) - \p^\beta P^{\ab}(x_i)| \nonumber \\
&\leq |\p^\beta \varphi^{\ab}(x) - \p^\beta P^{\bar{\al}}(x)| +|\p^\beta P^{\bar{\al}} (x)- \p^\beta P^{\ab}(x_i)| \nonumber \\
&\overset{(\ref{sob2})}{\lesssim} \| \varphi^{\ab}, P^{\ab} \|_{\J(0,\mu|_{9{Q_i}}; \delta_{Q_i})} \delta_{Q_i}^{m- |\beta| - n/p} + \Bigg| \sum_{0 < |\omega| \leq m-1-|\beta|} \p^{\beta + \omega} P^{\bar{\al}}(x_i) \frac{(x-x_i)^{\omega}}{\omega!}  \Bigg| \nonumber \\
&\overset{(\ref{ab4})}{\lesssim} \e_1 \delta_{Q_i}^{|\ab| - |\beta|} + \Bigg| \sum_{0 < |\omega| \leq m-1-|\beta|} \p^{\beta + \omega} P^{\bar{\al}}(x_i) \frac{(x-x_i)^{\omega}}{\omega!}  \Bigg| \nonumber \\
& \;\;\overset{\eqref{ab2}, \eqref{ab3}}{\lesssim} \left\{
\begin{aligned}
\epsilon_1 \delta_{Q_i}^{|\bar{\al}| - |\beta|} \;\; \mbox{if } \beta \geq \bar{\al}  \\
\epsilon_1^{-D} \delta_{Q_i}^{|\bar{\al}| - |\beta|} \;\; \mbox{for any } \beta,
\end{aligned}
\right.
\label{ab5}
\end{align}
where in the last line we have used that $\beta + \omega > \beta$ provided that $0 < |\omega| \leq m-1-|\beta|$, and hence, $\beta + \omega > \bar{\alpha}$ provided $\beta \geq \bar{\alpha}$. We insert (\ref{ab5}) into (\ref{ab1})-(\ref{ab3}) to deduce,
\begin{align}
&|\p^\beta P_x^{\bar{\al}}(x) - \delta_{\bar{\al} \beta}| \leq C \e_1 \delta_{Q_i}^{|\bar{\al}| - |\beta|} && (\beta \in \mm, \beta \geq \bar{\alpha}) \label{AB1}; \text{ and} \\
&|\p^\beta P_x^{\bar{\al}}(x)| \leq C \e_1^{-D} \delta_{Q_i}^{|\bar{\al}| - |\beta|} && (\beta \in \mm) \label{AB2}
\end{align}

We will use $\ab \in \mm$ and $P^{\ab}$ to construct $\mab < \ma$ and an $(\mab, x, \e_0/C_0, \delta_{Q_i^+})$-basis for $\sigma_J(x, \mu|_{3Q_i^+})$ for all $x \in \supp(\mu) \cap 3 Q_i^+$, indicating that $Q_i^+$ is OK, a contradiction. So (\ref{abcont}) cannot hold, and we have $\|0, P\|_{\J(\mu|_{9{Q_i}}; \delta_{Q_i})} \simeq  |P|_i$ for all $P \in \mpp$. This next portion of the proof follows Section 13 of \cite{arie3}.

Fix $x \in \supp(\mu) \cap 3 Q_i^+ \subset Q^\circ$. Recall from (\ref{basis2}) and (\ref{localbound}), that the auxiliary polynomials $(P_x^\al)_{\al \in \ma}$ form an $(\ma,x, C\e, 1)$-basis for $\sigma_J(x,\mu)$, satisfying:
\begin{align}
& P_x^\al \in C\e \cdot \sigma_J(x,\mu); \label{b0} \\
&\p^\beta P^\al_x (x) = \delta_{\al \beta} && ( \al, \beta \in \ma); \label{b1} \\
&|\p^\beta P^\al_x (x)| \leq C \e && (\al \in \ma, \beta \in \mm, \beta > \al); \label{b2} \\
&|\p^\beta P^\al_x (x)| \leq C && (\al \in \ma, \beta \in \mm); \text{ and} \label{b3} \\
&|\p^\beta P_x^\al(x)| \lesssim  \delta_{Q_i}^{|\al| - |\beta|} &&(\al \in \ma, \beta \in \mm). \label{b4a}
\end{align}

Define 
\[
\pt_x^{\ab} := P_x^{\ab} - \sum_{\al \in \ma, \al < \ab} \p^\al P_x^{\ab}(x) \cdot P_x^\al.
\]
We have $\sigma_J(x,\mu) \subset \sigma_J(x, \mu|_{9{Q_i}})$, $\delta_{Q_i} <1$, and $m-n/p > m -1 \geq |\al|$. So from (\ref{AB0}), (\ref{AB2}), and (\ref{b0}), we have
\begin{align}
\pt_x^{\ab} &\in \Big( C \e_1 \delta_{Q_i}^{|\ab|+ n/p -m} + \sum_{\al \in \ma, \al < \ab} (C \e_1^{-D} \delta_{Q_i}^{|\ab| - |\al|})(C\e) \Big) \cdot \sigma_J(x,\mu|_{9{Q_i}}) \nonumber \\
&\implies \pt_x^{\ab}\in C \left( \e_1 + \e_1^{-D} \e \right) \delta_{Q_i}^{|\ab| + n/p -m} \cdot \sigma_J(x, \mu|_{9{Q_i}}). \label{b5}
\end{align}
For $\beta \in \ma$, $\beta < \ab$, from (\ref{b1}),
\begin{align}
\p^\beta \pt_x^{\ab}(x) = \p^\beta P_x^{\ab}(x) - \p^\beta P_x^{\ab}(x) = 0. \label{b6}
\end{align}
For $\beta \in \mm$, $\beta \geq \ab$, from (\ref{AB1}), (\ref{AB2}), and (\ref{b2}),
\begin{align}
|\p^\beta \pt_x^{\ab}(x) - \delta_{\beta \ab}| &\leq |\p^\beta P_x^{\ab}(x) - \delta_{\beta \ab}| + |\p^\beta (P_x^{\ab}- \pt_x^{\ab})(x)| \nonumber \\
&\leq |\p^\beta P_x^{\ab}(x) - \delta_{\beta \ab}| + \sum_{\al \in \ma, \al < \ab}|\p^\al P_x^{\ab}(x)||\p^\beta P_x^\al(x)| \nonumber \\
&\leq C \e_1 \delta_{Q_i}^{|\ab|-|\beta|} + \sum_{\al \in \ma, \al < \ab} (C\e_1^{-D} \delta_{Q_i}^{|\ab| -|\al|})(C\e) \nonumber \\
&\lesssim (\e_1 + \e_1^{-D} \e) \delta_{Q_i}^{|\ab|-|\beta|}, \label{b7}
\end{align}
where the last inequality uses that $\delta_{Q_i} < 1$ and $|\beta| \geq |\al|$ for $\beta \geq \bar{\al} > \al$.

For $\beta \in \mm$ arbitrary, from (\ref{AB2}), (\ref{b4a}),
\begin{align}
|\p^\beta \pt_x^{\ab}(x)| &\leq |\p^\beta P_x^{\ab}(x)| + \sum_{\al \in \ma, \al < \ab}|\p^\al P_x^{\ab}(x)||\p^\beta P_x^\al(x)| \nonumber \\
&\leq C \e_1^{-D} \delta_{Q_i}^{|\ab|-|\beta|} + \sum_{\al \in \ma, \al < \ab} (C\e_1^{-D} \delta_{Q_i}^{|\ab| -|\al|})(C\delta_{Q_i}^{|\al| - |\beta|}) \nonumber \\
&\lesssim \e_1^{-D} \delta_{Q_i}^{|\ab| - |\beta|}. \label{b8}
\end{align}

Define 
\[
\mab = \{\ab\} \cup \{ \al \in \ma: \al < \ab \}.
\]
Then the minimal element of the symmetric difference $\mab \Delta \ma$ is $\bar{\al}$, which is in $\mab$. So  $\mab < \ma$, by definition of the order relation on multiindex sets.
We may assume $\e < \e_1^{D+1}$. Then, for small enough $\e_1$, (\ref{b7}) implies $\p^{\ab} \pt_x^{\ab}(x) \geq 1/2$. So $\ph_x^{\ab} : = \pt_x^{\ab} / (\p^{\ab} \pt_x^{\ab}(x))$ is well-defined, and due to  (\ref{b5})-(\ref{b8}), this polynomial satisfies:
\begin{align}
&\ph_x^{\ab} \in C (\e_1 + \e_1^{-D} \e) \delta_{Q_i}^{|\ab| + n/p - m} \cdot \sigma_J(x,\mu|_{9{Q_i}}); \label{bb0} \\
&\p^\beta \ph_x^{\ab}(x) = \delta_{\beta \ab} && (\beta \in \mab); \label{bb1} \\
&|\p^\beta \ph_x^{\ab}(x)| \leq C(\e_1 + \e_1^{-D} \e) \delta_{Q_i}^{|\ab|- |\beta|} && (\beta \in \mm, \beta> \ab); \text{ and} \label{bb2} \\
&|\p^\beta \ph_x^{\ab}(x)| \leq C \e_1^{-D} \delta_{Q_i}^{|\ab|- |\beta|} && (\beta \in \mm). \label{bb3}
\end{align}
For $\al \in \mab \setminus \{ \bar{\al}\}$, define 
\[
\ph_x^\al = P_x^\al - \p^{\ab}P_x^\al (x) \cdot \ph_x^{\ab}.
\]
Notice, if $\al \in \mab \setminus \{ \bar{\al}\}$ then $\al < \bar{\al}$. Also, $\delta_{Q_i} < 1$. Thus, thanks to \eqref{b2}, $|\p^{\ab}P_x^\al (x)| \leq C \e \leq C \e \delta_{Q_i}^{|\al| - |\ab|}$. Therefore, from  (\ref{b0})-(\ref{b3}), and (\ref{bb0})-(\ref{bb3}), for any $\al \in \mab \setminus \{ \bar{\al}\}$,
\begin{align}
&\ph_x^{\al} \in \left[ C \e  + C \e \delta_{Q_i}^{|\al| - |\ab|}  (\e_1 + \e_1^{-D} \e) \delta_{Q_i}^{|\ab| + n/p - m} \right] \cdot \sigma_J(x, \mu|_{9{Q_i}}); \label{bc0}\\
&\p^\beta \ph_x^{\al}(x) = \delta_{\al \beta} &&(\beta \in \mab); \text{ and} \label{bc1} \\
&|\p^\beta \ph_x^{\al}(x)| \leq C \e + C \e \delta_{Q_i}^{|\al| - |\ab|}  C\e_1^{-D} \delta_{Q_i}^{|\ab|- |\beta|} && (\beta \in \mm, \beta> \al). \label{bc2}
\end{align}
We suppose $\e < \e_1^{D+1}$. Since $|\al| + n/p - m <0$, and $\delta_{Q_i} < 1$, (\ref{bc0}) implies,
\begin{align}
\ph_x^{\al} \in C (\e_1 + \e_1^{-D} \e) \delta_{Q_i}^{|\al| + n/p - m} \cdot \sigma_J(x, \mu|_{9{Q_i}}). \label{bc00}
\end{align}
Similarly, because $\delta_{Q_i} < 1$ and $|\al| - | \beta| \leq 0$ for $\beta > \al$, \eqref{bc2} implies
\begin{align}
|\p^\beta \ph_x^{\al}(x)| \leq C \e_1  \delta_{Q_i}^{|\al|- |\beta|} \quad (\beta \in \mm, \beta> \al). \label{bc01}
\end{align}

Recall $x \in \supp(\mu) \cap 3{Q_i}^+$ is arbitrary. Together, (\ref{bb0})-(\ref{bb3}), (\ref{bc1}), (\ref{bc00}), and (\ref{bc01}) imply that, for each $x \in \supp(\mu) \cap 3 {Q_i}^+$, $\{ \ph_x^\al \}_{\al \in \mab}$ forms an $(\mab, x, C'(\e_1 + \e_1^{-D}\e), 30\delta_{{Q_i}^+})$-basis for $\sigma_J(x, \mu|_{9{Q_i}})$, and hence for $\sigma_J(x, \mu|_{3{Q_i}^+})$. Fix a universal constant $\e_1 > 0$, small enough so that the preceding arguments hold, and so that $\e_1 < \frac{\e_0}{2C' C_0}$. Recall that we have assumed $\e < \e_1^{D+1}$. Thus, $\{ \ph_x^\al \}_{\al \in \mab}$ forms an $(\mab, x, \e_0/ C_0, 30\delta_{{Q_i}^+})$-basis for $\sigma_J(x, \mu|_{3{Q_i}^+})$ for each $x \in \supp(\mu) \cap 3{Q_i}^+$, indicating that ${Q_i}^+$ is OK. But this contradicts the assumption that $Q_i$ is a CZ cube. We have reached the desired contradiction. This completes the proof of (\ref{locest2}), and with it, the proof of the lemma.
\end{proof}

\subsection{Patching Estimates}

\subsubsection{Patching Estimate on $K_{CZ}$}
Recall that we have defined a collection of disjoint dyadic cubes $CZ^\circ = \{Q_i\}_{i \in I} \in \Pi(Q^\circ)$ contained in $Q^\circ = (0,1]^n$. We set $K_{CZ} = \bigcup_{i \in I} Q_i$. Then $K_{CZ}$ is a relatively open subset of $Q^\circ$. Indeed, we showed that $K_p$ is a closed set in $\R^n$, and $K_p \subset \supp(\mu) \subset \frac{1}{10} Q^\circ$.

We associate to each cube $Q_i \in CZ^\circ$ a basepoint $x_i = \mathrm{ctr}(Q_i)$. We define the polynomial norms $|P|_i := |P|_{x_i,\delta_{Q_i}}$.

Recall that a Whitney field $\vec{P} \in Wh(\bpt)$ is an indexed collection of polynomials, $\vec{P} = \{ P_x \}_{x \in \bpt}$, associated to the CZ basepoints $\bpt = \{x_i\}_{i \in I}$.

For a relatively open set $\Omega \subset Q^\circ$ with $K_{CZ} \subset \Omega$, and for $\vec{P} \in Wh(\bpt)$, $f \in \J(\mu)$, and $F \in L^{m,p}(\Omega)$, we define: 
\begin{align}\label{new_Jfunc1:eqn}
&\|F,\vec{P}\|_{\J_*(f, \mu;\Omega,CZ^\circ)} =  \Big( \|F\|_{L^{m,p}(\Omega)}^p + \int_{\Omega} |F-f|^p d\mu  + \sum_{i \in I}  \|F-P_{x_i}\|_{L^p(Q_i)}^p/\delta_{Q_i}^{mp} \Big)^{1/p}; \text{ and} \\
&\|f,\vec{P}\|_{\J_*(\mu; \Omega,CZ^\circ)} = \inf \left\{ \|F,\vec{P}\|_{\J_*(f, \mu;\Omega,CZ^\circ)} : F \in L^{m,p}(\Omega) \right\}. \label{new_Jfunc2:eqn}
\end{align}
We define the seminormed vector space:
\begin{equation} \label{Jstar_space_2B:defn}
    \J_*(\mu;\Omega,CZ^\circ) = \left\{ (f,\vec{P}): f \in \J(\mu) \text{, } \vec{P} \in Wh(\bpt), \;  \|f,\vec{P}\|_{\J_*(\mu; \Omega, CZ^\circ)} < \infty \right\}. 
\end{equation}
In the previous definitions, we have in mind to take $\Omega = K_{CZ}$ or $\Omega = Q^\circ$.

\begin{lem}
Suppose we are given a collection of functions $\{G_i\}_{i \in I} \subset L^{m,p}(\R^n)$ and a Whitney field $\vec{P} = (P_{x_i})_{i \in I} \in Wh(\bpt)$.

Define $G: K_{CZ} \to \R$ by $G(x)= \sum_{i \in I} G_i(x) \cdot \theta_i(x)$, where $\{\theta_i\}_{i \in I}$ is a partition of unity satisfying \textbf{(POU1)-(POU4)} (see Section \ref{subsec:pou}). Then
\begin{align}
\|G, \vec{P} \|_{\J_*(f,\mu;K_{CZ},CZ^\circ)}^p \lesssim &  \sum_{i \in I} \|G_i, P_{x_i}\|_{\J_*(f,\mu|_{1.1Q_i \cap Q^\circ};1.1Q_i \cap Q^\circ)}^p + \sum_{i, i' \in I, \; i \lra i'} |P_{x_i} - P_{x_{i'}}|_i^p. \label{apatch1}
\end{align}
\label{kczpatchlm}
Here, we write $i \lra i'$ to denote that $\Cl(Q_i) \cap \Cl(Q_{i'}) \neq \emptyset$, and the $\J_*$-functionals on the right-hand side of (\ref{apatch1}) are defined in (\ref{Jstar1A:defn}).
\end{lem}

\begin{proof}

For $x \in Q_{i'}$, if $x \in \supp(\theta_i)$ then $x \in 1.1 Q_i$, hence $1.1 Q_i \cap Q_{i'} \neq \emptyset$ and thus $i \lra i'$ by the good geometry of the CZ cubes. Thus, by the condition $\sum_i \theta_i = 1$ on $K_{CZ}$, we deduce that $G(x) = G_{i'}(x) + \sum_{i \in I, 
\; i \lra i'} (G_i-G_{i'})(x) \cdot \theta_i(x)$. By the Leibniz rule, for any multiindex $\gamma$ such that $|\gamma| = m$, we have 
\[
\partial^\gamma G(x) = \partial^\gamma G_{i'}(x) +  \sum_{(\alpha,\beta ) : \alpha + \beta = \gamma} \sum_{i \in I : 
\; i \lra i'} \partial^\beta (G_i - G_{i'})(x) \partial^\alpha \theta_i(x) .
\]
Then, taking $p$'th powers, summing on $\gamma$ with $|\gamma| = m$, and integrating over $x \in Q_{i'}$, we have
\begin{align*}
\|G\|_{L^{m,p}(Q_{i'})}^p &\lesssim \|G_{i'}\|_{L^{m,p}(Q_{i'})}^p  + \sum_{(\alpha,\beta ): |\al| + |\beta| = m} \sum_{i \in I : i \lra i'} \int_{Q_{i'}} |\p^\beta(G_i -G_{i'})(x)|^p |\p^\al \theta_i(x)|^p dx. 
\end{align*}
Now note, by \textbf{(POU1)-(POU4)}, $| \partial^\alpha \theta_i(x)| \lesssim \delta_{Q_i}^{-|\alpha|}$, and $\theta_i$ is supported on $1.1 Q_i$. So, 
\begin{align}
\|G\|_{L^{m,p}(Q_{i'})}^p &\lesssim \|G_{i'}\|_{L^{m,p}(Q_{i'})}^p  + \sum_{(\alpha,\beta ): |\al| + |\beta| = m} \sum_{i \in I : i \lra i'} \delta_{Q_i}^{-|\alpha| p}  \int_{Q_{i'} \cap 1.1 Q_i} |\p^\beta(G_i -G_{i'})(x)|^p dx. \label{patch1}
\end{align}
If $\al = 0$ in the previous sum, then $|\beta| = m$, so
\[
\int_{Q_{i'} \cap 1.1 Q_i} |\p^\beta(G_i -G_{i'})(x)|^p dx \leq C \big( \|G_{i'}\|_{L^{m,p}(Q_{i'})}^p +  \|G_i\|_{L^{m,p}(1.1Q_i \cap Q^\circ)}^p \big).
\]
Now suppose $|\al|>0$ in the previous sum. Then  $|\beta| \leq m-1$, and 
\[
|\p^\beta(G_i -G_{i'})(x)| \leq |\p^\beta(G_i -P_{x_i})(x)| + |\p^\beta(P_{x_i} -P_{x_{i'}})(x)|+|\p^\beta(G_{i'} -P_{x_{i'}})(x)|.
\]
Thus, by integrating over $x \in Q_{i'} \cap 1.1 Q_i$, we can apply (\ref{sob}) twice, on the rectangle $1.1 Q_i \cap Q^\circ$ and on the square $Q_{i'}$, to obtain
\begin{align*}
\int_{Q_{i'} \cap 1.1 Q_i} |\p^\beta(G_i -G_{i'})&(x)|^p dx \leq C \Big(\|G_i, P_{x_i}\|_{\J_*(f,\mu|_{1.1Q_i \cap Q^\circ};1.1Q_i \cap Q^\circ)}^p \delta_{Q_i}^{mp  -|\beta|p} \\ 
 &+ \|G_{i'}, P_{x_{i'}}\|_{\J_*(f,\mu|_{Q_{i'}};Q_{i'})}^p \delta_{Q_i}^{mp  -|\beta|p} +  \max_{x \in 1.1 Q_i \cap Q_{i'}} |\p^\beta(P_{x_i} - P_{x_{i'}})(x)|^p\delta_{Q_i}^{n} \Big).
\end{align*}
Now because $|\al| + |\beta| = m$,
\begin{align*}
\delta_{Q_i}^{-|\alpha| p}  \int_{Q_{i'} \cap 1.1 Q_i} & |\p^\beta(G_i -G_{i'})(x)|^p  dx \leq C\cdot \Big(\|G_i, P_{x_i}\|_{\J_*(f,\mu|_{1.1Q_i \cap Q^\circ};1.1Q_i \cap Q^\circ)}^p \\
& +\|G_{i'}, P_{x_{i'}}\|_{\J_*(f,\mu|_{Q_{i'}};Q_{i'})}^p + \max_{x \in 1.1 Q_i \cap Q_{i'}} |\p^\beta(P_{x_i} - P_{x_{i'}})(x)|^p\delta_{Q_i}^{n + |\beta|p - mp}\Big).
\end{align*}
We use (\ref{movept}) to bound the third term in the parentheses  by $|P_{x_i} - P_{x_{i'}}|_{x_i, \delta_{Q_i}}^p = |P_{x_i} - P_{x_{i'}} |_i^p$. Returning to \eqref{patch1},
\begin{align*}
\|G\|_{L^{m,p}(Q_{i'})}^p  &\lesssim \|G_{i'}\|_{L^{m,p}(Q_{i'})}^p \\
&\quad +  \sum_{i \in I, i \lra i'} \Big( \|G_i, P_{x_i}\|_{\J_*(f,\mu|_{1.1Q_i \cap Q^\circ};1.1Q_i \cap Q^\circ)}^p +\|G_{i'}, P_{x_{i'}}\|_{\J_*(f,\mu|_{Q_{i'}};Q_{i'})}^p + |P_{x_i} - P_{x_{i'}}|_i^p \Big).
\end{align*}
We can bound $\|G_{i'}\|_{L^{m,p}(Q_{i'})}^p \leq  \|G_{i'}, P_{x_{i'}}\|_{\J_*(f,\mu|_{1.1Q_{i'} \cap Q^\circ};1.1Q_{i'} \cap Q^\circ)}^p$. Thus, by summing on $i' \in I$, and using that for each $i \in  I$ there are at most $C$ many  $i' \in I$ with $i \lra i'$, we have
\begin{align}
\|G\|_{L^{m,p}(K_{CZ})}^p &= \sum_{i' \in I} \|G\|_{L^{m,p}(Q_{i'})}^p \nonumber \\
& \lesssim \sum_{i \in I} \|G_i, P_{x_i}\|_{\J_*(f,\mu|_{1.1Q_i \cap Q^\circ};1.1Q_i \cap Q^\circ)}^p + \sum_{(i,i') : i \lra i'} |P_{x_i} - P_{x_{i'}}|_i^p. \label{apatchlmp}
\end{align}

Because $\theta_i \leq 1$, $\supp(\theta_i) \subset 1.1Q_i$, and $\sum \theta_i = 1$ on $K_{CZ}$, we have
\begin{align}
\int_{K_{CZ}} |G-f|^p d\mu &= \int_{K_{CZ}} \Big|\sum_{i \in I} (G_i - f)\cdot \theta_i\Big|^p d\mu \nonumber \\
&\lesssim \sum_{i \in I} \int_{1.1Q_i \cap Q^\circ} |G_i - f|^p d\mu \nonumber \\
& \lesssim \sum_{i \in I} \|G_i, P_{x_i}\|_{\J_*(f,\mu|_{1.1Q_i \cap Q^\circ};1.1Q_i \cap Q^\circ)}^p. \label{ap1}
\end{align}

First applying that $\sum \theta_{i'} = 1$ on $K_{CZ}$, and then  $x \in \supp(\theta_i') \implies x \in 1.1 Q_{i'}$, with (\ref{goodgeo}), (\ref{g1}), and Lemma \ref{polynorm:lem}, we obtain,
\begin{align}
\sum_{i \in I} \|G-P_{x_i}\|_{L^p(Q_i)}^p  / \delta_{Q_i}^{mp} &= \sum_{i \in I} \big\| \sum_{i' \in I: i' \lra i}(G_{i'} - P_{x_i})\theta_{i'}\big\|_{L^p(Q_i)}^p/\delta_{Q_i}^{mp} \nonumber \\
& \lesssim \sum_{i \in I} \sum_{i' \in I: i' \lra i} \left[ \|G_{i'} - P_{x_{i'}}\|_{L^p(1.1Q_{i'}\cap Q^\circ)}^p/\delta_{Q_{i'}}^{mp}+\|P_{x_i} - P_{x_{i'}}\|_{L^p(Q_i)}^p/\delta_{Q_i}^{mp} \right] \nonumber \\
& \lesssim \sum_{i \in I}  \|G_i-P_{x_i}\|_{L^p(1.1Q_i\cap Q^\circ)}^p / \delta_{Q_i}^{mp}  +  \sum_{(i,i') : i' \lra i} |P_{x_i} - P_{x_{i'}}|_i^p \nonumber \\
& \lesssim \sum_{i \in I} \|G_i, P_{x_i} \|_{\J_*(f,\mu|_{1.1Q_i\cap Q^\circ};1.1Q_i\cap Q^\circ)}^p  +  \sum_{(i,i') : i' \lra i} |P_{x_i} - P_{x_{i'}}|_i^p . \label{apatchlp}
\end{align}

From (\ref{apatchlmp}), (\ref{ap1}), and (\ref{apatchlp}), we conclude,
\begin{align*}
\|G, \vec{P} \|_{\J_*(f,\mu;K_{CZ}, CZ^\circ)}^p &= \|G\|_{L^{m,p}(K_{CZ})}^p + \int_{K_{CZ}} |G-f|^p d\mu  + \sum_{i \in I} \|G-P_{x_i}\|_{L^p(Q_i)}^p/\delta_{Q_i}^{mp} \\
&\lesssim \sum_{i \in I} \|G_i, P_{x_i}\|_{\J_*(f,\mu|_{1.1Q_i \cap Q^\circ};1.1Q_i \cap Q^\circ)}^p + \sum_{i' \lra i} |P_{x_i} - P_{x_{i'}}|_i^p.
\end{align*}
\end{proof}

\subsubsection{Patching Estimate on $Q^\circ$}
\label{sec:patchQ}

Recall that $K_{CZ} \subset Q^\circ$ and $K_p = Q^\circ \setminus K_{CZ}$. We showed that $K_p$ is a closed set in $\R^n$, and $K_p \subset \supp(\mu) \subset \frac{1}{10} Q^\circ$.

We have defined $\mathfrak{B}_{CZ} = \{ x_i \}_{i \in I}$, the set of all CZ basepoints, with $x_i$ the center of $Q_i$ for each $Q_i \in CZ^\circ$.

Given $\vec{S} \in Wh(K_p)$, $\vec{P} \in Wh(\bpt)$, we regard $(\vec{P},\vec{S}) \in Wh( \bpt \cup K_p)$ as a Whitney field on $K_p \cup \bpt$. 

\begin{lem}
Fix a collection of functions $\{G_i\}_{i \in I} \subset L^{m,p}(\R^n)$, and two Whitney fields $\vec{R}=(R_{x_i})_{i \in I} \in Wh(\bpt)$ and $\vec{S} = (S_x)_{x \in E} \in Wh(K_p)$. Let $G:Q^\circ \to \R^n$ be defined as
\begin{align*}
    G(x) = \begin{cases} \sum_{i \in I} G_i(x) \cdot \theta_i(x) & x \in K_{CZ} \\
    S_x(x) & x \in K_p,
    \end{cases}
\end{align*}
where $\{\theta_i \}_{i \in I}$ is a partition of unity satisfying \textbf{(POU1)-(POU4)} (see Section \ref{subsec:pou}). 

If $(G,\vec{R}, \vec{S})$ satisfy the conditions 
\[
\|G,\vec{R} \|_{\J_*(f,\mu;K_{CZ},CZ^\circ)} + \|\vec{S}\|_{L^{m,p}(K_p)} < \infty \mbox{ and } (\vec{S}, \vec{R})\in C^{m-1,1-n/p}(K_p \cup \bpt),
\]
then $G \in L^{m,p}(Q^{\circ})$, $J_xG = S_x$ for all $x \in K_p$, and
\begin{align}
    \| G \|_{L^{m,p}(Q^\circ)} \lesssim \|G,\vec{R} \|_{\J_*(f,\mu;K_{CZ},CZ^\circ)} + \|\vec{S}\|_{L^{m,p}(K_p)}. \label{apest}
\end{align} \label{lmppatch}
\end{lem}

\begin{remark}
In later applications of Lemma \ref{lmppatch}, the hypothesis $\|G,\vec{R} \|_{\J_*(f,\mu;K_{CZ},CZ^\circ)} < \infty$ will be verified using Lemma \ref{kczpatchlm}.
\end{remark}

\begin{proof}

By assumption, $ \|\vec{S}\|_{L^{m,p}(K_p)} < \infty$, hence for any $\eta>0$, there exists $H \in L^{m,p}(\R^n)$ satisfying $J_x H = S_x$ for all $x \in K_p$ and $ \|H\|_{L^{m,p}(\R^n)} \leq \|\vec{S}\|_{L^{m,p}(K_p)}+\eta$.

By assumption, $\|G,\vec{R} \|_{\J_*(f,\mu;K_{CZ},CZ^\circ)} < \infty$. In particular, $G \in L^{m, p}(K_{CZ})$. Thus, $J_x G \in \mpp$ is well-defined for $x \in K_{CZ}$. 

Define $\vec{P} \in Wh(Q^\circ)$ as
\begin{align*}
    P_x : = \begin{cases} J_x G &x \in K_{CZ} \\
    S_x = J_x H  & x \in K_p.
    \end{cases}
\end{align*}
By definition of $G$, observe that
\begin{equation}\label{eqn:pre1}
    P_x(x) = G(x) \mbox{ for all } x \in Q^\circ.
\end{equation}

Fix $\delta >0$, and fix a cube $\widehat{Q} \subset \R^n$ with $\delta_{\widehat{Q}} \leq \delta$. We will show that for all $x,y \in \widehat{Q} \cap Q^\circ$ we have
\begin{align}
|P_x - P_y|_{y, \delta} \lesssim \; &
\|H\|_{L^{m,p}(21\widehat{Q})}  +\sup_{Q_i \subseteq 35\widehat{Q}} \{ \|G,R_{x_i}\|_{\J_*(f,\mu;Q_i)} \} \nonumber \\
& + \sup \left\{ \|G\|_{L^{m,p}(\inte(B(z,r))\cap Q^\circ)}  : z \in 7\widehat{Q}, r \leq 7\delta, \inte(B(z,r)) \cap Q^\circ \subset K_{CZ} \right\}.
\label{patchbd}
\end{align}
Here, as usual, $B(z,r) = \{ w \in \R^n : |w-z| \leq r\}$, and $| \cdot |$ is the $\ell^\infty$ (sup) metric on $\R^n$. Thus, $B(z,r)$ is a closed cube centered at $z$ of sidelength $2r$. 

Fix $x,y,\widehat{Q},\delta$ as above. We shall split the proof of \eqref{patchbd} into cases depending on the relative positions of $x$, $y$, and $K_p$. Observe that 
\begin{equation}
    \label{eqn:xy}
    |x-y| \leq \delta_{\widehat{Q}} \leq \delta.
\end{equation}

\underline{Case 1:} Suppose $x,y \in K_p$. We apply \eqref{normdi} and the Sobolev Inequality on $\widehat{Q}$,
\begin{align}
    |P_x - P_y|_{y, \delta} &= |J_xH - J_yH|_{y, \delta} \leq |J_x H - J_y H|_{y,|x-y|} \lesssim \|H\|_{L^{m,p}(\widehat{Q})}. \label{case1p}
\end{align}
This completes the proof of  (\ref{patchbd}) in Case 1.

\underline{Case 2:} Suppose $x,y \in K_{CZ}$ satisfy $\dist(y,K_p) > 3|y-x|$ or $\dist(x,K_p) > 3|y-x|$. Because $|P|_{y,\delta} \simeq |P|_{x,\delta}$ for $\delta \geq |x-y|$ by (\ref{movept}), we may assume without loss of generality the first case occurs. By the triangle inequality, $B(y, |y-x|) \subset \R^n \setminus K_p$. Thus,  we can apply \eqref{normdi} and the Sobolev Inequality on $B(y, |y-x|) \cap Q^\circ \subset K_{CZ}$, and obtain
\begin{align*}
    |P_x - P_y|_{y, \delta} = |J_xG - J_yG|_{y, \delta}  \leq |J_x G - J_y G|_{y, |x-y|} \lesssim \|G\|_{L^{m,p}(B(y,|y-x|) \cap Q^\circ)}.
\end{align*}
Therefore, we can upper bound $|P_x - P_y|_{y, \delta}$ by the second supremum in (\ref{patchbd}). This completes the proof of (\ref{patchbd}) in Case 2.

\underline{Case 3:} Suppose $y \in K_{CZ}$ and $x \in K_p$, or $y \in K_p$ and $x \in K_{CZ}$. As in Case 2, without loss of generality, $y \in K_{CZ}$ and $x \in K_p$. Because $K_p$ is closed (see Lemma \ref{kpclosed}), there exists $z_y \in K_p$ satisfying $\dist(y,K_p) = |z_y-y|$. Because $x \in K_p$, and from (\ref{eqn:xy}), we have 
\begin{align}
&|z_y-y| \leq |x-y| \leq \delta_{\widehat{Q}} \leq \delta; \label{z1} \\
&B(y,|z_y-y|) \subset 3\widehat{Q}; \label{z2} \\
&|x-z_y| \leq |x-y| + |y-z_y| \leq 2 \delta_{\widehat{Q}} \leq 2 \delta. \label{z2.5}
\end{align} 

Because $z_y$ is a closest point of $K_p$ to $y$, $\inte B(y,|z_y-y|) \subset \R^n \setminus K_p$. We write $[y,z)$ for the segment $\{ y + t(z-y) : 0 \leq t < 1\}$. Then there exists a sequence $\{z_y^k\}_{k \in \N} \subset  \R^n$ satisfying
\begin{align}
& z_y^k \in [y,z_y) \subset \inte B(y,|z_y-y|) \subset \R^n \setminus K_p \text{ for all } k \in \N; \label{z0} \\    
&\lim_{k \to \infty} z_y^k = z_y. \label{zk}
\end{align} 
Observe that
\begin{equation}\label{z000}
\dist(z_y^k, K_p) \leq |z_y^k - z_y| \rightarrow 0 \mbox{ as } k \rightarrow \infty.
\end{equation}
By \eqref{z0} and \eqref{z1}, we have 
\begin{equation}
|z_y^k -y| \leq |z_y-y| \leq  |x-y| \leq \delta_{\widehat{Q}} \leq \delta \qquad (k \in \N).  \label{z3}
\end{equation}
Observe that $z_y^k \in [y, z_y) \subset Q^\circ$ by convexity of $Q^\circ$, so $z_y^k \in (\R^n \setminus K_p) \cap Q^\circ = K_{CZ}$. Thus, we can define a map $\tau_y:\N \to I$, such that
\begin{align*}
    \tau_y(k) =i \quad \quad \text{ if } z_y^k \in Q_i \in CZ^\circ.
\end{align*}
Consequently,
\begin{align}
    &\delta_{Q_{\tau_y(k)}} \leq \dist(z_y^k, K_p) \leq |z_y^k - z_y| \leq \dist(y,K_p) = |y - z_y| \leq |y-x| \leq \delta_{\widehat{Q}} \leq \delta \qquad (k \in \N);  \label{z4} \\
    &Q_{\tau_y(k)} \subset B(y, 2|y-x|) \subset 5\widehat{Q} \qquad (k \in \N). \label{z5}
\end{align}
For the proof of \eqref{z4}, we use Lemma \ref{kpdist}, which implies $\delta_{Q_{\tau_y(k)}} \leq \dist(Q_{\tau_y(k)}, K_p) \leq \dist(z_y^k, K_p) $, where the second inequality uses that $z_y^k \in Q_{\tau_y(k)}$; further, $\dist(z_y^k, K_p) \leq |z_y^k - z_y| \leq \dist(y, K_p) = |y-z_y|$ because $z_y^k$ is on the segment connecting $y$ to a nearest point $z_y$ of $K_p$, and the remaining inequalities in \eqref{z4} are immediate from \eqref{z1}. Lastly, the first inclusion of \eqref{z5} uses that $z_y^k\in Q_{\tau_y(k)}$, $|z_y^k-y| \leq |x-y|$ and $\delta_{Q_{\tau_y(k)}} \leq |x-y|$ (see \eqref{z3} and \eqref{z4}); the second inclusion of  \eqref{z5} uses that $y \in \widehat{Q}$ and $|y-x| \leq \delta_{\widehat{Q}}$ (see \eqref{z1}).

Let $k \in \N$. (We will later send $k \rightarrow \infty$.) Because $x \in K_p$ and $y \in K_{CZ}$, we have $P_x = S_x$ and $P_y = J_y G$. By (\ref{z3}), we have $|y-z_y^k| \leq |y-z_y| \leq |y-x| \leq \delta$. By applying the triangle inequality, and then \eqref{movept},
\begin{align}
    |P_x -  P_y|_{y,\delta} & = |S_x - J_yG|_{y,\delta} \nonumber \\
    &\leq |S_x - S_{z_y}|_{y,\delta}+|S_{z_y} - R_{x_{\tau_y(k)}}|_{y,\delta} + |R_{x_{\tau_y(k)}} - J_{z_y^k}G|_{y,\delta} + |J_{z_y^k}G - J_yG|_{y,\delta} \nonumber \\
     &\overset{\eqref{movept}}{\lesssim} |S_x - S_{z_y}|_{x,\delta}+|S_{z_y} - R_{x_{\tau_y(k)}}|_{z_y, \delta} + |R_{x_{\tau_y(k)}} - J_{z_y^k}G|_{z_y^k,\delta}  + |J_{z_y^k}G - J_yG|_{y, \delta} .\label{z6}
\end{align}
We analyze the four terms on the right-hand side of \eqref{z6}, one by one.

From (\ref{z1}), since $y \in \widehat{Q}$, we deduce $z_y \in 3\widehat{Q}$. Also, note that $|x-z_y| \leq 2 \delta$, according to (\ref{z2.5}). We apply (\ref{normdi}), and then the Sobolev Inequality on $3 \widehat{Q}$ to estimate
\begin{align}
    |S_x - S_{z_y}|_{x,\delta} \lesssim |S_x - S_{z_y}|_{x,|x-z_y|} = |J_xH - J_{z_y}H|_{x,|x-z_y|}\lesssim \|H\|_{L^{m,p}(3\widehat{Q})}. \label{z6.5}
\end{align}

Because $(\vec{S}, \vec{R}) \in C^{m-1,1-n/p}(K_p \cup \bpt)$, we have $R_{x_i} \rightarrow S_z$ whenever $x_i \in \bpt$, $x_i \rightarrow z$, $z \in K_p$. Observe, $|x_{\tau_y(k)} - z^k_y| \leq \delta_{Q_{\tau_y(k)}}$ (as both  $x_{\tau_y(k)}$ and $z^k_y$ belong to  $Q_{\tau_y(k)})$. Further, $\delta_{Q_{\tau_y(k)}} \leq \dist(z_y^k, K_p) \rightarrow 0$ as $k \rightarrow \infty$ (see \eqref{z000} and \eqref{z4}). When combined with (\ref{zk}), this implies $x_{\tau_y(k)} \rightarrow z_y$ as $k \rightarrow \infty$. Hence, $R_{x_{\tau_y(k)}} \rightarrow S_{z_y} $ as $k \rightarrow \infty$. Thus,
\begin{align}
    \lim_{k \to \infty} |S_{z_y} - R_{x_{\tau_y(k)}}|_{z_y, \delta} = 0. \label{z7}
\end{align}

From (\ref{z4}), $\delta_{Q_{\tau_y(k)}} \leq \delta$, So, applying (\ref{normdi}) and then (\ref{normest}), using that $z^k_y$ belongs to $Q_{\tau_y(k)}$, we have
\begin{align}
    |R_{x_{\tau_y(k)}} - J_{z_y^k}G|_{z_y^k,\delta} \leq |R_{x_{\tau_y(k)}} - J_{z_y^k}G|_{z_y^k,\delta_{Q_{\tau_y(k)}}} \lesssim \|G,R_{x_{\tau_y(k)}}\|_{\J_*(f,\mu;Q_{\tau_y(k)})}. \label{z8}
\end{align}

We recall that $|z_y^k - y| \leq \delta$ (see \eqref{z3}). Also note that by (\ref{z0}), $\inte (B(y,|z_y-y|)) \cap Q^\circ \subset (\R^n \setminus K_p) \cap Q^\circ = K_{CZ}$. So we can apply (\ref{normdi}) and then the Sobolev inequality on $\inte (B(y,|z_y-y|)) \cap Q^\circ$, and deduce
\begin{align}
    |J_{z_y^k}G - J_yG|_{y,\delta} \leq |J_{z_y^k}G - J_yG|_{y, |z_y^k - y|} \lesssim \|G\|_{L^{m,p}( \inte (B(y,|z_y-y|)) \cap Q^\circ)}. \label{z9}
\end{align}

Recalling (\ref{z5}), (\ref{z1}), (\ref{z0}), we let $k \rightarrow \infty$ in (\ref{z6}), and use (\ref{z6.5})-(\ref{z9}) to conclude 
\begin{align}
     |P_x - P_y|_{y,\delta}  \lesssim \; & \|H\|_{L^{m,p}(3\widehat{Q})}+\sup_{Q_i \subseteq 5\widehat{Q}} \{ \|G,R_{x_i}\|_{\J_*(f,\mu;Q_i)} \} \nonumber \\
     & + \sup \left\{ \|G\|_{L^{m,p}( \inte(  B(z,r)) \cap Q^\circ)} : z \in \widehat{Q}, r \leq \delta, \inte(B(z,r)) \cap Q^\circ \subset K_{CZ} \right\}. \label{case3}
\end{align}
This completes the proof of inequality (\ref{patchbd}) in Case 3.

\underline{Case 4:} Suppose $x,y \in K_{CZ}$ satisfy $\dist(y,K_p) \leq 3|y-x|$ and $\dist(x,K_p) \leq 3|y-x|$. Because $K_p$ is closed, there exist $z_x, z_y \in K_p$ such that  $\dist(x,K_p) = |z_x-x|$ and $\dist(y,K_p) = |z_y-y|$. Hence, we have 
\begin{align}
|z_x-x|, |z_y-y| \leq 3 |y-x| \leq 3 \delta_{\widehat{Q}} \leq 3 \delta. \label{az1}    
\end{align}
Consequently,
\begin{align}
    &|z_x -z_y| \leq |z_x -x| +|x-y|+|y-z_y| \leq 7 \delta_{\widehat{Q}} \leq 7\delta, \label{az2}
\end{align}
and $z_x, z_y \in 7 \widehat{Q}$ because $x,y \in \widehat{Q}$. Considering (\ref{az1}) and (\ref{az2}), we apply (\ref{movept}) and (\ref{normdi}) to estimate
\begin{align*}
    |P_x - P_y|_{y,\delta} \lesssim |P_x - P_{z_x}|_{x,7 \delta} + |P_{z_x} - P_{z_y}|_{z_x,7 \delta} + |P_{z_y} - P_y|_{y,7 \delta}.
\end{align*}
Because $z_x, z_y \in K_p$, while $z,y \in K_{CZ}$, we apply inequalities (\ref{case1p}) of Case 1 and (\ref{case3}) of Case 3 with the cube  $7 \widehat{Q}$ playing the role of $\widehat{Q}$, and $7 \delta$ playing the role of $\delta$, to further reduce this:
\begin{align*}
     |P_x - P_y|_{y,\delta} \lesssim \;& \|H\|_{L^{m,p}(21\widehat{Q})}+\sup_{Q_i \subseteq 35\widehat{Q}} \{ \|G,R_{x_i}\|_{\J_*(f,\mu;Q_i)} \}\\
     & + \sup \left\{ \|G\|_{L^{m,p}( \inte (B(z,r)) \cap Q^\circ )} : z \in 7\widehat{Q}, r \leq 7\delta, \inte(B(z,r)) \cap Q^\circ \subset K_{CZ} \right\}.
\end{align*}
This completes the proof of the inequality (\ref{patchbd}) in Case 4.

Since Cases 1--4 are exhaustive, we have proven (\ref{patchbd}).

We prepare to apply Corollary \ref{brlm} to show that $G \in L^{m,p}(Q^\circ)$.  Fix $\delta >0$ and fix a congruent $\delta$-packing $\pi \in \Pi_{\simeq}(Q^\circ)$. Thus, $\pi$ is a family of cubes in $Q^\circ$ of equal sidelength $\delta$, with pairwise disjoint interiors. 

For the next calculation we use the terminology of local approximation error, $E(G,Q)$, in Section \ref{subsec:char_ss}. For $\widehat{Q} \in \pi$, $z \in \widehat{Q}$, $\delta = \delta_{\widehat{Q}}$, we have
\begin{align*}
E(G,\widehat{Q})/\delta^m \leq \|G-P_z\|_{L^p(\widehat{Q})}/\delta^m &\leq \| G - P_z \|_{L^\infty(\widehat{Q})} \delta^{n/p - m} \\
&\leq \sup_{x,y \in \widehat{Q}} \{|G(y) - P_x(y)| \}\delta^{n/p-m} \\
&\leq \sup_{x, y \in \widehat{Q}} |P_y - P_x|_{y, \delta}, 
\end{align*}
where the last line follows from \eqref{eqn:pre1}. Combining this with (\ref{patchbd}),
\begin{align*}
      \sum_{\widehat{Q} \in \pi} & \big(E(G,\widehat{Q}\big)/ \delta^m)^p \lesssim \sum_{\widehat{Q} \in \pi} \sup_{x, y \in \widehat{Q}} |P_y - P_x|_{y, \delta}^p \\
      &\lesssim \sum_{\widehat{Q} \in \pi} \Big( \|H\|_{L^{m,p}(21\widehat{Q})}+\sup_{Q_i \subseteq 35\widehat{Q}} \{ \|G,R_{x_i}\|_{\J_*(f,\mu;Q_i)} \}  + \sup_{\substack{z \in 7\widehat{Q}, r \leq 7\delta\\ \inte(B(z,r)) \cap Q^\circ \subset K_{CZ}}} \{ \|G\|_{L^{m,p}(\inte(B(z,r)) \cap Q^\circ)} \}\Big)^p\\
      &\lesssim \|G, \vec{R}\|_{\J_*(f,\mu;K_{CZ}, CZ^\circ)}^p + \|H\|_{L^{m,p}(\R^n)}^p \\
      &\lesssim \|G, \vec{R}\|_{\J_*(f,\mu;K_{CZ}, CZ^\circ)}^p +(\|\vec{S}\|_{L^{m,p}(K_p)}+\eta)^p,
\end{align*}
where the second to last inequality follows because $y \in 7 \widehat{Q}$, $r \leq 7 \delta \implies B(y,r) \subset 35 \widehat{Q}$, and because $\{35 \widehat{Q}\}_{\widehat{Q} \in \pi}$ has bounded overlap (recall $\pi$ consists of cubes with pairwise disjoint  interiors and equal sidelength), and the last inequality follows because $\| H \|_{L^{m,p}(\R^n)} \leq \|\vec{S}\|_{L^{m,p}(K_p)}+\eta$. 
We let $\eta \to 0$, then take the supremum over $\pi \in \Pi_{\simeq}(Q^\circ)$ and apply Corollary \ref{brlm} to conclude $G \in L^{m,p}(Q^\circ)$ and 
\begin{align*}
    \|G\|_{L^{m,p}(Q^\circ)} \lesssim \|G, \vec{R}\|_{\J_*(f,\mu;K_{CZ}, CZ^\circ)} +\|\vec{S}\|_{L^{m,p}(K_p)}.
\end{align*}

Next, we will show $J_x G = S_x$ for all $x \in K_p$. We claim that $\vec{P} \in Wh(Q^\circ)$ satisfies $|P_x-P_y|_{y,|x-y|} \leq A < \infty$ for all $x,y \in Q^\circ$. Let $x,y \in Q^\circ$, and fix $Q \subset Q^\circ$ with $x,y \in Q$ and $\delta_Q = |x-y|$. By (\ref{patchbd}), we have
\begin{align*}
|P_x - P_y|_{y, |x-y|} &\lesssim 
\|H\|_{L^{m,p}(21Q)}  +\sup_{Q_i \subseteq 35Q} \{ \|G,R_{x_i}\|_{\J_*(f,\mu;Q_i)} \} + \sup_{\substack{z \in 7Q, r \leq 7\delta\\ \inte(B(z,r)) \cap Q^\circ \subset K_{CZ}}} \{ \|G\|_{L^{m,p}(\inte(B(z,r)) \cap Q^\circ)} \} \\
&\lesssim 
\|H\|_{L^{m,p}(\R^n)}  +  \|G,\vec{R}\|_{\J_*(f,\mu;K_{CZ}, CZ^\circ)}   +  \|G\|_{L^{m,p}(K_{CZ})} \\
&\lesssim \|H\|_{L^{m,p}(\R^n)} +  \|G,\vec{R}\|_{\J_*(f,\mu;K_{CZ}, CZ^\circ)} < \infty.
\end{align*}

Since $G(x)=P_x(x)$ for $x \in Q^\circ$ (see (\ref{eqn:pre1}), we can apply Lemma \ref{whitlm} to deduce $J_x G = P_x$ for all $x \in Q^\circ$. Since $P_x = S_x$ for $x \in K_p$, we have $J_x G = S_x$ for all $x \in K_p$. This completes the proof of Lemma \ref{lmppatch}.

\end{proof}

\subsubsection{Patching Estimates for Restriction of $\mu$}

We state variants of the last two lemmas for the restriction of the measure $\mu$ to a Borel set $E \subset \R^n$. Their proofs follow from the proofs of Lemma \ref{kczpatchlm} and Lemma \ref{lmppatch} with the measure $\mu$ replaced by $\mu|_E$.

\begin{lem}
Suppose we are given a collection of functions $\{G_i\}_{i \in I} \subset L^{m,p}(\R^n)$, a Whitney field $\vec{P} = (P_{x_i})_{i \in I} \in Wh(\bpt)$, and a Borel set $E \subset Q^\circ$.

Define $G: K_{CZ} \to \R$ by $G(x)= \sum_{i \in I} G_i(x) \cdot \theta_i(x)$, where $\{\theta_i\}_{i \in I}$ is a partition of unity satisfying \textbf{(POU1)-(POU4)} (see Section \ref{subsec:pou}). Then
\begin{align*}
\|G, \vec{P} \|_{\J_*(f,\mu|_E;K_{CZ},CZ^\circ)}^p \lesssim &  \sum_{i \in I} \|G_i, P_{x_i}\|_{\J_*(f,\mu|_{1.1Q_i \cap E};1.1Q_i \cap Q^\circ)}^p  + \sum_{i, i' \in I, \; i \lra i'} |P_{x_i} - P_{x_{i'}}|_i^p. 
\end{align*}
\label{rkczpatchlm}
\end{lem}

\begin{lem}
Fix a collection of functions $\{G_i\}_{i \in I} \subset L^{m,p}(\R^n)$, and Whitney fields $\vec{R}=(R_{x_i})_{i \in I} \in Wh(\bpt)$ and $\vec{S} = (S_x)_{x \in E} \in Wh(K_p)$. Let $E \subset Q^\circ$ be a Borel set. Let $G:Q^\circ \to \R^n$ be defined as
\begin{align*}
    G(x) = \begin{cases} \sum_{i \in I} G_i(x) \cdot \theta_i(x) & x \in K_{CZ} \\
    S_x(x) & x \in K_p,
    \end{cases}
\end{align*}
where $\{\theta_i \}_{i \in I}$ is a partition of unity satisfying \textbf{(POU1)-(POU4)}. 

If $(G,\vec{R}, \vec{S})$ satisfies the conditions $\|G,\vec{R} \|_{\J_*(f,\mu|_E;K_{CZ},CZ^\circ)} + \|\vec{S}\|_{L^{m,p}(K_p)} < \infty$ and  $(\vec{S}, \vec{R})\in C^{m-1,1-n/p}(K_p \cup \bpt)$ then $G \in L^{m,p}(Q^{\circ})$, $J_xG = S_x$ for all $x \in K_p$, and
\begin{align*}
    \| G \|_{L^{m,p}(Q^\circ)} \lesssim \|G,\vec{R} \|_{\J_*(f,\mu|_E;K_{CZ},CZ^\circ)} + \|\vec{S}\|_{L^{m,p}(K_p)}.
\end{align*} \label{rlmppatch}
\end{lem}

\section{Further Constraints on Extension}
\label{sec:further_cons}

Let $Q^\circ$, $CZ^\circ = \{Q_i\}_{i \in I}$, $K_{CZ}$, and $K_p$ be as defined in Section \ref{sec:cz_decomp}. Let $\mathfrak{B}_{CZ} = \{ x_i \}_{i \in I}$ be the set of all CZ basepoints.

\subsection{Definition and properties of the space $\J_*(\mu;Q^\circ, CZ^\circ; K_p)$.}

For  $(f,\vec{P},\vec{S}) \in \J(\mu) \times Wh(\bpt)\times Wh(K_p)$, we define: 
\begin{equation}\label{Jstar_func_3A:defn}
\|f,\vec{P}, \vec{S}\|_{\J_*(\mu;Q^\circ, CZ^\circ;K_p)} = \inf \left\{ \|F,\vec{P}\|_{\J_*(f, \mu;Q^\circ,CZ^\circ)} : \begin{aligned}& F \in L^{m,p}(Q^\circ) \text{ and} \\ 
& J_x(F) = S_x \text{ for all } x \in K_p \end{aligned} \right\}.
\end{equation}
Here, we have used the $\J_*(f, \mu;\Omega,CZ^\circ)$ functional defined in \eqref{new_Jfunc1:eqn}, with $\Omega = Q^\circ$.

We define the seminormed vector space:
\begin{align} \label{Jstar_space_3A:defn}
   \J_*(\mu;Q^\circ, CZ^\circ; K_p) = \left\{ (f,\vec{P},\vec{S}): \begin{aligned} &f \in \J(\mu), \; \vec{P} \in Wh(\bpt), \; \vec{S} \in Wh(K_p),  \\
   &\|f,\vec{P}, \vec{S}\|_{\J_*(\mu;Q^\circ,CZ^\circ;K_p)} < \infty \end{aligned} \right\}.
\end{align}

The next result gives a compatibility condition between the Whitney fields $\vec{P}$ and $\vec{S}$ whenever $(f,\vec{P}, \vec{S}) \in \J_*(\mu;Q^\circ, CZ^\circ; K_p)$.

\begin{prop}\label{spinprop}
Let $(f,\vec{P}, \vec{S}) \in \J_*(\mu;Q^\circ, CZ^\circ; K_p)$. Then $(\vec{P},\vec{S}) \in Wh(\bpt \cup K_p)$, and
\begin{align}
    \|(\vec{P},\vec{S})\|_{C^{m-1,1-n/p}(\bpt \cup K_p)} \leq C_H \|f, \vec{P}, \vec{S}\|_{\J_*(\mu;Q^\circ, CZ^\circ; K_p)}. \label{spineq}
\end{align}
Here, $C_H > 0$ is a universal constant, determined only by $m$, $n$, and $p$.
\end{prop}

\begin{proof}
Let $\eta >0$ be arbitrary. Let $H \in L^{m,p}(Q^\circ)$ satisfy $J_x H = S_x$ for all $x \in K_p$, and 
\begin{align}
\|H, \vec{P}\|_{\J_*(f,\mu; Q^\circ, CZ^\circ)} \leq  \|f,\vec{P},\vec{S}\|_{\J(\mu; Q^\circ, K_{CZ}; K_p)} +\eta. \label{j0}
\end{align}
\underline{Case 1:} As a consequence of the Sobolev Inequality, if $x, y \in K_p$, then
\begin{align}
|S_x - S_y|_{y,|y-x|} = |J_x H - J_y H|_{y, |y-x|} \lesssim \|H\|_{L^{m,p}(Q^\circ)} \lesssim \|H, \vec{P}\|_{\J_*(f,\mu; Q^\circ, CZ^\circ)}. \label{hj1}
\end{align}
\underline{Case 2:} If $x \in K_p$ and $y \in \bpt$, then $y = x_i$ for some $i \in I$, and thanks to Lemma \ref{kpdist}, $\delta_{Q_i} \leq \dist(x_i, K_p) \leq |x-x_i|$. Hence from the Sobolev Inequality, \eqref{normdi}, and (\ref{normest}),
\begin{align}
    |S_x -P_{x_i}|_{x_i,|x-x_i|} &\leq |J_xH - J_{x_i}H|_{x_i,|x-x_i|} + |J_{x_i}H - P_{x_i}|_{x_i,\delta_{Q_i}}\nonumber \\
    &\lesssim \|H\|_{L^{m,p}(Q^\circ)} + \|H,P_{x_i}\|_{\J_*(f,\mu;Q_i)} \nonumber \\
    &\lesssim \|H, \vec{P}\|_{\J_*(f,\mu; Q^\circ, CZ^\circ)}. \label{hj2}
\end{align}
Consequently, using \eqref{movept}, we get that $|S_x -P_{x_i}|_{x,|x-x_i|} \lesssim \|H, \vec{P}\|_{\J_*(f,\mu; Q^\circ, CZ^\circ)}$.\\
\underline{Case 3:} Similar to Case 2, if $x,y \in \bpt$ are distinct, then $y = x_i$, $x = x_j$ for some $i,j \in I$, and $\delta_{Q_i}, \delta_{Q_j} \leq C |x_i-x_j|$. From the Sobolev Inequality, (\ref{movept}), \eqref{normdi}, and (\ref{normest}),
\begin{align}
    |P_{x_{j}} -P_{x_i}|_{x_i,|x_j-x_i|} &\leq |J_{x_j}H - J_{x_i}H|_{x_i,|x_j-x_i|} + |J_{x_i}H - P_{x_i}|_{x_i,\delta_{Q_i}} + |J_{x_j}H - P_{x_{j}}|_{x_j,\delta_{Q_j}}\nonumber \\
    &\lesssim \|H\|_{L^{m,p}(Q^\circ)} + \|H,P_{x_i}\|_{\J_*(f,\mu;Q_i)} +\|H,P_{x_{j}}\|_{\J_*(f,\mu;Q_j)} \nonumber \\
    &\lesssim \|H, \vec{P}\|_{\J_*(f,\mu; Q^\circ, CZ^\circ)}. \label{hj3}
\end{align}
From (\ref{hj1}), (\ref{hj2}), and (\ref{hj3}), we have 
\begin{align*}
    \|(\vec{P},\vec{S})\|_{C^{m-1,1-n/p}(\bpt \cup K_p)} \lesssim \|H, \vec{P}\|_{\J_*(f,\mu; Q^\circ, CZ^\circ)}.
\end{align*}
In light of (\ref{j0}), and since $\eta > 0$ is arbitrary, we have proven inequality (\ref{spineq}).
\end{proof}

\subsection{Coherency}\label{subsec:coh}

We start by introducing two pieces of terminology. Recall we have fixed a multiindex set $\ma \subset \mm$.
 
\begin{define}[Coherency]
Let $K \subset \R^n$ and $P_0 \in \mpp$. We say that a Whitney field $(P_x)_{x \in K} \in Wh(K)$ is coherent with $P_0$ if  $\p^\al P_x(x) = \p^\al P_0(x)$ for all $x \in K$, and for all $\al \in \ma$.
\end{define}
\begin{define}[$K$-Coherency]
Let $K \subset Q^\circ \subset \R^n$ and $P_0 \in \mpp$. We say that a function $F \in C^{m-1,1-n/p}(Q^\circ)$ is $K$-coherent with $P_0$ if $\p^\al J_xF(x) = \p^\al F(x)= \p^\al P_0(x)$  for all $x \in K$, and for all $\al \in \ma$.
\end{define}

The next result will be used in Section \ref{subsec:kpt_jet}, to give the proofs of the lemmas therein.

\begin{prop} For $(f,P_0) \in \J(\mu; \delta_{Q^\circ})$,
\begin{align}
\inf \Bigg\{ \| f, \vec{P}, \vec{S} \|_{\J_*(\mu; Q^{\circ}, CZ^\circ;K_p)}: \begin{aligned} &\vec{P} \in Wh(\bpt) \text{ and } \vec{S} \in Wh(K_p) \text{ satisfy }  \nonumber \\
&(\vec{P},\vec{S}) \in Wh(\bpt \cup K_p) \text{ is coherent with } P_0 
\end{aligned}
\Bigg\} \leq C \| f, P_0 \|_{\J(\mu;\delta_{Q^{\circ}})}. 
\end{align}
\label{acohlm}
\end{prop}

The rest of this section is devoted to the proof of Proposition \ref{acohlm}.

\subsubsection{Proof of Proposition \ref{acohlm}}

Let  $(f,P_0) \in \J(\mu; \delta_{Q^\circ})$ be given. 

Given $G \in L^{m,p}(\R^n)$, we will define a function $F \in L^{m,p}(Q^{\circ})$, and the Whitney fields $\vec{P} \in Wh(\bpt)$, and $\vec{S} \in Wh(K_p)$, which satisfy the following properties:
\begin{align}
&\|F, \vec{P} \|_{\J_*(f,\mu;Q^{\circ}, CZ^\circ)} \leq C \| G, P_0 \|_{\J(f,\mu;\delta_{Q^{\circ}})},\label{acoh}\\ 
&(\vec{P}, \vec{S}) \mbox{ is coherent with } P_0,  \label{acoh1} \\
&J_{x}F = S_x \mbox{ for all } x \in K_p. \label{acoh2}
\end{align}
Notice that \eqref{acoh} and \eqref{acoh2} imply that $\| f, \vec{P}, \vec{S} \|_{\J_*(\mu; Q^\circ, CZ^\circ; K_p)} \leq C \| G, P_0 \|_{\J(f, \mu ; \delta_{Q^\circ})}$. By taking the infimum in this inequality over $G \in L^{m,p}(\R^n)$, the proposition follows. For $G \in L^{m,p}(\R^n)$, 
\begin{align*}
\|G-P_0,0\|_{\J(f-P_0,\mu;\delta_{Q^\circ})} = \|G, P_0\|_{\J(f,\mu;\delta_{Q^\circ})}.
\end{align*}
Therefore, in the proof of \eqref{acoh}--\eqref{acoh2} it suffices to assume that
\[
P_0 = 0.
\]
We now explain how to construct $F,\vec{P}, \vec{S}$ that satisfy \eqref{acoh}--\eqref{acoh2} for $P_0 = 0$. For reference, see (\ref{barf}) for the definition of $F$, see (\ref{defpi}) for the definition of $\vec{P} = (P_{x_i})_{i \in I}$, and see (\ref{barjet}) for the definition of $\vec{S} = (S_x)_{x \in K_p}$.

We shall make use of the auxiliary polynomials $P_x^\al \in \mpp$ and functions $\varphi_x^\al \in L^{m,p}(\R^n)$ ($x \in Q^\circ$, $\al \in \ma$) defined in Section \ref{subsubsec:aux_poly}. Recall $(P_x^\al)$ satisfies (\ref{basis})-(\ref{bound}) and $(\varphi_x^\al)$ satisfies (\ref{phi1})-(\ref{phi3}), while $J_x \varphi_x^\al = P_x^\al$. Each $\varphi_{x}^\al$ is defined in terms of another family of functions $(\varphi^\beta)_{\beta \in \ma}$, via (\ref{phi1}) which states that
\begin{align*}
    \varphi_{x}^\al= \sum_{\beta \in \ma} A^{x}_{\al \beta} \cdot \varphi^\beta,
\end{align*}
where $(A^{x}_{\al \beta})_{\al, \beta \in \mm}$ is a $(C, C\e)$ near-triangular matrix. The inverse of a near-triangular matrix is near-triangular with comparable parameters. So, for any $x, y \in Q^\circ$ and $ \al \in \ma$, we can write $\varphi_y^\al$ as a bounded linear combination of $(\varphi_x^\beta)_{\beta \in \ma}$. Precisely, there exist coefficients $(\omega_{\al \beta}^{xy})_{\al, \beta \in \ma, x,y \in Q^\circ} \subset \R$ such that
\begin{align}
&\varphi_y^\al = \sum_{\beta \in \ma} \omega_{\al \beta}^{xy} \cdot \varphi_x^\beta \; && (\al \in \ma), \mbox{ and }  \nonumber\\
&|\omega_{\al \beta}^{xy}| \leq C \; && (\al, \beta \in \ma). \label{omega}
\end{align}
From \eqref{phi2}, we have $J_{x_i} \varphi^\al_{x_i} = P^\al_{x_i}$. It follows from equation (\ref{localbound}) in Lemma \ref{localbdlm}  that the functions $\varphi^\al_{x_i}$ are locally bounded on $3Q_i^+$. Indeed, by taking $Q = Q_i$ and $y=x_i$ in (\ref{localbound}), we have 
\begin{align}
|\p^\beta \varphi_{x_i}^\al (x_i)| = |\p^\beta P_{x_i}^\al (x_i)| \leq C \delta_{Q_i}^{|\al| - |\beta|} \quad \quad (\al \in \ma, \; \beta \in \mm, \; i \in I). \label{cobd} 
\end{align}

We shall define functions $F_i$ on the CZ cubes, and patch them together on $K_{CZ}$ using a partition of unity. Define $F_i: \R^n \to \R$ and $P_{x_i} \in \mpp$ for each $i \in I$ as
\begin{align}
&F_i (x) = G(x) - \sum_{\al \in \ma} \p^\al G(x_i) \cdot \varphi_{x_i}^\al(x); \text{ and}
\label{deffi}\\
&P_{x_i} = J_{x_i}F_i = J_{x_i} G - \sum_{\al \in \ma} \p^\al G(x_i) \cdot P_{x_i}^\al . \label{defpi}
\end{align}
From (\ref{basis}), the polynomials $(P^\al_{x_i})_{\al \in \ma}$ satisfy $ \partial^\beta P^\al_{x_i}(x_i) = \delta_{\al \beta}$ for $\al, \beta \in \ma$. Thus,
\begin{equation}\label{Pcoh}
\p^\beta P_{x_i}(x_i) = \partial^\beta F_i(x_i) = 0 \mbox{ for }  i \in I, \beta \in \ma.
\end{equation}

Now we define $F: Q^\circ \to \R$. 
\begin{align}
F(x) = \begin{cases} 
\sum_{i \in I} F_i(x) \cdot \theta_i(x) & x \in K_{CZ}, \\
G(x) & x \in K_p.
\end{cases} \label{barf}    
\end{align}
where $\{\theta_i\}_{i \in I}$ is a partition of unity satisfying \textbf{(POU1)-(POU4)} (see Section \ref{subsec:pou}). By construction, $F \in C^{m-1}_{loc}(K_{CZ})$, and hence, $J_yF$ is well-defined for $y \in K_{CZ}$. 

For any $y \in Q^\circ$, we define the polynomial
\begin{align}
    S_y: = \begin{cases}
    J_yF & y \in K_{CZ}, \\
    J_yG - \sum_{\al \in \ma} \p^\al G(y) P_y^\al  & y \in K_p .
    \end{cases} \label{barjet}
\end{align} 
By restricting the family of polynomials $(S_y)_{y \in Q^\circ}$ to $K_p$, we obtain the Whitney field $\vec{S} \in Wh(K_p)$.

The basis $(P_y^\al)_{\al \in \ma}$ satisfies $\partial^\beta P_y^\al(y) = \delta_{\al \beta}$ for $\al, \beta \in \ma$. Hence, $\partial^\beta S_y(y) = 0$ for $y \in K_p$, $\beta \in \ma$. Thus,
\begin{equation}\label{Rcoh}
    (S_y)_{y \in K_p} \mbox{ is coherent with } P_0 = 0.
\end{equation}
The cutoff functions $\theta_i$ satisfy  $x_i \in \supp(\theta_j)$ only if $j = i$, and $J_{x_i} \theta_i = 1$, so from  (\ref{barjet}), (\ref{barf}), (\ref{defpi}), $S_{x_i} = J_{x_i}F = J_{x_i}F_i = P_{x_i}$.  Thus,
\begin{align}\label{barjet_prop}
S_{x_i} = P_{x_i} \qquad (i \in I).
\end{align}
By Lemma \ref{lem:auxpolyzeros}, we have $P^\al_y(y) = 0$ for $y \in K_p$, so by definition of $F$ (see \eqref{barf}), we have $S_y(y) = J_y G(y) = G(y) = F(y)$ for $y \in K_p$. Evidently, also $S_y(y) = F(y)$ for $y \in K_{CZ}$. Thus,
\begin{align}\label{barjet_prop2}
S_y(y) =F(y) \qquad (y \in Q^\circ).
\end{align}

\begin{lem}
Let $x,y \in K_p \cup \bpt$. Let $U \subset \R^n$ be a domain such that $x,y \in U$, and such that $U$ is the union of two $\eta$-non-degenerate boxes with an interior point in common (in particular, $U$ can be a cube, with $\eta = 1$). Then
\begin{align}
    |S_x - S_y|_{y,|y-x|} \lesssim_\eta \|G\|_{L^{m,p}(U)} + \|G, P_0\|_{\J(f,\mu;\delta_{Q^\circ})} \cdot \sum_{\beta \in \ma} \|\varphi^\beta\|_{L^{m,p}(U)}, \label{aqbd}
\end{align}
where $(\varphi^\beta)_{\beta \in \ma}$ is as defined in Proposition \ref{phiprop} of Section \ref{subsubsec:aux_poly}, and the constants in $\lesssim_\eta$ depend only on $m,n,p$ and $\eta$.
\label{aqbdprop}
\end{lem}

\begin{proof}
We define $F_y \in L^{m,p}(\R^n)$ ($y \in Q^\circ$) by
\begin{align}
    F_y(x) = G(x)- \sum_{\al \in \ma} \p^\al G(y) \cdot \varphi_{y}^\al(x). \label{aco0}
\end{align}
Observe, by definition of the polynomial $S_y$ in \eqref{barjet} and by property (\ref{phi2}) of $\varphi_y^\al$, that $S_y = J_y F_y$ for $y \in K_p$. On the other hand, note $F_{x_i} = F_i$ defined in \eqref{deffi}. Hence, by \eqref{barjet_prop}, $S_{x_i} = P_{x_i} = J_{x_i} F_i = J_{x_i} F_{x_i} $ for $i \in I$. Thus, 
\[
S_x = J_x F_x \qquad (x \in K_p \cup \bpt).
\]

Because $P_0=0$, by (\ref{sob}), we have 
\begin{align}
|\p^\beta G(z)|= |\p^\beta (G-P_0)(z)| \lesssim \|G,P_0\|_{\J(f,\mu;\delta_{Q^\circ})} \quad ( \beta \in \ma, \; z \in Q^\circ). \label{lm11c}
\end{align}

For distinct $x,y \in K_p \cup \bpt$, we want to bound
\[
|S_y-S_x|_{y,|x-y|} = \big(\sum_{\gamma \in \mm} |\p^\gamma (S_y - S_x)(y)|^p |x-y|^{|\gamma|p+n-mp}\big)^{1/p}.
\]
For $\gamma \in \mm$, we have
\begin{align}
|\p^\gamma (S_y - S_x)(y)| &= |\p^\gamma(F_y - J_x F_x)(y)| \nonumber \\
&\leq |\p^\gamma(F_x - F_y)(y)| + |\p^\gamma (J_x F_x - F_x)(y)| \label{lm10}
\end{align}
Now apply (\ref{aco0}) and (\ref{omega}) to estimate
\begin{align}
|\p^\gamma(F_x - F_y)(y)| &=  \Big|\p^\gamma \Big(\sum_{\al \in \ma} \p^\al G(x) \varphi_{x}^\al - \sum_{\al \in \ma} \p^\al G(y) \varphi_{y}^\al \Big)(y)\Big| \nonumber \\
&=\Big|\p^\gamma \Big(\sum_{\al \in \ma} \p^\al G(x) \sum_{\beta \in \ma}( \omega_{\al \beta}^{yx} \cdot \varphi_{y}^\beta) - \sum_{\beta \in \ma} \p^\beta G(y) \cdot \varphi_{y}^\beta \Big)(y)\Big| \nonumber \\
& \leq \sum_{\beta \in \ma} \Big|\sum_{\al \in \ma} \p^\al G(x) \cdot \omega_{\al \beta}^{yx} - \p^\beta G(y) \Big| |\p^\gamma \varphi_{y}^\beta (y) | \nonumber \\
&= \sum_{\beta \in \ma} |\p^\beta F_{x}(y)| |\p^\gamma \varphi_{y}^\beta (y) | \label{lm11}
\end{align}
Because $\ma$ is monotonic, if $\beta \in \ma$ and $\beta + \al \in \mm$ then $\beta+\al \in \ma$, so
\begin{equation}\label{lm11X}
\p^\beta J_xF_x(y) = \sum_{\beta+\al \in \mm} \p^{\beta+\al} F_x(x)\frac{(y - x)^\al}{\al!} =0 \quad \quad (\beta \in \ma).
\end{equation}
We claim that $|\p^\gamma \varphi_{y}^\beta (y) | \lesssim |y-x|^{|\beta| - |\gamma|}$ for $\gamma \in \mm$, $\beta \in \ma$. Indeed, this estimate follows by (\ref{acobd}) if $y \in K_p$, and by (\ref{alocalbound}) if $y=x_i$ for some $i \in I$ (note: if $y = x_i$ then $x \notin Q_i$, as $y,x$ are distinct, hence $\delta_{Q_i}/2 \leq |y-x| \leq 1$). Thus, using \eqref{lm11X} in \eqref{lm11}, and then using the Sobolev Inequality, we bound
\begin{align}
    |\p^\gamma(F_x - F_y)(y)| &\leq \sum_{\beta \in \ma} |\p^\beta F_{x}(y)| |\p^\gamma \varphi_{y}^\beta (y) | \nonumber\\
    &= \sum_{\beta \in \ma} |\p^\beta (F_x-J_xF_x)(y)| |\p^\gamma \varphi_y^\beta (y) | \nonumber \\
    &\lesssim \sum_{\beta \in \ma} \|F_x\|_{L^{m,p}(U)} |y-x|^{m- |\beta| - n/p} |y-x|^{|\beta| - |\gamma|} \nonumber \\
    & \lesssim \|F_x\|_{L^{m,p}(U)} |y-x|^{m- |\gamma| - n/p}. \label{lm11b}
\end{align}
On the other hand, by the Sobolev Inequality,
\[
|\p^\gamma (J_x (F_x) - F_x)(y)| \lesssim \| F_x \|_{L^{m,p}(U)} |x-y|^{m-|\gamma| - n/p}.
\]
Combining the previous inequalities in (\ref{lm10}), we have
\begin{equation}\label{lm12}
|\p^\gamma (S_y - S_x)(y)| \lesssim \| F_x \|_{L^{m,p}(U)} |x-y|^{m-|\gamma| - n/p} \quad (\gamma \in \mm).
\end{equation}

From (\ref{phi1}), we have $\varphi_x^\al := \sum_{\beta \in \ma} A^{x}_{\al \beta} \cdot \varphi^\beta$, where $(A^{x}_{\al \beta})_{\al, \beta \in \ma}$ is a $(C, C\e)$-near triangular matrix, and in particular $|A^x_{\al \beta}| \leq C$. Hence $\|\varphi_{x}^\al\|_{L^{m,p}(U)} \lesssim \sum_{\beta \in \ma} \|\varphi^\beta\|_{L^{m,p}(U)}$ for all $\al \in \ma$; applying this and (\ref{lm11c}), in the definition of $F_x$, we have
\begin{align}
 \|F_x\|_{L^{m,p}(U)} &\leq  \|G\|_{L^{m,p}(U)}+ \Big\| \sum_{\al \in \ma} \p^\al G(x) \varphi^\al_x \Big\|_{L^{m,p}(U)} \nonumber \\
 & \lesssim \|G\|_{L^{m,p}(U)} + \|G,P_0\|_{\J(f,\mu;\delta_{Q^\circ})} \sum_{\beta \in \ma} \|\varphi^\beta\|_{L^{m,p}(U)}. \label{aco8}
\end{align}
Substituting (\ref{aco8}) into (\ref{lm12}),
\begin{align}
     |\p^\gamma(S_y - S_x)(y)| \cdot & |x-y|^{|\gamma| + n/p - m} \lesssim \|G\|_{L^{m,p}(U)} + \|G,P_0\|_{\J(f,\mu;\delta_{Q^\circ})} \sum_{\beta \in \ma} \|\varphi^\beta\|_{L^{m,p}(U)}. \label{lm11d}
\end{align}
Thus, by definition of the norm $| \cdot |_{y, |x-y|}$, we conclude that
\begin{align*}
    |S_y - S_x|_{y, |x-y|}^p &= \sum_{\gamma \in \mm} |\p^\gamma (S_x - S_y)(y)|^p |y-x|^{|\gamma|p+n-mp} \\
    &\lesssim \|G\|_{L^{m,p}(U)}^p + \|G,P_0\|_{\J(f,\mu;\delta_{Q^\circ})}^p \cdot \sum_{\beta \in \ma} \| \varphi^\beta \|_{L^{m,p}(U)}^p.
\end{align*}
This completes the proof of (\ref{aqbd}).
\end{proof}

\begin{lem}
The function $F|_{K_{CZ}}: K_{CZ} \to \R$ and Whitney field $\vec{P} \in Wh(\bpt)$ satisfy
\begin{align}
   \|F, \vec{P}\|_{\J_*(f,\mu;K_{CZ}, CZ^\circ)} \lesssim \|G,P_0\|_{\J(f,\mu;\delta_{Q^\circ})}. \label{kczbd}
\end{align}
\end{lem}

\begin{proof}
By (\ref{apatch1}),
\begin{align*}
\|F, \vec{P}\|_{\J_*(f,\mu;K_{CZ},CZ^\circ)}^p \lesssim \sum_{i \in I} \|F_i, P_{x_i}\|_{\J_*(f,\mu|_{1.1Q_i \cap Q^\circ};1.1Q_i \cap Q^\circ)}^p + \sum_{i \lra i'} |P_{x_i} -P_{x_{i'}}|_i^p  \\
\end{align*}
Because $J_{x_i} F_i = P_{x_i}$, by (\ref{SobLp}), $\| F_i - P_{x_i} \|_{L^p(Q_i)}/\delta_{Q_i}^m \lesssim \| F_i \|_{L^{m,p}(Q_i)}$, hence we have
\begin{align}
\|F, \vec{P}\|_{\J_*(f,\mu;K_{CZ},CZ^\circ)}^p &\lesssim \sum_{i \in I} \left[ \|F_i\|_{L^{m,p}(1.1Q_i\cap Q^\circ)}^p + \int_{1.1Q_i} |F_i - f|^p d\mu \right] + \sum_{i \lra i'} |P_{x_i} -P_{x_{i'}}|_i^p. \label{aco5}
\end{align}

We now bound the term $\sum_{i \lra i'} |P_{x_i} -P_{x_{i'}}|_i^p$ in \eqref{aco5}. Recall from \eqref{barjet_prop} that $P_{x_i} = S_{x_i}$. If $i \lra i'$ then $Q_i$ and $Q_{i'}$ are neighboring CZ cubes, hence, $U_{i i'} := (1.1Q_i \cup 1.1 Q_{i'}) \cap Q^\circ$ is the union of two $C$-non-degenerate boxes with a common interior point, and $x_i,x_{i'} \in U_{i i'}$. Note also $|x_i-x_{i'}| \simeq \delta_{Q_i}$, by the good geometry of the CZ decomposition. Hence, from Lemma \ref{aqbdprop}, applied with $U = U_{ii'}$, for any $i \lra i'$,
\begin{align}
|P_{x_i} -P_{x_{i'}}|_i^p &= |S_{x_i} - S_{x_{i'}}|^p_{x_i, \delta_{Q_i}} \simeq |S_{x_i} - S_{x_{i'}}|^p_{x_i, |x_i-x_{i'}|} \nonumber \\
&\lesssim \| G \|_{L^{m,p}((1.1 Q_i \cup 1.1 Q_{i'}) \cap Q^\circ)}^p  + \| G,P_0 \|_{\J_*(f, \mu; Q^\circ)}^p \sum_{\beta \in \ma} \| \varphi^\beta \|_{L^{m,p}((1.1 Q_i \cup 1.1 Q_{i'}) \cap Q^\circ)}^p. \label{co3}
\end{align}
By summing \eqref{co3} over all $i,i' \in I$ with $i \lra i'$, and by using the bounded overlap of the regions $(1.1Q_i \cup 1.1 Q_{i'}) \cap Q^\circ \subset Q^\circ$, we find that
\begin{align*}
\sum_{i \lra i'} |P_{x_i} -P_{x_{i'}}|_i^p \lesssim \| G \|_{L^{m,p}(Q^\circ)}^p + \| G,P_0 \|_{\J_*(f, \mu; Q^\circ)}^p \sum_{\beta \in \ma} \| \varphi^\beta \|_{L^{m,p}(Q^\circ)}^p. 
\end{align*}
Further, from (\ref{mat0_new}) we have $\| \varphi^\beta \|_{L^{m,p}(Q^\circ)} \leq \| \varphi^\beta \|_{\J(0,\mu)}  \lesssim 1$ for $\beta \in \ma$. Also, $\| G \|_{L^{m,p}(Q^\circ)} \leq \| G,P_0 \|_{\J_*(f, \mu; Q^\circ)}$. Hence,
\begin{align}
\sum_{i \lra i'} |P_{x_i} -P_{x_{i'}}|_i^p \lesssim \| G,P_0 \|_{\J_*(f, \mu; Q^\circ)}^p. \label{aco6}
\end{align}

Recall $P_0=0$. For $\al \in \ma$, by (\ref{sob}), 
\begin{align}
|\p^\al G(x_i)|= |\p^\al (G-P_0)(x_i)| \leq C \|G,P_0\|_{\J_*(f,\mu;Q^\circ)}. \label{poly0}
\end{align}
From (\ref{phi1}), we have $\varphi_{x_i}^\al := \sum_{\beta \in \ma} A^{x_i}_{\al \beta} \cdot \varphi^\beta$ ($\al \in \ma$), where $(A^{x_i}_{\al \beta})_{\al, \beta \in \ma}$ is a $(C, C\e)$-near triangular matrix. In particular, $|A^{x_i}_{\al \beta}| \leq C$. We also have, from (\ref{mat0_new}), $\| \varphi^\beta \|_{L^p(d \mu)} \leq \|\varphi^\beta\|_{\J(0,\mu)} \lesssim 1$ for $\beta \in \ma$. Thus, by definition of $F_i$ (see \eqref{deffi}), and by the triangle inequality, we deduce that
\begin{align}
\sum_{i \in I}  \int_{1.1Q_i} |F_i-f|^p d\mu &\lesssim \sum_{i \in I} \Bigg[ \int_{1.1Q_i} |G-f|^p d\mu + \sum_{\al \in \ma} | \p^\al G(x_i)|^p \sum_{\beta \in \ma} |A_{\al \beta}^{x_i}|^p \cdot \int_{1.1Q_i} |\varphi^\beta|^p d\mu \Bigg] \nonumber \\
&\lesssim \sum_{i \in I} \Bigg[ \int_{1.1Q_i} |G-f|^p d\mu +  \|G,P_0\|_{\J_*(f,\mu;Q^\circ)}^p \sum_{\beta \in \ma} \int_{1.1Q_i} |\varphi^\beta|^p d\mu \Bigg] \nonumber \\
&\lesssim  \int_{Q^\circ} |G-f|^p d\mu + \|G,P_0\|_{\J_*(f,\mu;Q^\circ)}^p \sum_{\beta \in \ma}  \int_{Q^\circ} |\varphi^\beta|^p d\mu \nonumber \\
&\lesssim  \|G,P_0\|_{\J_*(f,\mu;Q^\circ)}^p. \label{co85}
\end{align}
Similarly, by instead using that $\| \varphi^\beta \|_{L^{m,p}(Q^\circ)} \leq \|\varphi^\beta\|_{\J(0,\mu)} \lesssim 1$ for $\beta \in \ma$, and by definition of $F_i$,  we have
\begin{align}
\sum_{i \in I}  \| F_i \|_{L^{m,p}(1.1 Q_i \cap Q^\circ)}^p &\lesssim \sum_{i \in I} \Bigg[  \| G \|_{L^{m,p}(1.1 Q_i \cap Q^\circ)}^p + \sum_{\al \in \ma} | \p^\al G(x_i)|^p \sum_{\beta \in \ma} |A_{\al \beta}^{x_i}|^p \|\varphi^\beta\|^p_{L^{m,p}(1.1 Q_i \cap Q^\circ)} \Bigg] \nonumber \\
&\lesssim \sum_{i \in I} \Bigg[\| G \|_{L^{m,p}(1.1 Q_i \cap Q^\circ)}^p   + \|G,P_0\|_{\J_*(f,\mu;Q^\circ)}^p \sum_{\beta \in \ma} \|\varphi^\beta\|^p_{L^{m,p}(1.1 Q_i \cap Q^\circ)} \Bigg] \nonumber \\
&\lesssim  \| G \|_{L^{m,p}(Q^\circ)}^p + \|G,P_0\|_{\J_*(f,\mu;Q^\circ)}^p \sum_{\beta \in \ma} \| \varphi^\beta \|_{L^{m,p}(Q^\circ)}^p \nonumber \\
&\lesssim  \|G,P_0\|_{\J_*(f,\mu;Q^\circ)}^p. \label{co86}
\end{align}

We substitute inequalities \eqref{aco6}, \eqref{co85}, and \eqref{co86} into \eqref{aco5} to deduce:
\begin{align*}
\|F, \vec{P}\|_{\J_*(f,\mu;K_{CZ}, CZ^\circ)}^p \lesssim \|G,P_0\|_{\J_*(f,\mu;Q^\circ)}^p.
\end{align*}
We apply (\ref{j5}) to conclude:
\begin{align*}
\|F, \vec{P} \|_{\J_*(f,\mu;K_{CZ}, CZ^\circ)} \lesssim\| G, P_0 \|_{\J(f,\mu;\delta_{Q^{\circ}})}.
\end{align*}
This completes the proof of \eqref{kczbd}. The proof of the lemma is complete.
\end{proof}

\begin{lem}
The polynomials $S_x$ ($x \in Q^\circ$) defined in (\ref{barjet}) satisfy
\[
\sup_{x,y \in Q^\circ} |S_x - S_y|_{x,|y-x|} \lesssim \|G,P_0\|_{\J(f,\mu;\delta_{Q^\circ})}.
\]
\label{holderbarf}
\end{lem}

\begin{proof}
Let $x,y \in Q^\circ$. We will prove the desired inequality by considering the geometry of $x$ and $y$ in relation to $K_p$ and $K_{CZ}$:

\underline{Case 1:} Suppose $x,y \in K_{CZ}$. Let $y \in Q_i \in CZ^\circ$, $x \in Q_j \in CZ^\circ$. Suppose $\delta_{Q_i} > 21 |y-x|$; then $x \in 1.1Q_i$. By applying the Sobolev inequality, we have
\begin{align}
    |S_x - S_y|_{x,|y-x|} &= |J_xF - J_yF|_{x,|y-x|} \lesssim \|F\|_{L^{m,p}(( 1.1Q_i)\cap Q^\circ)} \leq \|F\|_{L^{m,p}(K_{CZ})}  . \label{case1a}
\end{align}
Similary, if $\delta_{Q_j} > 21 |y-x|$; then $y \in 1.1Q_j$. By applying the Sobolev inequality, we have
\begin{align}
    |S_x - S_y|_{x,|y-x|} &= |J_xF - J_yF|_{x,|y-x|} \lesssim \|F\|_{L^{m,p}((1.1Q_j)\cap Q^\circ)} \leq \|F\|_{L^{m,p}(K_{CZ})}. \label{case1b}
\end{align}
Now suppose $\delta_{Q_i}, \delta_{Q_j} \leq 21 |y-x|$; then $|x_i - x_j| \leq \delta_{Q_i} + |x-y| + \delta_{Q_j} \leq 43 |x-y|$. We apply (\ref{movept}), the Sobolev Inequality, and Lemma \ref{aqbdprop} with $U=Q^\circ$ to deduce
\begin{align}
    |S_x - S_y|_{x,|y-x|} &\lesssim |S_y - S_{x_i}|_{y,\delta_{Q_i}}+ |S_{x_i} - S_{x_j}|_{x_i,|x_i-x_j|}+|S_{x_j} - S_x|_{x,\delta_{Q_j}} \nonumber \\
    &=|J_yF - J_{x_i}F|_{y,\delta_{Q_i}}+ |S_{x_i} - S_{x_j}|_{x_i,|x_i-x_j|} +|J_{x_j}F - J_xF|_{x,\delta_{Q_j}} \nonumber \\
    &\lesssim \|F\|_{L^{m,p}(Q_i)}+ \|F\|_{L^{m,p}(Q_j)} + \|G\|_{L^{m,p}(Q^\circ)}   + \|G, P_0\|_{\J(f,\mu;\delta_{Q^\circ})} \cdot \sum_{\beta \in \ma} \|\varphi^\beta\|_{L^{m,p}(Q^\circ)} \nonumber \\
    & \lesssim \|F\|_{L^{m,p}(K_{CZ})} + \|G, P_0\|_{\J(f,\mu;\delta_{Q^\circ})}, \label{case1c}
\end{align}
where in the last line we use (\ref{mat0_new}). By (\ref{kczbd}), we have $\|F\|_{L^{m,p}(K_{CZ})} \lesssim \|G,P_0\|_{\J(f,\mu;\delta_{Q^\circ})}$. Considering (\ref{case1a}), (\ref{case1b}), and (\ref{case1c}), we conclude for $x,y \in K_{CZ}$,
\begin{align}
    |S_x - S_y|_{x,|y-x|} &\lesssim  \|F\|_{L^{m,p}(K_{CZ})} + \|G,P_0\|_{\J(f,\mu;\delta_{Q^\circ})}\lesssim \|G,P_0\|_{\J(f,\mu;\delta_{Q^\circ})}. \label{case1}
\end{align}

\underline{Case 2:} Suppose $x \in K_p$, $y \in K_{CZ}$. Let $y \in Q_i \in CZ^\circ$. Because $|x-x_i| \lesssim |y-x|$ and $\delta_{Q_i} \leq |y-x|$ (see Lemma \ref{kpdist}), we can apply (\ref{movept}) and then the Sobolev Inequality to deduce
\begin{align}
    |S_x - S_y|_{x,|y-x|} &\lesssim |S_x - S_{x_i}|_{x,|x_i-x|} + |S_{x_i} - S_y|_{x_i,\delta_{Q_i}} \nonumber \\
    &=|S_x - S_{x_i}|_{x,|x_i-x|} + |J_{x_i} F - J_y F|_{x_i,\delta_{Q_i}} \nonumber \\
    & \lesssim |S_x - S_{x_i}|_{x,|x_i-x|} +\|F\|_{L^{m,p}(Q_i)}. \label{case2}
\end{align}
We apply Lemma \ref{aqbdprop} with $U=Q^\circ$, recalling for $\al \in \ma$, we have $\|\varphi^\al\|_{L^{m,p}(Q^\circ)} \leq \e/C_0$ (see (\ref{mat0_new})), to deduce
\begin{align}
    |S_x - S_{x_i}|_{x,|x_i-x|} &\lesssim \|G\|_{L^{m,p}(Q^\circ)} + \|G, P_0\|_{\J(f,\mu;\delta_{Q^\circ})} \cdot \sum_{\beta \in \ma} \|\varphi^\beta\|_{L^{m,p}(Q^\circ)} \nonumber \\ 
    &\lesssim \|G,P_0\|_{\J(f,\mu; \delta_{Q^\circ})}. \label{corwq} 
\end{align}
Substituting this into (\ref{case2}), and using again that $\| F \|_{L^{m,p}(K_{CZ})} \leq \| G, P_0 \|_{\J(f,\mu; \delta_{Q^\circ})}$, we have 
\begin{align*}
    |S_x - S_y|_{x,|y-x|} \lesssim \|G,P_0\|_{\J(f,\mu; \delta_{Q^\circ})}. 
\end{align*}

\underline{Case 3:} Suppose $x,y \in K_p$; then as in (\ref{corwq}) in Case 2, we can apply Lemma \ref{aqbdprop} to deduce
\begin{align*}
    |S_x - S_y|_{x,|y-x|} \lesssim \|G,P_0\|_{\J(f,\mu; \delta_{Q^\circ})}. 
\end{align*}
This completes the proof of Lemma \ref{holderbarf}.
\end{proof}

\begin{lem}
The function $F$ defined in (\ref{barf}) belongs to $L^{m,p}(Q^\circ)$, and 
\begin{align}
    \|F\|_{L^{m,p}(Q^\circ)} \lesssim \|G,P_0\|_{\J(f,\mu;\delta_{Q^\circ})}. \label{acohbd}
\end{align}
\end{lem}

\begin{proof}
We will use Corollary \ref{brlm}. Fix $\pi \in \Pi_{\simeq}(Q^\circ)$, a congruent $\delta$-packing of $Q^\circ$, with $\delta \leq \delta_{Q^\circ}$.

For $\widehat{Q} \in \pi$, $\delta_{\widehat{Q}} = \delta$, by H\"{o}lder's inequality,
\begin{align}
E(F,\widehat{Q})\delta^{-m} \leq \sup_{x \in \widehat{Q}} \|F-S_x\|_{L^p(\widehat{Q})}\delta^{-m} &\leq \sup_{x \in \widehat{Q}}\| F  -S_x \|_{L^\infty(\widehat{Q})} \delta^{n/p-m} \nonumber \\
&\leq \sup_{x, y \in \widehat{Q}} |S_y - S_x|_{y, \delta}, \label{appxerror}
\end{align}
where the last line follows because $F(y) = S_y(y)$ for $y \in Q^\circ$ -- see \eqref{barjet_prop2}.

Fix $\widehat{Q} \in \pi$ and $x, y \in \widehat{Q}$. Then, $|x-y| \leq \delta_{\widehat{Q}} = \delta$.

If $\widehat{Q} \subset K_{CZ}$, then by the Sobolev Inequality,
\begin{equation}
    |S_x - S_y|_{y, \delta} = |J_xF - J_yF|_{y, \delta} \lesssim \|F\|_{L^{m,p}(\widehat{Q})}. \label{Qhatinkcz}
\end{equation}

Now suppose $\widehat{Q} \cap K_p \neq \emptyset$. We will show
\begin{align}
|S_x - S_y|_{y, \delta} \lesssim 
    \; & \|G,P_0\|_{\J(f,\mu;\delta_{Q^\circ})} \cdot \sum_{\beta \in \ma} \| \varphi^\beta \|_{L^{m,p}(100\widehat{Q} \cap Q^\circ)} \nonumber\\
    & + \|G\|_{L^{m,p}(100\widehat{Q} \cap Q^\circ)} +  \sup_{1.1 Q_i \subset 100 \widehat{Q}} \|F\|_{L^{m,p}(1.1 Q_i\cap Q^\circ)}. \label{maxbd}
\end{align}

\underline{Case 1:} Let $x,y \in K_p$. Then apply (\ref{aqbd}) with $U = \widehat{Q}$,
\begin{align*}
    |S_x - S_y|_{y, \delta} \lesssim \|G\|_{L^{m,p}(\widehat{Q})} + \|G, P_0\|_{\J(f,\mu;\delta_{Q^\circ})} \cdot \sum_{\beta \in \ma} \|\varphi^\beta\|_{L^{m,p}(\widehat{Q})},
\end{align*}
which implies (\ref{maxbd}).

\underline{Case 2:} Let $y \in Q_i \in CZ^\circ$ and $\delta_{Q_i} > 21 |x-y|$ or $x \in Q_j \in CZ^\circ$ and $\delta_{Q_j} > 21 |x-y|$. Because $|P|_{x,\delta} \simeq |P|_{y,\delta}$ for $|x-y| \leq \delta$, we may assume without loss of generality that $y \in Q_i$, $\delta_{Q_i}>21|x-y|$. Then $x \in 1.1Q_i \subset K_{CZ}$.  As $x,y  \in K_{CZ}$, we have $S_x = J_x F$ and $S_y = J_y F$.  Because $\widehat{Q} \cap K_p \neq \emptyset$ and $x \in \widehat{Q}$, it follows that $\dist(x,K_p) \leq \delta_{\widehat{Q}}$. Also, by Lemma \ref{kpdist}), we have $\dist(Q_i,K_p) \geq \delta_{Q_i}$. Thus,
\[
\delta_{\widehat{Q}} \geq \dist(x, K_p ) \geq \dist(Q_i, K_p) \geq \delta_{Q_i}.
\]
Since $Q_i \cap \widehat{Q} \neq \emptyset$, we have $1.1 Q_i \subset 100 \widehat{Q}$, and by the Sobolev Inequality,
\begin{align*}
    |S_x - S_y|_{y, \delta} = |J_xF - J_yF|_{y, \delta}  &\lesssim \|F\|_{L^{m,p}(1.1 Q_i \cap Q^\circ)},
\end{align*}
completing the proof of (\ref{maxbd}).

\underline{Case 3:} Let $y \in Q_i \in CZ^\circ$ and $\delta_{Q_i} \leq 21 |x-y| \leq  21 \delta_{\widehat{Q}} = 21 \delta$. Then $1.1 Q_i \subset 100 \widehat{Q}$. Recall $J_{x_i}F = S_{x_i}$. Then from (\ref{movept}), By the Sobolev inequality,
\[
|S_{x_i} - S_y|_{y,\delta} = | J_{x_i} F - J_y F|_{y,\delta} \lesssim \| F \|_{L^{m,p}(Q_i)}.
\]
Thus, by the triangle inequality and \eqref{movept},
\begin{align}
|S_x - S_y|_{y, \delta} \lesssim |S_x - S_{x_i}|_{x,\delta} + |S_{x_i} - S_y|_{y, \delta} \lesssim |S_x - S_{x_i}|_{x,\delta} + \| F \|_{L^{m,p}(Q_i)}. \label{lmpcase3}
\end{align}
We continue this bound by splitting into subcases.

\quad \quad \underline{Subcase 3a:} Suppose $x \in K_p$. Note that $|x-x_i| \leq |x-y| + |y-x_i| \leq |x-y| + \delta_{Q_i} \leq 22 \delta$. Further, $x,x_i \in 100 \widehat{Q} \cap Q^\circ$. Thus,  from (\ref{normdi}), and (\ref{aqbd}) with $U = 100\widehat{Q} \cap Q^\circ$,
\begin{align*}
    |S_{x_i} - S_x|_{x,\delta} &\lesssim |S_{x_i} - S_x|_{x,|x-x_i|}  \\
    &\lesssim \|G\|_{L^{m,p}(100\widehat{Q} \cap Q^\circ)} + \|G, P_0\|_{\J(f,\mu;\delta_{Q^\circ})} \cdot \sum_{\beta \in \ma} \|\varphi^\beta\|_{L^{m,p}(100\widehat{Q} \cap Q^\circ)}.
\end{align*}
Substituting this into (\ref{lmpcase3}), we have
\begin{align*}
    |S_x - S_y|_{y, \delta} &\lesssim \|F\|_{L^{m,p}(Q_i)} +\|G\|_{L^{m,p}(100\widehat{Q} \cap Q^\circ)}+\|G, P_0\|_{\J(f,\mu;\delta_{Q^\circ})} \cdot \sum_{\beta \in \ma} \|\varphi^\beta\|_{L^{m,p}(100\widehat{Q} \cap Q^\circ)},
\end{align*}
completing the proof of (\ref{maxbd}).

\quad \quad \underline{Subcase 3b:} Suppose  $x \notin K_p$. Then $x \in Q_j$ for some $Q_j \in CZ^\circ$ and because of Case 2, we can assume $\delta_{Q_j} \leq 21 |x-y| \leq 21 \delta$. Hence, $1.1 Q_j \subset 100\widehat{Q}$. Also, $|x-x_j| \leq \delta_{Q_j}  \leq 21 \delta$, and 
\[
|x_i-x_j| \leq |x_j - x| + |x-y| + |y - x_i | \leq \delta_{Q_j} + |x-y| + \delta_{Q_i} \leq 43 \delta.
\]
Thus, from (\ref{movept}), (\ref{normdi}), and (\ref{aqbd}) with $U = 100\widehat{Q} \cap Q^\circ$, and the Sobolev Inequality,
\begin{align}
    |S_x- S_{x_i}|_{x,\delta} &\lesssim |S_{x_i} - S_{x_j}|_{x_j,\delta} + |S_{x_j} - S_x|_{x,\delta} \nonumber\\
    &\lesssim |S_{x_i} - S_{x_j}|_{x_j,|x_i - x_j|}  + |J_{x_j} F - J_x F|_{x,\delta} \nonumber \\
    &\lesssim \|G\|_{L^{m,p}(100\widehat{Q} \cap Q^\circ)}+\|G, P_0\|_{\J(f,\mu;\delta_{Q^\circ})} \cdot \sum_{\beta \in \ma} \|\varphi^\beta\|_{L^{m,p}(100\widehat{Q} \cap Q^\circ)} + \|F\|_{L^{m,p}(Q_j)}. \label{lmpcase31}
\end{align}
By substituting (\ref{lmpcase31}) into (\ref{lmpcase3}), we have
\begin{align*}
    |S_x - S_y|_{y, \delta} &\lesssim \|F\|_{L^{m,p}(Q_i)}  + \|F\|_{L^{m,p}(Q_j)} + \|G\|_{L^{m,p}(100\widehat{Q} \cap Q^\circ)}+\|G, P_0\|_{\J(f,\mu;\delta_{Q^\circ})} \cdot \sum_{\beta \in \ma} \|\varphi^\beta\|_{L^{m,p}(100\widehat{Q} \cap Q^\circ)}. 
\end{align*}
This completes the proof of (\ref{maxbd}).

Consequently, using \eqref{appxerror}, (\ref{Qhatinkcz}), and \eqref{maxbd}, we have 
\begin{align}
    \sum_{\widehat{Q} \in \pi} \big(E(F,\widehat{Q}\big)/ \delta_{\widehat{Q}}^m)^p &\lesssim \sum_{\widehat{Q} \in \pi} \sup_{x, y \in \widehat{Q}} |S_x - S_y|_{y, \delta}^p \nonumber \\
    &\lesssim \sum_{\widehat{Q} \in \pi, \widehat{Q} \cap K_p \neq \emptyset} \Big(\sup_{1.1 Q_i \subset 100 \widehat{Q}} \|F\|_{L^{m,p}(1.1 Q_i \cap Q^\circ)}^p + \|G\|_{L^{m,p}(100\widehat{Q} \cap Q^\circ)}^p \nonumber \\
    &\qquad+\|G, P_0\|_{\J(f,\mu;\delta_{Q^\circ})}^p \cdot \sum_{\beta \in \ma} \|\varphi^\beta\|_{L^{m,p}(100\widehat{Q} \cap Q^\circ)}^p \Big) + \sum_{\widehat{Q} \in \pi, \widehat{Q} \subset K_{CZ}} \|F\|_{L^{m,p}(\widehat{Q})}^p \nonumber \\
    &\lesssim \|F\|_{L^{m,p}(K_{CZ})}^p+\|G\|_{L^{m,p}(Q^\circ)}^p+\|G, P_0\|_{\J(f,\mu;\delta_{Q^\circ})}^p \cdot \sum_{\beta \in \ma} \|\varphi^\beta\|_{L^{m,p}(Q^\circ)}^p, \label{acoh10} 
\end{align}
where the last inequality follows because for $Q_i \in CZ^\circ$, $| \{\widehat{Q} \in \pi: 1.1 Q_i \subset 100 \widehat{Q} \}| \leq C$, and since $\{100 \widehat{Q} : \widehat{Q} \in \pi \}$ has $C$-bounded overlap (because $\pi$ is a congruent packing). From (\ref{mat0_new}), we have $\| \varphi^\beta \|_{L^{m,p}(Q^\circ)} \leq \e/C_0$ for all $\beta \in \ma$. From (\ref{kczbd}), we have $\| F \|_{L^{m,p}(K_{CZ})} \lesssim \| G , P_0 \|_{\J(f,\mu; \delta_{Q^\circ})}$. Therefore, from (\ref{acoh10}),
\begin{align*}
    \sum_{\widehat{Q} \in \pi} \big(E(F,\widehat{Q}\big)/ \delta_{\bar{Q}}^m)^p &\lesssim  \|G, P_0\|_{\J(f, \mu; \delta_{Q^\circ})}^p.
\end{align*}
Now by taking the supremum over $\pi \in \Pi_{\simeq}(Q^\circ)$, and applying Corollary \ref{brlm}, we conclude $F \in L^{m,p}(Q^\circ)$ and
\begin{align*}
    \|F\|_{L^{m,p}(Q^\circ)} \lesssim \|G, P_0\|_{\J(f, \mu; \delta_{Q^\circ})}. 
\end{align*}
This completes the proof of the lemma.
\end{proof}

We complete the section by giving the proof of conditions  (\ref{acoh}) -- (\ref{acoh2}).

We apply inequalities (\ref{kczbd}) and (\ref{acohbd}) and the identity $F|_{K_p} = G|_{K_p}$ (see (\ref{barf})) to bound
\begin{align*}
    \|F, \vec{P}\|_{\J_*(f, \mu; Q^\circ,CZ^\circ)}^p &\lesssim \|F, \vec{P}\|_{\J_*(f, \mu; K_{CZ},CZ^\circ)}^p+ \|F\|_{L^{m,p}(Q^\circ)}^p+\int_{K_p} |F-f|^p d\mu \\
    & \lesssim \|G,P_0\|_{\J(f, \mu; \delta_{Q^\circ})}^p.
\end{align*}
This establishes the inequality (\ref{acoh}) for $P_0 =0$, as desired.

As a consequence of Lemma \ref{holderbarf}, the Whitney field $(S_y)_{y \in Q^\circ} \in Wh(Q^\circ)$ defined in (\ref{barjet}) is in $C^{m-1,1-n/p}(Q^\circ)$, and thanks to \eqref{barjet_prop2} it satisfies $S_x(x) = F(x)$ for all $x \in Q^\circ$. Due to Lemma \ref{whitlm}, the function $F$ satisfies 
\begin{equation}\label{eqn:stuff}
J_xF = S_x, \quad  x \in Q^\circ.
\end{equation}

Then \eqref{acoh2} holds, thanks to (\ref{eqn:stuff}). From \eqref{Rcoh}, we have $\vec{S} = (S_y)_{y \in K_p}$ is coherent with $P_0 = 0$. From \eqref{Pcoh}, we have $\vec{P}$ is coherent with $P_0=0$. Thus,  $(\vec{P}, \vec{S}) \in Wh(\bpt \cup K_p)$ is coherent with $P_0=0$, proving \eqref{acoh1}. This completes the proof of  (\ref{acoh}) -- (\ref{acoh2}), thus completing the proof of Proposition \ref{acohlm}.

\subsection{Keystone Point Jets}
\label{subsec:kpt_jet}
Let $(f,P_0) \in \J(\mu; \delta_{Q^\circ})$. The goal of this section is to associate to the data $(f,P_0)$ a Whitney field $\vec{R}^* \in Wh(K_p)$, determined linearly by $(f,P_0)$, and satisfying the properties outlined below in Lemma \ref{kptlm}, Lemma \ref{kptlm_loc}, and Corollary \ref{acohkptcor}. These results will be  used later, in Section \ref{sec:opt_wh_field} and Section \ref{sec:proof_ml}.

Recall that we have fixed a multi-index set $\ma \subset \mm$. In the previous subsection we introduced the following notation: Given $K \subset \R^n$, a Whitney field $\vec{R} = (R_x)_{x \in K} \in Wh(K)$ is coherent with $P_0$ provided that $ \p^\al R_x(x)= \p^\al P_0(x)$ for all $x \in K$, $\al \in \ma$. A function $H \in C^{m-1,1-n/p}(\R^n)$ is $K$-coherent with $P_0$ provided that $ \p^\al H(x)= \p^\al P_0(x)$ for all $x \in K$ and for all $\al \in \ma$.

\begin{lem} For each $(f,P_0) \in \J(\mu; \delta_{Q^\circ})$, there exists a Whitney field $\vec{R}^* = \vec{R}^*(f,P_0) \in Wh(K_p)$ with the following properties.
\begin{enumerate}
\item $\vec{R}^*$ is coherent with $P_0$.
\item If $H \in L^{m,p}(\R^n)$ satisfies $\|H\|_{\J(f,\mu)} < \infty$ and $H$ is $K_p$-coherent with $P_0$, then
    \begin{align}
    J_x H = R^*_x \quad \forall x \in K_p. \label{kpt}
    \end{align}
\item $\displaystyle \| \vec{R}^* \|_{L^{m,p}(K_p)}^p + \int_{K_p} | R^*_x(x) - f(x)|^p d \mu \lesssim \| f , P_0 \|_{\J(\mu; \delta_{Q^\circ})}^p$.
\item $\vec{R}^*$ depends linearly on $(f,P_0)$.
\end{enumerate}
\label{kptlm}
\end{lem}

\begin{proof}

Due to Proposition \ref{acohlm}, for any $\eta > 0$ there exists $H_1 \in \J(f,\mu)$ that is $K_p$-coherent with $P_0$ and satisfies $\| H_1 \|_{\J(f,\mu)} \lesssim \| f, P_0 \|_{\J(\mu; \delta_{Q^\circ})} +\eta < \infty$. Define $\vec{R}^* \in Wh(K_p)$ as
\begin{align}
    \vec{R}^* = (R^*_x)_{x \in K_p}: = (J_x H_1)_{x \in K_p}. \label{kptr}
\end{align}
Next, we verify properties 1--4 of the lemma. In particular, we show that $\vec{R}^*$ is independent of $\eta$ and the choice of $H_1$ (property 2), and furthermore that $\vec{R}^*$ is a linear function of $(f,P_0)$ (property 4).

\textbf{Proof of property 1:} By definition, $R^*_x = J_x H_1$ for $x \in K_p$, where $H_1$ is $K_p$-coherent with $P_0$. Thus, $\vec{R}^*$ is coherent with $P_0$.

\textbf{Proof of property 2:}

We shall make use of the auxiliary polynomials $\left(P_y^\al \right)_{\al \in \ma}$ ($y \in Q^\circ$), defined in Lemma \ref{lem:polybasis}. In particular, from (\ref{basis}) we know that $\left(P_y^\al \right)_{\al \in \ma}$ forms an $(\ma,y, C\e/C_0, 1)$-basis for $\sigma_J(y,\mu)$ for all $y \in 100 Q^\circ$. 


Fix $\e_1 > 0$ satisfying $\e_1 < \min\{c_1, \e_0/(30^m C_0 C_1) \}$ and $\e_1 < \min \{(50^m C_3)^{-1/(D+1)}, C_4^{-1/(D+1)} \}$ where $C_3$ is from (\ref{alocalbound}), $C_4$ is from (\ref{acobd}),  $c_1$ and $C_1$ are from Lemma \ref{newdir}, and $D= |\mm|$. Fix $\e_2 <  \e_1^{2D+2}$.

For the sake of contradiction, suppose that (\ref{kpt}) does not hold. Thus, there exist $x\in K_p$ and $H_2 \in L^{m,p}(\R^n)$ satisfying $\|H_2\|_{\J(f,\mu)} < \infty$ such that 
\begin{equation}\label{qbd0}
\p^\al(H_1 - H_2)(x) = 0 \mbox{ for all } \al \in \ma, \mbox{ but } J_x(H_1 - H_2) \neq 0.
\end{equation} 
Let 
\[
\mb: = \{ \beta \in \mm \setminus \ma: \; \p^\beta(H_1-H_2)(x) \neq 0 \} = \{ \beta_i: i \in \{ 1, \dots, k \} \},
\] 
where $\beta_i$ are ordered:
\begin{align*}
\beta_1 < \beta_2 < \dots < \beta_k.
\end{align*}
Note that \eqref{qbd0} implies $\mb$ is nonempty. Also, $k = |\mb| \leq |\mm| = D$. Choose $\delta \in (0,1)$ such that for all $\beta \in \mb$,
\begin{align}
\delta^{m-|\beta| - n/p} \left( \|H_1\|_{\J(f,\mu)} + \|H_2\|_{\J(f,\mu)} \right) <  \e_2 |\p^\beta(H_1 - H_2)(x)|. \label{delta}
\end{align}
Choose a dyadic cube $Q \subset Q^\circ$ with $x \in Q$ and $\delta_Q < \delta<1$. Let 
\[
j: = \arg \max \{ | \p^{\beta_i}(H_1 - H_2)(x)| \cdot \delta_Q^{|\beta_i|}\e_1^{-i} : i \in \{ 1, \dots, k \}  \}.
\]
Then for all $i \in \{1,\cdots,k\}$,
\begin{equation}
\frac{| \p^{\beta_i}(H_1 - H_2)(x)|}{| \p^{\beta_j}(H_1 - H_2)(x)|} \leq \e_1^{i-j} \delta_Q^{|\beta_j|-|\beta_i|} \leq 
\left\{
\begin{aligned}
&\e_1 \delta_Q^{|\beta_j|-|\beta_i|} \qquad &&\mbox{if } \beta_i > \beta_j, \; i > j, \\
&\e_1^{-D} \delta_Q^{|\beta_j|-|\beta_i|} \qquad &&\mbox{if } \beta_i < \beta_j, \; i < j.
\end{aligned}
\right.
\label{qbd}
\end{equation}
Define $P_x^{\beta_j} : = J_x(H_1 -H_2) / (\p^{\beta_j} (H_1 - H_2)(x))$. By definition of $\mb$ and (\ref{qbd0}), $\partial^\beta P_x^{\beta_j}(x) = 0$ for $\beta \in \mm \setminus \mb = \ma$. Combining this with (\ref{qbd}),  we have
\begin{align}
&\p^{\beta_j} P_x^{\beta_j}(x) = 1; \label{qpt0} \\ 
&|\p^\beta P_x^{\beta_j}(x) | \leq \e_1 \delta_Q^{|\beta_j|-|\beta|} &&(\beta \in \mm, \; \beta> \beta_j); \text{ and} \label{qpt1}\\
& |\p^{\beta_i} P_x^{\beta_j}(x) | \leq \e_1^{-D}\delta_Q^{|\beta_j|-|\beta_i|} &&(\beta \in \mm). \label{qpt2}
\end{align}
Then from (\ref{delta}) and since $\delta_Q < \delta$, if $\varphi := (H_1-H_2)/| \partial^{\beta_j} (H_1-H_2)(x)|$ then $J_x \varphi = P_x^{\beta_j}$ and
\begin{align*}
\| \varphi \|_{\J(0, \mu|_{3Q})} = \frac{\|H_1 - H_2\|_{\J(0, \mu|_{3Q})}}{|\p^{\beta_j}(H_1 - H_2)(x)|} &\leq  \frac{\|H_1\|_{\J(f, \mu|_{3Q})}+\|H_2\|_{\J(f, \mu|_{3Q})}}{\delta_Q^{m - |\beta_j| - n/p} (\|H_1\|_{\J(f, \mu)}+\|H_2\|_{\J(f, \mu)} )} \cdot \e_2 \\
&\leq \e_2 \delta_Q^{|\beta_j|+n/p -m}.
\end{align*}
Hence,
\begin{align}
P_x^{\beta_j} \in \e_2 \delta_Q^{|\beta_j|+n/p -m} \cdot \sigma_J(x, \mu|_{3Q}). \label{qpt01}
\end{align}
From (\ref{qpt0})-(\ref{qpt01}), we see that $(\delta, x_0, \bb, \pt_{x_0}^{\bb}, \mu): = (\delta_Q, x, \beta_j, P_x^{\beta_j}, \mu|_{3Q})$ satisfies \textbf{(D6)} of Lemma \ref{newdir}. 

In (\ref{basis}), we saw that $\left(P_y^\al \right)_{\al \in \ma}$ forms an $(\ma,y, C\e, 1)$-basis for $\sigma_J(y,\mu)$ for all $y \in 100 Q^\circ$. We can assume $C\e < \e_2$. If $Q \subset Q^\circ$ then $4Q \subset 100 Q^\circ$. So, by (\ref{monotone}), 
\begin{equation}\label{qpt02}
\left(P_y^\al \right)_{\al \in \ma} \mbox{ forms an } (\ma, y, \e_2, \delta_Q)\mbox{-basis for  } \sigma_J(y, \mu|_{3Q}) \mbox{ for all } y \in 10Q.
\end{equation}

Fix $y \in 10Q$. Suppose first $y \in K_{CZ}$, so that $y \in Q_i$ for some $Q_i \in CZ^\circ$. Because $Q$ is not OK, we must have $\delta_{Q_i}/100 \leq \delta_Q <1$, and so, by applying (\ref{alocalbound}) for the cube $Q_i$ and for some $\widehat{\delta} \in [\delta_{Q_i}/2,1]$ satisfying $\delta_Q \leq \widehat{\delta} \leq 50 \delta_Q$, we deduce that
\begin{align}
    |\p^\beta P^\al_y(y)| \leq C_3 \widehat{\delta}^{|\al| - |\beta|} \leq 50^m C_3 \delta_Q^{|\al|-|\beta|} \quad \quad (\al \in \ma, \; \beta \in \mm). \label{d7kcz}
\end{align}
Now suppose $y \in K_p$. Then from (\ref{acobd}) with $\delta = \delta_Q$,
\begin{align}
    |\p^\beta P^\al_y(y)| \leq C_4 \delta_Q^{|\al|-|\beta|} \quad \quad (\al \in \ma, \; \beta \in \mm). \label{d7kp}
\end{align}
Because $\e_1 < \min \{(50^m C_3)^{-1/(D+1)}, C_4^{-1/(D+1)} \}$, in light of (\ref{d7kcz}) and (\ref{d7kp}), we have
\begin{equation}\label{qpt03}
    |\p^\beta P^\al_y(y)| \leq \e_1^{-D-1} \cdot \delta_Q^{|\al|-|\beta|} \qquad (y \in 10Q, \al \in \ma, \beta \in \mm)
\end{equation}
It is now evident from (\ref{qpt0})-(\ref{qpt01}), \eqref{qpt02}, and \eqref{qpt03} that properties \textbf{(D1)-(D7)} of Lemma \ref{newdir} hold with parameters:
\begin{align*}
\Big(\e_1,\e_2, \delta, \; \ma,\mu, E, &\left(\pt_x^\al \right)_{\al \in \ma, x \in E}, x_0, \bb, \pt_{x_0}^{\bb} \Big): = \Big(\e_1, \e_2, \delta_Q, \ma, \mu|_{3Q}, 10Q, \left(P_x^\al \right)_{\al \in \ma, x \in 10Q}, x, \beta_j, P_x^{\beta_j} \Big).
\end{align*}
Thus there exists $\mab<\ma$ so that for every $y \in 10Q$, $\sigma_J(y,\mu|_{3Q})$ contains an $(\mab, y, C_1 \e_1, \delta_Q)$-basis. Because $ C_1 \e_1 <  \e_0 /( 30^m C_0)$, we apply Lemma \ref{edlm} to deduce that for every $y \in 3Q$, $\sigma_J(y,\mu|_{3Q})$ contains an $(\mab, y, \e_0/C_0, 30 \delta_Q)$-basis, indicating that $Q$ is OK. This contradicts that $x \in K_p$ and $x \in Q$.

This completes the proof by contradiction of (\ref{kpt}). So we have proven property 2.

\textbf{Proof of property 3:} By definition,  $R_x^*(x) = H_1(x)$ for all $x \in K_p$, where $\| H_1 \|_{\J(f,\mu)} \lesssim \| f , P_0 \|_{\J(\mu; \delta_{Q^\circ}} +\eta$, and $\eta > 0$ is arbitrary. By property 2, $\vec{R}^*$ is independent of $\eta$. Then:
\begin{align*}
\int_{K_p} |R^*_x(x) - f(x)|^p d \mu(x) = \int_{K_p} |H_1 - f |^p d \mu &\leq \int_{\R^n} |H_1 - f |^p d \mu \leq \| H_1 \|_{\J(f,\mu)}^p \lesssim (\| f,P_0 \|_{\J(\mu; \delta_{Q^\circ})}+\eta)^p.
\end{align*}
Furthermore, since $R_x^* = J_x H_1$ for all $x \in K_p$, by definition of the $L^{m,p}(K_p)$ trace norm on $Wh(K_p)$, we have
\[
\| \vec{R}^* \|_{L^{m,p}(K_p)} \leq  \| H_1 \|_{L^{m,p}(\R^n)} \leq \| H_1 \|_{\J(f,\mu)} \lesssim \| f,P_0 \|_{\J(\mu; \delta_{Q^\circ})} +\eta.
\]
Now let $\eta \rightarrow 0$ in the previous inequalities. This completes the proof of property 3.

\textbf{Proof of property 4:} To complete the proof of the lemma, we will show that $\vec{R}^* = \vec{R}^*(f,P_0)$ depends linearly on $(f,P_0)$.

Let $(f_1,P_1)$, $(f_2,P_2) \in \J(\mu;\delta_{Q^\circ})$. Fix $H_j \in \J(f_j, \mu)$ such that $H_j$ is $K_p$-coherent with $P_j$ for $j=1,2$. By property 3, we have that $\vec{R}^*(f_j,P_j) = (J_x H_j)_{x \in K_p}$ for $j=1,2$.

Then, for $\lambda \in \R$, $H_1 + \lambda H_2 \in \J(f_1+\lambda f_2, \mu)$, and for $\al \in \ma$, $x \in K_p$, 
\begin{align*}
    \p^\al(H_1 + \lambda H_2)(x) &= \p^\al H_1(x) + \lambda \p^\al H_2(x) \\
    & = \p^\al P_1(x) + \lambda \p^\al P_2(x),
\end{align*}
indicating that $H_1 + \lambda H_2$ is $K_p$-coherent with $P_1 + \lambda P_2$. Because of (\ref{kpt}), 
\[
\vec{R}^*(f_1 + \lambda f_2, P_1 + \lambda P_2) = \{J_x (H_1 + \lambda H_2)\}_{x \in K_p} = \vec{R}^*(f_1,P_1)+ \lambda \vec{R}^*(f_2,P_2).
\]
Therefore, the map $(f,P_0) \mapsto \vec{R}^*(f,P_0)$ is linear, completing the proof of property 4.

\end{proof}

\begin{lem}\label{kptlm_loc}
Let $x \in K_p$ and $r > 0$. Set $\mu_{x,r} : = \mu|_{B(x,r)}$. Fix $(f,P_0) \in \J(\mu; \delta_{Q^\circ})$. Suppose $H \in L^{m,p}(\R^n)$, $\| H \|_{\J(f,\mu_{x,r})} < \infty$, and $H$ is $K_p$-coherent with $P_0$. Then $R^*_x(f,P_0) = J_x H$.
\end{lem}
\begin{proof}
We employ a proof by contradiction, following the proof of property 2 of Lemma \ref{kptlm},  with the measure $\mu$ replaced by $\mu_{x,r}$ in this proof. We make one change in our previous proof: When we choose the dyadic cube $Q \subset Q^\circ$ with $x \in Q$, we impose the additional condition $\delta_Q < r/3$. This condition implies $3 Q \subset B(x,r)$, so that $\mu_{x,r}|_{3Q} = \mu|_{3Q}$. Following our previous proof, we reach the conclusion that $\sigma_J(y, \mu_{x,r}|_{3Q})$ contains an $(\mab, y, \epsilon_0/C_0, 30 \delta_Q)$-basis for all $y \in 3Q$, for some $\mab < \ma$. So, $\sigma_J(y, \mu|_{3Q})$ contains an $(\mab, y, \epsilon_0/C_0, 30 \delta_Q)$-basis for all $y \in 3Q$, indicating that $Q$ is OK, a contradiction.
\end{proof}

As a consequence of Proposition \ref{acohlm}, the Whitney field $\vec{R}^* \in Wh(K_p)$ satisfies the following condition:

\begin{cor}
For $(f,P_0) \in \J(\mu; \delta_{Q^\circ})$, and for $\vec{R}^* = \vec{R}^*(f,P_0)$ as in Lemma \ref{kptlm}, 
\begin{align*}
\inf &\left\{ \| f, \vec{P}, \vec{R}^* \|_{\J_*(\mu; Q^{\circ}, CZ^\circ;K_p)}: \begin{aligned} &\vec{P} \in Wh(\bpt) \text{ satisfies} \nonumber \\
&\vec{P} \text{ is coherent with } P_0
\end{aligned}
\right\}  \lesssim \| f, P_0 \|_{\J(\mu;\delta_{Q^{\circ}})}. 
\end{align*} 
\label{acohkptcor}
\end{cor}

\section{Optimal Local Extension}
\label{sec:opt_local_ext}

In this section, we prove a general result, Lemma \ref{alinearmap2}, on the optimization of certain $L^p$-type norms by linear maps. In Section \ref{subsec:locext}, we apply this result to construct a Whitney field $\vec{R}'$ on the keystone basepoint set $\bkpt := \{ x_s \}_{s \in \bar{I}}$. We shall apply Lemma \ref{alinearmap2} once more, later, in Section \ref{sec:proof_maintheorems}, when we give the proofs of the main theorems. 

\subsection{Optimization by Linear Maps}
\begin{lem}[Linear Map Lemma] Let $(\nu,X,\Sigma)$ be a measure space, and let $V$ be a vector space. Let $k \geq 1$, let $p \geq 1$, and let $\Lambda : V \times \R^{k} \to L^p(d \nu)$ be a linear map. Then there exists a linear map $\xi: V \to \R^k$, satisfying:
\begin{align}
\int_{X} |\Lambda(v, \xi(v))|^p d \nu \leq C \inf_{w \in \R^k}  \int_{X} |\Lambda(v, w)|^p d \nu, \quad \mbox{ for all } v \in V, \label{linopt1}
\end{align}
where $C$ depends only on $k$ and $p$. \label{alinearmap}
\end{lem}
\begin{proof}
We will show this is true when $k=1$; then this can be iterated for the full result. We factor the linear map $\Lambda :V \times \R \rightarrow L^p(d \nu)$, as follows:  $\Lambda(v,w) = \hat{\Lambda}(v) - a \cdot w$ for a linear map $\hat{\Lambda}: V \to L^p(d\nu)$ and $a \in L^p(d\nu)$. Note, if $\| a \|_{L^p(d \nu)} = 0$ then we can take $\xi(v) = 0$, and the conclusion of the lemma will be satisfied. Thus, we may assume $\| a \|_{L^p(d \nu)} \neq 0$. Define $\xi: V \to \R$:
\begin{align*}
\xi(v) := \frac{\int_{\{a \neq 0\}}\frac{\hat{\Lambda}(v)}{a}|a|^p d\nu}{\int_{\{a \neq 0\}} |a|^p d\nu}.
\end{align*} 
Note the convergence of the integral in the numerator, by H\"{o}lder's inequality. Then $v \mapsto \xi(v)$ is linear.

For any $w \in \R$, we claim that $\int_{X} |\Lambda(v, \xi(v))|^p d\nu \leq C\int_{X} |\Lambda(v, w)|^p d\nu$. To see this, first note 
\begin{align*}
\int_{X} |\Lambda(v, w)|^p d\nu &= \int_{X} |\hat{\Lambda}(v) - a \cdot w|^p d\nu \\
&= \int_{\{a = 0\}} |\hat{\Lambda}(v)|^p d\nu + \int_{\{a \neq 0\}} |\hat{\Lambda}(v) - a \cdot w|^p d\nu.
\end{align*}
We bound
\begin{align*}
\int_{\{a \neq 0\}} |\hat{\Lambda}(v) - a \cdot \xi(v)|^p d\nu  &=  \int_{\{a \neq 0\}} \left|\frac{\hat{\Lambda}(v)}{a} -  \xi(v)\right|^p|a|^p d\nu \nonumber \\
& \leq 2^p \cdot \Big( \int_{\{a \neq 0\}} \left|\frac{\hat{\Lambda}(v)}{a} -  w\right|^p|a|^p d\nu + \int_{\{a \neq 0\}} \left| \frac{\int_{\{a \neq 0\}}\frac{\hat{\Lambda}(v)}{a}|a|^p d\nu}{\int_{\{a \neq 0\}}|a|^p d\nu} -  w\right|^p|a|^p d\nu \Big) \nonumber \\
& \leq (1+2^p) \cdot \int_{\{a \neq 0\}} \left|\frac{\hat{\Lambda}(v)}{a} -  w\right|^p|a|^p d\nu \\
&= (1+2^p) \cdot \int_{\{a \neq 0\}} \left|\hat{\Lambda}(v) -  a \cdot w\right|^p d\nu , 
\end{align*}
where the last inequality follows by applying Jensen's Inequality to the second term in the previous line. Thus, adding $\int_{\{a = 0\}} | \hat{\Lambda}(v)|^p d \nu$ to both sides of the above estimate, we obtain \eqref{linopt1} in the case $k=1$, with constant $C = 1 + 2^p$. By iterating $k$ times, we reach the conclusion of the lemma with a constant $C=(1+2^p)^k$.

\end{proof}

\begin{lem}[Linear Map Lemma II] Let $\mu_0$ be a Borel regular measure on $\R^n$, and let $V$ be a vector space. Let $k \geq 1$, and for each $\ell \in \N$, let $\lambda_\ell : V \times \R^{k} \to L^p(d \mu_0)$ be a linear map, and let $\phi_\ell :V \times \R^{k} \to \R$ be a linear functional. Let $\Psi : V \times \R^{k} \rightarrow \R^N$ be a linear map, such that $w \mapsto \Psi(0,w)$ is surjective ($N \leq k$).

Suppose that the functional
\[
M(v,w) := \sum_{\ell=1}^\infty \| \lambda_\ell(v,w) \|_{L^p(d \mu_0)}^p + \sum_{\ell=1}^\infty | \phi_\ell(v,w) |^p
\]
is finite for every $(v,w) \in V \times \R^k$.

Then there exists a linear map $\xi: V \to \R^k$, satisfying:
\begin{align*}
&\Psi(v,\xi(v)) = 0 \in \R^N,\\
&M(v, \xi(v)) \leq C \inf \{ M(v,w) : w \in \R^k, \; \Psi(v,w) = 0 \}, \quad \mbox{ for all } v \in V,
\end{align*}
where $C$ depends only on $k$ and $p$. \label{alinearmap2}
\end{lem}
\begin{proof}
We first suppose $N = 0$, i.e., the constraint map $\Psi$ is trivial. To prove Lemma \ref{alinearmap2} in this case, we apply Lemma \ref{alinearmap} to the product measure $\nu =  (\mu_0 + \delta_z) \times \mu_1 $ on the space $X = \left( \R^n \sqcup \{z\} \right) \times \N$, where $\mu_0$ is a given Borel regular measure on $\R^n$, $\delta_z$ is a Dirac delta measure supported at the point $z$ ($z \notin \R^n$), and $\mu_1$ is the counting measure on $\N$.

Now suppose $ N > 0$. Write $\Psi(v,w) = \Psi_1(v) - \Psi_2(w)$, for linear maps $\Psi_1 : V \to \R^N$ and $\Psi_2 : \R^k \to \R^N$. The condition that $w \mapsto \Psi(0,w)$ is surjective ensures that $\Psi_2$ is surjective. By Gaussian elimination, after possibly permuting the coordinates of $w = (w_1,\dots,w_k)$, the constraint set 
\[
\{ (v,w) : \Psi(v,w) = 0 \} = \{ (v,w) : \Psi_2(w) = \Psi_1(v) \} \subset V \times \R^k
\]
can be written in the form 
\begin{align*}
\{ (v,w^1,w^2) : v \in V, \; \; & w^1=(w_1,\dots, w_\ell) \in \R^\ell, \; w^2 = (w_{\ell+1},\dots, w_k) = \psi(v, w_1,\dots,w_\ell) \in \R^{k-\ell} \}
\end{align*}
for some $\ell \leq k$, and for a linear map $\psi : V \times \R^\ell \rightarrow \R^{\ell-k}$. (The condition that $\Psi_2$ is surjective ensures that the linear system $\Psi_2(w) = \Psi_1(v)$ admits a solution $w$ for every $v \in V$.)

We obtain the conclusion of the lemma then, by applying the version of Lemma \ref{alinearmap2} without constraints to the functional $\widetilde{M} : V \times \R^\ell \to \R_+$ given by 
\[
\widetilde{M}(v,w_1,\dots,w_\ell) := M(v,w_1,\dots,w_\ell, \psi(v,w_1,\dots,w_\ell)).
\]
\end{proof}

\subsection{Local Extension}\label{subsec:locext}

Fix a keystone cube $Q_s$. (See Definition \ref{keycuberef}.)  We will construct a polynomial $R_{x_s}' \in \mpp$ that is coherent with $P_0$ and approximately minimizes the expression $\|f, R_{x_s}'\|_{\J(\mu|_{9Q_s}, \delta_{Q_s})}$.

\begin{lem}
Let $Q_s \in CZ_{key}$, and let $(f,P_0) \in \J(\mu; \delta_{Q^\circ})$. There exists $R_{x_s}' \in \mathcal{P}$ that depends linearly on $(f|_{9Q_s},P_0)$ and satisfies $\p^\al R_{x_s}'(x_s) = \p^\al P_0(x_s)$ for all $\al \in \ma$, and 
\begin{align}
\|f, R_{x_s}'\|_{\J(\mu|_{9Q_s}, \delta_{Q_s})} \lesssim \inf_{R \in \mpp} \left\{ \|f, R\|_{\J(\mu|_{9Q_s}, \delta_{Q_s})}:  \p^\al R(x_s) = \p^\al P_0 (x_s)  \;\; \forall \al \in \ma  \right\}. \label{alocjet}
\end{align} \label{alocjetlm}
\end{lem}

\begin{proof}
Due to the good geometry of the CZ decomposition, $\delta_Q \rightarrow 0$ as $Q \rightarrow x$, $Q \in CZ^\circ$, $x \in K_p$. If $100Q_s \cap K_p \neq \emptyset$, then $100 Q_s$ intersects CZ cubes of arbitrarily small sidelength, due to the previous remark, which contradicts that $Q_s$ is a keystone cube. Therefore, $100 Q_s \cap K_p = \emptyset$. 

Let $CZ_s := \{\widehat{Q}^j\}_{1 \leq j \leq k}$ be the collection of all $Q \in CZ^\circ$ satisfying $1.1 Q \cap 100 Q_s \neq \emptyset$, and with $\widehat{Q}^1 = Q_s$. Let $x^j := \mathrm{ctr}(\widehat{Q}^j)$ and $D_j:=1.1\widehat{Q}^j \cap 9Q_s$ for  $j=1,\dots,k$. Then $\{D_j\}_{1 \leq j\leq k}$ is a cover of $9 Q_s$ by axis-parallel rectangles.

Let $P^1 \in \mathcal{P}$, and let $\vec{P}_* := (P^j)_{2 \leq j \leq k} \in \mathcal{P}^{k-1}$ be a $(k-1)$-tuple of elements of $\mathcal{P}$.

We apply \textbf{(AL1)-(AL3)} (see Section \ref{subsec:loc_ext_op}) to the CZ cube $\widehat{Q}^i$ and the Borel set $E_i = D_i \subset 3 \widehat{Q}^i$.  So, for each $i=1,2,\dots, k$, there exist a linear map $T_i:\J(\mu|_{D_i}) \times \mathcal{P} \rightarrow L^{m,p}(\R^n)$, a functional $M_i:\J(\mu|_{D_i}) \times \mathcal{P} \rightarrow \R_+$, and countable collections  of Borel sets $\{ A_\ell^i \}_{\ell \in \N}$ with
\begin{equation}\label{A_ell:eqn}
A_\ell^i \subset \supp(\mu|_{D_i}) \subset \supp(\mu|_{9Q_s}),
\end{equation}
and of linear maps $\phi_\ell^i: \J(\mu|_{D_i}) \times \mathcal{P} \to \R$, and $\lambda_\ell^i: \J(\mu|_{D_i}) \times \mathcal{P} \to L^p(d\mu)$ ($\ell \in \mathbb{N}$), that satisfy: For all $f \in \J(\mu|_{D_i})$ and $P \in \mpp$,
\begin{align*}
 &\textbf{(i) } M_i(f,P) \simeq \|T_i(f,P),P\|_{\J(f,\mu|_{D_i};\delta_{\widehat{Q}^i})} \simeq \|f,P\|_{\J(\mu|_{D_i};\delta_{\widehat{Q}^i})}; \text{ and} \\
 &\textbf{(ii) }M_i(f,P) = \Big(\sum_{\ell \in \N} \int_{A_\ell^i} |\lambda_\ell^i(f,P)-f|^p d\mu+\sum_{\ell \in \N} |\phi_\ell^i(f,P)|^p  \Big)^{1/p} < \infty.
\end{align*}

We now show that: For any $f \in \J(\mu|_{9 Q_s})$ and  $P^1 \in \mathcal{P}$,
\begin{align}
\|f & ,P^1\|_{\J(\mu|_{9Q_s}; \delta_{Q_s})}^p  \simeq \inf_{\vec{P}_* \in \mathcal{P}^{k-1} } \left\{ \sum_{i=1}^k \|f, P^i \|_{\J(\mu|_{D_i}; \delta_{\widehat{Q}^i})}^p + \sum_{i, j=1}^k |P^i - P^j|_{x^i,\delta_{\widehat{Q}^i}}^p  \right\}. \label{aoptlocext}
\end{align}


We prepare to apply Lemma \ref{rlmppatch} for the proof of the upper bound in (\ref{aoptlocext}).

We fix $P^1 \in \mpp$, and let $\vec{P}_* = (P^j)_{2 \leq j \leq k}$ be arbitrary.

We define a Whitney field $\vec{P} = (P_{x_i})_{i \in I} \in Wh(\bpt)$ by letting 
\[
P_{x_i} = \left\{\begin{aligned}
&P^j \mbox{ if } Q_i = \widehat{Q}^j,\mbox{ some } j=  1, \dots, k\mbox{; else} \\
&P^1 \mbox{ if } Q_i \in CZ^\circ \setminus \{\widehat{Q}^j\}_{1 \leq j \leq k}.
\end{aligned}
\right.
\]
Define $\vec{S} \in Wh(K_p)$ by letting $S_x = P^1$ for all $x \in K_p$. 
Define
\begin{align*}
F_i(x) = \left\{
        \begin{array}{ll}
        T_j(f,P^j) & \quad \mbox{if } Q_i = \widehat{Q}^j, \mbox{ some } j=  1, \dots, k \\
        P^1(x) & \quad \mbox{if } Q_i \in CZ^\circ \setminus \{\widehat{Q}^j\}_{1 \leq j \leq k}.
        \end{array}
    \right.
\end{align*}
Let $\{ \theta_i \}_{i \in I}$ be a partition of unity satisfying \textbf{(POU1)-(POU4)} (see Section \ref{subsec:pou}). Define $F: Q^\circ \to \R$ as
\begin{align*}
    F(x) = \begin{cases} \sum_{i \in I} \theta_i(x) \cdot F_i(x) & x \in K_{CZ} \\
    S_x(x) = P^1(x) & x \in K_p.
    \end{cases}
\end{align*} 

Because $\vec{S}$ is constant on $K_p$, $\|\vec{S} \|_{L^{m,p}(K_p)} = 0$. Further, $(\vec{P},\vec{S}) \in C^{m-1,1-n/p}(\bpt \cup K_p)$, because $(\vec{P},\vec{S})$ is constant except at finitely many points. Hence by Lemma \ref{rlmppatch}, if  $\|F,\vec{P}\|_{\J_*(f,\mu|_{9Q_s};K_{CZ},CZ^\circ)} < \infty$ then $F \in L^{m,p}(Q^\circ)$ and
\begin{align}\label{aole3.5}
    \|F\|_{L^{m,p}(Q^\circ)} \lesssim \|F,\vec{P}\|_{\J_*(f,\mu|_{9Q_s};K_{CZ},CZ^\circ)}.
\end{align}

We prepare to estimate $\|F,\vec{P}\|_{\J_*(f,\mu|_{9Q_s};K_{CZ},CZ^\circ)}$ and establish it to be finite. By Lemma \ref{rkczpatchlm},
\begin{align}
   \|F,\vec{P}\|_{\J_*(f,\mu|_{9Q_s};K_{CZ},CZ^\circ)}^p \lesssim & \sum_{i \in I} \|F_i, P_{x_i}\|_{\J_*(f,\mu|_{9Q_s \cap 1.1 Q_i} ;1.1 Q_i \cap Q^\circ)}^p  +\sum_{i \lra i'} |P_{x_i} - P_{x_{i'}}|_{x_i,\delta_{Q_i}}^p.\label{aole4}
\end{align}

From the definitions of $F_i$ and $P_{x_i}$, we see that all but finitely many of the terms in either of the preceding sums are equal to zero. Indeed,  for $Q_i \notin CZ_s = \{\widehat{Q}^j\}_{1 \leq j \leq k}$, $F_i = P_{x_i} = P^1$ and the support of the restricted measure $\mu|_{9Q_s}$, is disjoint from $1.1 Q_i$. Therefore, $\| F_i, P_{x_i} \|_{\J_*(f, \mu|_{9Q_s \cap 1.1Q_i}; 1.1Q_i \cap Q^\circ)} = 0$ for $Q_i \notin CZ_s$. Furthermore, by definition of $P_{x_i}$, note that $P_{x_i} \neq P_{x_{i'}}$ implies either $Q_i \in CZ_s$ or $Q_{i'} \in CZ_s$. Hence, the sum on the right-hand side of (\ref{aole4}) is finite. By reindexing the sums in (\ref{aole4}) to be over $j \in \{1,2,\dots,k\}$, and using (\ref{aole3.5}), we conclude that
\begin{align}
   \|F\|_{L^{m,p}(Q^\circ)}^p &\lesssim  \|F,\vec{P}\|_{\J_*(f,\mu|_{9Q_s};K_{CZ},CZ^\circ)}^p \nonumber \\
   &\lesssim  \sum_{j=1}^k \|T_j(f,P^j), P^j\|_{\J_*(f,\mu|_{9Q_s \cap 1.1\widehat{Q}^j} ;1.1 \widehat{Q}^j \cap Q^\circ)}^p +\sum_{j,j'=1}^k |P^j - P^{j'}|_{x^j,\delta_{\widehat{Q}^j}}^p.\label{aole4.5}
\end{align}

Note that $9 Q_s \cap \supp(\mu) \subset 9 Q_s \cap Q^\circ \subset K_{CZ}$, since $Q_s$ is keystone. Hence,
\begin{align}
    \int_{9 Q_s} | F -f |^p d \mu \leq \|F,\vec{P}\|_{\J_*(f,\mu|_{9Q_s};K_{CZ},CZ^\circ)}^p. \label{aole5}
\end{align}
Also, by Lemma \ref{polycubelm2},
\begin{align}
    \|F-P^1\|_{L^p(9 Q_s \cap Q^{\circ})}^p/\delta_{Q_s}^{mp} &\lesssim \| F \|_{L^{m,p}(9Q_s \cap Q^\circ)}^p + \|F-P^1\|_{L^p(Q_s)}^p/\delta_{Q_s}^{mp} \nonumber \\
    &\lesssim \|F,\vec{P}\|_{\J_*(f,\mu|_{9Q_s};K_{CZ},CZ^\circ)}^p,\label{aole6}
\end{align}
where the last line follows because $\widehat{Q}^1 = Q_s$, and so $P_{x_{s}} = P^1$ by definition of $\vec{P}$.  

Combining \eqref{aole4.5}, \eqref{aole5}, and \eqref{aole6}, we have
\begin{align*}
\| F,   P^1 \|_{\J_*(f, \mu|_{9 Q_s}; 9 Q_s \cap Q^\circ)}^p &\simeq \| F \|_{L^{m,p}(9 Q_s \cap Q^\circ)}^p + \int_{9 Q_s} |F- f |^p d \mu + \| F - P^1 \|_{L^p(9 Q_s \cap Q^\circ)}^p/\delta_{Q_s}^{mp}  \\
&  \lesssim  \sum_{j=1}^k \| T_j(f,P^j), P^j \|_{\J_*(f, \mu|_{D_j}; 1.1 \widehat{Q}^j \cap Q^\circ)}^p + \sum_{j,j'=1}^k |P^j - P^{j'}|_{x^j,\delta_{\widehat{Q}^j}}^p.
\end{align*}
By property \textbf{(i)} of $T_j$, and estimate \eqref{j5}, relating the $\J_*$ and $\J$-functionals, 
\begin{align*}
    \| T_j(f,P^j), P^j \|_{\J_*(f, \mu|_{D_j}; 1.1 \widehat{Q}^j \cap Q^\circ)}^p &\lesssim \| T_j(f,P^j), P^j \|_{\J(f,\mu|_{D_j};\delta_{\widehat{Q}^j})} \simeq \|f,P^j\|_{\J(\mu|_{D_j};\delta_{\widehat{Q}^j})},
\end{align*}
and thus, 
\begin{align*}
\| F,  P^1 \|^p_{\J_*(f, \mu|_{9 Q_s}; 9 Q_s \cap Q^\circ)}  \lesssim \sum_{j=1}^k \|f,P^j\|_{\J(\mu|_{D_j};\delta_{\widehat{Q}^j})}^p+ \sum_{j,j'=1}^k |P^j - P^{j'}|_{x^j,\delta_{\widehat{Q}^j}}^p.
\end{align*}
Thanks to \eqref{j8} and the definition of the $\J_*$-functional as an infimum, 
\[
\| f, P^1 \|_{\J(\mu|_{9Q_s}; \delta_{Q_s})} \simeq \| f, P^1 \|_{\J_*(\mu|_{9Q_s}; 9Q_s \cap Q^\circ)} \leq \| F,  P^1 \|^p_{\J_*(f, \mu|_{9 Q_s}; 9 Q_s \cap Q^\circ)}.
\]
Combining the above inequalities, we take the infimum over $P^2,\dots, P^k \in \mpp$, and obtain the upper bound in (\ref{aoptlocext}):
\begin{align*}
\|f,&P^1\|_{\J(\mu|_{9Q_s}, \delta_{Q_s})}^p \lesssim \inf \left\{ \sum_{j=1}^k \|f, P^j \|_{\J(\mu|_{D_j}; \delta_{\widehat{Q}^j})}^p + \sum_{j,j' =1}^k |P^j - P^{j'}|_{x^j,\delta_{\widehat{Q}^j}}^p: P^2,\dots,P^k \in \mpp  \right\}.
\end{align*}

Now, to complete the proof of (\ref{aoptlocext}), we'll show the reverse inequality:
\begin{align*}
 \|f,&P^1\|_{\J(\mu|_{9Q_s}, \delta_{Q_s})}^p \gtrsim \inf \left\{ \sum_{j=1}^k \|f, P^j \|_{\J(\mu|_{D_j}; \delta_{\widehat{Q}^j})}^p + \sum_{j,j' =1}^k |P^j - P^{j'}|_{x^j,\delta_{\widehat{Q}^j}}^p: P^2,\dots,P^k \in \mpp  \right\}.
\end{align*}
By taking $P^j = P^1$ for $j = 2, \dots, k$, we learn that
\begin{align*}
\inf \Big\{ \sum_{j=1}^k  \|f, P^j \|_{\J(\mu|_{D_j}; \delta_{\widehat{Q}^j})}^p + \sum_{j,j' =1}^k |P^j - P^{j'}&|_{x^j,\delta_{\widehat{Q}^j}}^p: P^2,\dots,P^k \in \mpp  \Big\} \\ 
&\leq \sum_{j=1}^k \|f, P^1 \|_{\J(\mu|_{D_j}; \delta_{\widehat{Q}^j})}^p \simeq \sum_{j=1}^k \|f, P^1 \|_{\J(\mu|_{D_j}; \delta_{Q_s})}^p,
\end{align*}
where the last line uses that $\delta_{Q_s} \simeq \delta_{\widehat{Q}^j}$ for $j=1,2,\dots,k$. Finally, note that $D_j \subset 9Q_s$, so
\begin{align*}
\sum_{j=1}^k \|f, P^1 \|_{\J(\mu|_{D_j}; \delta_{Q_s})}^p \leq k \|f, P^1 \|_{\J(\mu|_{9Q_s}; \delta_{Q_s})}^p.
\end{align*}
Since $k$ is bounded by a universal constant, this completes the proof of (\ref{aoptlocext}).

We prepare to apply (\ref{aoptlocext}) to construct the polynomial $R_{x_s}'$. By definition of the $|\cdot|_{x,\delta}$ norm (\ref{pnorm}), and conditions \textbf{(i)} and \textbf{(ii)} relating to properties of $T_i$ and $M_i$, we approximate the expression inside the infimum in (\ref{aoptlocext}) as follows:
\begin{align}
\sum_{i=1}^k \|f, &P^i \|_{\J(\mu|_{D_i}; \delta_{\widehat{Q}^i})}^p +  \sum_{i, j =1}^k |P^i - P^j|_{x^i,\delta_{\widehat{Q}^i}}^p \nonumber \\ 
& \simeq \sum_{i=1}^k \sum_{\ell \in \N} \big( \int_{A_\ell^i} |\lambda_\ell^i(f,P^i)-f|^p d \mu +|\phi_\ell^i(f,P^i)|^p \big)+ \sum_{i,j=1}^k \sum_{\al \in \mm} c_{\al,i,j} | \partial^\al(P^i - P^j)(x^i)|^p, \label{n1}
\end{align}
for some constants $c_{\al,i,j} \geq 0$. For $i,j \in \{1, 2, \dots ,k\}$, we define $\phi^{i,j, \al}: \J(\mu|_{9Q_s}) \times \mathcal{P} \times \mathcal{P}^{k-1} \to \R$ as $\phi^{i,j,\al}(f,P^1,\vec{P}_*) := c_{\al,i,j}^{1/p} \partial^\al(P^i - P^j)(x^i)$. By substituting these functionals into the right-hand side of (\ref{n1}) and reindexing the sum, we have
\begin{align}\label{exp1}
&\sum_{i=1}^k \|f, P^i \|_{\J(\mu|_{D_i}; \delta_{\widehat{Q}^i})}^p +  \sum_{i, j =1}^k |P^i - P^j|_{x^i,\delta_{\widehat{Q}^i}}^p 
\simeq M(f,P^1, \vec{P}_*)^p, \mbox{ where} \\
&M(f,P^1, \vec{P}_*)^p := \sum_{\ell \in \N} \int_{A_\ell} |\lambda_{\ell}(f,P^1,\vec{P}_*)-f|^p d \mu +  | \phi_\ell(f,P^1, \vec{P}_*)|^p, \nonumber
\end{align}
for countable collections of Borel sets $\{A_\ell\}_{\ell \in \N}$ with $A_\ell \subset \supp(\mu|_{9Q_s})$ (see \eqref{A_ell:eqn}), and of linear maps $\phi_\ell : \J(\mu|_{9Q_s}) \times \mathcal{P} \times \mathcal{P}^{k-1} \to \R$, and $\lambda_\ell : \J(\mu|_{9Q_s}) \times \mathcal{P} \times \mathcal{P}^{k-1} \to L^p(d\mu)$ ($\ell \in \mathbb{N}$). In combination with (\ref{aoptlocext}),
\begin{align}
\|f,P^1\|_{\J(\mu|_{9Q_s}; \delta_{Q_s})}^p \simeq  \inf \left\{ M(f,P^1,\vec{P}_*): \vec{P}_* \in \mathcal{P}^{k-1} \right\}.  \label{aole1}
\end{align}
According to \eqref{exp1}, the functional $M(f,P^1,\vec{P}_*)$ is finite-valued for every $(f,P^1,\vec{P}_*) \in \J(\mu|_{9Q_s}) \times \mpp \times \mpp^{k-1}$. So we are justified to apply Lemma \ref{alinearmap2} (with a trivial constraint map $\Psi$) to determine $\vec{P}_*=(P^2,\dots , P^k) := \xi(f,P^1) \in \mathcal{P}^{k-1}$ depending linearly on $(f|_{9Q_s},P^1) $ and satisfying
\begin{align}
M(f,P^1,\xi(f,P^1)) \leq C \cdot \inf \left\{ M(f,P^1,\vec{P}_*): \vec{P}_* \in \mathcal{P}^{k-1}  \right\}, \label{aole2}
\end{align}
where $C$ is a constant determined by $p$, $k$, and $D$. Observe that both $k$ and $D = \dim \mpp$ are bounded by constants depending on $m$ and $n$. Hence, $C$ is bounded by a constant determined by $p$, $m$, and $n$.

From (\ref{aole1}) and (\ref{aole2}),
\begin{align}
\inf_{P^1 \in \mpp} \Bigg\{ \|&f, P^1\|_{\J(\mu|_{9Q_s}; \delta_{Q_s})}: \p^\al P^1(x_s) = \p^\al P_0 (x_s) \text{ } \forall \al \in \ma
\Bigg\} \nonumber \\
&\simeq \inf_{P^1 \in \mpp} \Bigg\{ M(f,P^1, \xi(f,P^1)) : \p^\al P^1(x_s) = \p^\al P_0 (x_s) \;\forall \al \in \ma\Bigg\}. \label{aole3}
\end{align}
We apply Lemma \ref{alinearmap2} again, to the functional $M(f,P_0,P^1) := M(f,P^1,\xi(f,P^1))$, which is independent of $P_0$, where the constraint map $\Psi : \J(\mu) \times \mpp \times \mpp \rightarrow \R^{|\ma|}$ is chosen so that $\Psi(f,P_0,P^1) = 0$ encodes the constraints $\p^\al P^1(x_s) = \p^\al P_0 (x_s)$ (all $\alpha \in \ma$). Thus we determine $R_{x_s} = \xi_s(f, P_0)  \in \mathcal{P}$ depending linearly on $(f|_{9Q_s},P_0)$ and satisfying $\partial^\al R_{x_s}(x_s) = \partial^\al P_0(x_s)$ for all $\al \in \ma$, and
\begin{align*}
M(f,R_{x_s}, \xi(f,R_{x_s}))\leq C \cdot \inf_{P^1 \in \mpp } \Bigg\{ M(f,P^1, \xi(f,P^1)) : \p^\al P^1(x_s) = \p^\al P_0 (x_s) \; \forall \al \in \ma\Bigg\},
\end{align*}
where $C= C(p, D)$. Therefore, in light of (\ref{aole3}), $R_{x_s}' := R_{x_s}$ satisfies (\ref{alocjet}). This completes the construction and verification of the properties of the polynomial $R_{x_s}'$.

\end{proof}

\section{Decomposition of the Functional}
\label{sec:decomp_func}

\begin{lem}
Let $Q^\circ$, $CZ^\circ$, $K_{CZ}$, and $K_p$ be defined as in Section \ref{sec:cz_decomp}.  Let $\J_*(\mu;Q^\circ,CZ^\circ;K_p)$ be the space defined in \eqref{Jstar_space_3A:defn}.

\textbf{Existence of a Bounded Linear Map.} There exists a linear map $T: \J_*(\mu;Q^\circ,CZ^\circ;K_p) \to L^{m,p}(Q^{\circ})$ satisfying the following conditions. For any $(f,\vec{P},\vec{S}) \in \J_*(\mu;Q^\circ,CZ^\circ;K_p)$,
\begin{align}\label{top1}
    &J_x T(f,\vec{P},\vec{S}) = S_x \mbox{ for all } x \in K_p, \mbox{ and}\\
 &\|T(f,\vec{P},\vec{S}), \vec{P}\|_{\J_*(f,\mu; Q^{\circ}, CZ^\circ)} \leq C \cdot \|f,\vec{P}, \vec{S}\|_{\J_*(\mu; Q^{\circ}, CZ^\circ;K_p)}.  \label{top2}
\end{align}

\textbf{Characterization of the Function Space.} For  $f \in \J(\mu)$, $\vec{P} \in Wh(\bpt)$ and $\vec{S} \in Wh(K_p)$, consider the functionals $\mathcal{S}_0(f,\vec{P},\vec{S}) \leq \mathcal{S}_1(f,\vec{P},\vec{S}) \leq \mathcal{S}_2(f,\vec{P},\vec{S})$ valued in $[0,\infty]$, given by
\begin{align}
\mathcal{S}_0(f, \vec{P}, \vec{S}) &:=   \sum_{i \in I} \|f,P_{x_i}\|_{\J(\mu|_{1.1Q_i}; \delta_{Q_i})}^p  + \sum_{i \lra i'} |P_{x_i} - P_{x_{i'}}|_i^p   + \int_{K_p} |S_x(x) - f(x)|^p d \mu, \nonumber \\
\mathcal{S}_1(f, \vec{P}, \vec{S}) &:=   \mathcal{S}_0(f, \vec{P}, \vec{S}) + 1_{\ma = \emptyset} \cdot \|\vec{S} \|_{L^{m,p}(K_p)}^p,\label{func_decomp1}\\
\mathcal{S}_2(f,\vec{P},\vec{S}) &:=   \mathcal{S}_0(f, \vec{P}, \vec{S}) + \|\vec{S} \|_{L^{m,p}(K_p)}^p.\nonumber
\end{align}
Then $(f,\vec{P}, \vec{S}) \in \J_*(\mu; Q^{\circ}, CZ^\circ;K_p)$ if and only if
\begin{equation}\label{func_decomp2}
(\vec{P},\vec{S}) \in C^{m-1,1-n/p}(\bpt \cup K_p) \quad \mbox{and} \quad \mathcal{S}_2(f,\vec{P},\vec{S}) < \infty.
\end{equation}
Further, for all $(f,\vec{P},\vec{S}) \in \J_*(\mu;Q^\circ,CZ^\circ;K_p)$, 
\begin{align}
\|f,\vec{P}, \vec{S}\|_{\J_*(\mu; Q^{\circ}, CZ^\circ;K_p)}^p \simeq \mathcal{S}_1(f, \vec{P}, \vec{S}). \label{adecomp}
\end{align} 
In \eqref{func_decomp1}, the indicator term $1_{\ma = \emptyset} \|\vec{R^*}\|_{L^{m,p}(K_p)}^p$ is included in the expression $\mathcal{S}_1(f,\vec{R},\vec{R}^*)$ if and only if $\ma = \emptyset$.

\label{adecomplm}
\end{lem}

\begin{proof}
We establish half of (\ref{func_decomp2}). Specifically, we show that $(f,\vec{P}, \vec{S}) \in \J_*(\mu; Q^{\circ}, CZ^\circ;K_p)$ implies
\begin{equation} \label{nadecomp}
(\vec{P},\vec{S}) \in C^{m-1,1-n/p}(\bpt \cup K_p) \quad \mbox{and} \quad \mathcal{S}_2(f,\vec{P},\vec{S}) \lesssim \|f,\vec{P}, \vec{S}\|_{ \J_*(\mu; Q^{\circ}, CZ^\circ;K_p)}^p< \infty.
\end{equation}
Let $(f,\vec{P},\vec{S}) \in \J_*(\mu;Q^\circ,CZ^\circ;K_p)$; then from \eqref{spineq}, we have $(\vec{P},\vec{S}) \in C^{m-1,1-n/p}(\bpt \cup K_p)$. Let $\eta > 0$; then there exists $H \in L^{m,p}(Q^{\circ})$ satisfying 
\begin{align}
&J_xH = S_x \mbox{ for all } x \in K_p\mbox{, and}\\
&\|H,\vec{P}\|_{\J_*(f,\mu;Q^\circ, CZ^\circ)} \leq  \|f,\vec{P},\vec{S} \|_{\J_*(\mu; Q^{\circ}, CZ^\circ;K_p)} +\eta < \infty. \label{goodh}
\end{align}
As a consequence of Lemma \ref{restwh}, we have
\begin{align}
\|\vec{S}\|_{L^{m,p}(K_p)} \lesssim \|H\|_{L^{m,p}(Q^\circ)} \leq \|H,\vec{P}\|_{\J_*(f,\mu;Q^\circ, CZ^\circ)} \lesssim \|f,\vec{P},\vec{S} \|_{\J_*(\mu; Q^{\circ}, CZ^\circ;K_p)} +\eta < \infty. \label{vecsbd}    
\end{align}
Also,
\begin{align}
    \int_{K_p} | S_x(x) - f(x) |^p d \mu  &\leq \int |H-f|^p d \mu \leq \|H,\vec{P}\|_{\J_*(f,\mu;Q^\circ, CZ^\circ)}^p \nonumber \\
    &\lesssim (\|f,\vec{P},\vec{S} \|_{\J_*(\mu; Q^{\circ}, CZ^\circ;K_p)} +\eta)^p < \infty. \label{bdfors}
\end{align}
By Lemma \ref{polycubelm2},
\begin{align}
\|H - P_{x_i}\|_{L^p(1.1Q_i \cap Q^{\circ})}^p/\delta_{Q_i}^{mp} \lesssim \|H - P_{x_i}\|_{L^p(Q_i)}^p/\delta_{Q_i}^{mp} + \| H \|_{L^{m,p}(1.1Q_i \cap Q^\circ)}^p. \label{badeco1}
\end{align}
Because of (\ref{goodgeo}), (\ref{movept}), and (\ref{normdi}), it holds that $|P|_i \simeq |P|_{i'}$ for all $P \in \mpp$ if $i \lra i'$. Thus, by the Sobolev Inequality, (\ref{normest}), and the bounded overlap of the cubes $\{1.1Q_i\}_{i \in I}$, we have
\begin{align}
\sum_{i \lra i'} |P_{x_i} - P_{x_{i'}}|_i^p &\overset{(\ref{movept})}{\lesssim}  \sum_{i \lra i'} \left\{ |P_{x_i} - J_{x_i}H|_i^p + |J_{x_i}H - J_{x_{i'}}H|_i^p + |J_{x_{i'}}H - P_{x_{i'}}|_{i'}^p \right\} \nonumber \\
&\overset{\text{SI/}(\ref{normest})}{\lesssim} \sum_{i \in I} \|H,P_{x_i}\|_{\J_*(f,\mu|_{Q_i};Q_i)}^p + \sum_{i \lra i'} \|H\|_{L^{m,p}((1.1Q_i \cup 1.1Q_{i'}) \cap Q^\circ)}^p \nonumber \\
&\lesssim \sum_{i \in I} \|H,P_{x_i}\|_{\J_*(f,\mu|_{Q_i};Q_i)}^p + \| H \|_{L^{m,p}(K_{CZ})}^p. \label{badeco2}
\end{align}
We apply (\ref{j8}) (for the $C$-non-degenerate rectangular box $1.1Q_i \cap Q^\circ$), (\ref{badeco2}), (\ref{badeco1}), and (\ref{goodh}) to estimate
\begin{align}
\sum_{i \in I} \|f, &P_{x_i}\|_{\J(\mu|_{1.1Q_i}; \delta_{Q_i})}^p \overset{(\ref{j8})}{\lesssim} \sum_{i \in I} \|f, P_{x_i}\|_{\J_*(\mu|_{1.1Q_i}; 1.1Q_i \cap Q^{\circ})}^p \nonumber\\
& \; \leq \sum_{i \in I} \|H, P_{x_i}\|_{\J_*(f,\mu|_{1.1Q_i}; 1.1Q_i \cap Q^{\circ})}^p  \nonumber\\
& \; = \sum_{i \in I} \Big[ \|H\|_{L^{m,p}(1.1Q_i \cap Q^\circ)}^p + \int_{1.1Q_i}|H-f|^p d\mu + \|H - P_{x_i}\|_{L^p(1.1Q_i \cap Q^{\circ})}^p/\delta_{Q_i}^{mp} \Big] \nonumber\\
& \overset{(\ref{badeco1})}{\lesssim}\sum_{i \in I} \Big[ \|H\|_{L^{m,p}(1.1Q_i \cap Q^\circ)}^p + \int_{1.1Q_i}|H-f|^p d\mu  + \|H - P_{x_i}\|_{L^p(Q_i)}^p/\delta_{Q_i}^{mp} \Big]  \nonumber\\
& \; \lesssim \|H\|_{L^{m,p}(K_{CZ})}^p + \int_{K_{CZ}} |H-f|^p d\mu + \sum_{i \in I} \|H - P_{x_i}\|_{L^p(Q_i)}^p/\delta_{Q_i}^{mp} \nonumber\\
&\; \; =\|H, \vec{P} \|_{\J_*(f, \mu; K_{CZ}, CZ^\circ)}^p \nonumber \\
& \overset{(\ref{goodh})}{\leq}  (\|f,\vec{P},\vec{S} \|_{\J_*(\mu; Q^{\circ}, CZ^\circ;K_p)} +\eta )^p. \label{badeco3} 
\end{align}
Combining \eqref{vecsbd}, \eqref{bdfors}, \eqref{badeco2}, and \eqref{badeco3}, and letting $\eta \rightarrow 0$, we have established (\ref{nadecomp}) and consequently the sufficiency of (\ref{func_decomp2}).

Next we describe the construction of the map $T$. By applying  \textbf{(AL1)-(AL3)} (see Section \ref{subsec:loc_ext_op}) to the cube $Q_i \in CZ^\circ$ and  subset $E_i = 1.1Q_i \subset 3Q_i$, we obtain the existence of a linear map $T_i: \J(\mu|_{1.1Q_i};\delta_{Q_i}) \to L^{m,p}(\R^n)$, a functional $M_i: \J(\mu|_{1.1Q_i};\delta_{Q_i}) \to \R_+$, and countable collections of  Borel sets $\{A_\ell^i\}_{\ell \in \N}, A_\ell^i \subset \supp(\mu|_{1.1Q_i})$, and of linear maps $\{\phi_\ell^i: \J(\mu|_{1.1 Q_i} ;\delta_{Q_i}) \to \R\}_{\ell \in \N}$, and $\{\lambda_\ell^i: \J(\mu|_{1.1 Q_i} ;\delta_{Q_i}) \to L^p(d\mu)\}_{\ell \in \N}$, with the following properties.

Given $(f,P_{x_i}) \in \J(\mu|_{1.1Q_i};\delta_{Q_i})$,
\begin{align} 
&\|f,P_{x_i}\|_{\J(\mu|_{1.1Q_i}; \delta_{Q_i})} \simeq \|T_i(f,P_{x_i}), P_{x_i}\|_{\J(f,\mu|_{1.1Q_i};\delta_{Q_i})}  \simeq M_i(f,P_{x_i}) \label{amf1}; \text{ and} \\
&M_i(f,P_{x_i}) = \Big(\sum_{\ell \in \N} \int_{A_\ell^i} |\lambda_\ell^i(f,P_{x_i})-f|^pd\mu +\sum_{\ell \in \N} |\phi_\ell^i(f,P_{x_i})|^p \Big)^{1/p}. \label{amf2}
\end{align}
Further, the maps $T_i$ and $M_i$ are $\Omega_i'$-constructible for a family of linear functionals $\Omega_i' \subset \J(\mu|_{1.1Q_i})^*$, i.e., the maps satisfy \textbf{(AL4)-(AL6)} for $E_i = 1.1 Q_i$

Given $(f,\vec{P}, \vec{S}) \in \J(\mu)\times Wh(\bpt) \times Wh(K_p)$, define a function $T(f,\vec{P},\vec{S}): Q^{\circ} \to \R$ as
\begin{align}
T(f,\vec{P},\vec{S})(x) = \begin{cases} \sum_{i \in I} T_i(f,P_{x_i})(x) \cdot \theta_i(x) & x \in K_{CZ} \\
S_x(x) & x \in K_p \end{cases} \label{tfdef}
\end{align}
where $\{\theta_i\}_{i \in I}$ satisfy \textbf{(POU1)-(POU4)} (see Section \ref{subsec:pou}). We apply (\ref{apatch1}), (\ref{j5}), and (\ref{amf1}) to deduce
\begin{align}
\|\tf, \vec{P} \|_{\J_*(f,\mu;K_{CZ}, CZ^\circ)}^p 
&\overset{(\ref{apatch1})}{\lesssim} \sum_{i \in I} \|T_i(f,P_{x_i}), P_{x_i}\|_{\J_*(f,\mu|_{1.1Q_i \cap Q^\circ};1.1Q_i\cap Q^\circ)}^p + \sum_{i' \lra i} |P_{x_i} - P_{x_{i'}}|_i^p \nonumber \\
&\overset{(\ref{j5})}{\lesssim} \sum_{i \in I} \|T_i(f,P_{x_i}), P_{x_i}\|_{\J(\mu|_{1.1Q_i}; \delta_{Q_i})}^p + \sum_{i' \lra i} |P_{x_i} - P_{x_{i'}}|_i^p \nonumber \\
&\overset{(\ref{amf1})}{\lesssim}  \sum_{i \in I} \|f, P_{x_i}\|_{\J(\mu|_{1.1Q_i}; \delta_{Q_i})}^p + \sum_{i' \lra i} |P_{x_i} - P_{x_{i'}}|_i^p  . \label{bamf3} 
\end{align}

We have not proved the map $T$ is bounded yet, but we return to the proof of the necessity of (\ref{func_decomp2}): Suppose $
(\vec{P},\vec{S}) \in C^{m-1,1-n/p}(\bpt \cup K_p)$ and $\mathcal{S}_2(f,\vec{P},\vec{S}) < \infty$. In light of (\ref{bamf3}) and the definition of the $\mathcal{S}_2$ functional,
\[
\|\tf, \vec{P} \|_{\J_*(f,\mu;K_{CZ}, CZ^\circ)}^p + \| \vec{S} \|_{L^{m,p}(K_p)}^p \lesssim \mathcal{S}_2(f,\vec{P},\vec{S})  < \infty.
\]
Consequently, we may apply Lemma \ref{lmppatch} to deduce that the function $\tf$ defined in \eqref{tfdef} satisfies
\begin{align}
&\tf \in L^{m,p}(Q^\circ), \label{tfmember} \\
& J_x\tf = S_x \mbox{ for all } x \in K_p, \mbox{ and} \label{tfjet}\\
&\| \tf \|_{L^{m,p}(Q^\circ)} \lesssim \|\tf, \vec{P} \|_{\J_*(f,\mu;K_{CZ}, CZ^\circ)} + \|\vec{S}\|_{L^{m,p}(K_p)} < \infty. \label{tfbound}
\end{align}
Because $\tf(x) = S_x(x)$ for $x \in K_p$, we have
\begin{align}
    \int_{K_p} |\tf - f|^p d \mu &= \int_{K_p} | S_x(x) - f(x)|^p d \mu(x) \leq \mathcal{S}_2 (f, \vec{P},\vec{S}) < \infty. \label{tfbound2}
\end{align}
In combination with (\ref{tfmember})-(\ref{tfbound}) and (\ref{bamf3}), we bound
\begin{align}
\|f,\vec{P},\vec{S} &\|_{\J_*(\mu; Q^{\circ},CZ^\circ;K_p)}^p \leq \|\tf, \vec{P} \|_{\J_*(f,\mu;Q^\circ, CZ^\circ)}^p \nonumber \\ 
& \quad \quad  \leq \| \tf, \vec{P} \|_{\J_*(f,\mu; K_{CZ}, CZ^\circ)}^p + \| \tf \|_{L^{m,p}(Q^\circ)}^p + \int_{K_p} |\tf - f|^p d \mu
\nonumber \\
& \quad \quad \lesssim \sum_{i \in I} \|f, P_{x_i}\|_{\J(\mu|_{1.1Q_i}; \delta_{Q_i})}^p + \sum_{i' \lra i} |P_{x_i} - P_{x_{i'}}|_i^p + \|\vec{S}\|_{L^{m,p}(K_p)}^p + \int_{K_p} |S_x(x) - f(x)|^p d\mu \nonumber \\
& \quad \quad \lesssim \mathcal{S}_2 (f, \vec{P},\vec{S}) <\infty, \label{nb2}
\end{align}
completing the proof of (\ref{func_decomp2}). 

Suppose $(f,\vec{P}, \vec{S}) \in \J_*(\mu; Q^{\circ}, CZ^\circ;K_p)$; we proved the extension operator $\tf$ defined in (\ref{tfdef}) satisfies $J_x T(f,\vec{P},\vec{S}) = S_x$ for all $x \in K_p$, and
\begin{equation} \label{nb3}
\mathcal{S}_2(f,\vec{P},\vec{S}) \lesssim \|f,\vec{P}, \vec{S}\|_{ \J_*(\mu; Q^{\circ}, CZ^\circ;K_p)}^p
\end{equation}
(see (\ref{nadecomp})). In combination with (\ref{nb2}), this proves (\ref{top1}) and (\ref{top2}): 
\begin{align*}
\|\tf, \vec{P} \|_{\J_*(f,\mu;Q^\circ, CZ^\circ)} \lesssim \|f,\vec{P}, \vec{S}\|_{ \J_*(\mu; Q^{\circ}, CZ^\circ;K_p)},
\end{align*}
as well as (\ref{adecomp}) in the case $\ma = \emptyset$: In this case, $\mathcal{S}_2(f,\vec{P},\vec{S})=\mathcal{S}_1(f,\vec{P},\vec{S})$. Together with (\ref{nb2}) and (\ref{nb3}), we have
\begin{align*}
\|f,\vec{P}, \vec{S}\|_{\J_*(\mu; Q^{\circ}, CZ^\circ;K_p)}^p \simeq \mathcal{S}_1(f, \vec{P}, \vec{S}).
\end{align*} 

Suppose $\ma \neq \emptyset$. 
Thanks to Lemma \ref{measkplem}, the set $K_p$ has Lebesgue measure $0$, and so 
\begin{align*}
    \| \tf \|_{L^{m,p}(Q^\circ)} = \| \tf \|_{L^{m,p}(K_{CZ})} \qquad (\ma \neq \emptyset),
\end{align*}
and therefore,
\begin{align}
    & \| \tf, \vec{P}  \|_{\J_*(f,\mu;Q^{\circ}, CZ^\circ)}^p = \| \tf \|_{\J_*(f,\mu;K_{CZ}, CZ^\circ)}^p + \int_{K_p} |\tf - f|^p d \mu. \label{tfeq2}
\end{align}
Combining \eqref{bamf3}, \eqref{tfbound2}, \eqref{tfeq2}, 
\begin{align}
\|\tf,  \vec{P} \|_{\J_*(f,\mu;Q^{\circ}, CZ^\circ)}^p &\overset{\eqref{tfeq2}}{=} \|\tf, \vec{P} \|_{\J_*(f,\mu;K_{CZ}, CZ^\circ)}^p + \int_{K_p} |\tf - f|^p d \mu \nonumber \\
&\overset{\eqref{tfbound2}}{=} \|\tf, \vec{P} \|_{\J_*(f,\mu;K_{CZ}, CZ^\circ)}^p + \int_{K_p} | S_x(x) - f(x)|^p d \mu(x) \nonumber \\
&\overset{\eqref{bamf3}}{\lesssim}  \sum_{i \in I} \|f, P_{x_i}\|_{\J(\mu|_{1.1Q_i}; \delta_{Q_i})}^p + \sum_{i \lra i'} |P_{x_i} - P_{x_{i'}}|_i^p + \int_{K_p} | S_x(x) - f(x)|^p d \mu(x) \nonumber \\
&= \mathcal{S}_1(f,\vec{P},\vec{S}). \label{amf3a}
\end{align}
Therefore,
\[
\|f,\vec{P},\vec{S} \|_{\J_*(\mu; Q^{\circ},CZ^\circ;K_p)}^p \leq \|\tf,  \vec{P} \|_{\J_*(f,\mu;Q^{\circ}, CZ^\circ)}^p \lesssim \mathcal{S}_1 (f, \vec{P}, \vec{S}).
\]
Combining this estimate with (\ref{nb3}) and $\mathcal{S}_1 (f, \vec{P}, \vec{S}) \leq \mathcal{S}_2(f,\vec{P},\vec{S})$, we have proven \eqref{adecomp} under the assumption that $\ma \neq \emptyset$, completing the proof of Lemma \ref{adecomplm}.

\end{proof}

\section{Optimal Whitney Field}\label{sec:opt_wh_field}

We consider the set of all Calder\'{o}n-Zygmund cubes $CZ^\circ = \{Q_i\}_{i \in I}$ in $Q^\circ$. We denote $K_{CZ} = \bigcup_{i \in I} Q_i$, and $K_p = Q^\circ \setminus K_{CZ}$. Then $CZ^\circ$ is a partition of the set $K_{CZ}$ into disjoint dyadic cubes. We denote $CZ_{key} = \{Q_s\}_{s \in \bar{I}} \subset \{Q_i\}_{i \in I}$ for the set of keystone cubes (see page \pageref{keycuberef} for the definition and basic properties of keystone cubes). Recall that $\bpt = \{x_i \}_{i \in I}$ and $\bkpt = \{x_s\}_{s \in \bar{I}}$ are the sets of centers of CZ cubes, and keystone cubes, respectively.

 For each $Q_i \in CZ^\circ$, we apply Lemma \ref{ageolm} to produce a sequence of CZ cubes $\Se(Q_i)$. Then either $\Se(Q_i)= \{Q^{i,k}\}_{k=1}^L$ and $Q^{i,L} \in CZ_{key}$ or $\Se(Q_i)= \{Q^{i,k}\}_{k \in \N}$ and $x^{i,k} \to x \in K_p$ as $k \rightarrow \infty$, where $x^{i,k}$ is the center of $Q^{i,k}$. In either case, $Q^{i,1} = Q_i$, $Q^{i,k} \lra Q^{i,k+1}$ for each $k$, and for $1 \leq \ell \leq k$,
\begin{align}
&\delta_{Q^{i,k}} \leq C\cdot c^{k-\ell} \delta_{Q^{i,\ell}}, \label{geoseq}
\end{align}
for universal constants $C \geq 1$ and $c \in (0,1)$.

Now, define a mapping $\ka: \bpt \to \bkpt \cup K_p$ by
\begin{align}
\ka(x_i) = 
&\begin{cases} 
x_s & \quad \Se(Q_i)= \{Q^{i,k}\}_{k=1}^L, \text{ and } Q^{i,L}=Q_s \in CZ_{key} \\
x & \quad \Se(Q_i)= \{Q^{i,k}\}_{k \in \N}, \text{ and } \lim_{k \to \infty}x^{i,k} = x \in K_p.
\end{cases}\label{defn:ka}
\end{align}


The next lemma contains further elementary properties of the sequences $\Se(Q_i)$, and of the mapping $\ka$, that will be used throughout the section.

\begin{lem}\label{keygeomlem}
For each $i \in I$, let $x_i$ be the center of $Q_i$, and let $x^{i,j}$ be the center of the cube $Q^{i,j}$ in the sequence $\Se(Q_i)$.
\begin{enumerate}
    \item If $\ka(x_i) \neq x_i$ then $|x_i-\ka(x_i)| \simeq \delta_{Q_i}$.
    \item $|x_i - \ka(x_i)| \lesssim |x_i - x|$ for any $x \in K_p$.
    \item If $Q^{i,j} \in \Se(Q_i)$ then $Q^{i,j} \subset C Q_i$ and  $|x^{i,j} - \ka(x_i)| \lesssim \delta_{Q^{i,j}}$.
    \item For any keystone cube $Q_s$ and $x_i \in \ka^{-1}(x_s)$, $Q_s \subset CQ_i$, and in particular  $\delta_{Q_i} \geq c\delta_{Q_s}$. Furthermore, for any fixed $\delta > 0$ and $s \in \bar{I}$,
    \[
    |\{ i\in I :  x_i \in \ka^{-1}(x_s), \delta_{Q_i} = \delta \}| \leq C.
    \]
\end{enumerate}
\end{lem}
\begin{proof}
(1): Because $\ka(x_i) \neq x_i$, the sequence $\Se(Q_i)$ has length at least $2$, and $\ka(x_i)$ does not belong to $\inte(Q_i)$ (either $\ka(x_i) \in Q_s$ for $s \neq i$, or $\ka(x_i) \in K_p$). Recall the sequence $\Se(Q_i) = \{Q^{i,k}\}_{k \geq 1}$ (finite or infinite) of CZ cubes satisfies  $Q^{i,k} \lra Q^{i,k+1}$ and $Q^{i,1} = Q_i$. Hence, $|x^{i,k} - x^{i,k+1}| \lesssim \delta_{Q^{i,k}}$, $x^{i,1}=x_i$, and either $x^{i,L} = \ka(x_i)$ for some $L < \infty$ or $x^{i,k} \rightarrow \ka(x_i)$ as $k \rightarrow \infty$. We use inequality (\ref{geoseq}) for $\ell=1$ to bound:
\begin{align*}
    |x_i - \ka(x_i)| &\leq \sum_{k \geq 1} |x^{i,k} - x^{i,k+1}| \lesssim \sum_{k \geq 1} \delta_{Q^{i,k}} \lesssim \sum_{k \geq 1} \delta_{Q_i} c^k \lesssim \delta_{Q_i}.
\end{align*}
Because $x_i = \ctr(Q_i)$ and $\ka(x_i)$ does not belong to $\inte(Q_i)$, we also have $\delta_{Q_i} \lesssim |x_i-\ka(x_i)|$.

(2): Let $x \in K_p$. We may assume $\ka(x_i) \neq x_i$. Because $Q_i \subset K_{CZ}$ and $x_i = \ctr(Q_i)$, we have $\dist(x_i,K_p) \gtrsim \delta_{Q_i}$. In combination with (1),
\[
|x-x_i| \geq \dist(x_i,K_p) \gtrsim \delta_{Q_i} \gtrsim |x_i - \ka(x_i)|.
\]

(3): Let $Q^{i,j} \in \Se(Q_i)$. From (\ref{geoseq}), $\delta_{Q^{i,j}} \lesssim \delta_{Q_i}$ and $\dist(Q^{i,j},Q_i) \lesssim \sum_{k=1}^j \delta_{Q^{i,k}} \lesssim \delta_{Q_i}$, implying there exists $C>0$ such that $Q^{i,j} \subset CQ_i$. We use (1) to bound 
\[
|x^{i,j} - \ka(x_i)| \leq |x^{i,j}-x_i| + |x_i -\ka(x_i)| \lesssim \delta_{Q_i}.
\]

(4): If $Q_s$ is a keystone cube and $x_i \in \ka^{-1}(x_s)$ then $\Se(Q_i)$ is finite and $Q_s = Q^{i,L}$ for some $L < \infty$. In light of (\ref{geoseq}), $\delta_{Q_s} = \delta_{Q^{i,L}} \lesssim \delta_{Q_i}$. From (3) applied with $j=L$, we have $Q_s \subset CQ_i$. By a volume counting argument,
\[
|\{ x_i \in \ka^{-1}(x_s): \delta_{Q_i} = \delta \}| \leq | \{ Q_i: \delta_{Q_i} = \delta \text{ and } |x_i-x_s| \simeq \delta \}|  \leq C.
\]

\end{proof}

The next result concerns a certain  operator mapping Whitney fields on $K_p \cup \bpt$ to Whitney fields on $\bpt$.

\begin{lem}
Given $\vec{P}^*=(P^*_x)_{x \in K_p} \in Wh(K_p)$ and  $\vec{P} = (P_{x_i})_{i \in I} \in Wh(\bpt)$, define $\vec{P}^{**} \in Wh(\bpt)$ by
\[
P^{**}_{x_i}: = \left\{ 
\begin{aligned}
&P_{\ka(x_i)} && \ka(x_i) \in \bkpt\\
&P^*_{\ka(x_i)} && \ka(x_i) \in K_p.
\end{aligned}
\right.
\]
If $(\vec{P}^*, \vec{P}) \in C^{m-1,1-n/p}(K_p \cup \mathfrak{B}_{CZ})$, then for $F \in L^{m,p}(Q^\circ)$, 
\begin{align}
    \sum_{i \in I} |P_{x_i}-P^{**}_{x_i}|_i^p \lesssim \|F, \vec{P} \|_{\J_*(f, \mu; K_{CZ}, CZ^\circ)}^p. \label{acapprop}
\end{align}
\end{lem}

\begin{proof}
From Lemma \ref{keygeomlem}, $|x_i-\ka(x_i)| \lesssim \delta_{Q_i}$. Therefore, from (\ref{movept}), we have
\begin{align}
    \sum_{i \in I} |P_{x_i}-P^{**}_{x_i}|_i^p = \sum_{i \in I} |P_{x_i}-P^{**}_{x_i}|_{x_i,\delta_{Q_i}}^p &\lesssim \sum_{i \in I} |P_{x_i}-P^{**}_{x_i}|_{\ka(x_i), \delta_{Q_i}}^p \nonumber \\
    &= \sum_{i \in I} \sum_{\al \in \mm} |\p^\al(P_{x_i} -P^{**}_{x_i})(\ka(x_i))|^p \delta_{Q_i}^{n-mp+|\al|p}. \label{cap1}
\end{align}

We will compare polynomials $P^{i,k}$ ($k \geq 1$) associated to the cubes $Q^{i,k} \in \Se(Q_i)$ to bound (\ref{cap1}). We define these polynomials in terms of the Whitney field $\vec{P} \in Wh(\bpt)$, Note that each $Q^{i,k} \in \Se(Q_i)$ is in $CZ^\circ$, hence, $Q^{i,k} = Q_j$ for $j \in I$ -- then we define $P^{i,k} = P_{x_{j}}$.

Evidently, $Q^{i,1} = Q_i$, so 
\[
P^{i,1} = P_{x_i}.
\]

If $\Se(Q_i) = \{Q^{i,k}\}_{1 \leq k \leq L}$ is finite, then its terminal cube is a keystone cube, $Q^{i,L} = Q_{s}$, for $\ka(x_i) = x_s \in \bkpt$. In this case, $P^{i,L} = P_{\ka(x_i)}  = P^{**}_{x_i}$. So,
\[
P^{i,L} =  P^{**}_{x_i} \mbox{ if } \Se(Q_i) = \{Q^{i,k}\}_{1 \leq k \leq L} \mbox{ is finite}.
\]

If $\Se(Q_i) = \{Q^{i,k}\}_{ k \geq 1}$ is infinite, then $x^{i,k} = \ctr(Q^{i,k})$ converges to $x = \ka(x_i) \in K_p$ as $k \rightarrow \infty$. Because $(\vec{P}^*, \vec{P}) \in C^{m-1,1-n/p}(K_p \cup \mathfrak{B}_{CZ})$, also $P^{i,k}$ converges to $P^*_x$ as $k \rightarrow \infty$. But $P^*_x = P^{**}_{x_i}$, by definition of $\vec{P}^{**}$. Thus,
\[
\lim_{k \rightarrow \infty} P^{i,k} = P^{**}_{x_i} \mbox{ if } \Se(Q_i) \mbox{ is infinite}.
\]

Evidently, by use of the telescoping series formula, and by the above properties, for $\al \in \mm$ and $i \in I$,
\begin{align*}
    |\p^\al&(P_{x_i} - P^{**}_{x_i})(\ka(x_i))| \leq \sum_{k \geq 1} |\p^\al(P^{i,k} -P^{i,k+1})(\ka(x_i))|,
\end{align*} 
where we write $\sum_{k \geq 1}$ to indicate the summation $\sum_{k=1}^{L-1}$ in the case $\Se(Q_i) = \{Q^{i,k}\}_{1 \leq k \leq L}$ is finite, and the summation $\sum_{k=1}^\infty$ in the case $\Se(Q_i)$ is infinite

Fix a universal constant $\e' \in (0, 1-n/p)$, and let $p' \in (1,\infty)$ be the conjugate exponent of $p \in (1,\infty)$, so that $1/p + 1/p'=1$. We apply H\"older's inequality to deduce
\begin{align}
    |\p^\al(P_{x_i} -P^{**}_{x_i})(\ka(x_i))| &\leq \bigg( \sum_{k \geq 1} |\p^\al(P^{i,k} -P^{i,k+1})(\ka(x_i))|^p \delta_{Q^{i,k}}^{-mp+n+|\al|p+\e'p} \bigg)^{\frac{1}{p}} \cdot \bigg( \sum_{k \geq 1}\delta_{Q^{i,k}}^{(m-n/p-|\al|-\e')p'} \bigg)^{\frac{1}{p'}} \nonumber \\
    &\lesssim \delta_{Q_i}^{m-n/p-|\al|-\e'} \big( \sum_{k \geq 1} |\p^\al(P^{i,k} -P^{i,k+1})(\ka(x_i))|^p \delta_{Q^{i,k}}^{-mp+n+|\al|p+\e'p} \big)^{\frac{1}{p}}, \label{cap2}
\end{align}
where the last inequality uses that $m-n/p-|\al|-\e' > 0$ for $|\al| \leq m-1$, so (\ref{geoseq}) implies 
\begin{align*}
\big( \sum_{k \geq 1}\delta_{Q^{i,k}}^{(m-n/p-|\al|-\e')p'} \big)^{1/p'} \lesssim \delta_{Q_i}^{m-n/p-|\al|-\e'}.    
\end{align*}
Because $Q^{i,k} \in \Se(Q_i)$, we have $|x^{i,k} - \ka(x_i)| \lesssim \delta_{Q^{i,k}}$ (see  Lemma \ref{keygeomlem}). By substituting (\ref{cap2}) into (\ref{cap1}), using the definition of the $| \cdot|_{\ka(x_i),\delta_{Q^{i,k}}}$ norm, and then applying (\ref{movept}), 
\begin{align}
     \sum_{i \in I} |P_{x_i}-P^{**}_{x_i}|_i^p  &\lesssim  \sum_{i \in I}  \delta_{Q_i}^{-\e'p} \sum_{k \geq 1} \sum_{\al \in \mm} |\p^\al(P^{i,k} -P^{i,k+1})(\ka(x_i))|^p \delta_{Q^{i,k}}^{-mp+n+|\al|p+\e'p} \nonumber \\
     &=  \sum_{i \in I}  \delta_{Q_i}^{-\e'p} \sum_{k \geq 1}  |P^{i,k} -P^{i,k+1}|_{\ka(x_i),\delta_{Q^{i,k}}} \delta_{Q^{i,k}}^{\e'p} \nonumber \\
     &\lesssim  \sum_{i \in I} \delta_{Q_i}^{-\e'p} \sum_{k \geq 1} |P^{i,k} -P^{i,k+1}|_{x^{i,k},\delta_{Q^{i,k}}}^p \cdot \delta_{Q^{i,k}}^{\e'p} \nonumber \\
     &\lesssim \sum_{\substack{j,j' \in I \\ j \lra j'}} \sum_{\substack{i \in I  \\ Q_j \in \Se(Q_i)}} \Big(\frac{\delta_{Q_j}}{\delta_{Q_i}}\Big)^{\e'p} |P_{x_{j}} - P_{x_{j'}}|_{x_j,\delta_{Q_j}}^p. \label{cap3}
\end{align}
For fixed $Q_j \in CZ^\circ$ and any dyadic length scale $\delta>0$,
\begin{align*}
| \{i \in I: Q_j \in \Se(Q_i), \; &\delta_{Q_i} = \delta \}| \leq |\{ i \in I : Q_j \subset CQ_i, \; \delta_{Q_i} =\delta \}| \leq C,
\end{align*}
where the first inequality uses property 3 in Lemma \ref{keygeomlem}, and the second inequality uses that $Q_i$ and $Q_j$ are dyadic cubes. Furthermore, by the first inequality above,
\[
| \{i \in I : Q_j \in \Se(Q_i), \; \delta_{Q_i} = \delta \} |= 0 \mbox{ if }\delta < c \delta_{Q_j}.
\]
Thus we may continue from (\ref{cap3}):
\begin{align}
     \sum_{i \in I} |P_{x_i}-P^{**}_{x_i}|_i^p  
     &\lesssim \sum_{j \lra j'} |P_{x_{j}} - P_{x_{j'}}|_j^p \sum \Big\{ \Big(\frac{\delta_{Q_j}}{\delta_{Q_i}}\Big)^{\e'p}: i \in  I, \; Q_j \in \Se(Q_i) \Big\} \nonumber \\
     &\lesssim \sum_{j \lra j'} |P_{x_{j}} - P_{x_{j'}}|_j^p \sum \Big\{ \Big(\frac{\delta_{Q_j}}{\delta}\Big)^{\e'p}: \delta \text{ a dyadic length scale, } \delta \geq c \delta_{Q_j} \Big\} \nonumber \\ 
     &\lesssim \sum_{j \lra j'} |P_{x_{j}} - P_{x_{j'}}|_j^p, \label{cap4}
\end{align}
where the final inequality follows because we are computing the sum of a geometric series, and $\e' > 0$ is a universal constant, dependent on $m$, $n$, and $p$. From the triangle inequality, the Sobolev Inequality, (\ref{sob}), and (\ref{movept}), for $F \in L^{m,p}(Q^\circ)$, 
\begin{align*}
    \sum_{j \lra j'} |P_{x_{j}} - P_{x_{j'}}|_j^p &\lesssim \sum_{j \lra j'} \left[ |P_{x_{j}} - J_{x_j}(F)|_j^p + |J_{x_{j'}}(F) - P_{x_{j'}}|_j^p + \|F\|_{L^{m,p}((1.1Q_j \cup 1.1Q_{j'})\cap Q^\circ)}^p \right] \\
    &\lesssim \sum_{j \lra j'} \left[ \|F,P_{x_{j}}\|_{\J_*(f, \mu; Q_j)}^p+ \|F\|_{L^{m,p}((1.1Q_j \cup 1.1Q_{j'})\cap Q^\circ)}^p \right]\\
    &\lesssim \|F,\vec{P}\|_{\J_*(f, \mu; K_{CZ}, CZ^\circ)}^p.
\end{align*}
Substituting this into (\ref{cap4}), we complete the proof of the lemma:
\[
\sum_{i \in I} |P_{x_i}-P^{**}_{x_i}|_i^p  \lesssim\|F,\vec{P}\|_{\J_*(f, \mu; K_{CZ}, CZ^\circ)}^p.
\]
\end{proof}

\subsection{Whitney Fields on $K_p$ and $\bpt$}

Fix the data $(f,P_0) \in \J(\mu; \delta_{Q^\circ})$.

From this data, we defined a Whitney field $\vec{R}^* \in Wh(K_p)$ in Section \ref{subsec:kpt_jet}. The defining properties of $\vec{R}^*$ are stated in Lemma \ref{kptlm}. We also defined a family of polynomials $R_{x_s}' \in \mpp$ ($s \in \bar{I}$), satisfying the conditions in Lemma \ref{alocjetlm}.  This Whitney field and these polynomials depend linearly on $(f,P_0)$, and they are coherent with $P_0$ in the sense that $\partial^\al R^*_x(x) = \partial^\al P_0(x)$ for all $x \in K_p$, $\al \in \ma$, while $\partial^\al R_{x_s}'(x_s) = \partial^\al P_0(x_s)$ for all $s \in \bar{I}$, $\al \in \ma$. Because $\ma$ is monotonic it holds that if $\partial^\al P(x) = 0$ for all $\al \in \ma$, then $\partial^\al P \equiv 0$ on $\R^n$ for all $\al \in \ma$, for any $P \in \mpp$. Therefore,
\begin{equation}\label{ccc:eqn}
\partial^\al (R_{x_s}' - P_0) \equiv 0, \;\; \partial^\al (R_x^* - P_0) \equiv 0 \mbox{ for all } x \in K_p, s \in \bar{I}, \al \in \ma.
\end{equation}

We define the Whitney field $\vec{R} = (R_{x})_{x \in \bpt} \in Wh(\bpt)$ ($\bpt = \{x_i\}_{i \in I}$) by
\begin{align}
    R_{x_i} := \begin{cases} R_{x_s}' & \mbox{if } \ka(x_i)=x_s \in \bkpt \\
    R_x^* &\mbox{if }  \ka(x_i) = x \in K_p,
    \end{cases} \label{vecr}
\end{align}
where $\ka : \bpt \rightarrow \bkpt \cup K_p$ is the mapping defined in \eqref{defn:ka}.

Evidently, the Whitney field $(\vec{R}^*,\vec{R})\in Wh(K_p \cup \mathfrak{B}_{CZ})$ depends linearly on $(f,P_0)$.  From \eqref{ccc:eqn}, it holds that $\partial^\alpha (R_{x_i} - P_0) \equiv 0$ for $\al \in \ma$ and $x_i \in \bpt$ -- thus, $\vec{R}$ is coherent with $P_0$. As mentioned above, $\vec{R}^*$ is coherent with $P_0$. So, $(\vec{R}^*,\vec{R})$ is coherent with $P_0$.

\begin{lem}
The Whitney field $(\vec{R}^*,\vec{R}) \in Wh(K_p \cup \bpt)$ satisfies 
\begin{align*}
    \|(\vec{R}^*,\vec{R})\|_{C^{m-1,1-n/p}(K_p \cup \mathfrak{B}_{CZ})} \leq C \|f,P_0\|_{\J(\mu;\delta_{Q^\circ})}.
\end{align*} \label{vecholder}
\end{lem}

\begin{proof}
Let $\eta > 0$. From Corollary \ref{acohkptcor}, there exists $H \in L^{m,p}(Q^\circ)$ and $\vec{P} \in Wh(\bpt)$ satisfying $J_xH = R_x^*$ for all $x \in K_p$, $\vec{P}$ is coherent with $P_0$, and 
\begin{align}
    \|H,\vec{P}\|_{\J_*(f,\mu;Q^\circ,CZ^\circ)} \leq \|f,\vec{P}, \vec{R}^*\|_{\J_*(\mu;Q^\circ,CZ^\circ;K_p)} +\eta/2 \lesssim \|f,P_0\|_{\J(\mu; \delta_{Q^\circ})}+\eta. \label{hchoice}
\end{align}
Consequently, since $\vec{R}^* \in Wh(K_p)$ is coherent with $P_0$, also $H$ is $K_p$-coherent with $P_0$. From Proposition \ref{spinprop} and \eqref{hchoice}, we have
\begin{align}
    \|(\vec{R}^*,\vec{P})\|_{C^{m-1,1-n/p}(K_p \cup \mathfrak{B}_{CZ})} \lesssim \|f,P_0\|_{\J(\mu; \delta_{Q^\circ})}+\eta. \label{holder}
\end{align}
To prove Lemma \ref{vecholder}, we must prove the following inequalities:
\begin{equation}\label{vecholder2}
\left\{
    \begin{aligned}
    &|R_x^* - R_y^*|_{x, |x-y|} \leq C \| f,P_0 \|_{\J(\mu; \delta_{Q^\circ})} && x,y \in K_p \mbox{ distinct}; \\
    &|R_x^* - R_{x_i}|_{x_i, |x-x_i|} \leq C \| f,P_0 \|_{\J(\mu; \delta_{Q^\circ})} && x \in K_p, \; i \in I; \\
    &|R_{x_i} - R_{x_j}|_{x_j, |x_j-x_i|} \leq C \| f,P_0 \|_{\J(\mu; \delta_{Q^\circ})} && i,j \in I \mbox{ distinct}.
    \end{aligned}
\right.
\end{equation}
We proceed with the proof of \eqref{vecholder2} in cases, below.

\underline{Case 1:} For distinct $x,y \in K_p$, we apply (\ref{holder}):
\begin{align*}
   |R_x^* - R_y^*|_{x,|y-x|} \lesssim \|f,P_0\|_{\J(\mu; \delta_{Q^\circ})} +\eta.
\end{align*}
By letting $\eta \to 0$, we deduce the first line of \eqref{vecholder2}.

\underline{Case 2:} Suppose $x \in K_p$ and $i \in I$; we will prove the second line of \eqref{vecholder2}: First, suppose that $\ka(x_i) = y \in K_p$. Then by definition of $\vec{R}$ (see \eqref{vecr}), $R_{x_i} = R_y^*$. By part 2 of Lemma \ref{keygeomlem}, $|x - y| \leq |x-x_i| + |x_i - y| \lesssim |x-x_i|$. So, by (\ref{movept}), and the analysis of Case 1,
\begin{align}
|R_x^* - R_{x_i}|_{x,|x_i-x|} = |R_x^* - R_y^*|_{x,|x_i-x|} &\lesssim  |R_x^* - R_y^*|_{x,|x-y|} \lesssim \|f,P_0\|_{\J(\mu, \delta_{Q^\circ})}.\label{holder4}
\end{align}
This estimate proves the second line of \eqref{vecholder2} for $\ka(x_i) = y \in K_p$.

Next suppose that $\ka(x_i) = x_s \in \bkpt$. By definition of $\vec{R}$, $R_{x_i} = R'_{x_s}$, with $x_s \in \bpt$  the center of the keystone cube $Q_{s}$. We apply (\ref{holder}) to deduce
\begin{align}
    |R_x^* - R_{x_i}|_{x_i,|x_i-x|} &\leq  |R^*_x - P_{x_i}|_{x_i,|x_i-x|} + |R_{x_i} - P_{x_i}|_{x_i,|x_i-x|}\nonumber \\
    &\lesssim \|f,P_0\|_{\J(\mu, \delta_{Q^\circ})} + |R_{x_i} - P_{x_i}|_{x_i,|x_i-x|}+\eta \label{holder2}.
\end{align}
By part 2 of Lemma \ref{keygeomlem}, since $x \in K_p$, $|x_i - x_s| = |x_i - \ka(x_i)| \lesssim |x-x_i|$, and by part 4 of Lemma \ref{keygeomlem}, $\delta_{Q_{s}} \lesssim \delta_{Q_i} \leq |x-x_i|$, so we can apply (\ref{movept}) and (\ref{holder}) to deduce
\begin{align}
    |R_{x_i} - P_{x_i}|_{x_i,|x_i-x|} &= |R'_{x_s} - P_{x_i}|_{x_i,|x_i-x|} \nonumber \\
    &\lesssim |R'_{x_s} - P_{x_s}|_{x_{s},|x_i-x|} + |P_{x_s} - P_{x_i}|_{x_{s},|x_i-x_{s}|} \nonumber \\
    &\lesssim |R'_{x_s} - P_{x_s}|_{x_s,\delta_{Q_{s}}} +\|f,P_0\|_{\J(\mu, \delta_{Q^\circ})}+\eta. \label{holder1} 
\end{align}
 Because $R'_{x_s}, P_{x_s} \in \mpp$ are each coherent with $P_0$, and $\ma$ is monotonic, we have $\partial^\al (R'_{x_s}-P_{x_s}) \equiv 0$ for all $\al \in \ma$. We apply (\ref{movept}), (\ref{locest2}), and the triangle inequality to bound
\begin{align}
    |R'_{x_s} - P_{x_s}|_{x_s,\delta_{Q_s}} &\lesssim \|0,R'_{x_s} - P_{x_s}\|_{\J(\mu|_{9Q_s}; \delta_{Q_s})} \nonumber \\
    &\lesssim \|f,R'_{x_s}\|_{\J(\mu|_{9Q_s}; \delta_{Q_s})} + \|f,P_{x_s}\|_{\J(\mu|_{9Q_s)}; \delta_{Q_s})} \nonumber \\
    &\lesssim \|f,P_{x_s}\|_{\J(\mu|_{9Q_s}; \delta_{Q_s})} \label{holder3},
\end{align}
where the last inequality follows by the defining properties of $R'_{x_s}$ (see Lemma \ref{alocjetlm}), since $P_{x_s}$ is coherent with $P_0$. We apply (\ref{j8}) (for the $C$-non-degenerate rectangular box $9Q_s \cap Q^\circ$) and (\ref{cub2}) (for $R_1 = Q_s$, $R_2 = 9Q_s \cap Q^\circ$), and finally (\ref{hchoice}), to estimate
\begin{align}
\|f,P_{x_s}\|_{\J(\mu|_{9Q_s}, \delta_{Q_s})}^p &\overset{(\ref{j8})}{\lesssim}  \|f, P_{x_s}\|_{\J_*(\mu|_{9Q_s}; 9Q_s \cap Q^{\circ})}^p \nonumber \\
&\; \leq \|H, P_{x_s}\|_{\J_*(f, \mu|_{9Q_s}; 9Q_s \cap Q^{\circ})}^p \nonumber\\
&\; =  \|H\|_{L^{m,p}(9Q_s \cap Q^{\circ})}^p + \int_{9Q_s} |H-f|^p d\mu + \|H - P_{x_s}\|_{L^p(9Q_s \cap Q^{\circ})}^p / \delta_{Q_s}^{mp} \nonumber \\
&\overset{\eqref{cub2}}{\lesssim} \|H\|_{L^{m,p}(Q^\circ)}^p + \int_{9Q_s} |H-f|^p d\mu + \| H - P_{x_s} \|_{L^p(Q_s)}^p /\delta_{Q_s}^{mp} \nonumber \\
& \; \leq \|H,\vec{P}\|_{\J_*(f,\mu;Q^\circ,CZ^\circ)}^p \nonumber \\
&\overset{(\ref{hchoice})}{\lesssim} (\|f,P_0\|_{\J(\mu, \delta_{Q^\circ})}+\eta)^p.\label{hold4}
\end{align}

Combining \eqref{holder2}--\eqref{hold4}, we have for $\kappa(x_i) =x_s \in \bkpt$,
\begin{align*}
    |R_x^* - R_{x_i}|_{x_i,|x_i-x|} \lesssim \|f,P_0\|_{\J(\mu, \delta_{Q^\circ})}+3\eta.
\end{align*}
By letting $\eta \to 0$, we complete the proof of the second line of \eqref{vecholder2}. 

\underline{Case 3:} Let $i,j \in I$ be distinct. Then $x_i \in Q_i \in CZ^\circ$ and $x_j \in Q_j \in CZ^\circ$. We aim to prove the third line of  \eqref{vecholder2}: Suppose $\ka(x_i) = x \in K_p$, so that $R_{x_i}=R_x^*$. By Lemma \ref{keygeomlem}, $|x-x_i| \simeq \delta_{Q_i} \lesssim |x_i -x_j|$ and $|x-x_j|\leq |x-x_i| + |x_i - x_j| \lesssim |x_i-x_j|$, so from (\ref{movept}),
\begin{align*}
    |R_{x_i} -R_{x_j}|_{x_j,|x_j-x_i|} &\lesssim |R_x^* - R_{x_j}|_{x_j, |x-x_j|}\lesssim \|f,P_0\|_{\J(\mu, \delta_{Q^\circ})},
\end{align*}
where the last inequality follows by the analysis of Case 2. This proves the third line of \eqref{vecholder2} if $\ka(x_i) \in K_p$.

If $\ka(x_j) \in K_p$, then by \eqref{movept}, $|R_{x_i} -R_{x_j}|_{x_j,|x_j-x_i|} \lesssim |R_{x_i} -R_{x_j}|_{x_i,|x_j-x_i|}$, and we repeat the preceding analysis to prove the third line of \eqref{vecholder2} in this case.

Now suppose $\ka(x_i) = x_s, \ka(x_j) = x_t \in \bkpt$. Then $R_{x_i} = R'_{x_s}$ and $R_{x_j} = R'_{x_t}$. Because $\delta_{Q_s} \lesssim \delta_{Q_i} \simeq |x_i-x_s|\lesssim |x_i-x_j|$ (see Lemma \ref{keygeomlem}), we can apply (\ref{movept}) and (\ref{holder}) to deduce
\begin{align}
    |R_{x_i} - P_{x_i}|_{x_i,|x_i-x_j|} &= |R'_{x_s} - P_{x_i}|_{x_i,|x_i-x_j|} \nonumber \\
    &\lesssim |R'_{x_s} - P_{x_s}|_{x_s,|x_i-x_j|} + |P_{x_s} - P_{x_i}|_{x_s,|x_i-x_s|} \nonumber \\
    &\lesssim |R'_{x_s} - P_{x_s}|_{x_s,\delta_{Q_s}} +\|f,P_0\|_{\J(\mu, \delta_{Q^\circ})} +\eta. \label{aholder1} 
\end{align}
From \eqref{holder3} and \eqref{hold4}, we have
\begin{align}\label{aholder2}
|R'_{x_s} - P_{x_s}|_{x_s,\delta_{Q_s}} \lesssim \|f,P_0\|_{\J(\mu, \delta_{Q^\circ})} +\eta.
\end{align}

Combining \eqref{aholder1}, \eqref{aholder2}, we learn that
\[
|R_{x_i} - P_{x_i}|_{x_i,|x_i-x_j|} \lesssim \|f,P_0\|_{\J(\mu, \delta_{Q^\circ})} + 2\eta.
\]
Similarly,
\[
|R_{x_j} - P_{x_j}|_{x_j,|x_i-x_j|} \lesssim \|f,P_0\|_{\J(\mu, \delta_{Q^\circ})}+2 \eta.
\]
We apply the triangle inequality and \eqref{movept}, and then the preceding estimates and (\ref{holder}) to deduce
\begin{align*}
    |R_{x_i} - R_{x_j}|_{x_i, |x_i - x_j|} &\lesssim |R_{x_i} - P_{x_i}|_{x_i, |x_i - x_j|} + |P_{x_i} - P_{x_{j}}|_{x_i, |x_i - x_j|} + |P_{x_{j}} - R_{x_j}|_{x_j, |x_i - x_j|} \\
    &\lesssim \|f,P_0\|_{\J(\mu, \delta_{Q^\circ})}+5 \eta.
\end{align*}
By letting $\eta \to 0$, we complete the proof of the third line of \eqref{vecholder2} in Case 3, concluding the proof of the lemma.

\end{proof}

\begin{prop}
For $(f,P_0) \in \J(\mu; \delta_{Q^\circ})$, we have $(f,\vec{R},\vec{R}^*) \in \J_*(\mu; Q^{\circ}, CZ^\circ;K_p)$, and
\begin{align*}
\|f,\vec{R}, \vec{R}^* \|_{\J_*(\mu; Q^{\circ}, CZ^\circ;K_p)} \leq C \|f, P_0\|_{\J(\mu;\delta_{Q^{\circ}})}.
\end{align*} 
Furthermore, $(\vec{R},\vec{R}^*)$ depends linearly on $(f,P_0)$.
\label{aoptfieldlm}
\end{prop}

\begin{proof}
From Lemma \ref{kptlm}, Lemma \ref{alocjetlm}, and the definition of $\vec{R}$ in (\ref{vecr}), $(\vec{R},\vec{R}^*)$ depends linearly on $(f,P_0)$.

Because of Corollary \ref{acohkptcor}, it suffices to show that for any $\vec{P} \in Wh(\bpt)$, satisfying $\vec{P}$ is coherent with $P_0$  and $(f,\vec{P},\vec{R}^*) \in \J_*(\mu;Q^\circ,CZ^\circ;K_p)$, we have 
\begin{align}\label{goal1}
    \|f,\vec{R},\vec{R}^*\|_{\J_*(\mu; Q^{\circ}, CZ^\circ;K_p)}  \leq C \|f,\vec{P},\vec{R}^*\|_{\J_*(\mu; Q^{\circ}, CZ^\circ;K_p)}.
\end{align}

Let $\vec{P}= \{P_{x_i}\}_{i \in I} \in Wh(\bpt)$ satisfy $\vec{P}$ is coherent with $P_0$ and $(f,\vec{P},\vec{R}^*) \in \J_*(\mu;Q^\circ,CZ^\circ;K_p)$. Observe that the Whitney field $(\vec{R}^*, \vec{P}) \in Wh(K_p \cup \mathfrak{B}_{CZ})$ is in the class $C^{m-1,1-n/p}(K_p \cup \bpt)$, thanks to Proposition \ref{spinprop}. Define $\vec{P}^{**} \in Wh(\bpt)$ by 
\begin{equation}
P^{**}_{x_i}: = \left\{ 
\begin{aligned}
&P_{\ka(x_i)} && \mbox{ if } \ka(x_i) \in \bkpt\\
&R^*_{\ka(x_i)} && \mbox{ if } \ka(x_i) \in K_p.
\end{aligned}
\right.\label{vecP**}
\end{equation}

We will demonstrate that $(f, \vec{P}^{**}, \vec{R}^*) \in \J_*(\mu;Q^\circ,CZ^\circ;K_p)$. To see this, suppose that $F \in L^{m,p}(Q^\circ)$ is arbitrary, satisfying $J_x(F)=R^*_x$ for all $x \in K_p$. Recalling \eqref{new_Jfunc1:eqn}, and applying (\ref{acapprop}) to the pair of Whitney fields $(\vec{R}^*,\vec{P})$, we have
\begin{align*}
   \|F,\vec{P}^{**}\|_{\J_*(f,\mu; Q^{\circ}, CZ^\circ)}^p
    &= \|F\|_{L^{m,p}(Q^\circ)}^p + \int_{Q^\circ} |F-f|^p d\mu + \sum_{i \in I} \|F-P^{**}_{x_i}\|_{L^p(Q_i)}^p/\delta_{Q_i}^{mp}\\
    &\lesssim \|F\|_{L^{m,p}(Q^\circ)}^p + \int_{Q^\circ} |F-f|^p d\mu + \sum_{i \in I} \left\{ \|F-P_{x_i}\|_{L^p(Q_i)}^p/\delta_{Q_i}^{mp}+ |P_{x_i}-P^{**}_{x_i}|_i^p \right\}\\
    &= \|F,\vec{P}\|_{\J_*(f,\mu; Q^{\circ}, CZ^\circ)}^p + \sum_{i \in I}  |P_{x_i}-P^{**}_{x_i}|_i^p \\ 
    & \overset{(\ref{acapprop})}{\lesssim} \|F,\vec{P}\|_{\J_*(f,\mu; Q^{\circ}, CZ^\circ)}^p.
\end{align*}
Here, the first $\lesssim$ uses the triangle inequality as well as the fact that $\| P \|_{L^p(Q_i)} \delta_{Q_i}^{-m} \simeq |P|_i$ for $P \in \mpp$. By taking the infimum in the above inequality over $F \in L^{m,p}(Q^{\circ})$ satisfying $J_x(F) = R_x^*$ for all $x \in K_p$,
\begin{align}
    \|f,\vec{P}^{**}, \vec{R}^*\|_{\J_*(\mu; Q^{\circ}, CZ^\circ;K_p)}^p \lesssim \|f,\vec{P},\vec{R}^*\|_{\J_*(\mu; Q^{\circ}, CZ^\circ;K_p)}^p < \infty. \label{apbarp}
\end{align}
Thus, $(f, \vec{P}^{**}, \vec{R}^*) \in \J_*(\mu;Q^\circ,CZ^\circ;K_p)$, as claimed.

From Lemma \ref{vecholder}, we have $(\vec{R}^*, \vec{R}) \in C^{m-1,1-n/p}(K_p \cup \mathfrak{B}_{CZ})$. Therefore, by Lemma \ref{adecomplm}, to demonstrate that $(f, \vec{R}, \vec{R}^*) \in \J_*(\mu;Q^\circ,CZ^\circ;K_p)$ it suffices to prove that
\begin{align}
\mathcal{S}_2(f,\vec{R},\vec{R}^*) := \sum_{i \in I} \|f,R_{x_i}  \|_{\J(\mu|_{1.1Q_i}; \delta_{Q_i})}^p  & + \sum_{i \lra i'}  |R_{x_i} - R_{x_{i'}}|_i^p \nonumber  \\
& + \|\vec{R}^*\|_{L^{m,p}(K_p)}^p  + \int_{K_p} |R^*_x(x) - f(x)|^p d \mu < \infty. \label{goal2}
\end{align}

By the triangle inequality,  (\ref{locest1}), (\ref{movept}),  (\ref{adecomp}),  \eqref{func_decomp1}, and (\ref{apbarp}),
\begin{align}
\sum_{i \in I} & \|f,R_{x_i}\|_{\J(\mu|_{1.1Q_i}; \delta_{Q_i})}^p + \sum_{i \lra i'} |R_{x_i} - R_{x_{i'}}|_i^p  \nonumber \\
&\lesssim \sum_{i \in I} \Big\{\|f,P^{**}_{x_i}\|_{\J(\mu|_{1.1Q_i}; \delta_{Q_i})}^p  + \|0,P^{**}_{x_i} - R_{x_i}\|_{\J(\mu|_{1.1Q_i}; \delta_{Q_i})}^p \Big\}  \nonumber  \\
& \qquad \qquad + \sum_{i \lra i'} \Big\{ |R_{x_i} - P^{**}_{x_i}|_i^p + |P^{**}_{x_i} - P^{**}_{x_{i'}}|_i^p + |R_{x_{i'}} - P^{**}_{x_{i'}}|_i^p \Big\} \nonumber \\
&\overset{(\ref{locest1}), (\ref{movept})}{\lesssim} \sum_{i \in I} \Big\{\|f,P^{**}_{x_i}\|_{\J(\mu|_{1.1Q_i}; \delta_{Q_i})}^p  + |P^{**}_{x_i} - R_{x_i}|_i^p \Big\} \nonumber \\
& \qquad \qquad + \sum_{i \lra i'} \Big\{ |R_{x_i} - P^{**}_{x_i}|_i^p + |P^{**}_{x_i} - P^{**}_{x_{i'}}|_i^p + |R_{x_{i'}} - P^{**}_{x_{i'}}|_{i'}^p \Big\} \nonumber \\
&\quad \; \; \lesssim  \sum_{i \in I} \|f,P^{**}_{x_i}\|_{\J(\mu|_{1.1Q_i}; \delta_{Q_i})}^p  + \sum_{i \lra i'}  |P^{**}_{x_i} - P^{**}_{x_{i'}}|_i^p +  \sum_{i \in I}  |P^{**}_{x_i} - R_{x_i}|_i^p \nonumber \\
&\; \overset{(\ref{adecomp}), \eqref{func_decomp1}}{\lesssim} \|f,\vec{P}^{**},\vec{R}^*\|_{\J_*(\mu; Q^{\circ}, CZ^\circ;K_p)}^p + \sum_{i \in I} |R_{x_i} - P^{**}_{x_i}|_i^p \nonumber\\
&\quad \overset{(\ref{apbarp})}{\lesssim} \|f,\vec{P},\vec{R}^*\|_{\J_*(\mu; Q^{\circ}, CZ^\circ;K_p)}^p + \sum_{i \in I} |R_{x_i} - P^{**}_{x_i}|_i^p. \label{aoptfield}
\end{align}

We wish to bound the sum $\sum_{i \in I} |R_{x_i} - P^{**}_{x_i}|_i^p$. If $\ka(x_i)  \in K_p$, then $R_{x_i} = R_{\kappa(x_i)}^* = P^{**}_{x_i}$, by \eqref{vecr} and \eqref{vecP**}, and so the summand vanishes. Else, suppose $\ka(x_i) \in \bkpt$. Then $R_{x_i} = R'_{\ka(x_i)}$ and $P^{**}_{x_i} = P_{\ka(x_i)}$, by \eqref{vecr} and \eqref{vecP**}. Hence,
\begin{align}
    \sum_{i \in I} |R_{x_i} - P^{**}_{x_i}|_i^p = \sum_{i \in I: \ka(x_i) \in \bkpt} |R_{\kappa(x_i)}' - P_{\ka(x_i)}|_i^p. \label{aowf1}
\end{align}
For $\ka(x_i) \in \bkpt$, we have $|x_i - \ka(x_i)| \leq C\delta_{Q_i}$ (see Lemma \ref{keygeomlem}), so, by recalling the definition of the $|\cdot|_i$ polynomial norm, and by (\ref{movept}), 
\begin{align*}
\sum_{i \in I: \ka(x_i) \in \bkpt} |R_{\ka(x_i)}' - P_{\ka(x_i)}|_{x_i, \delta_{Q_i}}^p  & \leq \sum_{i \in I: \ka(x_i) \in \bkpt} |R_{\ka(x_i)}' - P_{\ka(x_i)}|_{\ka(x_i), \delta_{Q_i}}^p \\
&=  \sum_{s \in \bar{I}} \sum_{x_i \in \ka^{-1}(x_s)} |P_{x_s} - R_{x_s}'|_{x_s, \delta_{Q_i}}^p  \\
&= \sum_{s \in \bar{I}} \sum_{x_i \in \ka^{-1}(x_s)} \sum_{|\beta| \leq m-1} |\p^\beta(P_{x_s} - R_{x_s}')(x_s)|^p \cdot \delta_{Q_i}^{|\beta|p + n -mp}. 
\end{align*}
By Lemma \ref{keygeomlem}, for any $s \in \bar{I}$ and $x_i \in \ka^{-1}(x_s)$, we have $\delta_{Q_i} \geq c\delta_{Q_s}$; furthermore, for any dyadic lengthscale $\delta > 0$, and fixed $s$,
\begin{align*}
    | \{ i \in I : x_i \in \ka^{-1}(x_s), \delta_{Q_i} = \delta \}| \leq C.
\end{align*}
Thus, in combination with the inequality, $|\beta|p + n -mp <0$ for all $\beta \in \mm$, we have
\begin{align*}
 \sum_{x_i \in \ka^{-1}(x_s)} \delta_{Q_i}^{|\beta|p + n -mp} \lesssim   \delta_{Q_s}^{|\beta|p + n -mp}.
\end{align*}
So, using the previous three equation lines, we reduce (\ref{aowf1}),
\begin{align}
\sum_{i \in I} |R_{x_i} - P^{**}_{x_i}|_i^p &\lesssim  \sum_{s \in \bar{I}} \sum_{|\beta| \leq m-1} |\p^\beta(P_{x_s} - R_{x_s}')(x_s)|^p \delta_{Q_s}^{|\beta|p + n -mp}  =\sum_{s \in \bar{I}} |P_{x_s} - R_{x_s}'|_s^p. \label{aowf5}
\end{align}
Because $\vec{P}$ and $\vec{R}$ in $Wh(\mathfrak{B}_{CZ})$ are both coherent with $P_0$ , we can manipulate the right-hand side of the inequality (\ref{aowf5}) using (\ref{locest2}), (\ref{sublinearj}), (\ref{alocjet}), and (\ref{j8}), to obtain
\begin{align}
\sum_{i \in I} |R_{x_i} - P^{**}_{x_i}|_i^p &\overset{(\ref{locest2})}{\lesssim}  \sum_{s \in \bar{I}} \|0, P_{x_s} - R_{x_s}'\|_{\J(\mu|_{9Q_s}, \delta_{Q_s})}^p \nonumber \\
&\overset{(\ref{sublinearj})}{\lesssim} \sum_{s \in \bar{I}} \|f, P_{x_s}\|_{\J(\mu|_{9Q_s}, \delta_{Q_s})}^p +  \|f, R_{x_s}'\|_{\J(\mu|_{9Q_s}, \delta_{Q_s})}^p \nonumber \\
&\overset{(\ref{alocjet})}{\lesssim} \sum_{s \in \bar{I}} \|f, P_{x_s}\|_{\J(\mu|_{9Q_s}, \delta_{Q_s})}^p \nonumber \\
&\overset{(\ref{j8})}{\lesssim} \sum_{s \in \bar{I}}  \|f, P_{x_s}\|_{\J_*(\mu|_{9Q_s}; 10Q_s \cap Q^{\circ})}^p. \label{aowf2}
\end{align}
Consider an arbitrary $H \in L^{m,p}(Q^{\circ})$ satisfying $J_xH = R_x^*$ for all $x \in K_p$. Recall from Lemma \ref{10q}, for any keystone cube $Q_s$, $\left| \{ s' \in \bar{I}: 10Q_s \cap 10Q_{s'} \neq \emptyset \} \right| \leq C$. So from (\ref{aowf2}), we have
\begin{align}
\sum_{i \in I} |R_{x_i} - P^{**}_{x_i}|_i^p
&\lesssim \sum_{s \in \bar{I}}  \|H, P_{x_s}\|_{\J_*(f,\mu|_{10Q_s};10Q_s \cap Q^{\circ})}^p \nonumber \\
&\lesssim   \|H\|_{L^{m,p}(Q^{\circ})}^p + \int |H-f|^p d\mu + \sum_{s \in \bar{I}} \|H - P_{x_s}\|_{L^p(10Q_s \cap Q^{\circ})}^p / \delta_{Q_s}^{mp}. \label{aowf3}
\end{align}
Now apply (\ref{cub2}) (with $R_1 = Q_s$ and $R_2 = 10 Q_s \cap Q^\circ$), to obtain
\[
    \|H - P_{x_s}\|_{L^p(10Q_s \cap Q^{\circ})}^p / \delta_{Q_s}^{mp} \lesssim \|H - P_{x_s}\|_{L^p(Q_s)}^p / \delta_{Q_s}^{mp} + \|H\|_{L^{m,p}(10 Q_s \cap Q^\circ)}^p.
\]
Then summing on $s$, and using the bounded overlap condition on $\{10 Q_s : s \in \bar{I}\}$ again, we obtain
\begin{align*}
\sum_{s \in \bar{I}} \|H - P_{x_s}\|_{L^p(10Q_s \cap Q^{\circ})}^p / \delta_{Q_s}^{mp} \lesssim \sum_{s \in \bar{I}}  \|H - P_{x_s}\|_{L^p(Q_s)}^p / \delta_{Q_s}^{mp} + \|H\|_{L^{m,p}(Q^\circ)}^p
\end{align*}
Using this in (\ref{aowf3}),
\begin{align*}
\sum_{i \in I} |R_{x_i} - P^{**}_{x_i}|_i^p &\lesssim  \|H\|_{L^{m,p}(Q^{\circ})}^p + \int |H-f|^p d\mu + \sum_{i \in I} \|H - P_{x_i}\|_{L^p(Q_i)}^p / \delta_{Q_i}^{mp} \\
&= \|H, \vec{P} \|_{\J_*(f,\mu;Q^{\circ}, CZ^\circ)}^p.
\end{align*} 
Taking the infimum with respect to $H \in L^{m,p}(Q^{\circ})$ satisfying $J_xH = R_x^*$ for all $x \in K_p$, we have
\begin{equation}\label{int1}
\sum_{i \in I} |R_{x_i} - P^{**}_{x_i}|_i^p \lesssim  \|f, \vec{P},\vec{R}^* \|_{\J_*(\mu; Q^{\circ}, CZ^\circ;K_p)}^p.
\end{equation}
Furthermore, for any $H \in L^{m,p}(Q^{\circ})$ satisfying $J_xH = R_x^*$ for all $x \in K_p$, we have
\begin{align*}
\| \vec{R}^* \|_{L^{m,p}(K_p)}^p & + \int_{K_p} | R^*_x(x) - f(x)|^p d \mu \lesssim \| H \|_{L^{m,p}(Q^\circ)}^p + \int_{K_p} |H - f|^p d \mu \leq \| H, \vec{P} \|_{\J_*(f, \mu; Q^\circ, CZ^\circ)}^p.
\end{align*}
For the first inequality, we have used Lemma \ref{restwh}. The second inequality follows by definition of the $\J_*(f, \mu; Q^\circ, CZ^\circ)$ functional. Now, taking the infimum over all $H$ in the previous inequality,
\begin{equation}\label{th1}
\| \vec{R}^* \|_{L^{m,p}(K_p)}^p + \int_{K_p} | R^*_x(x) - f(x)|^p d \mu \lesssim  \|f, \vec{P},\vec{R}^* \|_{\J_*(\mu; Q^{\circ}, CZ^\circ;K_p)}^p.
\end{equation}

From (\ref{aoptfield}), \eqref{int1}, \eqref{th1}, we conclude that $\mathcal{S}_2(f,\vec{R},\vec{R}^*)$ defined in \eqref{goal2} satisfies
\begin{align*}
& \mathcal{S}_2(f,\vec{R},\vec{R}^*) \lesssim \|f, \vec{P},\vec{R}^* \|_{\J_*(\mu; Q^{\circ}, CZ^\circ;K_p)}^p < \infty.
\end{align*}
Thus, we have proven \eqref{goal2}, as desired.

We conclude by Lemma \ref{adecomplm} that $(f,\vec{R},\vec{R}^*) \in \J_*(\mu; Q^{\circ}, CZ^\circ;K_p)$ and
\begin{align*}
\|f, \vec{R},\vec{R}^* \|_{\J_*(\mu; Q^{\circ}, CZ^\circ;K_p)}^p &\simeq \mathcal{S}_1(f,\vec{R},\vec{R}^*) \leq \mathcal{S}_2(f,\vec{R},\vec{R}^*) \lesssim \|f, \vec{P},\vec{R}^* \|_{\J_*(\mu; Q^{\circ}, CZ^\circ;K_p)}^p.
\end{align*}
This completes the proof of \eqref{goal1}. This completes the proof of the proposition.
\end{proof}

\section{Proof of the Main Lemma for $\mathcal{A}$.}
\label{sec:proof_ml}

The next lemmas tell us that the Main Lemma for $\mathcal{A}$ is true. That is, we shall establish the Extension Theorem for $(\mu,\delta)$. Recall, we have rescaled and translated $\mu$ and $\delta$, so that $\diam(\supp(\mu)) < \delta = 1/10$ (see (\ref{aqsubset})).

\begin{lem}
There exist a linear map $T: \J(\mu;1/10) \to L^{m,p}(\R^n)$, a map $M: \J(\mu;1/10) \to \R_+$, $K \subset \Cl(\supp(\mu))$, a linear map $\vec{S}: \J(\mu) \to Wh(K)$, and a countable collection of linear maps $\{\lambda_\ell\}_{\ell \in \mathbb{N}}$, $\lambda_\ell: \J(\mu;1/10) \to L^p(d\mu)$, that satisfy for each $(f,P_0) \in \J(\mu;1/10)$, (\ref{aet1}), (\ref{aet2}), and (\ref{aet3}) hold with $\delta = 1/10$. Furthermore,if $\ma \neq \emptyset$ then $K = \emptyset$, and so the map $M$ satisfies (\ref{aextthm3plus}). \label{localpflm}
\end{lem}

\begin{proof}

We recall from Proposition \ref{aoptfieldlm}, there exists $(\vec{R}, \vec{R}^*)$ depending linearly on $(f,P_0)$, so that $(f,\vec{R},\vec{R}^*) \in \J_*(\mu; Q^\circ, CZ^\circ; K_p)$, and
\begin{align}\label{f4}
\|f,\vec{R}, \vec{R}^* \|_{\J_*(\mu; Q^{\circ}, CZ^\circ;K_p)} \leq C \|f, P_0\|_{\J(\mu;\delta_{Q^{\circ}})}.
\end{align} 
Because $(f, \vec{R}, \vec{R}^*) \in \J_*(\mu;Q^\circ,CZ^\circ;K_p)$, we can apply Lemma \ref{adecomplm} to produce $T(f, \vec{R}, \vec{R}^*) \in L^{m,p}(Q^\circ)$ satisfying $J_xT(f,\vec{R}, \vec{R}^*) = R_x$ for $x \in K_p$, and
\begin{align}
\|T(f,\vec{R}, \vec{R}^*), \vec{R}\|_{\J_*(f,\mu; Q^{\circ}, CZ^\circ)} &\leq C \|f,\vec{R}, \vec{R}^*\|_{\J_*(\mu; Q^{\circ}, CZ^\circ;K_p)}. \label{f5}
\end{align}

Define $T:\J(\mu;\delta) \to L^{m,p}(\R^n)$ as $T(f,P_0): = \theta \cdot T(f, \vec{R}, \vec{R}^*) + (1-\theta)P_0$, where $\theta$ is a smooth cutoff function satisfying $\supp(\theta) \subset Q^\circ$, $\theta |_{0.99Q^\circ}=1$, $|\theta(x)| \leq 1$, and $|\p^\al \theta(x)| \leq C$ for $\al \in \mm$. Then we apply (\ref{j1}) and the definition of the $\J_*(f,\mu; Q^\circ)$ and $\J_*(f,\mu; Q^{\circ}, CZ^\circ)$ functionals to estimate 
\begin{align}
\|T(f,P_0), P_0\|_{\J(f, \mu; \delta_{Q^\circ})}  &\lesssim
\|T(f, \vec{R}, \vec{R}^*),P_0\|_{\J_*(f,\mu; Q^\circ)} \nonumber \\
&\lesssim
\|T(f, \vec{R}, \vec{R}^*),\vec{R}\|_{\J_*(f,\mu; Q^{\circ}, CZ^\circ)} + \|T(f, \vec{R}, \vec{R}^*)-P_0 \|_{L^p(Q^\circ)}. \label{f6} 
\end{align}
Let $Q_1 \in CZ^\circ$ satisfy $\delta_{Q_1} \geq c$ (such a cube exists by (\ref{g2}) -- recall $\delta_{Q^\circ} = 1$). Then, using \eqref{cub1} and \eqref{cub2} with $R_1 = Q_1$, $R_2 = Q^\circ$,
\begin{align}
   \|\tfr  - & P_0\|_{L^p(Q^\circ)} \leq \|\tfr - R_{x_1} \|_{L^p(Q^\circ)}  +\|R_{x_1} -P_0 \|_{L^p(Q^\circ)} \nonumber  \\
    &\lesssim \|\tfr\|_{L^{m,p}(Q^\circ)} + \| \tfr -R_{x_1}\|_{L^p(Q_1)} +\|R_{x_1} -P_0 \|_{L^p(Q_1)} \nonumber \\
    &\lesssim \|\tfr\|_{L^{m,p}(Q^\circ)} + \|\tfr - R_{x_1}\|_{L^p(Q_1)}/\delta_{Q_1}^m +\|R_{x_1} -P_0 \|_{L^p(Q_1)} \nonumber \\
    &\lesssim \|\tfr, \vec{R}\|_{\J_*(f,\mu;Q^\circ, CZ^\circ)} + \| R_{x_1} -P_0 \|_{L^p(Q_1)}. \label{f1}
\end{align}

Combining \eqref{f4}, \eqref{f5}, \eqref{f6}, and \eqref{f1}, we have
\begin{align}
\|T(f,P_0),P_0\|_{\J(f, \mu; \delta_{Q^\circ})} &\lesssim \|f,\vec{R}, \vec{R}^* \|_{\J_*(\mu; Q^{\circ}, CZ^\circ;K_p)} + \| R_{x_1} -P_0 \|_{L^p(Q_1)} \nonumber\\
&\lesssim \|f, P_0\|_{\J(\mu;\delta_{Q^{\circ}})} + \| R_{x_1} -P_0 \|_{L^p(Q_1)}.
\label{f0}
\end{align}

From Corollary \ref{acohkptcor}, for $\eta>0$ there exist $H \in L^{m,p}(Q^\circ)$ and $\vec{P} \in Wh(\bpt)$ satisfying $J_xH = R_x^*$ for all $x \in K_p$,  $\vec{P}$ is coherent with $P_0$, and 
\begin{align}
    \|H,\vec{P}\|_{\J_*(f,\mu;Q^\circ,CZ^\circ)} \leq \|f,\vec{P}, \vec{R}^*\|_{\J_*(\mu;Q^\circ,CZ^\circ;K_p)} +\eta/2 \lesssim \|f,P_0\|_{\J(\mu; \delta_{Q^\circ})} +\eta. \label{hchoice2}
\end{align} 
Consequently, since $\vec{R}^*$ is coherent with $P_0$, also $H$ is $K_p$-coherent with $P_0$.

We now work to estimate the term $\| R_{x_1} - P_0 \|_{L^p(Q_1)}$ in \eqref{f0}.

\emph{Case I:} First suppose that $\ka(x_1) = x \in K_p$. Recalling the definition of $\vec{R} = (R_{x_i})_{i \in I}$ in (\ref{vecr}), we have $R_{x_1}=R_x^* = J_xH$. Due to the fact that $\delta_{Q_1} \simeq 1$, we have $\| P \|_{L^p(Q_1)} \simeq | P |_{x_1,\delta_{Q_1}}$ for any polynomial $P \in \mathcal{P}$. Also, because $\vec{R}^*$ is coherent with $P_0$, and $\ma$ is monotonic, we have $\p^\al(R_x^*-P_0) \equiv  0$ for all $\al \in \ma$. By these remarks,  (\ref{locest2}), and the triangle inequality, we have
\begin{align*}
    \| R_{x_1} - P_0 \|_{L^p(Q_1)}^p  &=\|R_x^* - P_0\|_{L^p(Q_1)}^p \nonumber \simeq |R_x^* -P_0|_{x_1,\delta_{Q_1}}^p \overset{(\ref{locest2})}{\simeq}  \|0,R_x^* -P_0\|_{\J(\mu|_{9Q_1};\delta_{Q_1})}^p    \\
    &\;\; \leq \|f,R_x^* \|_{\J(\mu;\delta_{Q^\circ})}^p + \|f,P_0\|_{\J(\mu;\delta_{Q^\circ})}^p.
\end{align*}
Continuing by using (\ref{j3}), (\ref{SobLp}), and (\ref{hchoice2}), we obtain
\begin{align*}
\|f,R_x^*  \|_{\J(\mu;\delta_{Q^\circ})}^p
     \overset{(\ref{j3})}{\lesssim} \|f,R_x^* \|_{\J_*(\mu;Q^\circ)}^p   & \leq \|H,J_xH \|_{\J_*(f,\mu;Q^\circ)}^p  \\
     & \;\; = \|H\|_{L^{m,p}(Q^\circ)}^p + \| H - f \|_{L^p(d \mu)}^p +  \| H - J_x H \|_{L^p(Q^\circ)}^p   \\
    &\overset{(\ref{SobLp})}{\lesssim} \|H\|_{L^{m,p}(Q^\circ)}^p + \| H - f \|_{L^p(d \mu)}^p  \\
    &\; \; \leq \|H,\vec{P}\|_{\J_*(f,\mu;Q^\circ, CZ^\circ)}^p    \\
    &\overset{(\ref{hchoice2})}{\lesssim} (\|f,P_0\|_{\J(\mu;\delta_{Q^\circ})}+\eta)^p.
\end{align*}
Combining the previous two lines, and letting $\eta \to 0$, we have
\begin{align}
    \| R_{x_1} - P_0 \|_{L^p(Q_1)}  \lesssim \|f,P_0\|_{\J(\mu;\delta_{Q^\circ})}  \qquad (\mbox{if } \ka(x_1) \in K_p) \label{f2}.
\end{align}

\emph{Case II:} Next suppose  that $\ka(x_1) = x_s \in \bkpt$. Recalling the definition of $\vec{R} = (R_{x_i})_{i \in I}$ in (\ref{vecr}), we have $R_{x_1} = R_{x_s}'$, where $Q_s$ is a keystone cube, and $\delta_{Q_s} \lesssim \delta_{Q_1} \simeq 1$. 
We have
\begin{align}
    \| R_{x_1} -P_0 \|_{L^p(Q_1)} &= \|R_{x_s}'-P_0\|_{L^p(Q_1)} \leq \|R_{x_s}'-P_{x_s}\|_{L^p(Q_1)} +\|P_{x_s}-P_0\|_{L^p(Q_1)}. \label{r1}
\end{align}

In what follows, we bound the two terms on the right-hand side of \eqref{r1}.

We first look to bound the term $\|P_{x_s}-P_0\|_{L^p(Q_1)}$ on the right-hand side of \eqref{r1}. Because $\vec{P}$ is coherent with $P_0$, in particular, $P_{x_s}$ is coherent with $P_0$. Hence, because $\ma$ is monotonic, $\p^\al(P_{x_s}-P_0) \equiv 0$ for all $\al \in \ma$. We apply (\ref{locest2}), the triangle inequality, and $\delta_{Q_1} \simeq \delta_{Q^\circ}=1$ to deduce
\begin{align}
    \|P_{x_s}-P_0\|_{L^p(Q_1)} &\lesssim |P_{x_s} -P_0|_{x_1, \delta_{Q_1}}\lesssim \|0, P_{x_s} - P_0\|_{\J(\mu|_{9Q_1}; \delta_{Q_1})} \nonumber \\
    &\leq \|f, P_{x_s}\|_{\J(\mu|_{9 Q_1}; \delta_{Q_1})}+\|f, P_0\|_{\J(\mu|_{9Q_1}; \delta_{Q_1})} \nonumber \\
    &\lesssim \|f, P_{x_s}\|_{\J(\mu|_{9Q_1}; \delta_{Q_1})}+\|f, P_0\|_{\J(\mu; \delta_{Q^\circ})}. \label{r2}
\end{align} 
Applying (\ref{j8}) (note that $9Q_1 \cap Q^\circ$ is $C$-non-degenerate), 
\begin{align}
   \|f,  P_{x_s}\|_{\J(\mu|_{9Q_1}; \delta_{Q_1})}^p &\lesssim  \|f, P_{x_s}\|_{\J_*(\mu; 9{Q_1}\cap Q^\circ)}^p \nonumber \\
   &\leq \|H, P_{x_s}\|_{\J_*(f, \mu; 9{Q_1}\cap Q^\circ)}^p \nonumber \\
   &=\|H\|_{L^{m,p}(9Q_1\cap Q^\circ)}^p + \int_{9Q_1} |H-f|^p d\mu + \|H-P_{x_s}\|_{L^p(9Q_1\cap Q^\circ)}^p/\delta_{Q_1}^{mp}. \label{r3}
\end{align}
By applying \eqref{cub2}, with $R_1 = Q_s$ and $R_2 = 9Q_1 \cap Q^\circ$, we have
\begin{equation}
    \|H-P_{x_s}\|_{L^p(9Q_1\cap Q^\circ)}/\delta_{Q_1}^{m} \lesssim \|H\|_{L^{m,p}(9Q_1\cap Q^\circ)} + \|H - P_{x_s}\|_{L^p(Q_s)}/\delta_{Q_s}^m. \label{r4}
\end{equation}

Substituting (\ref{r4}) into (\ref{r3}), and using \eqref{hchoice2}, we see
\begin{align*}
    \|f, P_{x_s}\|_{\J(\mu|_{9Q_1}; \delta_{Q_1})}^p &\lesssim \| H \|_{L^{m,p}(Q^\circ)}^p + \| H  - P_{x_s} \|_{L^p(Q_s)}^p/\delta_{Q_s}^{mp} + \int_{9Q_1} |H-f|^p d \mu \\
    &\lesssim  \|H,\vec{P}\|_{\J_*(f,\mu;Q^\circ, CZ^\circ)}^p \overset{\eqref{hchoice2}}{\lesssim} (\|f,P_0\|_{\J(\mu; \delta_{Q^\circ})}+\eta)^p.
\end{align*}
Substituting the previous equation into (\ref{r2}), we have
\begin{align}
    \|P_{x_s}-P_0\|_{L^p(Q_1)} \lesssim \|f,P_0\|_{\J(\mu; \delta_{Q^\circ})} +\eta. \label{r5}
\end{align}

We next look to bound the term $\|R_{x_s}'-P_{x_s}\|_{L^p(Q_1)}$ on the right-hand side of \eqref{r1}.  Because $R_{x_s}'$ and $P_{x_s}$ are each coherent with $P_0$, and because $\ma$ is monotonic, $\p^\al(R_{x_s}'-P_{x_s}) \equiv 0$ for all $\al \in \ma$. Using $\delta_{Q_s} \lesssim \delta_{Q_1} \simeq 1$, $|x_s-x_1| \lesssim  \delta_{Q_1}$, \eqref{movept} and \eqref{normdi}, (\ref{locest2}), and the triangle inequality, we deduce
\begin{align}
    \|R_{x_s}'-P_{x_s}\|_{L^p(Q_1)} \simeq |R_{x_s}'-P_{x_s}|_{x_1,\delta_{Q_1}} 
    &\lesssim |R_{x_s}'-P_{x_s}|_{x_s,\delta_{Q_s}} \nonumber \\ 
    &\lesssim \|0, R_{x_s}'-P_{x_s}\|_{\J(\mu|_{9Q_s};\delta_{Q_s})} \nonumber \\
    &\leq \|f,R_{x_s}'\|_{\J(\mu|_{9Q_s};\delta_{Q_s})}+\|f,P_{x_s}\|_{\J(\mu|_{9Q_s};\delta_{Q_s})}  \label{r6}
\end{align}
Recalling our choice of $R_{x_s}'$ in Lemma \ref{alocjetlm} (because $P_{x_s}$ is assumed to be coherent with $P_0$), we must have
\begin{align}
\|f,R_{x_s}'\|_{\J(\mu|_{9Q_s};\delta_{Q_s})} \lesssim \|f,P_{x_s}\|_{\J(\mu|_{9Q_s};\delta_{Q_s})}. \label{r0}  
\end{align}
Applying (\ref{j8}) (note: $10Q_s \cap Q^\circ$ is $C$-non-degenerate), (\ref{cub2}) with $R_1 = Q_s$ and $R_2 = 10 Q_s \cap Q^\circ$, and (\ref{hchoice2}), we have 
\begin{align}
    \|f, P_{x_s}\|_{\J(\mu|_{9Q_s}  ;\delta_{Q_s})}^p  &\overset{(\ref{j8})}{\lesssim} \|f,P_{x_s}\|_{\J_*(\mu|_{9Q_s};10Q_s \cap Q^\circ)}^p \leq \|H, P_{x_s}\|_{\J_*(f, \mu|_{9Q_s}; 10Q_s \cap Q^\circ)}^p \nonumber \\
    & \; \; \lesssim \| H \|_{L^{m,p}(10Q_s \cap Q^\circ)}^p + \int_{10 Q_s} |H - f|^p d \mu + \| H - P_{x_s} \|_{L^p(10 Q_s \cap Q^\circ)}^p/\delta_{Q_s}^{m,p} \nonumber\\
    &\overset{(\ref{cub2})}{\lesssim} \| H \|_{L^{m,p}(10Q_s \cap Q^\circ)}^p + \int_{10 Q_s} |H - f|^p d \mu + \| H - P_{x_s} \|_{L^p{ (Q_s })}^p/\delta_{Q_s}^{m,p} \nonumber \\
    &\;\;\leq \|H,\vec{P}\|_{\J_*(f,\mu;Q^\circ,CZ^\circ)} \nonumber \\
    &\overset{(\ref{hchoice2})}{\lesssim} \|f, P_0\|_{\J(\mu; \delta_{Q^\circ})}+\eta. \label{r7}
\end{align}
Substituting (\ref{r5}), \eqref{r6}, (\ref{r0}), and (\ref{r7}) into (\ref{r1}), and letting $\eta \to 0$ in (\ref{hchoice2}), we have 
\begin{equation} \label{f3}
     \|R_{x_1}-P_0\|_{L^p(Q_1)} \lesssim \|f, P_0\|_{\J(\mu; \delta_{Q^\circ})} \qquad (\mbox{if } \ka(x_1) \in \bkpt).
\end{equation}
That completes the analysis of Cases I and II.

Combining \eqref{f2} and \eqref{f3}, since either $\ka(x_1) \in \bkpt$ or $\ka(x_1) \in K_p$, we have shown
\begin{align}
     \|R_{x_1}-P_0\|_{L^p(Q_1)} \lesssim \|f, P_0\|_{\J(\mu; \delta_{Q^\circ})}. \label{r8}
\end{align}
Using this in (\ref{f0}), we conclude that
\begin{align}\label{r9}
   \|T(f,P_0),P_0\|_{\J(f, \mu; 1/10)} \simeq  \|T(f,P_0),P_0\|_{\J(f, \mu; \delta_{Q^\circ})} \lesssim \|f, P_0\|_{\J(\mu; \delta_{Q^\circ})}\simeq \|f, P_0\|_{\J(\mu; 1/10)},
\end{align}
where we have used that $\delta_{Q^\circ} = 1 \simeq 1/10$. Thus, we have proven the upper bound $\|T(f,P_0), P_0\|_{\J(f, \mu; \delta)} \leq C \cdot \|f,P_0\|_{\J(\mu; \delta)}$ in (\ref{aet1}), for $\delta=1/10$. The matching lower bound is an immediate consequence of the definition of the $\J(\cdots)$ functional.

We now prepare to approximate the quantity $\| f, P_0 \|_{\J(\mu; \delta_{Q^\circ})}$. From \eqref{f4} and \eqref{r8} we have
\[
\|f,\vec{R}, \vec{R}^* \|^p_{\J_*(\mu; Q^{\circ}, CZ^\circ;K_p)} + \| R_{x_1} - P_0 \|_{L^p(Q_1)}^p \lesssim \|f, P_0\|_{\J(\mu;\delta_{Q^{\circ}})}^p.
\]
Furthermore, from \eqref{f0} we have
\begin{align*}
\|f, P_0\|_{\J(\mu;\delta_{Q^{\circ}})}^p &\leq  \| T(f,P_0), P_0 \|_{\J(f,\mu; \delta_{Q^\circ})}^p \lesssim  \|f,\vec{R}, \vec{R}^* \|^p_{\J_*(\mu; Q^{\circ}, CZ^\circ;K_p)} + \| R_{x_1} -P_0 \|_{L^p(Q_1)}^p.
\end{align*}
Combining the above two inequalities, and using $\| R_{x_1} - P_0 \|_{L^p(Q_1)} \simeq | R_{x_1} - P_0 |_{x_1,\delta_{Q_1}}$, we have
\begin{equation}\label{r10}
\|f, P_0\|_{\J(\mu;\delta_{Q^{\circ}})}^p \simeq  \|f,\vec{R}, \vec{R}^* \|^p_{\J_*(\mu; Q^{\circ}, CZ^\circ;K_p)} + | R_{x_1} - P_0 |_{x_1,\delta_{Q_1}}^p. 
\end{equation}

In Lemma \ref{adecomplm} we constructed an equivalent expression $\mathcal{S}_1(f,\vec{R},\vec{R}^*)$ for $\|f,\vec{R}, \vec{R}^* \|^p_{\J_*(\mu; Q^{\circ}, CZ^\circ;K_p)}$. Namely, we showed that $\|f,\vec{R}, \vec{R}^* \|^p_{\J_*(\mu; Q^{\circ}, CZ^\circ;K_p)} \simeq \mathcal{S}_1(f,\vec{R},\vec{R}^*)$, with
\begin{align*}
    &\mathcal{S}_1(f,\vec{R},\vec{R}^*) = \sum_{i \in I} \|f,R_{x_i}\|_{\J(\mu|_{1.1Q_i}; \delta_{Q_i})}^p + \sum_{i \lra i'} |R_{x_i} - R_{x_{i'}}|_i^p    + \int_{K_p}|R^*_x(x) -f(x)|^p d\mu  + 1_{\ma = \emptyset}  \|\vec{R^*}\|_{L^{m,p}(K_p)}^p.
\end{align*}
The indicator term $1_{\ma = \emptyset} \|\vec{R^*}\|_{L^{m,p}(K_p)}^p$ is included in the expression $\mathcal{S}_1(f,\vec{R},\vec{R}^*)$ if and only if $\ma = \emptyset$.

Next we apply (\ref{amf1}) and (\ref{amf2}) to replace the $\Sigma_{i \in I}$ sum in $\mathcal{S}_1$ by an equivalent expression, involving the countable collections of Borel sets $\{A_\ell^i\}_{\ell \in \N}, A_\ell^i \subset \supp(\mu|_{1.1Q_i})$, and of linear maps $\{\phi_\ell^i: \J(\mu|_{1.1 Q_i} ; \delta_{Q_i}) \to \R\}_{\ell \in \N}$, and $\{\lambda_\ell^i: \J(\mu|_{1.1 Q_i} ; \delta_{Q_i}) \to L^p(d\mu)\}_{\ell \in \N}$ (see (\ref{amf1}) and (\ref{amf2})). We also use the definition of the $| \cdot |_i$ polynomial norm to show that
\begin{align}
    &\|f,\vec{R}, \vec{R}^* \|^p_{\J_*(\mu; Q^{\circ}, CZ^\circ;K_p)} \simeq \widehat{\mathcal{S}}_1(f,\vec{R},\vec{R}^*), \nonumber \\
    & \widehat{\mathcal{S}}_1(f,\vec{R},\vec{R}^*) := \sum_{i \in I} \sum_{\ell \in \N} \big( \int_{A_\ell^i} |\lambda_\ell^i(f,R_{x_i})-f|^p d\mu + |\phi_\ell^i(f,R_{x_i})|^p \big)+ \sum_{i \lra i'} \sum_{\al \in \mm} c_{\al,i,i'} | \partial^\al(R_{x_i} - R_{x_{i'}})(x_i)|^p \nonumber  \\
    & \qquad\qquad\qquad + \int_{K_p}|R^*_x(x) -f(x)|^p d\mu  + 1_{\ma = \emptyset} \cdot \|\vec{R^*}\|_{L^{m,p}(K_p)}^p,\label{amffin}
\end{align}
for some constants $c_{\al,i,i'} \geq 0$. 

Recall that the Whitney fields $\vec{R} = (R_{x_i})$ and $\vec{R}^* = (R^*_x)$ depend linearly on $(f,P_0)$. Thus, for $i \lra i'$, we can define the linear functional $\phi^{i,i',\al}:\J(\mu;\delta_{Q^\circ}) \to \R$ by 
\begin{equation}
    \label{phi_func:defn}
\phi^{i,i',\al}(f,P_0):= c_{\al,i,i'}^{1/p}  \cdot \p^\al(R_{x_i} - R_{x_{i'}})(x_i).
\end{equation}
Also,  define the map $\lambda': \J(\mu; \delta_{Q^\circ}) \to L^p(d\mu)$ by
\begin{align}
    \lambda'(f,P_0)(x) = \begin{cases} R_x^*(x) & x \in K_p \\
    0 & x \in \R^n \setminus K_p.
    \end{cases} \label{lamp}
\end{align}   
We have  
\[
\int_{K_p}|R^*_x(x) -f(x)|^p d\mu = \int_{K_p} | \lambda'(f,P_0)-f|^p d \mu. 
\]
Making substitutions of these maps in equation (\ref{amffin}), we have
\begin{align}
\|f,\vec{R}, \vec{R}^*\|_{\J_*(\mu; Q^{\circ}, CZ^\circ;K_p)}^p &\simeq 1_{\ma = \emptyset} \cdot \|\vec{R^*}\|_{L^{m,p}(K_p)}^p +  \sum_{i \in I}  \sum_{\ell \in \N} \big( \int_{A_\ell^i} |\lambda_\ell^i(f,R_{x_i})-f|^p d\mu + |\phi_\ell^i(f,R_{x_i})|^p \big) \nonumber\\
& \qquad\qquad\qquad + \sum_{i \lra i'} \sum_{\al \in \mm} |\phi^{i,i',\al}(f,P_0)|^p + \int_{K_p}|\lambda'(f,P_0)-f|^p d\mu.
\label{amffin2}
\end{align}

Returning to \eqref{r10}, we write
\begin{align}
|R_{x_1} - P_0|_{x_1,\delta_{Q_1}}^p = \sum_{\al \in \mm} | \zeta_\al(f,P_0)|^p, \label{zeta}
\end{align}
for  linear functionals $\zeta_\al : \J(\mu; \delta_{Q^\circ}) \rightarrow \R$ of the form
\begin{equation}\label{zeta2}
    \zeta_\al(f,P_0):= c_\al (\p^\al (R_{x_1} -P_0)(x_1)) \qquad (\al \in \mm).
\end{equation}

Combining \eqref{r10}, \eqref{amffin2}, and \eqref{zeta}, and reindexing the sums,
\begin{align}
\|f,P_0\|_{\J(\mu; \delta_{Q^{\circ}})}^p & \simeq  \|f,\vec{R}, \vec{R}^* \|^p_{\J_*(\mu; Q^{\circ}, CZ^\circ;K_p)} + | R_{x_1} - P_0 |_{x_1,\delta_{Q_1}}^p \nonumber \\
&\simeq 1_{\ma = \emptyset} \cdot \|\vec{R^*}\|_{L^{m,p}(K_p)}^p +  \sum_{i \in I}  \sum_{\ell \in \N} \big( \int_{A_\ell^i} |\lambda_\ell^i(f,R_{x_i})-f|^p d\mu + |\phi_\ell^i(f,R_{x_i})|^p \big) \nonumber\\
& \qquad + \sum_{i \lra i'} \sum_{\al \in \mm} |\phi^{i,i',\al}(f,P_0)|^p + \int_{K_p}|\lambda'(f,P_0)-f|^p d\mu + \sum_{\al \in \mm} | \zeta_\al(f,P_0)|^p  \nonumber\\
&= 1_{\ma = \emptyset} \cdot \|\vec{R^*}\|_{L^{m,p}(K_p)}^p + \sum_{\ell \in \N} \int_{A_\ell} |\lambda_\ell(f,P_0)-f|^p d\mu + \sum_{\ell \in \N} | \phi_\ell(f,P_0)|^p, 
\label{amffin3}
\end{align}
for certain Borel sets $A_\ell \subset \supp(\mu)$ and linear maps $\phi_\ell: \J(\mu;\delta_{Q^\circ}) \to \R$, $\lambda_\ell: \J(\mu;\delta_{Q^\circ}) \to L^p(d\mu)$ ($\ell \in \mathbb{N}$). 

We define the map $M: \J(\mu; \delta_{Q^\circ}) \to \R_+$ as the $(1/p)^{th}$ power of the right hand side of (\ref{amffin3}). By the above, and the fact that $\|f,P_0\|_{\J(\mu; \delta_{Q^{\circ}})} \simeq \|f,P_0\|_{\J(\mu; 1/10)}$ and \\ $\| T(f,P_0),P_0 \|_{\J(f,\mu; \delta_{Q^\circ})} \simeq \|T(f,P_0),P_0 \|_{\J(f,\mu; 1/10)}$, we have
\[
M(f,P_0) \simeq  \|f,P_0\|_{\J(\mu; 1/10)} \simeq \| T(f,P_0),P_0 \|_{\J(f,\mu; 1/10)},
\]
establishing \eqref{aet2}.

Next we show that $M$ has the form \eqref{aet3}.

First suppose $\ma = \emptyset$. Then, by Lemma \ref{kptlm}, if $H \in L^{m,p}(\R^n)$ and $ \|H\|_{\J(f,\mu)} < \infty$, then $\vec{R}^* = (R_x^*)_{x \in K_p} = (J_xH)_{x \in K_p}$ -- indeed, any $H$ is $K_p$-coherent with $P_0$, vacuously, if $\ma = \emptyset$. This remark shows that $\vec{R}^*(f,P_0) = \vec{R}^*(f)$ depends only on $f$ when $\ma = \emptyset$. By definition of $M(f,P_0)$ via \eqref{amffin3},
\[
M(f,P_0) = \left(\| \vec{R}^*(f) \|_{L^{m,p}(K_p)}^p + \sum_{\ell \in \N} \int_{A_\ell}  |\lambda_\ell(f,P_0)-f|^p d\mu + \sum_{\ell \in \N} | \phi_\ell(f,P_0)|^p \right)^{1/p}.
\]

Therefore,  $M(f,P_0)$ has the form \eqref{aet3}, with $K:=K_p$, and with $\vec{S}(f):=\vec{R}^*(f)$ depending linearly on $f$.

If $\ma \neq \emptyset$ then $M(f,P_0) = \left(\sum_{\ell \in \N} \int_{A_\ell} |\lambda_\ell(f,P_0)-f|^p d\mu + \sum_{\ell \in \N} | \phi_\ell(f,P_0)|^p \right)^{1/p}$ via \eqref{amffin3}, so $M$ has the form \eqref{aet3} with $K = \emptyset$. In this case, the map $M$ has the form in (\ref{aextthm3plus}).

Thus, the maps $T$ and $M$ satisfy \eqref{aet1}--\eqref{aet3},  while $M$ has the form \eqref{aextthm3plus} provided $\ma \neq \emptyset$, as desired.

\end{proof}

In the previous lemma, we established the main properties of $M$ and $T$ in the Extension Theorem for $(\mu,\delta=1/10)$. Next, we establish properties of the Whitney field $ \vec{R}^*(f,P_0) \in Wh(K_p)$. In particular, the next lemma is a tool used in the proof that the map $T$ is $\Omega'$-constructible.

\begin{lem}\label{Rstarlm}
For $x \in K_p$, the linear map $(f,P_0) \in \J(\mu;\delta_{Q^\circ}) \mapsto R_x^*(f,P_0) \in \mpp$ defined in Lemma \ref{kptlm} has the form
\[
R_x^*(f,P_0) = \sum_{\al \in \mm} \omega^\al_{x}(f) \cdot v_{\al} + \ot_x(P_0),
\]
where $\ot_x : \mpp \rightarrow \mpp$ is a linear map, $\{v_\al\}_{\al \in \mm}$ is a basis for $\mpp$, and $\omega^\al_{x}: \J(\mu) \rightarrow \R$ ($\al \in \mm$) are linear functionals satisfying $\supp(\omega^\al_{x}) \subset \{x\}$.
\end{lem}

\begin{proof}
Fix $x \in K_p$. We write $R^*_x(f,P_0) = \lambda_x(f) + \ot_x(P_0)$ for linear maps $\lambda_x : \J(\mu) \rightarrow \mpp$ and $\ot_x : \mpp \rightarrow \mpp$. Fix a basis $\{ v_\al\}_{\al \in \mm}$ for $\mpp$, and write
\[
\lambda_x(f) = \sum_{\al \in \mm} \omega^\al_{x}(f) \cdot v_\al
\]
for linear functionals $\omega^\al_{x} \in \J(\mu)^*$. To demonstrate that $\supp( \omega^\al_{x}) \subset \{x\}$, we will show for any open neighborhood $U \ni x$ and any $f_1, f_2 \in \J(\mu; \delta_{Q^\circ})$ satisfying
\begin{align}
    f_1|_{U} = f_2|_{U} \label{ball},
\end{align}
we have $\lambda_{x,\al}(f) = \lambda_{x,\al}(f)$ ($\al \in \mm$). Equivalently, it suffices to show that for any $f_1,f_2$ as in \eqref{ball}, we have
\[
    R_x^*(f_1,P_0) = R_x^*(f_2,P_0).
\]
Because of (\ref{ball}), there exists $\eta>0$ such that $f_1|_{B(x,\eta)}=f_2|_{B(x,\eta)}$. Set $\mu_{x, \eta} = \mu|_{B(x,\eta)}$.  Thus, for any $H \in L^{m,p}(\R^n)$, $\| H \|_{\J(f_1,\mu_{x,\eta})} = \| H \|_{\J(f_2, \mu_{x,\eta})}$. Due to Proposition \ref{acohlm}, there exists $H \in L^{m,p}(\R^n)$ such that $\| H \|_{\J(f_1,\mu)} < \infty$ and $H$ is $K_p$-coherent with $P_0$. According to the above, $\| H \|_{\J(f_j,\mu_{x,\eta})} < \infty$ for $j=1,2$. We apply Lemma \ref{kptlm_loc} to deduce that $R^*_x(f_1,P_0) = J_x H = R^*_x(f_2,P_0)$, concluding the proof. 
\end{proof}

At the end of the proof of Lemma \ref{localpflm}, we showed that if $\ma = \emptyset$ then $\vec{R}^*(f,P_0) = \vec{R}^*(f)$ is independent of $P_0$. Therefore, we have:

\begin{cor}
For $\ma = \emptyset$, $\vec{R}^*(f,P_0) = \vec{R}^*(f)$, and consequently, for $x \in K_p$,
\[
R_x^*(f) = \sum_{\al \in \mm} \omega^\al_{x}(f) \cdot v_{\al},
\]
where $\{v_\al\}_{\al \in \mm}$ is any basis for $\mpp$, and $\omega^\al_{x}: \J(\mu) \rightarrow \R$ ($\al \in \mm$) are linear functionals satisfying $\supp(\omega^\al_{x}) \subset \{x\}$. \label{rxcor}
\end{cor}

In the next two lemmas, we verify that the maps $T$ and $M$ constructed in Lemma \ref{localpflm} are $\Omega'$-constructible, as claimed in the Extension Theorem for $(\mu,\delta)$, where $\Omega'$ is a set of linear functionals on $\J(\mu)$ whose supports have bounded overlap.

\begin{lem}
The map $T$ in Lemma \ref{localpflm} is $\Omega'$-constructible: There exists a collection of linear functionals $\Omega'=\{\omega_{t}\}_{t \in \Upsilon} \subset \J(\mu)^*$ satisfying the collection of sets $\{\supp(\omega_t)\}_{t \in \Upsilon}$ has $C$-bounded overlap, and for each $y \in \R^n$, there exists a finite subset $\Upsilon_y \subset \Upsilon$ and a collection of polynomials $\{v_{t,y} \}_{t \in \Upsilon_y} \subset \mpp$ such that $|\Upsilon_y| \leq C$ and 
\begin{align}\label{Tconst}
J_y T(f,P_0) =  \sum_{t \in \Upsilon_y}  \omega_{t} (f) \cdot v_{t,y} +  \ot_y (P_0),
\end{align}
where $\ot_y: \mpp \to \mpp$ is a linear map. \label{tconstlm}
\end{lem}

\begin{proof}
Any functional $\oh \in \J(\mu;\delta)^*$ admits a unique decomposition $\oh = \omega + \widetilde{\omega}$, where $\omega \in \J(\mu)^*$ and $\ot \in \mpp^*$.
We have defined $T:\J(\mu;\delta) \to L^{m,p}(\R^n)$ as $T(f,P_0): = \theta \cdot T(f, \vec{R}, \vec{R}^*) + (1-\theta)P_0$, where $\theta$ is a smooth function satisfying $\theta \equiv 1$ on $0.9 Q^\circ$ and $\supp(\theta) \subset Q^\circ$. Consequently,
\begin{align}\label{hhh1}
J_y T(f,P_0) = \begin{cases} P_0 &y \in (Q^\circ)^c \\
J_yT(f,\vec{R},\vec{R^*}) \odot_y J_y\theta  + (1 - J_y \theta) \odot_y P_0 & y \in Q^\circ.
\end{cases}
\end{align}
Here, $\odot_y$ is the jet product on $\mpp$ defined by $P \odot_y P' = J_y(PP')$. Recall, from (\ref{tfdef}) in the proof of Lemma \ref{adecomplm}, we defined $T(f,\vec{R},\vec{R}^*) \in L^{m,p}(Q^\circ)$ as 
\begin{align}\label{hhh2}
T(f,\vec{R},\vec{R}^*)(x) = \begin{cases} \sum_{i \in I} T_i(f,R_{x_i})(x) \cdot \theta_i(x) & x \in K_{CZ} \\
R^*_x(x) & x \in K_p \end{cases} 
\end{align}
where $\{\theta_i\}_{i \in I}$ is a partition of unity satisfying \textbf{(POU1)-(POU4)} (see Section \ref{subsec:pou}), the maps $T_i(f,R_{x_i})$ are defined just above (\ref{amf1}), and the Whitney fields $\vec{R}= (R_{x_i})_{i \in I}$, $ \vec{R}^* =(R^*_y)_{y \in K_p}$ are defined in \eqref{vecr} and Lemma \ref{kptlm}, depending linearly on $(f,P_0)$. By Lemma \ref{adecomplm},
\begin{equation}\label{hhh3}
J_y\tfr = R_y^* \mbox{ for all } y \in K_p.
\end{equation}

In Lemma \ref{Rstarlm} we describe the form of the linear maps $(f,P_0) \rightarrow R^*_y(f,P_0)$ ($y \in K_p$). There exists a basis $\{v^\gamma\}_{\gamma \in \mm}$ for $\mpp$, a linear map $\widetilde{\omega}_y : \mpp \to \mpp$, and linear functionals $\omega_y^\gamma : \J(\mu) \rightarrow \R$ satisfying  $\supp(\omega_y^\gamma) \subset \{y\}$ ($\gamma \in \mm$), such that 
\begin{equation}\label{hhh4}
R_y^*(f,P_0) = \sum_{\gamma \in \mm} \omega_y^\gamma(f) \cdot v^\gamma + \widetilde{\omega}_y(P_0) \qquad (y \in K_p).
\end{equation}

In (\ref{vecr}), we defined $\vec{R} = (R_{x_i})_{i \in I} \in Wh(\bpt)$ by
\begin{align}\label{Rdefn}
    R_{x_i} := \begin{cases} R_{x_s}' &\ka(x_i)=x_s \in \bkpt \\
    R_z^* &\ka(x_i) = z \in K_p,
    \end{cases}
\end{align}
We now discuss the form of the linear maps $(f,P_0) \mapsto R_{x_s}'(f,P_0)$ ($s \in \bar{I}$), defined in Lemma \ref{alocjetlm}.  Recall that $R_{x_s}'  \in \mathcal{P}$ depends linearly on $(f|_{9Q_s},P_0)$. Thus, given a basis $\{v^\gamma\}_{\gamma \in \mm}$ for $\mpp$, there exist linear functionals $\omega_{x_s}^\gamma :  \J(\mu|_{9 Q_s}) \rightarrow \R$ satisfying 
\begin{equation}\label{supp_cond_Omega2:eqn}
\supp(\omega_{x_s}^\gamma) \subset \supp(\mu|_{9Q_s}) \subset 10Q_s  \quad (\gamma \in \mm),
\end{equation}
and there exists a linear map $\ot_{x_s} : \mpp \rightarrow \mpp$, such that
\begin{equation}\label{hhh5}
    R_{x_s}'(f,P_0) = \sum_{\gamma \in \mm} \omega_{x_s}^\gamma(f) \cdot v^\gamma + \ot_{x_s}(P_0).
\end{equation}

We now discuss the form of the linear maps $T_i$, defined just above (\ref{amf1}). For each $i \in I$, the map $T_i$ is $\Omega_i'$-constructible, satisfying \textbf{(AL4)} for $E_i = 1.1 Q_i$ (see Section \ref{subsec:loc_ext_op}). Thus, there exists a collection of linear functionals
\begin{align}
    \Omega_i' = \{\omega^i_{t}\}_{t \in \Upsilon^i} \subset \J(\mu|_{1.1Q_i})^* \subset \J(\mu)^* \quad \mbox{(see Remark \ref{rem:J_space})}, \label{Omegai}
\end{align}
 such that the collection of sets $\{\supp(\omega^i_t)\}_{t \in \Upsilon^i}$ has $C$-bounded overlap, and by \eqref{supp_cond_Omega:eqn}, 
\begin{equation} \label{supp_cond_Omega3:eqn}
\supp(\omega^i_t) \subset \supp(\mu|_{1.1Q_i}) \subset 1.3 Q_i  \mbox{ for all } i \in I, t \in \Upsilon^i.
\end{equation}
Further, for each $y \in \R^n$, there exists a finite subset $\Upsilon^i_y \subset \Upsilon^i$ and a collection of polynomials $\{v^i_{t,y} \}_{t \in \Upsilon^i} \subset \mpp$ such that $|\Upsilon^i_y| \leq C$ and
\begin{align}
J_y T_i(f,R_{x_i}) = \sum_{t \in \Upsilon^i_y} \omega^i_{t}(f) \cdot v^i_{t,y}+\ot^i_y(R_{x_i}), \label{jety}
\end{align}
where $\ot^i_y: \mpp \to \mpp$ is a linear map.

We define collections of functionals in $\J(\mu)^*$:
\begin{align*}
    &\Omega_0: = \{ \omega_z^\gamma: z \in K_p, \; \gamma \in \mm \} \mbox{ indexed by } \Upsilon_0 := \{ (z,\gamma) : z \in K_p, \; \gamma \in \mm \};\\
    &\Omega_1: = \{ \omega^\gamma_{x_s}: x_s \in \bkpt, \; \gamma \in \mm \}  \mbox{ indexed by } \Upsilon_1 := \{ (x_s, \gamma) : x_s \in \bkpt, \; \gamma \in \mm \}; \text{ and } \\
    &\Omega_2: =  \bigcup_{i \in I} \Omega'_i = \{ \omega^i_t: i \in I, \; t \in \Upsilon^i \}  \mbox{ indexed by } \Upsilon_2 := \{ (i,t) : i \in I, \; t \in \Upsilon^i \}.
\end{align*}
We define the complete collection of functionals as
\begin{align}
    \Omega' : = \Omega_0 \cup \Omega_1 \cup \Omega_2 \subset \J(\mu)^*, \label{Om}
\end{align}
indexed by 
\begin{align*}
\Upsilon = \Upsilon_0 \cup \Upsilon_1 \cup \Upsilon_2.
\end{align*}

Because $\supp(\omega_z^\gamma) \subset \{z\}$,  $\{ \supp(\omega) : \omega \in \Omega_0\}$ has $C$-bounded overlap for $C = | \mm|$, a universal constant. 

For $x_s \in \bkpt$, $\supp(\omega_{x_s}^\gamma) \subset 10Q_s$ (see \eqref{supp_cond_Omega2:eqn}). Because of Lemma \ref{10q}, $\left| \{ s' \in \bar{I}: 10Q_s \cap 10Q_{s'} \neq \emptyset \} \right| \leq C$ for fixed $s$. Thus, the supports of the functionals in $\Omega_1$ have $C$-bounded overlap. 

From (\ref{g1}), $x \in 1.3Q$ for at most $C$ cubes $Q \in CZ^\circ$. The supports of the functionals in each collection $\Omega_i'$ are contained in $1.3Q_i$ (see \eqref{supp_cond_Omega3:eqn}), and have $C$-bounded overlap. Thus, the supports of the functionals in $\Omega_2$ have $C$-bounded overlap.

We now show that the map $T : \J(\mu; \delta_{Q^\circ}) \rightarrow L^{m,p}(\R^n)$ is $\Omega'$-constructible. We do so by verifying the identity  \eqref{Tconst}  for each $y \in \R^n$, for the set of functionals $\Omega' \subset \J(\mu)^*$ defined above.

For $y \in \R^n \setminus Q^\circ$, we have by \eqref{hhh1},
\[
J_y T(f,P_0) = P_0.
\]
Thus, the identity \eqref{Tconst} holds with $\Upsilon_y = \emptyset$ and $\ot_y(P_0) = P_0$.

For $y \in K_p$, we have $J_y \theta = 1$ (recall $K_p \subset \frac{1}{9} Q^\circ$ and $\theta \equiv 1$ on $0.9 Q^\circ$). So,  by \eqref{hhh1}, \eqref{hhh3}, \eqref{hhh4},
\begin{align*}
    J_y T(f,P_0) = R^*_y(f,P_0) =\sum_{\gamma \in \mm} \omega_y^\gamma(f) \cdot v^\gamma + \ot_y(P_0).
\end{align*} 
Thus, the identity \eqref{Tconst} holds with $\Upsilon_y = \{ (y, \gamma) : \gamma \in \mm\} \subset \Upsilon$.

Finally, suppose $y \in Q^\circ \setminus K_p$. Then $y \in Q_j \in CZ^\circ$ for a unique $j \in I$. Since $\supp(\theta_i) \subset 1.1 Q_i$, by the good geometry of the CZ cubes, we have that $y \in \supp(\theta_i)$ if and only if $i \lra j$. Thus,
\begin{align}
J_yT(f,\vec{R},\vec{R}^*) &= J_y \Big( \sum_{i : i \lra j} T_i(f,R_{x_i}) \cdot \theta_i \Big) = \sum_{i : i \lra j} J_yT_i(f,R_{x_i}) \odot_y J_y\theta_i.\label{c1}
\end{align}
If $i \in I$ is such that $\ka(x_i) = z \in K_p$, then $R_{x_i} = R_{z}^*$ by definition of $R_{x_i}$ in \eqref{Rdefn}. So, by \eqref{hhh4},
\begin{equation}\label{vecR_form1:eqn}
R_{x_i}(f,P_0) = \sum_{\gamma \in \mm} \omega_{\ka(x_i)}^\gamma(f) \cdot v^\gamma + \ot_{\ka(x_i)}(P_0) \quad (\mbox{if } \ka(x_i) \in K_p).
\end{equation}
On the other hand, if $i \in I$ is such that $\ka(x_i) =x_s \in \bkpt$, then $R_{x_i} = R'_{x_s}$ by definition of $R_{x_i}$ in \eqref{Rdefn}. So, by \eqref{hhh5}, 
\begin{equation}\label{vecR_form2:eqn}
R_{x_i}(f,P_0) = \sum_{\gamma \in \mm} \omega_{\ka(x_i)}^\gamma(f) \cdot v^\gamma + \ot_{\ka(x_i)}(P_0) \quad (\mbox{if } \ka(x_i) \in \bkpt).
\end{equation}
Using the above identities in \eqref{jety}, we have, for all $i \in I$,
\begin{align}
J_y T_i(f,R_{x_i}) &= \sum_{t \in \Upsilon_y^i} \omega^i_t(f) v^i_{t,y} +  \ot^i_y\Big(\sum_{\gamma \in \mm} \omega_{\ka(x_i)}^\gamma(f) \cdot v^\gamma + \ot_{\ka(x_i)}(P_0)\Big) \nonumber \\
    & = \sum_{t \in \Upsilon_y^i} \omega^i_t(f) v^i_{t,y} +  \sum_{\gamma \in \mm} \omega_{\ka(x_i)}^\gamma(f) \cdot \ot^i_y(v^\gamma) + \ot^i_y\big(\ot_{\ka(x_i)}(P_0)\big). \label{ri}
\end{align}

We return to the formula \eqref{c1}. We substitute (\ref{ri}) into (\ref{c1}) to write, for $y \in Q_j$,
\begin{align*}
J_yT(f,\vec{R},\vec{R}^*)
& = \sum_{i : i \lra j} \sum_{t \in \Upsilon^i_y} \omega^i_{t}(f) \cdot \left\{ v^i_{t,y}  \odot_y J_y \theta_i \right\} \\
& \qquad +\sum_{i : i \lra j} \sum_{\gamma \in \mm} \omega_{\ka(x_i)}^\gamma(f) \cdot \left\{ \ot^i_y(v^\gamma)\odot_y J_y \theta_i \right\}  +  \ot^i_y\big(\ot_{\ka(x_i)}(P_0)\big) \odot_y J_y \theta_i.
\end{align*}
Using \eqref{hhh1}, we write, for $y \in Q_j$,
\begin{align}
J_yT(f,P_0)
& = \sum_{i : i \lra j} \sum_{t \in \Upsilon^i_y} \omega^i_{t}(f) \cdot \left\{ v^i_{t,y}  \odot_y J_y \theta_i \odot_y J_y \theta  \right\} \nonumber \\
& \qquad +\sum_{i : i \lra j} \sum_{\gamma \in \mm} \omega_{\ka(x_i)}^\gamma(f) \cdot \left\{ \ot^i_y(v^\gamma)\odot_y J_y  \theta_i \odot_y J_y \theta  \right\} \nonumber \\
& \qquad +  \ot^i_y\big(\ot_{\ka(x_i)}(P_0)\big) \odot_y J_y \theta_i \odot_y J_y \theta  + (1 - J_y \theta) \odot_y P_0. \label{finalform} 
\end{align}
For $y \in Q^\circ \setminus K_p$, there is a unique $j = j(y) \in I$ so that $y \in Q_j$; we define,
\begin{align*}
   & \Upsilon_{1,y}:= \left\{(i,t): i \in I, \; i \lra j(y), \; t \in \Upsilon_y^i \right\}; \\
   &\Upsilon_{2,y} := \left\{ (z, \gamma): \ka(x_i) = z \in K_p \mbox{ for some } i \in I, \; i \lra j(y), \;  \gamma \in \mm \right\}; \text{ and } \\
   &\Upsilon_{3,y}:= \left\{ (x_s, \gamma): \ka(x_i) = x_s \in \bkpt \mbox{ for some } i \in I, \; i \lra j(y), \;  \gamma \in \mm\right\}.   
\end{align*}
Note that the last term on the right-hand side of \eqref{finalform} is just a linear map applied to $P_0$. By splitting the second sum on $i \in I$ on the right-hand side of \eqref{finalform} into cases depending on whether $\ka(x_i) = z \in K_p$ or $\ka(x_i) = x_s \in \bkpt$, and using the definitions of the index sets $\Upsilon_{r,y}$ ($r=1,2,3$), we have
\begin{align}
J_yT(f,P_0) & = \sum_{(i,t) \in \Upsilon_{1,y}} \omega^i_{t}(f) \cdot \big(v^i_{t,y} \odot_y J_y(\theta_i) \odot_y J_y (\theta)  \big) \nonumber \\
& \qquad + \sum_{(z, \gamma) \in \Upsilon_{2,y}} \omega_z^\gamma(f) \cdot  \sum_{i \in I: \; \ka(x_i) = z, i \lra j}  \ot^i_y(v_z^\gamma)\odot_y J_y(\theta_i) \odot_y J_y (\theta) 
\nonumber \\
& \qquad + \sum_{(x_s, \gamma) \in \Upsilon_{3,y}} \omega_{x_s}^\gamma(f) \cdot \sum_{i \in I: \; \ka(x_i) = x_s, i \lra j}  \ot^i_y(v^\gamma)\odot_y J_y(\theta_i)\odot_y J_y (\theta)  + \ot_y(P_0), \label{finalform2}
\end{align}
for some linear map $\ot_y: \mpp \to \mpp$.

Using \eqref{finalform2}, we see that the identity \eqref{Tconst} holds with $\Upsilon_y : = \Upsilon_{1,y} \cup \Upsilon_{2,y} \cup \Upsilon_{3,y}$. Note that
\begin{align*}
|\Upsilon_y|\leq C,    
\end{align*}
because $|\mm| = D$, $|\{i \in I: i \lra j \}| \leq C$ for fixed $j$, from (\ref{g1}), and $|\Upsilon^i_y| \leq C$ for fixed $(i,y)$.  
\end{proof}

\begin{lem}
The map $M$ defined in Lemma \ref{localpflm} is $\Omega'$-constructible, where $\Omega' = \{\omega_s\}_{s \in \Upsilon}$ is the collection defined in Lemma \ref{tconstlm}. Precisely, the objects $\lambda_\ell$, $\phi_\ell$, and $\vec{S}$ arising in the description of $M$ in \eqref{aet3}, satisfy the following: 

\textbf{(1)} For each $\ell \in \N$ and $y \in \supp(\mu)$, there exists a finite subset $\bar{\Upsilon}_{\ell,y} \subset {\Upsilon}$ and constants $\{\eta_{s,y}\}_{s \in \bar{\Upsilon}_{\ell,y}}$ such that $|\bar{\Upsilon}_{\ell,y}| \leq C$, and the map $(f,P_0) \mapsto \lambda_\ell(f,P_0)(y)$ has the form
\begin{align}
&\lambda_\ell(f,P_0)(y) = \sum_{s \in \bar{\Upsilon}_{\ell,y}} \eta_{s,y} \omega_{s}(f) + \widetilde{\lambda}_{y,\ell}(P_0), \label{m1}
\end{align}
where $\widetilde{\lambda}_{y,\ell} : \mpp \rightarrow \R$ is a linear functional. 

\textbf{(2)} For each $\ell \in \N$, there exists a finite subset $\bar{\Upsilon}_{\ell} \subset {\Upsilon}$ and constants $\{\eta_{s}\}_{s \in \bar{\Upsilon}_{\ell}}$ such that $|\bar{\Upsilon}_{\ell}| \leq C$, and the map $\phi_\ell$ has the form 
\begin{align}
&\phi_\ell(f,P_0) = \sum_{s \in \bar{\Upsilon}_{\ell}} \eta_{s} \omega_{s}(f) + \widetilde{\lambda}_{\ell}(P_0), \label{m2}
\end{align}
where $\widetilde{\lambda}_{\ell} : \mpp \rightarrow \R$ is a linear functional.

\textbf{(3)} For $y \in K$, there exist $\{\omega^\al_{y}\}_{\al \in \mm} \subset \Omega'$ satisfying for all $\al \in \mm$, $\supp(\omega^\al_{y}) \subset \{y\}$, and the map $f \mapsto S_y(f)$ has the form
\begin{align}
S_y(f) = \sum_{\al \in \mm} \omega^\al_{y}(f) \cdot v_{\al}, \label{m3}
\end{align}
where $\{v_\al\}_{\al \in \mm}$ is a basis for $\mpp$. 
\end{lem}

\begin{proof}
In the proof of Lemma \ref{localpflm}, we defined $M: \J(\mu; \delta) \rightarrow \R$ as the $(1/p)^{th}$ power of the right hand side of (\ref{amffin3}). Also, we defined $K = \emptyset$ for $\ma \neq \emptyset$, and $K = K_p$, $\vec{S}(f) = \vec{R}^*(f)$ for $\ma = \emptyset$. Then formula (\ref{m3}) follows immediately from Corollary \ref{rxcor}, in the case $\ma = \emptyset$; while formula (\ref{m3}) is vacuously true if $\ma \neq \emptyset$ (for then $K = \emptyset$).

We now investigate the form of the various terms arising on the right-hand side of \eqref{amffin3}. We shall demonstrate that each of these terms either involves a map of the form \eqref{m1} or \eqref{m2}. 

We first recall the form of the maps $\lambda_\ell^i(f,P)$ and $\phi_\ell^i(f,P)$ arising in \textbf{(AL5)} and \textbf{(AL6)} for $E_i = 1.1 Q_i$ (see Section \ref{subsec:loc_ext_op}). These maps arise in the reindexed sums on the right-hand side of \eqref{amffin3}.

By hypothesis \textbf{(AL5)}, for each $\ell \in \N$, $i \in I$, and $y \in \supp(\mu)$, there exists a finite subset $\bar{\Upsilon}^i_{\ell,y} \subset {\Upsilon}^i$ and constants $\{\eta^{\ell,i}_{s,y}\}_{s \in \bar{\Upsilon}^i_{\ell,y}} \subset \R$ such that $|\bar{\Upsilon}^i_{\ell,y}| \leq C$, and the map $(f,P) \mapsto \lambda_\ell^i(f,P)(y)$ has the form
\begin{align}
\lambda_\ell^i(f,P)(y) = \sum_{s \in \bar{\Upsilon}^i_{\ell,y}} \eta^{\ell,i}_{s,y} \cdot \omega^i_{s}(f) + \widetilde{\lambda}^i_{y,\ell}(P), \label{cm1}
\end{align}
where $\widetilde{\lambda}^i_{y,\ell} : \mpp \rightarrow \R$ is a linear functional. Also, from \textbf{(AL6)}, for each $\ell \in \N$, and $i \in I$, there exists a finite subset $\bar{\Upsilon}^i_{\ell} \subset {\Upsilon}^i$ and constants $\{\eta^{\ell,i}_{s}\}_{s \in \bar{\Upsilon}^i_{\ell}} \subset \R$ such that $|\bar{\Upsilon}^i_{\ell}| \leq C$, and the map $(f,P) \rightarrow \phi^i_\ell(f,P)$ has the form 
\begin{align}
\phi^i_\ell(f,P) = \sum_{s \in \bar{\Upsilon}^i_{\ell}} \eta_{s}^{\ell,i} \cdot \omega^i_{s}(f) + \widetilde{\lambda}^i_{\ell}(P), \label{cm2}
\end{align}
where $\widetilde{\lambda}^i_{\ell} : \mpp \rightarrow \R$ is a linear functional.

For each $i \in I$, either $\ka(x_i) \in K_p$ or $\ka(x_i) \in \bkpt$. In \eqref{vecR_form1:eqn} and \eqref{vecR_form2:eqn}, we have shown
\begin{align}
    R_{x_i}(f,P_0) = \sum_{\gamma \in \mm} \omega^\gamma_{\ka(x_i)}(f) \cdot v^\gamma + \ot_{\ka(x_i)}(P_0), \label{rid}
\end{align}
where $\{v^\gamma\}_{\gamma \in \mm}$ is a basis for $\mpp$, $ \omega_{\ka(x_i)}^\gamma$ ($\gamma \in \mm$) are linear functionals in $\Omega'$ (for the definition of $\Omega'$, see line \eqref{Om} and the containing paragraph), and where $\ot_{\ka(x_i)}: \mpp \to \mpp$ is a linear map. Consequently, 
\begin{align}
    \p^\al [R_{x_i}(f,P_0)] = \sum_{\gamma \in \mm} \omega^\gamma_{\ka(x_i)}(f) \cdot \p^\al [v^\gamma] + \p^\al[ \ot_{\ka(x_i)}(P_0)]. \label{rid2}
\end{align}

For $j \in \N$, after the reindexing in \eqref{amffin3}, a map $\lambda_j: \J(\mu; \delta) \to L^p(d\mu)$ arising in \eqref{amffin3} is either of the form $\lambda_j(f,P_0) = \lambda_\ell^i(f,R_{x_i})$ for  $i \in I$ and $\ell \in \N$, or $\lambda_j(f,P_0) =\lambda'(f,P_0)$, where $\lambda'$ is defined in (\ref{lamp}). 

In the former case, we substitute (\ref{rid}) into (\ref{cm1}) and write
\begin{align*}
\lambda_\ell^i(f,R_{x_i})(y) &=\sum_{s \in \bar{\Upsilon}^i_{\ell,y}} \eta^{\ell,i}_{s,y} \cdot \omega^i_{s}(f) + \widetilde{\lambda}^i_{y,\ell}(R_{x_i}) \\
    &=\sum_{s \in \bar{\Upsilon}^i_{\ell,y}} \eta^{\ell,i}_{s,y} \cdot \omega^i_{s}(f) + \widetilde{\lambda}^i_{y,\ell}\big(\sum_{\gamma \in \mm} \omega^\gamma_{\ka(x_i)}(f) \cdot v^\gamma + \ot_{\ka(x_i)}(P_0)\big) \\
    &=\sum_{s \in \bar{\Upsilon}^i_{\ell,y}} \eta^{\ell,i}_{s,y} \cdot \omega^i_{s}(f) + \sum_{\gamma \in \mm} \omega^\gamma_{\ka(x_i)}(f) \widetilde{\lambda}^i_{y,\ell}(v^\gamma) + \widetilde{\lambda}^i_{y,\ell}\big(\ot_{\ka(x_i)}(P_0)\big).
\end{align*}
Above, the number of terms in the two sums on the right-hand side  is at most $|\bar{\Upsilon}^i_{\ell,y}| + |\mm| \leq C$, and each of the functionals $\omega^i_{s}(f)$ and $\omega^\gamma_{\ka(x_i)}(f)$ belongs to the family $\Omega' = \{\omega_s\}_{s \in \Upsilon}$ defined in (\ref{Om}). Further, the third term on the right-hand side is a linear function of $P_0$. This proves the identity (\ref{m1}) for $\lambda_j$, provided that $\lambda_j(f,P_0) = \lambda_\ell^i(f,R_{x_i})$.

Now suppose $\lambda_j(f,P_0) =\lambda'(f,P_0)$ with $\lambda'$ defined in \eqref{lamp}. If $x \in \R^n \setminus K_p$ then $\lambda'(f,P_0)(x) = 0$. If $x \in K_p$, then from Lemma \ref{Rstarlm}, 
\begin{align*}
    \lambda'(f,P_0)(x)&= R_x^*(f,P_0)(x) \\
    &= \sum_{\al \in \mm} \omega_{x}^\al(f) \cdot v_{\al}(x) + \widetilde{\omega}_x(P_0)(x),
\end{align*}
where $\widetilde{\omega}_x: \mpp \to \mpp$ is a linear map. Above, the number of terms in the sum on the right-hand side is at most $|\mm| \leq C$, and each of the functionals $\omega^\al_{x}(f)$ belongs to the family $\Omega' = \{\omega_s\}_{s \in \Upsilon}$ defined in (\ref{Om}). Further, the second term on the right-hand side is a linear function of $P_0$. Therefore, we have succeeded in verifying the identity (\ref{m1}) for $\lambda_j$, provided that $\lambda_j(f,P_0) = \lambda'(f,P_0)$.

From (\ref{amffin3}) for $j \in \N$, a map $\phi_j: \J(\mu; \delta) \to \R$ is either of the form $\phi_j(f,P_0) = \phi_\ell^i(f,R_{x_i})$ for  $i \in I$ and $\ell \in \N$, $\phi_j(f,P_0) = \phi^{i,i',\al}(f,P_0)$ for $i \lra i'$ and $\al \in \mm$ (see \eqref{phi_func:defn}), or $\phi_j(f,P_0) = \zeta_\al(f,P_0)$ for $\al \in \mm$ (see (\ref{zeta2})). 

In the first case, if $\phi_j(f,P_0) = \phi_\ell^i(f,R_{x_i})$ for  $i \in I$ and $\ell \in \N$, we substitute (\ref{rid}) into (\ref{cm2}):
\begin{align*}
\phi_\ell^i(f,R_{x_i})   &=\sum_{s \in \bar{\Upsilon}^i_{\ell}} \eta^{\ell,i}_{s} \cdot \omega^i_{s}(f) + \widetilde{\lambda}^i_{\ell}(R_{x_i}) \\
    &=\sum_{s \in \bar{\Upsilon}^i_{\ell}} \eta^{\ell,i}_{s} \cdot \omega^i_{s}(f) + \widetilde{\lambda}^i_{\ell}\big(\sum_{\gamma \in \mm} \omega^\gamma_{\ka(x_i)}(f) \cdot v^\gamma + \ot_{\ka(x_i)}(P_0)\big) \\
    &=\sum_{s \in \bar{\Upsilon}^i_{\ell}} \eta^{\ell,i}_{s} \cdot \omega^i_{s}(f) + \sum_{\gamma \in \mm} \omega^\gamma_{\ka(x_i)}(f) \widetilde{\lambda}^i_{\ell}(v^\gamma) + \widetilde{\lambda}^i_{\ell}\big(\ot_{\ka(x_i)}(P_0)\big).
\end{align*}
In the expression above, the last term is a linear function of $P_0$, and the number of terms in both sums is at most $|\bar{\Upsilon}^i_{\ell}| + |\mm| \leq C$. Further, each of the functionals $\omega^i_{s}(f)$ and $\omega^\gamma_{\ka(x_i)}$ belongs to the family $\Omega' = \{\omega_s\}_{s \in \Upsilon}$ defined in (\ref{Om}). This proves the identity (\ref{m2}) for $\phi_j$, provided that $\phi_j(f,P_0) = \phi_\ell^i(f,R_{x_i})$.

Suppose $\phi_j(f,P_0) = \phi^{i,i',\al}(f,P_0)= c_{\al,i,i'}^{1/p} \p^\al_x(R_{x_i} - R_{x_{i'}})(x_i)$ for $i \lra i'$ and $\al \in \mm$ (see \eqref{phi_func:defn}). Using (\ref{rid2}), we have
\begin{align*}
\phi^{i,i',\al}(f,P_0)  & = c_{\al,i,i'}^{1/p}  \sum_{\gamma \in \mm} \omega^\gamma_{\ka(x_i)}(f) \cdot \p^\al_x [ v^\gamma](x_i)  - c_{\al,i,i'}^{1/p} \sum_{\gamma \in \mm} \omega^\gamma_{\ka(x_{i'})}(f) \cdot \p^\al_x [ v^\gamma](x_i)  \\
& \qquad +  c_{\al,i,i'}^{1/p} \cdot \p^\al_x [ \ot_{\ka(x_i)}(P_0) - \ot_{\ka(x_{i'})}(P_0)](x_i).
\end{align*}
In the expression above, the last term is a linear function of $P_0$, and the number of terms in both sums is at most $2 |\mm| \leq C$. Further, each of the functionals $\omega^\gamma_{\ka(x_i)}$ and $\omega^\gamma_{\ka(x_{i'})}$ belongs to the family $\Omega' = \{\omega_s\}_{s \in \Upsilon}$ defined in (\ref{Om}). This proves the identity (\ref{m2}) for $\phi_j$, provided that $\phi_j(f,P_0) = \phi^{i,i',\al}(f,P_0)$.

Finally, suppose $\phi_j(f,P_0) = \zeta_\al(f,P_0):= c_\al \p^\al (R_{x_1} -P_0)(x_1)$ for  $\al \in \mm$. Using (\ref{rid2}) at $x=x_1$, we have
\begin{align*}
    \zeta_\al(f,P_0) &=c_\al \sum_{\gamma \in \mm} \omega^\gamma_{\ka(x_i)}(f) \cdot \p^\al v^\gamma(x_1) +  c_\al \left( \p^\al \ot_{\ka(x_i)}(P_0)(x_1) -\p^\al P_0(x_1)\right).
\end{align*}
In the expression above, the last term is a linear function of $P_0$, and the number of terms in the sum is at most $ |\mm| \leq C$. Further, each of the functionals $\omega^\gamma_{\ka(x_i)}$ belongs to the family $\Omega' = \{\omega_s\}_{s \in \Upsilon}$ defined in (\ref{Om}). Therefore, we have verified the identity (\ref{m2}) for $\phi_j$, provided that $\phi_j(f,P_0) = \zeta_\al(f,P_0)$. 

This completes the proof that the map $M$ is $\Omega'$-constructible.
\end{proof}

The previous lemmas complete the inductive argument, begun in Section \ref{sec:ind_hyp}, by showing that under the assumption that the Main Lemma holds for all $\ma' < \ma$, the Main Lemma also holds for $\ma$. 

\section{Proofs of the Main Theorems}
\label{sec:proof_maintheorems}

\begin{proof}[Proofs of Theorems \ref{amainthm} and \ref{constructthm}]
Let $\mu$ be a compactly supported Borel regular measure on $\R^n$. \\
Fix $\delta> \diam(\supp(\mu))$, so that  $\supp(\mu)$ is contained in a cube of sidelength $\delta$. From (\ref{j3}) and (\ref{j4}),
\begin{align}
    \|f\|_{\J(\mu)} \simeq \inf_{P \in \mpp} \|f, P\|_{\J(\mu;\delta)}. \label{afin1}
\end{align}
Per the Extension Theorem for $(\mu, \delta)$ (Proposition \ref{aextthmeq}), there exist a linear map $T: \J(\mu;\delta) \to L^{m,p}(\R^n)$, a map $M: \J(\mu;\delta) \to \R_+$, a set $K \subset \Cl(\supp(\mu))$, a linear map $\vec{S}: \J(\mu) \to Wh(K)$, and countable families of  Borel sets $\{A_\ell\}_{\ell \in \N}, A_\ell \subset \supp(\mu)$, and linear maps $\{\phi_\ell: \J(\mu;\delta) \to \R\}_{\ell \in \mathbb{N}}$, and $\{\lambda_\ell: \J(\mu;\delta) \to L^p(d\mu)\}_{\ell \in \mathbb{N}}$, that satisfy the conclusion of the extension theorem, as described below.

For each $(f,P_0) \in \J(\mu;\delta)$,  
\begin{align}
&\textbf{(i) }\|f,P_0\|_{\J(\mu; \delta)} \leq  \|T(f,P_0), P_0\|_{\J(f, \mu; \delta)} \leq C \cdot \|f,P_0\|_{\J(\mu; \delta)};  \label{afin2} \\
&\textbf{(ii) } c \cdot M(f,P_0) \leq \|T(f,P_0), P_0\|_{\J(f, \mu; \delta)} \leq C \cdot M(f,P_0); \text{ and} \label{afin3} \\
&\textbf{(iii) }M(f,P_0) = \Big(\sum_{\ell \in \mathbb{N}} \int_{A_\ell} |\lambda_\ell(f,P_0)-f|^p d\mu+ \sum_{\ell \in \N} | \phi_\ell(f,P_0)|^p + \|\vec{S}(f)\|_{L^{m,p}(K)}^p \Big)^{1/p}. \label{afin4}
\end{align}
Hence,
\begin{align}
    \inf_{P \in \mpp} \|f, P\|_{\J(\mu; \delta)} &\simeq \inf_{P \in \mpp} M(f,P) \nonumber \\
    & = \inf_{P \in \mpp} \Big(  \sum_{\ell \in \N} \int_{A_\ell} |\lambda_\ell(f,P)-f|^p d\mu+ \sum_{\ell \in \N} | \phi_\ell(f,P)|^p + \|\vec{S}(f)\|_{L^{m,p}(K)}^p \Big)^{1/p} \label{afin5}.
\end{align}
We apply Lemma \ref{alinearmap2} (with trivial constraint map $\Psi$, i.e., with $N=0$) to compute $\xi(f) \in \mpp$, depending linearly on $f$ and satisfying:
\begin{align}
    \big(\sum_{\ell \in \N} \int_{A_\ell} |\lambda_\ell(f, \xi(f))-f|^p d\mu+& \sum_{\ell \in \N} | \phi_\ell(f,\xi(f))|^p\big)^{1/p} \nonumber \\ 
    &\lesssim \inf_{P \in \mpp} \big(\sum_{\ell \in \N} \int_{A_\ell} |\lambda_\ell(f, P)-f|^p d\mu+ \sum_{\ell \in \N} | \phi_\ell(f,P)|^p\big)^{1/p}. \label{afin6} 
\end{align}
Define $Mf:=M(f,\xi(f))$. Then
\begin{align}
    Mf &= \Big(\sum_{\ell \in \N} \int_{A_\ell} |\lambda_\ell(f, \xi(f))-f|^p d\mu + |\phi_\ell(f, \xi(f))|^p \Big)^{1/p} = \Big(\sum_{\ell \in \N} \big( \int_{A_\ell} |\zeta_{\ell}(f) -f|^p d\mu + |\psi_{\ell}(f) |^p \big) \Big)^{1/p}. \label{mfinal}
\end{align}
for linear maps $\zeta_{\ell}: \J(\mu) \to L^p(d\mu)$, and linear functionals $\psi_{\ell}: \J(\mu) \to \R$ ($\ell \in \N$), defined by $\zeta_\ell(f)= \lambda_\ell(f,\xi(f))$ and $\psi_\ell(f) = \phi_\ell(f,\xi(f))$. 

Because of (\ref{afin2})-(\ref{afin4}), $Tf:=T(f, \xi(f))$ satisfies
\begin{align}
    \|Tf, \xi(f)\|_{\J(f, \mu; \delta)} \simeq \| f, \xi(f) \|_{\J(\mu; \delta)} \simeq M(f, \xi(f))=Mf \label{afin7}.
\end{align}
Together, (\ref{afin1}) and (\ref{afin5})-(\ref{afin7}) imply that
\begin{align*}
    \|f\|_{\J(\mu)} \simeq \|Tf\|_{\J(f, \mu)} \simeq Mf,
\end{align*}
completing the proof of Theorem \ref{amainthm}.

Per the Extension Theorem for $(\mu, \delta)$, there exists a collection of linear functionals $\Omega'=\{\omega_{s}\}_{s \in \Upsilon} \subset \J(\mu)^*$ such that the collection of sets $\{\supp(\omega_s)\}_{s \in \Upsilon}$ has $C$-bounded overlap, and for each $y \in \R^n$, there exists a finite subset $\Upsilon_y \subset \Upsilon$ and a collection of polynomials $\{v_{s,y} \}_{s \in \Upsilon_y} \subset \mpp$ such that $|\Upsilon_y| \leq C$ and 
    \begin{align}
    J_y T(f,P_0) =  \sum_{s \in \Upsilon_y}  \omega_{s} (f) \cdot v_{s,y} +  \ot_y (P_0) \quad \quad \text{ for }(f,P_0) \in \J(\mu;\delta),
    \label{aconstructlm}
    \end{align}
where $\ot_y: \mpp \to \mpp$ is a linear map. 

We have defined $T(f): = T(f,\xi(f))$. Because $\xi: \J(\mu) \to \mpp$ is a linear map, we can write
\begin{align}
    \xi(f) = \sum_{\gamma \in \mm} \omega_\gamma (f) \cdot v^\gamma, \label{xi}
\end{align}
for linear functionals $\omega_\gamma : \J(\mu) \rightarrow \R$ ($\gamma \in \mm$), and where $\{v^\gamma\}_{\gamma \in \mm}$ is a basis for $\mpp$. Substituting this into (\ref{aconstructlm}), we have
\begin{align}
    J_y T(f) = J_yT(f,\xi(f)) &=  \sum_{s \in \Upsilon_y}  \omega_{s} (f) \cdot v_{s,y} +  \ot_y (\sum_{\gamma \in \mm} \omega_\gamma (f) \cdot v^\gamma) \nonumber \\
    &=  \sum_{s \in \Upsilon_y} \omega_{s} (f) \cdot v_{s,y} +   \sum_{\gamma \in \mm} \omega_\gamma (f) \cdot \ot_y(v^\gamma) \label{last}.
\end{align}
Define the collection of functionals $\Omega := \Omega' \cup \{\omega_\gamma\}_{\gamma \in \mm}$. Because $|\mm|=D$ and the collection of supports of the functionals in $\Omega'$ has $C$-bounded overlap, the collection of supports of the functionals in $\Omega$ has $C'$-bounded overlap. Because $|\Upsilon_y| \leq C$, the number of terms in both sums in \eqref{last} is at most $C$. Therefore, $ J_y T(f)$ is of the desired form in Theorem \ref{constructthm} (see \eqref{l4}). That is, we have shown that $T$ is $\Omega$-constructible.

We have defined $Mf:=M(f,\xi(f))$. Next we verify that the map $M$ is $\Omega$-constructible for the set $\Omega =  \J(\mu)^*$ defined above. We shall verify conditions \textbf{(1)--(3)} of Theorem \ref{constructthm}.

From the Extension Theorem for $(\mu,\delta)$, for $y \in K$, there exists $\{\omega^\al_{y}\}_{\al \in \mm} \subset \Omega'$ satisfying for all $\al \in \mm$, $\supp(\omega^\al_{y}) \subset \{y\}$, and the map $f \mapsto S_y(f)$ has the form
\begin{align*}
S_y(f) = \sum_{\al \in \mm} \omega^\al_{y}(f) \cdot v_{\al}, 
\end{align*}
where $\{v_\al\}_{\al \in \mm}$ is a basis for $\mpp$. This implies condition \textbf{(3)} of Theorem \ref{constructthm}, because $\Omega' \subset \Omega$. 

To prove that $M$ is $\Omega$-constructible, we must show that the maps $\zeta_{\ell}: \J(\mu) \to L^p(d\mu)$ and $\psi_{\ell}: \J(\mu) \to \R$ in (\ref{mfinal}) satisfy conditions \textbf{(1)} and \textbf{(2)} of Theorem \ref{constructthm}.

From the Extension Theorem for $(\mu, \delta)$, we have for each $\ell \in \N$ and $y \in \supp(\mu)$, there exists a finite subset $\bar{\Upsilon}_{\ell,y} \subset {\Upsilon}$ and constants $\{\eta^{\ell}_{s,y}\}_{s \in \bar{\Upsilon}_{\ell,y}} \subset \R$ such that $|\bar{\Upsilon}_{\ell,y}| \leq C$, and the map $ \lambda_\ell$ has the form
\begin{align*}
&\lambda_\ell(f,P)(y) = \sum_{s \in \bar{\Upsilon}_{\ell,y}} \eta^{\ell}_{s,y} \cdot \omega_{s}(f) + \widetilde{\lambda}_{y,\ell}(P), 
\end{align*}
where $\widetilde{\lambda}_{y,\ell} : \mpp \rightarrow \R$ is a linear functional. Using (\ref{xi}) in the previous equation, we deduce that the map $f \mapsto \zeta_\ell(f)(y)$ has the form
\begin{align*}
\zeta_\ell(f)(y)=\lambda_\ell(f,\xi(f))(y) &= \sum_{s \in \bar{\Upsilon}_{\ell,y}} \eta^{\ell}_{s,y} \cdot \omega_{s}(f) + \widetilde{\lambda}_{y,\ell}\bigg(\sum_{\gamma \in \mm} \omega_\gamma (f) \cdot v^\gamma\bigg) \\
&=\sum_{s \in \bar{\Upsilon}_{\ell,y}} \eta^{\ell}_{s,y} \cdot \omega_{s}(f) + \sum_{\gamma \in \mm} \omega_\gamma (f) \cdot \widetilde{\lambda}_{y,\ell}(v^\gamma).
\end{align*}
Therefore, $\zeta_\ell(f)$ has the form stated in condition \textbf{(1)} of Theorem \ref{constructthm}, for the set $\Omega = \Omega' \cup \{\omega_\gamma \}_{\gamma \in \mm}$.

Similarly, from the Extension Theorem for $(\mu,\delta)$, for $\ell \in \N$, there exists a finite subset $\bar{\Upsilon}_{\ell} \subset {\Upsilon}$ and constants $\{\eta^\ell_{s}\}_{s \in \bar{\Upsilon}_{\ell}} \subset \R$ such that $|\bar{\Upsilon}_{\ell}| \leq C$, and the functional $\phi_\ell$ has the form 
\begin{align*}
&\phi_\ell(f,P) = \sum_{s \in \bar{\Upsilon}_{\ell}} \eta_{s}^\ell \cdot \omega_{s}(f) + \widetilde{\lambda}_{\ell}(P), 
\end{align*}
where $\widetilde{\lambda}_{\ell} : \mpp \rightarrow \R$ is a linear functional. Again, using (\ref{xi}), we deduce that the map $\psi_\ell$ has the form
\begin{align*}
\psi_\ell(f)=\phi_\ell(f,\xi(f)) &= \sum_{s \in \bar{\Upsilon}_{\ell}} \eta_{s}^\ell \cdot \omega_{s}(f) + \widetilde{\lambda}_{\ell}\bigg(\sum_{\gamma \in \mm} \omega_\gamma (f) \cdot v^\gamma\bigg)\\
&= \sum_{s \in \bar{\Upsilon}_{\ell}} \eta_{s}^\ell \cdot \omega_{s}(f) + \sum_{\gamma \in \mm} \omega_\gamma (f) \cdot \widetilde{\lambda}_{\ell}( v^\gamma).
\end{align*}
Therefore, $\psi_\ell(f)$ has the form stated in condition \textbf{(2)} of Theorem \ref{constructthm}, for the set $\Omega = \Omega' \cup \{\omega_\gamma \}_{\gamma \in \mm}$.

This completes the proof of Theorem \ref{constructthm}.
\end{proof}

\begin{proof}[Proof of Theorem \ref{finitethm}]
Let $\mu$ be a finite Borel regular measure on $\R^n$ with compact support. By rescaling, we may assume $\supp(\mu) \subset \frac{1}{10}Q^\circ$ for $Q^\circ = (0,1]^n$. In the proof of the Extension Theorem for $(\mu,\delta=1/10)$, the Calder\'on-Zygmund decomposition $CZ^\circ$ defined in Section \ref{sec:cz_decomp} is finite and $\bigcup_{Q \in CZ^\circ} Q = Q^\circ$ (see Lemma \ref{nottrivial}). Therefore,  $K_p = \emptyset$ (see Definition \ref{keypt:def}). Consequently, $Mf:\J(\mu) \to \R$ in Theorem \ref{amainthm} has the form 
\begin{align*}
Mf = \big(\sum_{\ell = 1}^{K} \int_{A_\ell} |\zeta_\ell(f)-f|^p d\mu + |\psi_{\ell}(f) |^p \big)^{1/p}, 
\end{align*}
where for each $\ell$, $A_\ell \subset \supp(\mu)$ is a Borel set, $\zeta_\ell: \J(\mu) \to L^p(d\mu)$, and $\psi_\ell: \J(\mu) \to \R$ are linear maps, and $K$ is finite.
Again because the CZ decomposition is finite, in Theorem \ref{constructthm}, we have that the collection of linear functionals $\Omega = \{\omega_r\}_{r \in \Upsilon} \subset \J(\mu)^*$ is finite ($|\Upsilon|<\infty$).
\end{proof}

\begin{proof}[Proofs of Theorems \ref{introtracethm} and \ref{arbethm}]
Theorems  \ref{introtracethm} and \ref{arbethm} are immediate consequences of Theorems \ref{amainthm} and \ref{constructthm} applied to the Borel measure $\mu = \mu_E$ in \eqref{muE:defn}. For details, see the discussion after \eqref{muE:defn}.
\end{proof}

\newpage
\bibliographystyle{amsplain}
\bibliography{main}

\providecommand{\bysame}{\leavevmode\hbox to3em{\hrulefill}\thinspace}
\providecommand{\MR}{\relax\ifhmode\unskip\space\fi MR }
\providecommand{\MRhref}[2]{%
  \href{http://www.ams.org/mathscinet-getitem?mr=#1}{#2}
}
\providecommand{\href}[2]{#2}
\begin{thebibliography}{10}

\bibitem{bmp1}
E.~Bierstone, P.~Milman, and W.~Paw\l{}ucki, \emph{Differentiable functions
  defined in closed sets. {A} problem of {W}hitney}, Invent. Math. \textbf{151}
  (2003), no.~2, 329--352.

\bibitem{br1}
Y.~A. Brudnyi, \emph{Spaces that are definable by means of local
  approximations}, Trudy Moscov. Math Obshch \textbf{24} (1971), 69--132.

\bibitem{brbr}
Y.~A. Brudnyi and A.~Brudnyi, \emph{Methods of geometric analysis in extension
  and trace problems}, Monographs in Mathematics, vol.~1, Birkh\"auser Basel,
  2012.

\bibitem{brshv0}
Y.~A. Brudnyi and P.~Shvartsman, \emph{Generalizations of {W}hitney's extension
  theorem}, Int. Math. Research Notices \textbf{3} (1994), no.~129-139.

\bibitem{brshv}
\bysame, \emph{The {W}hitney problem of existence of a linear extension
  operator}, J. Geom. Anal. \textbf{7} (1997), no.~4, 515--574.

\bibitem{f6}
C.~Fefferman, \emph{A generalized sharp {W}hitney theorem for jets}, Rev. Mat.
  Iberoamericana \textbf{21} (2005), no.~2, 577--688.

\bibitem{f1}
\bysame, \emph{A sharp form of {W}hitney's extension theorem}, Ann. of Math.
  \textbf{161} (2005), no.~1, 509--577.

\bibitem{f2}
\bysame, \emph{{W}hitney's extension problem for {$C^m$}}, Ann. of Math.
  \textbf{164} (2006), no.~1, 313--359.

\bibitem{f3}
\bysame, \emph{{$C^m$} extension by linear operators}, Ann. of Math.
  \textbf{163} (2007), no.~3, 779--835.

\bibitem{f4}
\bysame, \emph{The structure of linear extension operators for {$C^m$}}, Rev.
  Mat. Iberoamericana \textbf{23} (2007), no.~1, 269--280.

\bibitem{f5}
\bysame, \emph{Fitting a {$C^m$}-smooth function to data {III}}, Ann. of Math.
  \textbf{170} (2009), no.~1, 427--441.

\bibitem{arie4}
C.~Fefferman, A.~Israel, and G.~K. Luli, \emph{The structure of {S}obolev
  extension operators}, Rev. Mat. Iberoamericana \textbf{30} (2012), no.~2,
  419--429.

\bibitem{arie3}
\bysame, \emph{{S}obolev extension by linear operators}, J. Amer. Math. Soc.
  \textbf{27} (2014).

\bibitem{arie6}
\bysame, \emph{Fitting a {S}obolev function to data {I}}, Rev. Mat. Iberoam.
  \textbf{32} (2016), no.~1, 273--374.

\bibitem{arie7}
\bysame, \emph{Fitting a {S}obolev function to data {II}}, Rev. Mat. Iberoam.
  \textbf{32} (2016), no.~2, 649--750.

\bibitem{arie8}
\bysame, \emph{Fitting a {S}obolev function to data {III}}, Rev. Mat. Iberoam.
  \textbf{32} (2016), no.~3, 1039--1126.

\bibitem{f7}
C.~Fefferman and B.~Klartag, \emph{Fitting a {$C^m$}-smooth function to data
  {I}}, Ann. of Math. \textbf{169} (2009), no.~1, 315--346.

\bibitem{f8}
\bysame, \emph{Fitting a {$C^m$}-smooth function to data {II}}, Rev. Mat.
  Iberoamericana \textbf{25} (2009), no.~1, 49--273.

\bibitem{gl}
G.~Glaeser, \emph{Etude de quelques algebres {T}ayloriennes}, J. Analyse Math.
  \textbf{6} (1958), 1--125.

\bibitem{arie5}
A.~Israel, \emph{A bounded linear extension operator for
  {$L^{2,p}(\mathbb{R}^2)$}}, Ann. of Math. \textbf{178} (2013), no.~1,
  183--230.

\bibitem{shv2}
P.~Shvartsman, \emph{Sobolev {$W^1_p$}-spaces on closed subsets of
  {$\mathbb{R}^n$}}, Advances in Math. \textbf{220} (2009), no.~6, 1842--1922.

\bibitem{shv4}
\bysame, \emph{On the sum of a {S}obolev space and a weighted {$L_p$}-space},
  Advances in Math. \textbf{248} (2013), 155--228.

\bibitem{shv5}
\bysame, \emph{{S}obolev {$L^2_p$} functions on closed subsets of {$\R^2$}},
  Advances in Math. \textbf{252} (2014), 22--113.

\bibitem{st}
E.~Stein, \emph{Singular integrals and differentiability properties of
  functions}, Princeton University Press, 1970.

\bibitem{wh1}
H.~Whitney, \emph{Differentiable functions defined in closed sets {I}}, Trans.
  Amer. Math Soc. (1934), no.~36, 360--389.

\end{thebibliography}

\end{document}